\documentclass[english]{scrartcl}

\usepackage[T1]{fontenc}
\usepackage[utf8]{inputenc}
\usepackage[english]{babel}
\usepackage{lmodern}
\usepackage{exscale}
\usepackage[babel]{microtype}
\usepackage{amsmath, amssymb, mathtools, mathrsfs}
\usepackage{xstring}
\usepackage{xparse}
\usepackage{graphicx}
\usepackage{xcolor}
\usepackage{array, booktabs}
\usepackage[inline]{enumitem}
\usepackage{subcaption}

\usepackage{imakeidx}
\makeindex

\usepackage{csquotes}
\usepackage[style=alphabetic, sorting=nyt, maxnames=10, maxalphanames=5]{biblatex}

\usepackage{tikz}
\usepackage{tikz-cd}
\usetikzlibrary{3d,arrows,calc,decorations,decorations.pathmorphing,decorations.pathreplacing,matrix,positioning,shapes,through,fit}

\usepackage[numbered]{bookmark}
\usepackage{hyperref}
\usepackage{amsthm, thmtools}
\usepackage[xspace]{ellipsis}
\usepackage[all]{hypcap}

\numberwithin{figure}{section}
\numberwithin{equation}{section}
\numberwithin{table}{section}

\numberwithin{mycounter}{section}
\NewDocumentCommand{\declthm}{m m O{plain}}{
  \declaretheorem[
  name = #2,
  sibling = mycounter,
  style = #3
  ]
  {#1}
}
\declthm{theorem}{Theorem}
\declthm{proposition}{Proposition}
\declthm{corollary}{Corollary}
\declthm{lemma}{Lemma}
\declthm{definition}{Definition}[definition]
\declthm{example}{Example}[remark]
\declthm{remark}{Remark}[remark]
\declthm{conjecture}{Conjecture}

\declaretheorem[name = Theorem]{theoremintro}

\declaretheorem[name = Corollary, sibling = theoremintro]{corollaryintro}

\declaretheorem[name = Remark, sibling = theoremintro]{remarkintro}

\newcommand{\K}{\Bbbk}
\newcommand{\Q}{\mathbb{Q}}
\newcommand{\R}{\mathbb{R}}
\newcommand{\Z}{\mathbb{Z}}
\newcommand{\N}{\mathbb{N}}
\DeclareMathOperator{\id}{id}
\newcommand{\vol}{\mathrm{vol}}
\newcommand{\fiber}{\mathrm{fiber}}
\newcommand{\triv}{\mathrm{triv}}
\newcommand{\tree}{\mathrm{tree}}
\newcommand{\PP}{\mathsf{P}}
\newcommand{\QQ}{\mathsf{Q}}
\newcommand{\MM}{\mathsf{M}}
\newcommand{\CC}{\mathsf{C}}
\newcommand{\DD}{\mathsf{D}}
\newcommand{\NN}{\mathsf{N}}
\DeclareMathOperator{\Tw}{Tw}
\DeclareMathOperator{\Def}{Def}
\DeclareMathOperator{\im}{im}
\DeclareMathOperator{\cone}{cone}
\DeclareMathOperator{\colim}{colim}
\DeclareMathOperator{\holim}{holim}
\DeclareMathOperator{\hoker}{hoker}
\DeclareMathOperator{\TotCof}{TotCof}
\DeclareMathOperator{\Fr}{Fr}
\DeclareMathOperator{\Harr}{Harr}
\DeclareMathOperator{\gr}{gr}
\DeclareMathOperator{\HS}{HS}
\DeclareMathOperator{\MC}{MC}
\newcommand{\op}{\mathsf}
\newcommand{\alg}[1]{\mathfrak{#1}} 
\newcommand{\mF}{\mathcal{F}}       

\newcommand{\p}{\partial}
\newcommand{\Conf}{\mathrm{Conf}}
\newcommand{\OmPA}{\Omega_{\mathrm{PA}}}

\newcommand{\SO}{\mathrm{SO}}
\newcommand{\Ha}{\mathsf{H}}

\newcommand{\enV}[1][n]{{\mathsf{e}_{#1}^{\vee}}}
\newcommand{\Lie}{\mathsf{Lie}}
\newcommand{\hoLie}{\mathsf{hoLie}}
\newcommand{\Com}{\mathsf{Com}}
\newcommand{\hCom}{\Com_{\infty}}
\newcommand{\FM}{\mathsf{FM}}
\newcommand{\SFM}{\mathsf{SFM}}
\newcommand{\aFM}{\mathsf{aFM}}
\newcommand{\mFM}{\mathsf{mFM}}
\newcommand{\EE}{\mathsf{E}}    
\newcommand{\SC}{\mathsf{SC}}

\newcommand{\GG}[1]{\mathsf{G}_{#1}}
\newcommand{\GGt}[1]{\tilde{\mathsf{G}}_{#1}}
\newcommand{\Gra}{\mathsf{Gra}}

\newcommand{\Graphs}{\mathsf{Graphs}}
\newcommand{\SGra}{\mathsf{SGra}}
\newcommand{\SGraphs}{\mathsf{SGraphs}}

\newcommand{\aGraphs}{\mathsf{aGraphs}}
\newcommand{\mGraphs}{\mathsf{mGraphs}}
\newcommand{\GC}{\mathrm{GC}}
\newcommand{\fGC}{\mathrm{fGC}}

\newcommand{\SGC}{\mathrm{SGC}}
\newcommand{\fSGC}{\mathrm{fSGC}}
\newcommand{\aGC}{\mathrm{aGC}}
\newcommand{\mGC}{\mathrm{mGC}}
\newcommand{\KGC}{\mathrm{KGC}}

\newcommand{\dGC}{\mathrm{dGC}}

\newcommand{\zz}[1]{\mathrm{z}_{#1}}

\newcommand{\ze}{\zz{\varepsilon}}

\newcommand{\col}[1]{\mathfrak{#1}}
\newcommand{\OO}{\mathcal{O}}
\newcommand{\BF}{\mathcal{BF}}
\newcommand{\hotimes}{\mathbin{\hat{\otimes}}}
\newcommand{\lotimes}{\otimes^\mathbb{L}}

\NewDocumentCommand{\qiso}{s}{
  \IfBooleanTF #1
  {\xleftarrow{\sim}}
  {\xrightarrow{\sim}}%
}

\tikzset{
  ext/.style={circle, draw, fill=white, inner sep=2pt},
  int/.style={circle,draw,fill,inner sep=2pt},nil/.style={inner sep=2pt},
  unkv/.style={circle, fill = gray, draw, inner sep = 2pt}
}

\usepackage{accents}

\newcommand{\beq}[1]{\begin{equation}
\label{#1}
}
\newcommand{\eeq}{\end{equation}}

\newcommand{\tadpole}{
\begin{tikzpicture}
\node[int](v) at (0,0){};
\draw (v) edge[loop] (v);
\end{tikzpicture}
}

\newcommand{\notadp}
{{
		\begin{tikzpicture}[baseline=-.55ex,scale=.2, every loop/.style={}]
		\node[circle,draw,fill,inner sep=.5pt] (a) at (0,0) {};
		\draw (a) edge[loop] (a);
		\draw (-.2,-.2) -- (.2,.5);
		\end{tikzpicture}}}
\newcommand{\ntstG}{\Graphs^{\notadp}}

\author{Ricardo Campos\thanks{Institut de Mathématiques de Toulouse, UMR5219, Université de Toulouse, CNRS, UPS, F-31062 Toulouse Cedex 9, France. \href{mailto:ricardo.campos@math.univ-toulouse.fr}{ricardo.campos@math.univ-toulouse.fr}} \and Najib Idrissi\thanks{Université de Paris and Sorbonne Université, CNRS, IMJ-PRG, F-75013 Paris, France. \href{mailto:najib.idrissi-kaitouni@u-paris.fr}{najib.idrissi-kaitouni@u-paris.fr}} \and Pascal Lambrechts\thanks{Université catholique de Louvain, I.R.M.P., 2 Chemin du cyclotron, 1348 Louvain-la-Neuve, Belgium. \href{mailto:pascal.lambrechts@uclouvain.be}{pascal.lambrechts@uclouvain.be}} \and Thomas Willwacher\thanks{Department of Mathematics, ETH Zurich, R\"amistrasse 101, 8092 Zurich, Switzerland. \href{mailto:thomas.willwacher@math.ethz.ch}{thomas.willwacher@math.ethz.ch}}}
\title{Configuration Spaces of Manifolds with Boundary}

\hypersetup{pdftitle={Configuration Spaces of Manifolds with Boundary}, pdfauthor={Ricardo Campos, Najib Idrissi, Pascal Lambrechts, Thomas Willwacher}}

\begin{document}

\maketitle

\begin{abstract}
  We study ordered configuration spaces of compact manifolds with boundary.
  We show that for a large class of such manifolds, the real homotopy type of the configuration spaces only depends on the real homotopy type of the pair consisting of the manifold and its boundary.
  We moreover describe explicit real models of these configuration spaces using three different approaches.
  We do this by adapting previous constructions for configuration spaces of closed manifolds which relied on Kontsevich's proof of the formality of the little disks operads.
  We also prove that our models are compatible with the richer structure of configuration spaces, respectively a module over the Swiss-Cheese operad, a module over the associative algebra of configurations in a collar around the boundary of the manifold, and a module over the little disks operad.
\end{abstract}


\tableofcontents

\section*{Introduction}

Configuration spaces of points on manifolds are classical and yet intriguing objects in topology.
The ordered configuration space of $k$ points on a space $X$ is given by
\[ \Conf_k(X) = \{(x_1,\dots,x_k) \in X^k \mid x_i\neq x_j, \text{ for $i\neq j$} \}. \]
Despite the apparent simplicity of the definition, understanding their homotopy type, or even their rational homotopy type, has been a long-standing endeavor.

The first results in this direction were obtained by Arnold~\cite{Arnold1969} (in 2 dimensions) and Cohen \cite{Cohen1976}, who computed the cohomology of the configuration spaces of points in $\R^n$.
The real homotopy type of configuration spaces of points on smooth projective varieties were independently computed by Kriz~\cite{Kriz1994} and Totaro~\cite{Totaro1996}.

For closed simply connected manifolds, a way of computing the Betti numbers of these configuration space has been described by Lambrechts and  Stanley~\cite{LambrechtsStanley2008a}.
It has been a long-established conjecture that for such manifolds, the rational homotopy type of the configuration space depends only on the rational homotopy type of the manifold~\cite[Problem 8, p. 518]{FHT2001}.
For non-simply connected manifolds, this conjecture has a negative answer, as shown in~\cite{LoSa2005}.
Recently, three of the authors~\cite{CamposWillwacher2016,Idrissi2018b} proved that the conjecture is true when restricted to \emph{real} homotopy types.

Studying configuration spaces of manifolds with boundary is in some aspects harder, as the their homotopy type should a priori depend on the homotopy types of the manifold, its boundary, and the inclusion between the two.
Some results for computing the Betti number of configuration spaces of manifolds with boundary are known~\cite{Petersen2017}.
However, due to these difficulties, there has not been such a thorough study of the homotopy theory of configuration spaces on manifolds with boundary and in particular, the question of determining whether configuration spaces of compact manifold with boundary is a homotopy invariant remains open.
In this work, we prove that, for compact, simply connected manifolds $M$ with simply connected boundary satisfying $\dim M \ge 4$, the real homotopy type of $\Conf_{k}(M)$ only depends on the real homotopy type of the pair $\partial M \subset M$.

In mathematical physics the study of configuration spaces on manifolds with boundary is also very relevant.
For instance, in the  BV-BFV formalism~\cite{CattaneoMnevReshetikhin2018,CattaneoMnevWernli2017}, in order to perturbatively quantize gauge theory in the presence of a boundary, one needs a good understanding of the real homotopy type of configuration spaces on manifolds with boundary.
The construction of propagators and the computation of integrals given by Feynman rules admit parallels in these contexts.

The aforementioned results mostly focus on the algebro-topological properties of the configuration spaces on their own.
It has long been known, however, that configuration spaces carry rich algebraic structures using ``gluing''.

More concretely, we consider the operad of little $n$-disks, initially introduced by Boardman--Vogt~\cite{BoardmanVogt1968}, and which consists of configuration of disjoint $n$-disks (instead of points) inside the unit $n$-disk.
By considering the centers of each disk, we obtain a homotopy equivalence between this configuration space and the configuration space of points in $\R^{n}$.
However, there is a new algebraic structure on configuration of little disks, namely that of an operad.
This operadic structure is given by ``composition products'', obtained by plugging a configuration of $k$ disks inside one of the disks in a configuration of $l$ disks, to obtain a configuration of $k+l-1$ disks.

Here, for technical reason, we actually use a different model of the little $n$-disks operad.
We consider the Fulton-MacPherson compactification $\FM_n(k)$ of $\Conf_{k}(\R^{n})$~\cite{AxelrodSinger1994,FultonMacPherson1994}, obtained by allowing points to become infinitesimally close.
The collection $\FM_{n} = \{ \FM_{n}(k) \}_{k \ge 0}$ can be made into a topological operad, which is equivalent in homotopy to the little $n$-disks operad.

For a closed parallelized manifold $M$, there exists a similar compactification $\FM_M$ of the configuration space of points on $M$.
This collection $\FM_{M}$ carries the structure of an operadic right $\FM_n$-module, again using insertion of configuration.
More generally, for any $M$, one can define an operad $\FM_n^M$ in topological spaces over $M$, built from fiberwise configuration spaces.
Even if $M$ is not parallelizable, the collection $\FM_{M}$ is endowed by the structure of an operadic right module over $\FM_{n}^{M}$.

These various operadic structures on configuration spaces have received growing interest over the decades.
The configuration spaces of points on manifolds with their operadic module structure have recently seen a surge in interest, due to their central appearance in the Goodwillie--Weiss embedding calculus~\cite{BoavidaWeiss2013,Turchin2013} and factorization homology~\cite{AF2015}.
These applications require understanding of the homotopy type of the configuration spaces together with their natural operadic structures.

The first result in this direction was the rational formality of the little disks operads, shown by Tamarkin~\cite{Tamarkin2003} (for $n=2$, over $\Q$), Kontsevich~\cite{Kontsevich1999} (for all $n$, over $\R$), with further contributions over the years~\cite{GNPR2005,LambrechtsVolic2014,Petersen2014,FresseWillwacher2015}.

For closed connected orientable $M$, the real homotopy types of the configuration spaces $\FM_M$ together with the operadic structure have recently computed by three of the authors~\cite{CamposWillwacher2016, Idrissi2018b}, where ``workable'' combinatorial models were given.
This paper is a follow-up to these works, extending the methods and generalizing the results to compact orientable manifolds with boundary.

\subsection*{Summary of results}

Let $M$ be a compact orientable manifold with non-empty boundary $\p M$.
We study configuration spaces of $r$ points in the interior and $s$ points on the boundary:
\[
  \Conf_{r,s}(M) = \{(x_1,\dots,x_r, y_1,\dots ,y_s) \times (\p M)^r \times \mathring{M}^s \mid x_i\neq x_j, y_i\neq y_j \text{ for $i\neq j$} \}.
\]

There are essentially two approaches to defining algebraic structures on those spaces: one that has to do with the action of the Swiss-Cheese operad, and one that has to do with how configuration spaces behave when one glues manifolds along their boundaries.
We will describe ``graphical'' models for both approaches.
We will also define ``small'' models for configurations on the interior of $M$ together with its action of the little disks operad.

\subsubsection*{Graphical models: Swiss-Cheese action}

We can compactify $\Conf_{r,s}(M)$ in the spirit of Axelrod--Singer~\cite{AxelrodSinger1994} to obtain a compact space $\SFM_M(r,s)$, cf.\ Section~\ref{sec:configspaces} below.
These compactified spaces come with a natural operadic right action of the fiberwise Swiss-Cheese operad $\SC_n^M$ (and the $\mathsf{S}$ in $\SFM$ stands for ``Swiss-Cheese'').

More concretely, in the second color, the fiberwise little disks operad $\FM_n^M$ acts on $\SFM_M$ by insertion of configurations of points ``infinitesimally close to'' a given point in the configuration.
In the first color, we have a similar operation of insertion of configurations of points on the upper half-space, fiberwise over $\p M$.
The operations are depicted in the following illustration:
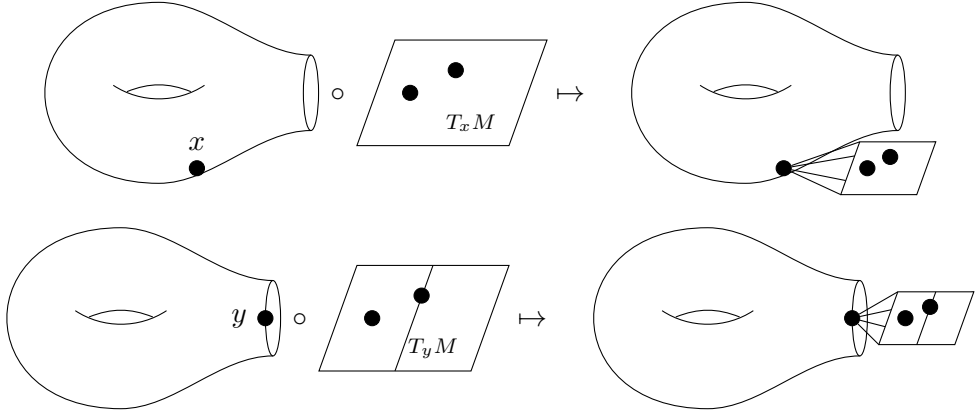
\[
  \begin{tikzpicture}[yshift=-.5cm,baseline=-.65ex]
    \draw (0,0)
    .. controls (-.7,0) and (-1.3,-.7) .. (-2,-.7)
    .. controls (-4,-.7) and (-4,1.7) .. (-2,1.7)
    .. controls (-1.3,1.7) and (-.7,1) .. (0,1);
    \begin{scope}[xshift=-2cm, yshift=.6cm, scale=1.2]
      \draw (-.5,0) .. controls (-.2,-.2) and (.2,-.2) .. (.5,0);
      \begin{scope}[yshift=-.07cm]
        \draw (-.35,0) .. controls (-.1,.1) and (.1,.1) .. (.35,0);
      \end{scope}
    \end{scope}
    \draw (0,.5) ellipse (.1cm and .5cm);
    \node[int,label=$x$] (v) at (-1.5,-.5) {};
  \end{tikzpicture}
  \, \circ \,
  \begin{tikzpicture}[baseline=-.65ex]
    \draw (-1,-.7)--(1,-.7)--(1.5,.7)--(-.5,.7)--cycle;
    \node[int] (w1) at (-.3,0){};
    \node[int] (w2) at (.3,.3){};
    \node at (.5,-.4) {$\scriptstyle T_x M$};
  \end{tikzpicture}
  \mapsto
  \begin{tikzpicture}[yshift=-.5cm,baseline=-.65ex]
    \draw (0,0)
    .. controls (-.7,0) and (-1.3,-.7) .. (-2,-.7)
    .. controls (-4,-.7) and (-4,1.7) .. (-2,1.7)
    .. controls (-1.3,1.7) and (-.7,1) .. (0,1);
    \begin{scope}[xshift=-2cm, yshift=.6cm, scale=1.2]
      \draw (-.5,0) .. controls (-.2,-.2) and (.2,-.2) .. (.5,0);
      \begin{scope}[yshift=-.07cm]
        \draw (-.35,0) .. controls (-.1,.1) and (.1,.1) .. (.35,0);
      \end{scope}
    \end{scope}
    \draw (0,.5) ellipse (.1cm and .5cm);
    \node[int] (v) at (-1.5,-.5) {};
    \begin{scope}[scale=.5,xshift=-.5cm,yshift=-1cm]
      \draw (v) --(-1,-.7) (v)--(1,-.7) (v)--(-.5,.7) (v)--(1.5,.7);
      \draw[fill=white] (-1,-.7)--(1,-.7)--(1.5,.7)--(-.5,.7)--cycle;
      \node[int] (w1) at (-.3,0){};
      \node[int] (w2) at (.3,.3){};
    \end{scope}
  \end{tikzpicture}
\]
\[
  \begin{tikzpicture}[yshift=-.5cm,baseline=-.65ex]
    \draw (0,0)
    .. controls (-.7,0) and (-1.3,-.7) .. (-2,-.7)
    .. controls (-4,-.7) and (-4,1.7) .. (-2,1.7)
    .. controls (-1.3,1.7) and (-.7,1) .. (0,1);
    \begin{scope}[xshift=-2cm, yshift=.6cm, scale=1.2]
      \draw (-.5,0) .. controls (-.2,-.2) and (.2,-.2) .. (.5,0);
      \begin{scope}[yshift=-.07cm]
        \draw (-.35,0) .. controls (-.1,.1) and (.1,.1) .. (.35,0);
      \end{scope}
    \end{scope}
    \draw (0,.5) ellipse (.1cm and .5cm);
    \node[int,label=180:{$y$}] (v) at (-.1,.5) {};
  \end{tikzpicture}
  \, \circ \,
  \begin{tikzpicture}[baseline=-.65ex]
    \draw (-1,-.7)--(1,-.7)--(1.5,.7)--(-.5,.7)--cycle;
    \draw (0,-.7) --(.5,.7);
    \node[int] (w1) at (-.3,0){};
    \node[int] (w2) at (.35,.3){};
    \node at (.5,-.4) {$\scriptstyle T_y M$};
  \end{tikzpicture}
  \mapsto
  \begin{tikzpicture}[yshift=-.5cm,baseline=-.65ex]
    \draw (0,0)
    .. controls (-.7,0) and (-1.3,-.7) .. (-2,-.7)
    .. controls (-4,-.7) and (-4,1.7) .. (-2,1.7)
    .. controls (-1.3,1.7) and (-.7,1) .. (0,1);
    \begin{scope}[xshift=-2cm, yshift=.6cm, scale=1.2]
      \draw (-.5,0) .. controls (-.2,-.2) and (.2,-.2) .. (.5,0);
      \begin{scope}[yshift=-.07cm]
        \draw (-.35,0) .. controls (-.1,.1) and (.1,.1) .. (.35,0);
      \end{scope}
    \end{scope}
    \draw (0,.5) ellipse (.1cm and .5cm);
    \node[int] (v) at (-.1,.5) {};
    \begin{scope}[scale=.5,xshift=1.5cm,yshift=1cm]
      \draw (v) --(-1,-.7) (v)--(1,-.7) (v)--(-.5,.7) (v)--(1.5,.7);
      \draw[fill=white] (-1,-.7)--(1,-.7)--(1.5,.7)--(-.5,.7)--cycle;
      \draw (0,-.7) --(.5,.7);
      \node[int] (w1) at (-.3,0){};
      \node[int] (w2) at (.35,.3){};
    \end{scope}
  \end{tikzpicture}
\]

Our first main result is the construction of a CDGA model $\SGraphs_{A,A_{\partial}}$ for the right $\SFM_n^M$-module $\SFM_M$, where $(A,A_{\partial})$ is a CDGA model for $(M, \partial M)$ (see Section~\ref{sec.model-conf-space-closed} for what is precisely expected of $(A,A_{\partial})$).
The proof mostly follows analogous results and constructions of Kontsevich, a strategy already used in previous works~\cite{CamposWillwacher2016,Idrissi2018b}.

Let us briefly describe this model.
Elements of $\SGraphs_{A,A_{\partial}}(r,s)$ are directed graphs with vertices of 4 kinds: aerial external vertices, numbered from $1,\dots, s$, representing the $s$ points in the interior of $M$; terrestrial external vertices, numbered $1,\dots, r$, representing the $r$ points on the boundary; and internal ``unidentifiable'' vertices, either aerial or terrestrial.
In addition, aerial vertices may be decorated by elements of $A$, and terrestrial vertices are decorated by elements of $A_{\partial}$.
Note also that edges may not start at terrestrial vertices.
For detailed construction and some more technical conditions we refer the reader to Section~\ref{sec.model-right-hopf}.

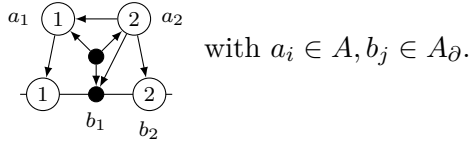
\[
  \begin{tikzpicture}[baseline=.5cm]
    \draw(-1,0) -- (1,0);
    \node[ext,label=0:{$\scriptstyle a_2$}] (v1) at (0.5,1) {$\scriptstyle 2$};
    \node[ext,label=180:{$\scriptstyle a_1$}] (v2) at (-.5,1) {$\scriptstyle 1$};
    \node[int] (v3) at (0,.5) {};
    \node[ext, fill=white] (w1) at (-0.7,0) {$\scriptstyle 1$};
    \node[ext, fill=white,label=-90:{$\scriptstyle b_2$}] (w2) at (0.7,0) {$\scriptstyle 2$};
    \node[int,label=-90:{$\scriptstyle b_1$}] (w3) at (0,0) {};
    \draw[-latex] (v1) edge (v2) edge (w2) edge (w3) (v3) edge (v1) edge (v2) edge (w3) (v2) edge (w1);
  \end{tikzpicture}
  \text{ with } a_i \in A, b_j \in A_{\partial}.
\]

We can define similarly a graphical model $\SGraphs_n^A$ of the fiberwise Swiss-Cheese operad, and there is a coaction of $\SGraphs_n^A$ on $\Graphs_{A,A_{\partial}}$ by explicit combinatorial formulas on graphs.
Furthermore, all graphical models have a natural dg commutative algebra structure, given by gluing diagrams at external vertices.

Our first main result is then that these graphical objects are indeed models for the topological configuration spaces, in the sense that they are quasi-isomorphic to the CDGAs of differential forms on those spaces:
\begin{theoremintro}[See Theorem~\ref{thm.main-qiso}]
  \label{thm.intro.sgraphs-qiso}
  Let $M$ be an oriented compact manifold with boundary $\p M\neq \emptyset$.
  Then there is zigzag of quasi-isomorphisms relating the pairs
  \[
    (\SGraphs_{A,A_{\partial}},\SGraphs_n^A,\Graphs_{n}^{A}) \simeq (\OmPA(\SFM_M), \OmPA(\SFM_n^M), \OmPA(\FM_{n}^{M}))
  \]
  compatible with all structures, i.e., with the dg commutative algebra structure and the operadic action of the second member of the pairs on the first.
\end{theoremintro}

We note that the graded object $\SGraphs_{A,A_{\partial}}$ depends on $M$ only through the homotopy type of $M$, while that is certainly not true for the real homotopy type of the configuration spaces.
The dependence on $M$ in $\Graphs_M$ comes from the differential.
More concretely, the differential, and hence all dependence on $M$, is neatly encoded by a Maurer--Cartan element $Z_M$ in a certain graph complex.
Physically, this Maurer--Cartan element corresponds to the partition function in the underlying topological field theory, taking values in the complex of vacuum Feynman diagrams.
In this paper we will hence call this special MC element $Z_M$ that governs the real homotopy type of our configuration spaces the ``partition function'', although we will not discuss any connections to physics.

The partition function can actually be evaluated under good conditions.
In particular, if $M$ and $\p M$ are simply connected, and $\dim(M)\geq 5$, then the partition function only depends on the real homotopy type of $M$:
\begin{corollaryintro}[Corollary~\ref{cor.inv-htp1}]\label{cor:mainSC}
  If $M$ and $\p M$ are simply connected, and $\dim(M)\geq 5$ then the real homotopy type of $\SFM_M$ (as space, and as right module under the fiberwise Swiss-Cheese operad) only depends on the real homotopy type of the map $\p M \to M$.
\end{corollaryintro}
More concretely, the real homotopy type is precisely encoded in the tree piece of the partition function.
The above result hence means that there are no ``quantum corrections''.

Note that for $\dim M \le 4$, using the (generalized) Poincaré conjecture (see Freedman~\cite{Freedman1982} for $n=4$ and Perelman~\cite{Perelman2002,Perelman2003} for $n = 3$), if $M$ and its boundary are simply connected then $M$ is homeomorphic to $D^{3}$ or $D^{4}$ (or, depending on convention, $D^{0}$ if $\varnothing$ is considered simply connected.
Therefore, in dimension $\le 4$, if $(M, \partial M)$ and $(M', \partial M')$ are homotopy equivalent, then they are homeomorphic.
Since homeomorphic spaces have homeomorphic configuration spaces, Corollary~\ref{cor:mainSC} thus holds vacuously for the real homotopy type of the \emph{spaces} $\SFM_{M}$, as there is at most one possible manifold in each dimension to consider.
However, the right action of the fiberwise Swiss-Cheese operad depends on tangent spaces, and thus the diffeomorphism types of the manifold.
In dimension $\le 3$, $M$ must also be diffeomorphic to $D^{3}$, so Corollary~\ref{cor:mainSC} holds in full.
The existence of exotic $\R^{4}$ prevents us from concluding in dimension $4$ and we do not know if there exists a counterexample to real homotopy invariance in dimension $4$.

\subsubsection*{Graphical models: gluing at the boundary}

There is an alternative viewpoint on the configuration spaces of points on manifolds with boundary, that both gives rise to simpler models, and to algebraic structure which is not easily extracted from (although contained in) the Swiss-Cheese action above.
Let us now only consider configuration spaces
\[
  \Conf_{r}(M) \coloneqq \Conf_{0,r}(M)
\]
of points in the interior, with no points on the boundary.
Also consider the configuration space of points on $\p M\times I$, i.e., $\Conf_{r'}(\p M \times I)$, where $I = (0,+\infty)$.
The collection of the latter spaces (for various $r'$) naturally forms an algebra object (more precisely, an $E_1$-algebra object), the product being the gluing of the intervals
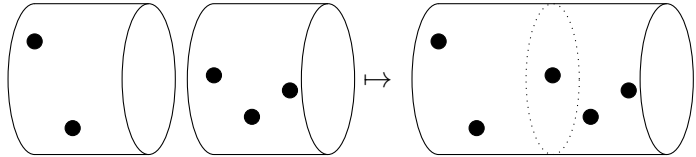
\[
  \begin{tikzpicture}[scale=.5, baseline=-.65ex]
    \begin{scope}
      \clip (0,-2.1) rectangle (-1,2.1);
      \draw (0,0) ellipse (.7cm and 2 cm);
    \end{scope}
    \draw (3,0) ellipse (.7cm and 2 cm);
    \draw (0,2) -- (3,2) (0,-2)--(3,-2);
    \node[int] at (0,1) {};
    \node[int] at (1,-1.3) {};
  \end{tikzpicture}
  \begin{tikzpicture}[scale=.5, baseline=-.65ex]
    \begin{scope}
      \clip (0,-2.1) rectangle (-1,2.1);
      \draw (0,0) ellipse (.7cm and 2 cm);
    \end{scope}
    \draw (3,0) ellipse (.7cm and 2 cm);
    \draw (0,2) -- (3,2) (0,-2)--(3,-2);
    \node[int] at (2,-.3) {};
    \node[int] at (1,-1) {};
    \node[int] at (0,.1) {};
  \end{tikzpicture}
  \mapsto
  \begin{tikzpicture}[scale=.5, baseline=-.65ex]
    \begin{scope}
      \clip (0,-2.1) rectangle (-1,2.1);
      \draw (0,0) ellipse (.7cm and 2 cm);
    \end{scope}
    \draw[dotted] (3,0) ellipse (.7cm and 2 cm);
    \draw (6,0) ellipse (.7cm and 2 cm);
    \draw (0,2) -- (6,2) (0,-2)--(6,-2);
    \node[int] at (0,1) {};
    \node[int] at (1,-1.3) {};
    \node[int] at (5,-.3) {};
    \node[int] at (4,-1) {};
    \node[int] at (3,.1) {};
  \end{tikzpicture}
\]
This $E_{1}$-algebra naturally acts on $\Conf_M(r)$ by gluing at the boundary:
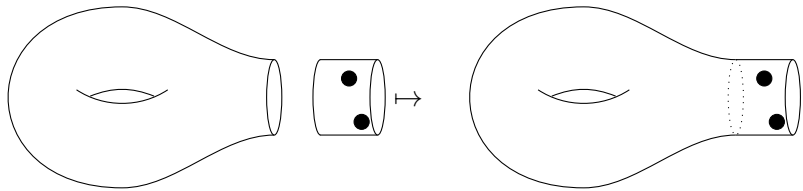
\tikzset{
  show curve controls/.style={
      postaction={
          decoration={
              show path construction,
              curveto code={
                  \draw [blue]
                  (\tikzinputsegmentfirst) -- (\tikzinputsegmentsupporta)
                  (\tikzinputsegmentlast) -- (\tikzinputsegmentsupportb);
                  \fill [red, opacity=0.5]
                  (\tikzinputsegmentsupporta) circle [radius=.5ex]
                  (\tikzinputsegmentsupportb) circle [radius=.5ex];
                }
            },
          decorate
        }}}
\[
  \begin{tikzpicture}[yshift=-.5cm,baseline=-.65ex]
    \draw (0,0)
    .. controls (-.7,0) and (-1.3,-.7) .. (-2,-.7)
    .. controls (-4,-.7) and (-4,1.7) .. (-2,1.7)
    .. controls (-1.3,1.7) and (-.7,1) .. (0,1);
    \begin{scope}[xshift=-2cm, yshift=.6cm, scale=1.2]
      \draw (-.5,0) .. controls (-.2,-.2) and (.2,-.2) .. (.5,0);
      \begin{scope}[yshift=-.07cm]
        \draw (-.35,0) .. controls (-.1,.1) and (.1,.1) .. (.35,0);
      \end{scope}
    \end{scope}
    \draw (0,.5) ellipse (.1cm and .5cm);
  \end{tikzpicture}
  \begin{tikzpicture}[yshift=-.5cm, baseline=-.65ex]
    \begin{scope}
      \clip (0,-.1) rectangle (-.5,1.1);
      \draw (0,.5) ellipse (.1cm and .5cm);
    \end{scope}
    \draw (.75,.5) ellipse (.1cm and .5cm);
    \draw (0,0)--(.75,0) (0,1)--(.75,1);
    \node[int] at (.375,.75) {};
    \node[int] at (.54,.175) {};
  \end{tikzpicture}
  \mapsto
  \begin{tikzpicture}[yshift=-.5cm,baseline=-.65ex]
    \draw (0,0)
    .. controls (-.7,0) and (-1.3,-.7) .. (-2,-.7)
    .. controls (-4,-.7) and (-4,1.7) .. (-2,1.7)
    .. controls (-1.3,1.7) and (-.7,1) .. (0,1);
    \begin{scope}[xshift=-2cm, yshift=.6cm, scale=1.2]
      \draw (-.5,0) .. controls (-.2,-.2) and (.2,-.2) .. (.5,0);
      \begin{scope}[yshift=-.07cm]
        \draw (-.35,0) .. controls (-.1,.1) and (.1,.1) .. (.35,0);
      \end{scope}
    \end{scope}
    \draw[dotted] (0,.5) ellipse (.1cm and .5cm);
    \draw (.75,.5) ellipse (.1cm and .5cm);
    \draw (0,0)--(.75,0) (0,1)--(.75,1);
    \node[int] at (.375,.75) {};
    \node[int] at (.54,.175) {};
  \end{tikzpicture}
\]
Understanding these gluing operations is of high importance, because they allow the configuration space of points on a glued manifold $X=M\sqcup_{\p M} N$ to be expressed through the configuration spaces of the pieces, as a ``derived tensor product''
\[
  \Conf(X) \simeq \Conf(M) \lotimes_{\Conf(\p M\times I)} \Conf(N) .
\]
We refer the reader to~\cite{AF2015} for more explanations.

The algebraic operations above (of algebra and module, by gluing at the boundary) are encoded in our models in the Swiss-Cheese action, but not in an accessible form.
We may however describe simpler models in this case that capture those gluing operations more nicely.
To this end it is also suitable to define slightly different compactifications $\mFM_M$ and $\aFM_{\p M}$ (where ``a'' stands for ``algebra'' and ``m'' for module, see Section~\ref{sec:typeIIcompactification}).
The algebra and module structures described above are defined on the nose for these models of the configuration spaces.

In this setting we can construct significantly simpler combinatorial models for our configuration spaces $\mGraphs_A$ and $\aGraphs_{A_{\partial}}$.
Concretely, elements of $\mGraphs_A(r)$ are directed graphs with only two types of vertices, external vertices numbered $1\dots,r$, and internal vertices.
All vertices are decorated by an element on $A$.
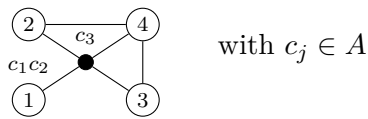
\[
  \begin{tikzpicture}[baseline=.65ex, yshift=-.5cm]
    \node[ext,label={$\scriptstyle c_1c_2$}] (v1) at (0,0) {$\scriptstyle 1$};
    \node[ext] (v2) at (0,1) {$\scriptstyle 2$};
    \node[ext] (v3) at (1.5,0) {$\scriptstyle 3$};
    \node[ext] (v4) at (1.5,1) {$\scriptstyle 4$};
    \node[int,label={$\scriptstyle c_3$}] (v) at (.75,.5) {};
    \draw (v) edge (v1) edge (v2) edge (v3) edge (v4) (v4) edge (v3) edge (v2);
  \end{tikzpicture}
  \qquad
  \text{with }
  c_j\in A
\]
The construction of $\aGraphs_{A_{\partial}}$ is similar.

All the algebraic operations may then be encoded combinatorially on these diagrams.
Our second main result is then:
\begin{theoremintro}[See Section~\ref{sec:secondmodel}]\label{thm:B}
  For $M$ a compact oriented manifold with boundary $\p M$, we have a zig-zag of quasi-isomorphisms
  \[
    (\mGraphs_A, \aGraphs_{A_{\partial}}, \Graphs_{n}^A) \simeq (\OmPA(\mFM_M), \OmPA(\aFM_{\p M}), \OmPA(\FM_n^M))
  \]
  respecting all algebraic structures, i.e., the CDGA structures, the operadic right actions, and the $E_{1}$-algebra and module structure obtained by gluing at the boundary.
\end{theoremintro}

Again, the objects here depend on a certain Maurer--Cartan element in a graph complex called the partition function.
This partition function can can be evaluated under good conditions, just like for $\SFM_{M}$.
The conditions here are weaker than those of Corollary \ref{cor:mainSC} (see page~\pageref{proof:cor-d}), which can heuristically be explained because the graph complex for $\SFM_{M}$ depends on a model for $\Conf(\partial M)$, which requires $\dim \partial M \ge 4 \iff \dim M \ge 5$.

\begin{corollaryintro}\label{cor:mainoGC}
  If $M$ is simply connected, and $\dim(M)\geq4$ then the real homotopy type of $\mFM_M$ (as space, and as right module over $\aFM_{N}$) only depends on the real homotopy type of the map $\p M \to M$.
\end{corollaryintro}

In Section~\ref{sec.conn-betw-sgraphs}, we connect the two models $\SGraphs_{A,A_{\partial}}(0,-)$ (all the external vertices are in the interior) and $\mGraphs_{A}$ as comodules over $\Graphs_{n}^{A}$.
The only possible manifold in dimension $\le 3$ is $D^{3}$ by the Poincaré conjecture, so the corollary vacuously in low dimensions.

\subsubsection*{Small models and coaction of the cohomology}

Under some hypotheses about the connectivity and the dimension of $M$, we will also find some ``small'' models for $\Conf_{k}(M)$, inspired by the Lambrechts--Stanley models for configuration spaces of closed manifolds (see~\cite{Idrissi2018b} and~\cite[Appendix~A]{CamposWillwacher2016}).

Suppose both $M$ and $\partial M $ are simply connected and that $\dim M \geq 7$, so that the pair $(M, \partial M)$ admits a Poincaré--Lefschetz duality model, a notion we define in Section~\ref{sec.pretty-nice}, and let $P$ be the resulting model of $M$.
We may then use the same construction as in~\cite{Idrissi2018b} to get a CDGA $\GG{P}(k)$.
If $\partial M = \varnothing$, then $P$ is a Poincaré duality model of $M$ and we recover the Lambrechts--Stanley model of $\Conf_{k}(M)$.
We show in Theorem~\ref{thm.ga-same-cohom} that there is an isomorphism of graded vector spaces between $H^{*}(\GG{P}(k))$ and $H^{*}(\Conf_{k}(M))$ over $\Q$, which generalizes the result of~\cite{LambrechtsStanley2008a}.

However, $\GG{P}(k)$ is not a CDGA model of $\Conf_{k}(M)$ in general.
Instead we consider a ``perturbed'' version $\GGt{P}(k)$, which is isomorphic to $\GG{P}(k)$ as a dg-module but not as an algebra.
We show that $\GGt{P}(k)$ is a CDGA model of $\Conf_{k}(M)$.
Moreover, we prove that $\GGt{P}$ is a right Hopf $\enV$-comodule if $\partial M \neq \varnothing$, and if $M$ is framed then we prove that the quasi-isomorphism $\GGt{P} \simeq \OmPA^{*}(\SFM_{M}(\varnothing, -))$ is compatible with the comodule structures (over $\enV$ and $\OmPA^{*}(\FM_{n})$, respectively).

\begin{theoremintro}[See Theorems~\ref{thm.model-ga}--\ref{thm.ext-formal}]
  \label{thm.A}
  Let $M$ be a smooth, simply connected connected compact $n$-manifold with simply connected boundary of dimension at least $5$.
  Assume that either $M$ admits a surjective pretty model, or that $n \geq 7$ so that $M$ admits a Poincaré--Lefschetz duality model.
  Let $P$ be the model built either out of the surjective pretty model or the Poincaré--Lefschetz duality model.

  Then for all $k \geq 0$, the CDGA $\GGt{P}(k)$ is weakly equivalent to $\OmPA^{*}(\SFM_{M}(0,k))$, and the equivalence is compatible with the action of the symmetric group $\Sigma_{k}$; in particular, it is a model of $\Conf_{k}(M)$.
  Moreover, if $M$ is framed, then the right Hopf comodule $(\GGt{P}, \enV)$ is weakly equivalent to $(\OmPA^{*}(\SFM_{M}(0,-)), \OmPA^{*}(\FM_{n}))$.

  The same result holds with $P = H^{*}(M)$ for simply connected manifolds with simply connected boundary satisfying $\dim M \in \{4,5,6\}$.
\end{theoremintro}

The advantage of this small model is that it can be used to do some computations, e.g., of factorization homology (see~\cite[Section~5]{Idrissi2018b}) or embedding calculus.
Note that despite the notation, $\GGt{P}$ depends not just on the model $P$ of $M$ but on the full Poincaré--Lefschetz duality model of the pair $(M, \partial M)$.

\begin{remarkintro}
  All of our models are compatible with the symmetric group actions.
  Therefore, we obtain models of the unordered configuration spaces $B_{k}(M) = \Conf_{k}(M) / \Sigma_{k}$ by considering the sub-CDGA of elements invariant under the symmetric group action.
  Note however that the unordered configuration spaces are not simply connected even if $M$ is (if $\dim M \geq 3$ and $M$ is simply connected then $\pi_{1}(B_{k}(M)) = \Sigma_{k}$) so this may give less information than expected.
  This is still sufficient to compute the cohomology, for example.
\end{remarkintro}

\subsection*{Outline}

\newcommand{\na}{{\tiny n/a} & {\tiny n/a}}
\begin{center}
  \footnotesize
  \begin{tabular}{crllll}
    \toprule
     &
     & Closed mfd
     & Swiss-Cheese
     & $E_{1}$-algebra
     & $E_{1}$-module
    \\ \midrule
    Compactif.
     & local
     & $\FM_{n}$ §\ref{sec:axelr-sing-fult}
     & $\SFM_{n}$ §\ref{sec:swiss-cheese-operad}
     & \na
    \\
     & fibered
     & $\FM_{n}^{M}$ §\ref{sec:fiberwiseLD}
     & $\SFM_{n}^{M}$ §\ref{sec:swiss-cheese-operad}
     & \na
    \\
     & global
     & $\FM_{M}$ §\ref{sec:comp-conf-space-closed}
     & $\SFM_{M}$ §\ref{sec.constr-sfmm}
     & $\aFM_{\partial M}$
     & \& $\mFM_{M}$ §\ref{sec:typeIIcompactification}
    \\ \midrule
    Propagator
     & local
     & \cite{Kontsevich1999}
     & \cite{Willwacher2015}
     & \na
    \\
     & fibered
     & \cite{CDW17}
     & §\ref{sec:locMC2}
     & \na
    \\
     & global
     & §\ref{sec:closed propagator}
     & §\ref{sec.propagator-sfm}
     & §\ref{sec.propagator-afm}
     & §\ref{sec.propagator-mfm}
    \\ \midrule
    Model
     & local
     & $\Graphs_{n}$ §\ref{sec.extens-swiss-cheese}
     & $\SGraphs_{n}$ §\ref{sec.extens-swiss-cheese}
     & \na
    \\
     & fibered
     & $\Graphs_{n}^{A}$ §\ref{sec:model-fiberwiseLD}
     & $\SGraphs_{n}^{A}$ §\ref{sec:locMC2}
     & \na
    \\
     & global
     & $\Graphs_{A}$ §\ref{sec.model-conf-space-closed}
     & $\SGraphs_{A,A_{\partial}}$ §\ref{sec.model-right-hopf}
     & $\aGraphs_{A_{\partial}}$
     & \& $\mGraphs_{A}$ §\ref{sec:agraphs-mgraphs}
    \\ \midrule
    MC elements
     & local
     & $\mu \in \GC^{\vee}_{n}$ \eqref{eq.mu}
     & $c \in \SGC_{n}^{\vee}$ \eqref{eq.konts-coeff}
     & \na
    \\
     & fibered
     & $z \in A \hotimes \GC_{n}^{\vee}$ §\ref{sec:localMC1}
     & $z^{\partial} \in A_{\partial} \hotimes \SGC_{n}^{\vee}$ §\ref{sec:locMC2}
     & \na
    \\
     & global
     & $Z \in \GC_{A}^{\vee}$
     & $Z \in \SGC_{A,A_{\partial}}^{\vee}$ §\ref{sec.global-maurer-cartan}
     & $w \in \aGC_{A_{\partial}}^{\vee}$
     & $W \in \mGC_{A}^{\vee}$ §\ref{sec:oGCMC}
    \\ \bottomrule
  \end{tabular}
\end{center}

\begin{description}
  \item[Section~\ref{sec.backgr-recoll}]
    We recall some background on cooperads and comodules over them, operads over spaces, the cohomology of compact manifolds with boundary, and pretty models.
  \item[Section~\ref{sec.pretty-nice}]
    We define Poincaré--Lefschetz duality models, a generalization of surjective pretty models, and we prove that any simply connected manifold with simply connected boundary of dimension at least $7$ admits a Poincaré--Lefschetz duality model.
  \item[Section~\ref{sec:configspaces}]
    We recall and define various compactifications for configuration spaces of Euclidean (half-)spaces and manifolds with and without boundary, inspired by the Axelrod--Singer--Fulton--MacPherson compactifications.
  \item[Section~\ref{sec.propagator}]
    We explain how to construct the ``propagators'' which will be used to define integrals on these compactified configuration spaces, using the usual Feynman rules.
  \item[Section~\ref{sec.model-closed}]
    We recall the construction of models for configuration spaces of closed manifolds~\cite{CamposWillwacher2016,Idrissi2018b} that we will generalized for compact manifolds with boundary.
    We also explain in what sense the graphical models we build are ``functorial'', which will be used in the rest of the paper.
  \item[Section~\ref{sec.model-right-hopf}]
    We build our first graphical model $\SGraphs_{A,A_{\partial}}$, and we prove that it is a model of $\Conf_{\bullet, \bullet}(M)$ as an operadic module over the Swiss-Cheese operad.
  \item[Section~\ref{sec:secondmodel}]
    We build our second graphical model, $\mGraphs_{A}$, and we prove that it is a model of $\Conf_{0,\bullet}(M)$ as a module over the $E_{1}$-algebra $\Conf_{\partial M \times \R_{>0}}$.
  \item[Section~\ref{sec.model-conf-spac}]
    We build a first small dg-module $\GG{P}(k)$, and we prove that under some hypotheses, it computes the Betti numbers of $\Conf_{k}(M)$.
    We then prove that a ``perturbed'' variant $\GGt{P}(k)$ is a CDGA model for $\Conf_{k}(M)$ as a module over the cohomology $\enV$ of the little $n$-disks operad.
    We also make precise the connection between $\SGraphs_{A,A_{\partial}}$ and $\mGraphs_{A}$.
  \item[Appendix~\ref{sec:appxA}]
    We compute the cohomology of several graph complexes that appear throughout the paper.
\end{description}

\subsection*{Acknowledgments}

R.C.\ thanks useful discussions with Konstantin Wernli and acknowledges support by the Swiss National Science Foundation Early Postdoc.Mobility grant number \texttt{P2EZP2\_174718}.
N.I.\ thanks his (now former) advisor Benoit Fresse for guidance during the completion of his PhD thesis, part of which is contained in the present work; Hisatoshi Kodani for useful remarks; acknowledges support by the European Research Council (ERC StG 678156--GRAPHCPX); and contributes to the IdEx University of Paris ANR-18-IDEX-0001.
T.W.\ acknowledges partial support by the European Research Council (ERC StG 678156--GRAPHCPX) and the NCCR SwissMap funded by the Swiss National Science Foundation.

\section{Background and recollections}
\label{sec.backgr-recoll}

We will generally work in the category of cohomologically graded cochain complexes, with a suspension defined by $(V[k])^{n} = V^{k+n}$ and appropriate signs.
A ``CDGA'' will then be a (unital) commutative differential graded algebra in this category.

We will typically aim to prove that two $\N$-graded CDGAs are quasi-isomorphic, i.e., connected by a zig-zag of CDGA morphism which induce isomorphisms on cohomology.
The intermediary CDGAs will sometimes be $\Z$-graded; recall that if two \emph{connected} $\mathbb{N}$-graded CDGAs are quasi-isomorphic in the category of $\Z$-graded CDGAs, then they are too in the category of $\mathbb{N}$-graded CDGAs (see~\cite[Proposition 4]{Idrissi2018b}).
The same remark will also apply to Hopf cooperads below (see the discussion in~\cite[Section~1.2]{Idrissi2018b}).

\subsection{Colored {(co)}operads and {(co)}modules}
\label{sec.cooperads-comodules}²

We work with (symmetric) operads, which we will generally index by finite sets, rather than by numbers.
Thus, for us an operad\index{operad} $\PP$ is a symmetric collection\index{symmetric collection}, i.e., a functor from the category of finite sets and bijections to some symmetric monoidal category (e.g., topological spaces, chain complexes\dots), equipped with a unit $\eta : 1 \to \PP(\{*\})$ as well as composition maps, for all pairs $(U, W \subset U)$:
\begin{equation}
  \circ_{W} : \PP(U/W) \otimes \PP(W) \to \PP(U),
\end{equation}
where $U/W = (U - W) \sqcup \{*\}$ is the quotient, satisfying the usual associativity, equivariance, and unitality axioms.
If we let $\underline{k} = \{ 1, \dots, k \}$ then we recover the classical notion by setting $\PP(k) \coloneqq \PP(\underline{k})$ and the composition maps $\circ_{i} : \PP(k) \otimes \PP(l) \to \PP(k+l-1)$ are given by $\circ_{\{i, \dots, i+l-1\}}$.

A right module\index{operadic right module} over an operad $\PP$ is a symmetric collection $\MM$ equipped with composition maps
\begin{equation}
  \circ_{W} : \MM(U/W) \otimes \PP(W) \to \MM(U)
\end{equation}
satisfying the usual axioms.
Cooperads and comodules are defined dually.
We will add the adjective ``Hopf''\index{Hopf operad, module\dots} to these dual structures when they are considered in the base category of CDGAs.

We also deal with special types of two-colored operads, called relative operads~\cite{Voronov1999},\index{relative operad} or Swiss-Cheese type operads.

\begin{definition}\label{def:rel-operad}
  Given an operad $\PP$, a \textbf{relative operad over} $\PP$ is an operad in the category of right $\PP$-modules.
\end{definition}

Let us now give a concrete description of relative operad.
We can encode a relative operad over $\PP$ as a bisymmetric collection\index{bisymmetric collection} $\QQ$, i.e., as a functor from the category of pairs of sets and pairs of bijections to dg-modules.
The first set in the pair corresponds to open inputs, and the second to closed inputs.
There is an identity $\eta_{\mathfrak{o}} \in \QQ(\{*\}, \varnothing)$, and the operadic composition structure maps are given by:
\begin{align}
  \circ_{T} : \QQ(U, V/T) \otimes \PP(T)
   & \to \QQ(U,V)
   & T \subset V;           \\
  \circ_{W,T} : \QQ(U/W, V) \otimes \QQ(W,T)
   & \to \QQ(U, V \sqcup T)
   & W \subset U.
\end{align}

As mentioned in the definition, we can equivalently view $\QQ$ as an operad in the category of right $\PP$-modules.
The $\PP$-module in arity $U$ is given by $\QQ(U, -)$, and one checks that the operad structure maps $\circ_{W, -} : \QQ(U/W, -) \otimes \QQ(W, -) \to \QQ(U, -)$ are morphisms of right $\PP$-modules.

\begin{remark}\label{rmk:rel-col-operad}
  The data of a relative operad over $\PP$ is also equivalent to the data of an operad with two colors (traditionally called the ``closed'' color $\col{c}$ and the ``open'' color $\col{o}$) such that operations with a closed output may only have closed inputs and are given by $\PP$.
  Indeed, given a relative operad $\QQ$ over $\PP$, we can define a colored operad $\QQ'$ by:
  \begin{align*}
    \QQ'(\underbrace{\col{o}, \dots, \col{o}}_{U}, \underbrace{\col{c}, \dots, \col{c}}_{V}; \col{c})
     & =
    \begin{cases}
      \PP(V), & \text{if } U = \varnothing, \\
      0,      & \text{otherwise}.
    \end{cases}
    \\
    \QQ'(\underbrace{\col{o}, \dots, \col{o}}_{U}, \underbrace{\col{c}, \dots, \col{c}}_{V}; \col{o})
     & = \QQ(U,V).
  \end{align*}
  and conversely.
\end{remark}

Given a (one-colored) cooperad $\CC$, a relative cooperad over $\CC$ is defined dually as a bisymmetric collection $\DD$ equipped with structure maps (which are morphisms of CDGAs):
\begin{align}
  \circ^{\vee}_{T} : \DD(U,V)
   & \to \DD(U, V/T) \otimes \CC(W)
   & T \subset V;                      \\
  \circ^{\vee}_{W,T} : \DD(U, V \sqcup T)
   & \to \DD(U/W, V) \otimes \DD(W, T)
   & W \subset U.
\end{align}

Finally, a comodule over a relative $\CC$-cooperad $\DD$ is given by a bisymmetric collection $\NN$ equipped with structure maps (which are morphisms of CDGAs):
\begin{align}
  \circ^{\vee}_{T} : \NN(U,V)
   & \to \NN(U, V/T) \otimes \CC(T)
   & T \subset V;                      \\
  \circ^{\vee}_{W,T} : \NN(U, V \sqcup T)
   & \to \NN(U/W, V) \otimes \DD(W, T)
   & W \subset U.
\end{align}

We can also define relative Hopf cooperads as relative cooperads in the category of CDGAs.

\subsection{Operads over spaces and right multimodules}\label{sec:operads_multimodules}

Let $X$ be a topological space.
We will frequently consider operads, operadic right modules and relative operads in the category $\mathrm{Top}/X$ of spaces over $X$, equipped with the symmetric monoidal product the categorical product $\times_X$.
Dually, for some CDGA $A$, one may consider the under-category $A/\mathrm{CDGA}$ of CDGAs under $A$ equipped with the symmetric monoidal product $\otimes_{A}$, and then consider (relative) Hopf cooperads in that category.

Next, let us recall the notion of operadic right multimodule\index{operadic right multimodule} from~\cite{CDW17}.
To this end fix an operad $\PP$ in spaces over $X$.
A right operadic $\PP$-multimodule is a collection of spaces $\op M$ equipped with maps
\beq{equ:MMtoX}
\op M(U) \to X^U,
\eeq
and operadic right actions
\begin{align}
  \circ_{T} : \op M(U/T) \times_X \PP(T)
   & \to \op M(U)
   & T \subset U,
\end{align}
where the fiber product is taken relative to the map $\op M(U/T)\to X$ corresponding to $T$.
The maps $\circ_{T}$ are required to satisfy natural compatibility relations, which are obvious generalizations of the the axioms required from an operadic right action.
Operadic right $\PP$-modules are trivial special cases of operadic right $\PP$-multimodules, for which the map \eqref{equ:MMtoX} takes values in the diagonal $X\subset X^U$.

Next consider a relative operad $\QQ$ over $\PP$, such that the pair $(\PP,\QQ)$ generates a two-colored operad in the category of spaces over $X$.
Then we may naturally extend the notion of right operadic multimodule to the colored setting as follows.
We say that a ($\PP$-relative) right operadic $\QQ$-multimodule $\op M$ is a right operadic $\QQ$-multimodule in operadic right $\PP$-multimodules.
More explicitly, such data consists of a bisymmetric collection $\op M$, equipped with maps
\[
  \op M(U,V) \to X^U\times X^V,
\]
and with right actions
\begin{align}
  \circ_{T} : \op M(U, V/T) \times_X \PP(T)
   & \to \op M(U,V)
   & T \subset V;             \\
  \circ_{W,T} : \op M(U/W, V) \times_X \QQ(W,T)
   & \to \op M(U, V \sqcup T)
   & W \subset U.
\end{align}
Here the fiber products $\times_X$ are taking relative to the map $\op M(U, V/T)\to X$ corresponding to $T$, or respectively to the map $\op M(U/W, V) \to X$ corresponding to $W$.
The composition morphisms again are required to satisfy natural compatibility relations.

To avoid too clumsy notation, we will call $\op M$ as above just a right $(\PP,\QQ)$-multimodule for short.

\begin{example}
  For the purposes of this paper, the relevant examples of the notions above are as follows.
  The fiberwise little disks operad $\FM_n^M$ defined in Section~\ref{sec:fiberwiseLD} is an operad in spaces over the smooth manifold $X=M$ of dimension $n$. If $M$ has a non-empty boundary, then the half-space-configuration spaces of points in the tangent spaces of boundary points may be made into a relative operad $\SFM_n^M$ for $\FM_n^M$.
  (Note here that the maps $\SFM_n^M(U)\to M$ take values in $\p M\subset M$.)
  Finally the configuration spaces of points on $M$ can be compactified into a right $(\FM_n^M,\SFM_n^M)$-multimodule.
\end{example}

The notions above have evident dualizations.
Hence, given a CDGA $A$, a cooperad $\CC$ in CDGAs under $A$, and a relative cooperad $\DD$ over $\CC$, we will use the notion of right $(\PP,\QQ)$-multicomodules.
Such a $(\PP,\QQ)$-multicomodule $\NN$ is concretely a bisymmetric collection of CDGAs equipped with maps
\[
  A^{\otimes U}\otimes A^{\otimes V}
  \to
  \NN(U,V),
\]
and furthermore coactions
\begin{align}
  \circ^{\vee}_{T} : \NN(U,V)
   & \to \NN(U, V/T) \otimes_A \CC(T)
   & T \subset V;                        \\
  \circ^{\vee}_{W,T} : \NN(U, V \sqcup T)
   & \to \NN(U/W, V) \otimes_A \DD(W, T)
   & W \subset U
\end{align}
satisfying natural compatibility relations.

\subsection{A note on the cohomology of compact orientable manifolds with boundary}
\label{sec:note-cohom-comp}

\begin{lemma}\label{lem:abc}
  Let $M$ be a compact orientable manifold with boundary $\p M\neq \emptyset$.
  Then there exist collections of forms $\{\alpha_i\}_{i=1,\dots,k}$, $\{\beta_i\}_{i=1,\dots,k}$, $\{\gamma_i\}_{i=1,\dots,l}$ on $M$ with the following properties.
  \begin{itemize}
    \item The $\alpha_i$ and $\gamma_j$ together form (i.e., represent) a basis of $H(M)$, and in particular $d\alpha_i=d\gamma_j=0$.
    \item The (boundary values of) $\alpha_i$ and $\beta_j$ together form a basis of $H(\partial M)$, in particular $d\beta_j\mid_{\p M}=0$.
    \item The $\gamma_i$ and $d\beta_j$ together form a basis of $H(M,\p M)$.
    \item We have $\gamma_j\mid_{\p M}=0$. Together with the other properties above this means that the maps in the long exact sequence
          \[
            \cdots \to H(M,\p M) \xrightarrow{f} H(M) \xrightarrow{g} H(\p M) \xrightarrow{h} H(M,\p M)[-1] \cdots
          \]
          may be expressed in our basis as
          \begin{align}
            f(\gamma_i) & = \gamma_i & f(d\beta_j) & =[d\beta_j]=0 \\
            g(\alpha_i) & = \alpha_i & g(\gamma_j) & = 0           \\
            h(\beta_i)  & = d\beta_i & h(\alpha_j) & = 0 \, .
          \end{align}
    \item The choices above can be made such that the Poincaré duality pairing on $H(\p M)$ and the Poincaré--Lefschetz duality pairing between $H(M)$ and $H(M,\p M)$ satisfy the following.
          \begin{align}
            \label{equ:abc1}
            \int_{\p M} \beta_j \alpha_i & = \int_M d\beta_j \alpha_i = \delta_{ij}
            \\
            \label{equ:abc2}
            \int_{\p M} \alpha_i\alpha_j & =  \int_{\p M} \beta_i\beta_j  = 0
            \\
            \label{equ:abc3}
            \int_M \alpha_i \alpha_j     & = \int_M \alpha_i \gamma_j =  \int_M d\beta_i \gamma_j = 0.
          \end{align}
  \end{itemize}
\end{lemma}
\begin{proof}
  We choose the $\gamma_i$ to be representatives of a basis of the image of $f$ above.
  We choose (closed extensions of representatives of) a basis $\alpha_i$ of the image of $g$.
  Irrespective of the choice of such representatives we have by Stokes theorem
  \[
    \int_{\p M}\alpha_i \alpha_j = \int_{M}d(\alpha_i \alpha_j )=0.
  \]
  We may extend the set $\{\alpha_i\}_{i=1,\dots,k}$ to a basis of $H(\p M)$ by adding elements represented by some forms $\{\beta_j\}_{j=1,\dots,k'}$.
  We understand these forms as boundary values of forms $\beta_j$ defined on $M$, by picking some (possibly non-closed) continuation to $M$. By construction of the cohomology long exact sequence above the $d\beta_j$ then span the image of $h$.
  We have thus constructed bases of $H(M)$, $H(M,\p M)$ and $H(\p M)$ as asserted above.
  The non-degeneracy of the pairing between $H(M)$ and $H(M,\p M)$ implies that the two spaces must be of equal dimension in complementary degrees, and hence $k'=k$. Retrospectively, we may hence have picked the $\beta_j$ so as to form a dual basis to the $\alpha_i$, i.e., such that \eqref{equ:abc2} and the second equality of \eqref{equ:abc1}
  hold.
  By Stokes' Theorem we also have that
  \beq{equ:abcx}
  \int_M d\beta_i \gamma_j = 0.
  \eeq
  None of the above properties are altered if we add to the $\alpha_i$ arbitrary linear combinations of the $\gamma_j$'s and $d\beta_j$'s.
  By \eqref{equ:abc1} we may add suitable linear combinations of the $\beta_j$ to ensure that
  \[
    \int_M \alpha_i \alpha_j = 0.
  \]
  (Here we denote the modified $\alpha_i$ by $\alpha_i$ again, i.e., we assume that suitable choices had already been taken from the start.)
  Finally we note that by Poincaré--Lefschetz duality and \eqref{equ:abcx} the matrix $(\int_M\gamma_i\wedge \gamma_j)_{i,j}$ must be non-degenerate and hence we may add suitable linear combinations of $\gamma_j$'s to the $\alpha_i$'s to ensure that \eqref{equ:abc3} holds.
\end{proof}

\section{Poincaré--Lefschetz duality models}
\label{sec.pretty-nice}

\subsection{Motivation}
\label{sec.motivation}

The Lambrechts--Stanley model from~\cite{Idrissi2018b} relied on a particular feature of closed orientable manifolds: Poincaré duality, i.e., the fact that for all $k \in \Z$, the pairing $H^{k}(M) \otimes H^{n-k}(M) \to \R$ given by $x \otimes y \mapsto \int_{M} x \wedge y$ is non-degenerate.
On the model-theoretic level, a Poincaré duality CDGA is a CDGA $P$ equipped with an orientation, i.e., a chain map $\varepsilon_{P} : P \to \R[-n]$ (in other words a linear map $P^{n} \to \R$ vanishing on coboundaries) such that the induced pairing $P^{k} \otimes P^{n-k} \to \R, \, x \otimes y \mapsto \varepsilon_{P}(xy)$ is non-degenerate for all $k \in \Z$.
Lambrechts and Stanley~\cite{LambrechtsStanley2008} proved that any simply connected manifold admits a Poincaré duality model.

The cohomology of a compact orientable manifold with boundary does not satisfy Poincaré duality.
Instead, it satisfies Poincaré--Lefschetz duality: for all $k \in \Z$, there is a non-degenerate pairing $H^{k}(M) \otimes H^{n-k}(M,\partial M) \to \R$ given by $x \otimes y \mapsto \int_{M} x \wedge y$.
Cordova Bulens, Lambrechts, and Stanley~\cite{CordovaBulensLambrechtsStanley2015} defined ``surjective pretty models'' to encode such a duality, roughly speaking when $M$ is obtained as the complement of a high-codimensional sub-polyhedron in a closed manifold.
In this section, we define a more flexible notion of pairs of CDGAs satisfying an appropriate version of Poincaré--Lefschetz, and which we call ``Poincaré--Lefschetz duality pairs''.

The rough idea is the following.
Suppose we have an orientable compact manifold with boundary $(M,\partial M)$.
Then we first want a model of the pair, i.e., a model of the inclusion, which should be given by a morphism of CDGAs $\lambda : B \to B_{\partial}$ where $B$ is a model of $M$ and $B_{\partial}$ is a model of $\partial M$.
In Poincaré--Lefschetz duality, we want to say something about the relative cohomology $H^{*}(M, \partial M)$.
On the level of models, relative cohomology would be represented by the homotopy kernel of $\lambda$.
The homotopy kernel is too big for duality, in general; however, if $\lambda$ is surjective, then the homotopy kernel is the kernel $K = \ker(\lambda)$.
We thus ask for $\lambda$ to be surjective.
If $\dim B_{\partial} \neq 0$ (i.e., if $\partial M \neq \varnothing$), then $\dim K < \dim B$, and there is no way of producing a non-degenerate pairing between $K$ and $B$.
As $K$ is a sub-CDGA of $B$, we should morally be considering a quotient $P = B / I$, where we kill enough elements to compensate what we have removed from $B$ to get $K$.
Poincaré--Lefschetz duality is then the statement that we can find an appropriate $I$ such that the quotient map $B \to P$ is a quasi-isomorphism and such that the pairing between $B$ and $K$ descends to a non-degenerate pairing between $P$ and $K$.
The final picture to keep in mind is the following:
\begin{equation}
  \label{eq.zigzag2}
  \begin{tikzcd}
    {} &
    K \coloneqq \ker \lambda
    \ar[d, hook]
    \ar[dl, bend right, dashed, leftrightarrow, "\text{\tiny non degen.}" {sloped,above}, "\text{\tiny pairing}" {sloped,below}]
    \\
    P \coloneqq B/I
    &
    B \ar[d, "\lambda", two heads] \ar[l, "\pi", "\sim" swap, two heads]
    \\
    {}
    & B_{\partial}
  \end{tikzcd}
\end{equation}
In the next section, we formally define Poincaré--Lefschetz duality models and pairs, and we prove their existence under good conditions.

\subsection{Definition and existence}
\label{sec.definition}

For convenience, we first define cones:

\begin{definition}
  The mapping cone of a chain map $f : X \to Y$, denoted either by $\cone(f)$ or $Y \oplus_{f} X[1]$, is given as a graded vector space by $Y \oplus X[1]$, and the differential is given by $d(y,x) = (d_{Y}y + f(x), d_{X}x)$.
\end{definition}

The data of a chain map $\varepsilon : \cone(f) \to \R$ is equivalent to the data of two linear maps $\varepsilon : Y^{0} \to \R$ and $\varepsilon_{\partial} : X^{-1} \to \R$ such that $\varepsilon_{\partial} \circ d = 0$ and $\varepsilon \circ d = \varepsilon_{\partial} \circ f$, i.e., the ``Stokes formula''.
We will use the two points of view in the sequel.

\begin{definition}
  A Poincaré--Lefschetz duality pair\index{Poincaré--Lefschetz duality pair} (or ``PLD pair'' for short) of formal dimension $n$ is a CDGA morphism $\lambda : B \to B_{\partial}$ between two connected CDGAs, equipped with a chain map $\varepsilon : \cone(\lambda) \to \R[1-n]$, and such that:
  \begin{itemize}
    \item The pair $(B_{\partial}, \varepsilon_{B_{\partial}})$ is a Poincaré duality CDGA of formal dimension $n-1$;
    \item Let $K \coloneqq \ker \lambda$, and let $\theta_{B} : B \to K^{\vee}[-n]$ be defined by $\theta_{B}(b)(k) = \varepsilon(bk)$; then we require $\theta_{B}$ to be surjective and to be a quasi-isomorphism.
  \end{itemize}
\end{definition}

The following proposition motivates the definition:

\begin{proposition}
  Let $(B \xrightarrow{\lambda} B_{\partial}, \varepsilon)$ be a PLD pair.
  Then the quotient map $B \to P = B / \ker \theta_{B}$ is a quasi-isomorphism.
  Moreover, the pairing $P^{i} \otimes K^{n-i} \to \R$, given by $x \otimes k \mapsto \varepsilon_{B}(ak)$, is well-defined and non-degenerate for all $i \in \mathbb{Z}$; equivalently, the map $\theta : P \to K^{\vee}[-n]$ induced by $\theta_{B}$ is an isomorphism of $B^{\otimes 2}$-modules.
\end{proposition}

\begin{proof}
  Since $\theta_{B}$ is a quasi-isomorphism, its kernel $\ker \theta_{B}$ is acyclic, therefore $B \to P$ is indeed a quasi-isomorphism.
  The map $\theta : P = B / \ker \theta_{B} \to K^{\vee}[-n]$ is clearly well-defined.
  It is surjective because $\theta_{B}$ is surjective, and it is injective because we modded out by the kernel.
  Hence it is an isomorphism.
\end{proof}

\begin{definition}
  A Poincaré--Lefschetz duality model\index{Poincaré--Lefschetz duality model} of $(M, \partial M)$ is a PLD pair $\lambda : B \to B_{\partial}$ which is a model of the inclusion $\partial M \subset M$, in the sense that we can fill out the following diagram:
  \[ \begin{tikzcd}
      B \ar[d, "\lambda"] & R \ar[d, "\rho"] \ar[l, dashed, "\sim" swap] \ar[r, "\sim", "f" swap] & \OmPA^{*}(M) \ar[d, "\text{res}"] \\
      B_{\partial} & R_{\partial} \ar[l, dashed, "\sim" swap] \ar[r, "\sim", "f_\partial" swap] & \OmPA^{*}(\partial M)
    \end{tikzcd}
    . \]
\end{definition}

\begin{remark}
  We can also assume that the maps $f$, $f_{\partial}$ factor through the sub-CDGAs of trivial forms in order to integrate along the fibers of the canonical projections, see~\cite[Appendix~C]{CamposWillwacher2016}.
  Indeed, this sub-CDGA is quasi-isomorphic to the full CDGA $\OmPA^{*}(-)$.
\end{remark}

\begin{remark}
  Given a PLD model as in the definition, we then have $H^{*}(P) \cong H^{*}(M)$, $H^{*}(K) \cong H^{*}(M, \partial M)$, and the isomorphism $\theta$ induces Poincaré--Lefschetz duality $H^{*}(M) \cong H_{n-*}(M, \partial M)$.
\end{remark}

\begin{example}
  \label{exa.model-dn}
  Consider the manifold $M = D^{n}$ with boundary $\partial M = S^{n-1}$.
  Intuitively, we can see $M$ as a sphere with a (thick) point removed.
  We thus get a PLD model of $(M, \partial M)$ as follows.
  The Poincaré duality model of $\partial M = S^{n-1}$ is merely $B_{\partial} = H^{*}(S^{n-1}) = \R 1 \oplus \R v$, where $\deg v = n-1$ and $v^{2} = 0$.
  The model of $M = D^{n}$ is given by $B = \R 1 \oplus \R v \oplus \R w$, with $\deg w = n$, $d v = w$, and all nontrivial products vanish.
  The surjective model of the restriction, $\lambda : B \to B_{\partial}$, is given by $\lambda(1) = 1$, $\lambda(v) = v$, and $\lambda(w) = 0$.
  Finally the orientations are given by $\varepsilon_{\partial}(v) = 1$ and $\varepsilon(w) = 1$.

  The kernel $K = \ker(\lambda)$ is then one-dimensional, spanned by $w$.
  The map $\theta_{B} : B \to K^{\vee}[-n]$ is given by $\theta_{B}(1) = w^{\vee}$, $\theta_{B}(v) = \theta_{B}(w) = 0$.
  It follows that the quotient $P = B / \ker(\theta_{B})$ is also one-dimensional, spanned by the unit $1$; in other words, $P = H^{*}(D^{n})$.
  The induced non-degenerate pairing $P^{0} \otimes K^{n} \to \R$ is the obvious one.
\end{example}

\begin{example}
  We can generalize the previous example.
  Suppose that $M$ is obtained from a closed manifold $N$ by removing an open disk.
  Suppose that $N$ has a Poincaré duality model $Q$.
  Then $\partial M = S^{n-1}$ and we can again take $B_{\partial} = H^{*}(S^{n-1}) = \R 1 \oplus \R v$.
  The model $B$ is given by $B = (Q \oplus \R v, dv = \vol_{Q})$ where $\vol_{Q}$ is the volume form of $Q$.
  The map $\lambda$ is the sum of the augmentation of $Q$ and of $\lambda(v) = v$.
  The kernel $K$ is thus the augmentation ideal $\bar{Q}$, while the quotient $P$ is $Q / \vol_{Q}$, and the non-degenerate pairing between the two is induced from the non-degenerate pairing of $Q$ with itself.
\end{example}

\begin{proposition}
  \label{prop.more-general-model}
  Let $M$ be a simply connected $n$-manifold with simply connected boundary, and assume that $n \geq 7$.
  Then $(M, \partial M)$ admits a PLD model.
\end{proposition}

The proof is heavily inspired from the proof of the main theorem of~\cite{LambrechtsStanley2008}.

\begin{proof}
  We start with some surjective model $\rho : R \to R_{\partial}$ and we will build a PLD model out of it:
  \begin{equation}
    \begin{tikzcd}
      R \ar[d, two heads, "\rho"] \ar[r, "g", "\sim" swap] & \OmPA^{*}(M) \ar[d, "\operatorname{res}"] \\
      R_{\partial} \ar[r, "g_{\partial}", "\sim" swap] & \OmPA^{*}(\partial M)
    \end{tikzcd}
  \end{equation}

  We keep the terminology of~\cite{LambrechtsStanley2008}.
  We can find a surjective quasi-isomorphism $g_{\partial} : R_{\partial} \to B_{\partial}$, where $B_{\partial}$ is a Poincaré duality CDGA, by~\cite{LambrechtsStanley2008}.
  We let $K_{R}$ be the kernel of $\rho' \coloneqq g_{\partial} \circ \rho$ and we define the chain map $\varepsilon_{R} : \cone(\rho') \to \R[-n+1]$ as in the previous section.
  Let $\OO \coloneqq \ker \theta^{\vee}_{R} \subset K_{R}$ be the ideal of ``orphans'', i.e.\
  \begin{equation}
    \OO \coloneqq \ker \theta^{\vee}_{R} = \{ y \in K_{R} \mid \forall x \in R, \; \theta^{\vee}_{R}(y)(x) = \varepsilon_{R}(xy) = 0 \}.
  \end{equation}
  We could consider (for a moment) the new short exact sequence:
  \begin{equation}
    0 \to (K \coloneqq K_{R}/\OO) \to (B \coloneqq R/\OO) \xrightarrow{\lambda} B_{\partial} \to 0.
  \end{equation}
  There is an induced chain map $\varepsilon_{B} : \cone(\lambda) \to \R[-n+1]$, and $\theta^{\vee}_{B} : K \to B^{\vee}[-n]$ is injective because we killed all the orphans.
  Thus we do obtain an isomorphism $B / \ker \theta_{B} \cong K^{\vee}[-n]$ induced by $\theta_{B}$.

  The problem is that the ideal $\OO$ is not necessarily acyclic, thus $\lambda$ is not necessarily a model of $(M, \partial M)$ anymore.
  Indeed, by Poincaré--Lefschetz duality, all we know is that a cycle $o \in \OO$ is always the boundary of some element $z \in K$; but one may not always choose $z \in \OO$.
  The idea, just like in~\cite{LambrechtsStanley2008}, is to extend the CDGA $R$ by acyclic cofibrations (over $B_{\partial}$) in order to make the ideal of orphans acyclic.

  Thanks to our connectivity assumptions on the manifold and its boundary, we can assume that the model $\rho : R \to R_{\partial}$ of $(M, \partial M)$ and the chain map $\varepsilon_{R} : \cone(\rho') \to \R[-n+1]$ satisfy:
  \begin{itemize}
    \item $\rho$ is surjective, hence so is $\rho' \coloneqq g_{\partial} \circ \rho$, and we let $K_{R} = \ker(\rho')$;
    \item both $R$ and $R_{\partial}$ are of finite type;
    \item both $R$ and $R_{\partial}$ are simply connected, i.e., $R^{0} = R_{\partial}^{0} = \R$ and $R^{1} = R_{\partial}^{1} = 0$;
    \item we have $R_{\partial}^{2} \subset \ker d$ and $K^{2} \subset \ker d$;
    \item the morphisms $\theta_{R}$, $\theta^{\vee}_{R}$ induced by $\varepsilon_{R}$ are quasi-isomorphisms.
  \end{itemize}

  Note that since we require $\rho$ to be surjective, it is possible that not all elements of $R^{2}$ are cycles, because some of the classes from $H^{2}(\partial M)$ may need to be removed (see Example~\ref{exa.model-dn} for $M = D^{3}$).

  We say that the pair $(\rho,\varepsilon_{R})$ is a ``good'' pair if it satisfies all these assumptions.
  Let us say that its orphans are $k$-half-acyclic if $H^{i}(\OO) = 0$ for $n/2 +1 \leq i \leq k$.
  This condition is void when $k = n/2$, and if $k = n+1$ then Poincaré--Lefschetz duality implies that $\OO$ is acyclic (see~\cite[Proposition~3.6]{LambrechtsStanley2008}).

  We will now work by induction (starting at $k = n/2$) and we assume that we are given a good pair $(\rho,\varepsilon_{R})$ whose orphans are $(k-1)$-half-acyclic.
  We will build an extension:
  \begin{equation}
    \begin{tikzcd}
      0 \rar
      & K_{R} \rar \dar[hookrightarrow]{\sim}
      & R \rar{\rho} \dar[hookrightarrow]{\sim}
      & R_{\partial} \rar \dar{=}
      & 0
      \\
      0 \rar
      & \hat{K}_{R} \rar
      & \hat{R} \rar{\hat{\rho}}
      & R_{\partial} \rar
      & 0
    \end{tikzcd}
  \end{equation}
  and an extension $\hat{\varepsilon}_{R} : \cone(\hat{\rho}') \to \R[-n+1]$ of $\varepsilon_{R}$ such that $(\hat{\rho},\hat{\varepsilon})$ is good and its orphans are $k$-half-acyclic.

  We follow closely~\cite[Sections~4 and~5]{LambrechtsStanley2008}, adapting the proof where needed.
  Let $l = \dim(\OO^{k} \cap \ker d) - \dim(d(\OO^{k-1}))$ and choose $l$ linearly independent cycles $\alpha_{1}, \dots, \alpha_{l} \in \OO^{k}$ such that
  \begin{equation}
    \OO^{k} \cap \ker d = d(\OO^{k-1}) \oplus \langle \alpha_{1}, \dots, \alpha_{l} \rangle.
  \end{equation}
  These are the obstructions to $\OO$ being $k$-half acyclic.
  Because $\theta^{\vee}_{R}$ is a quasi-isomorphism and $\theta^{\vee}_{R}(\alpha_{i}) = 0$, there exists $\gamma'_{i} \in K^{k-1}$ such that $d\gamma'_{i} = \alpha_{i}$.

  Let $m \coloneqq \dim H^{*}(R) = \dim H^{*}(K_{R})$, and choose cycles $h_{1}, \dots, h_{m} \in R$ such that $([h_{1}], \dots, [h_{m}])$ is a basis of $H^{*}(R)$.
  By Poincaré--Lefschetz duality there exists cycles $h'_{1}, \dots, h'_{m}$ in $K_{R}$ such that $\varepsilon(h_{i} h'_{j}) = \delta_{ij}$ (and these form a basis for $H^{*}(K)$).
  Let
  \begin{equation}
    \gamma_{i} \coloneqq \gamma'_{i} - \sum_{j} \varepsilon(\gamma'_{i} h_{j}) h'_{j} \in K^{k-1},
  \end{equation}
  and let $\Gamma$ be the subspace of $K^{k-1}$ generated by the $\gamma_{i}$.
  Then a proof similar to the proof of~\cite[Lemma~4.1]{LambrechtsStanley2008} shows that $d\gamma_{i} = \alpha_{i}$, and if $y \in R$ is a cycle, then $\varepsilon(\gamma_{i} y) = \theta^{\vee}(\gamma_{i})(y) = 0$.

  Now let
  \begin{equation}
    \hat{R} \coloneqq \bigl( R \otimes S(c_{1}, \dots, c_{l}, w_{1}, \dots, w_{l}), dc_{i} = \alpha_{i}, dw_{i} = c_{i} - \gamma_{i} \bigr),
  \end{equation}
  where the $c_{i}$ and the $w_{i}$ are new variables of degrees $k-1$ and $k-2$.

  Extend $\rho$ to $\hat{\rho} : \hat{R} \to R_{\partial}$ by declaring that $\hat{\rho}(c_{i}) = \hat{\rho}(w_{i}) = 0$.
  It follows that
  \begin{equation}
    \hat{K}_{R} \coloneqq \ker \hat{\rho} = (K_{R} \otimes S(c_{i}, w_{j})_{1 \leq i, j \leq l}, d).
  \end{equation}
  It is not hard to see that $R \hookrightarrow \hat{R}$ is an acyclic cofibration (compare with~\cite[Lemma~4.2]{LambrechtsStanley2008}), and so is $K \hookrightarrow \hat{K}$.
  Finally, since all the $\gamma_{i}$ and $\alpha_{i}$ are in $\ker \hat{\rho}$, one can extend $\varepsilon_{R}$ to $\hat{\varepsilon}_{R} : \cone(\hat{\rho}) \to \R[-n+1]$ by formulas identical to~\cite[Equation~4.5]{LambrechtsStanley2008}, which works because $n \geq 7$.
  It is clear that $(\hat{\rho}, \hat{\varepsilon}_{R})$ is still a good pair.
  We let $\hat{\theta}_{R}$, $\hat{\theta}^{\vee}_{R}$ be the quasi-isomorphisms induced by the pairing, and we finally let $\hat{\OO} \coloneqq \ker \hat{\theta}^{\vee}_{R}$.

  It remains to check that $\hat{\OO}$ is $k$-half-acyclic, knowing that $\OO$ is $(k-1)$-half-acyclic.
  First, we can reuse the proofs of~\cite[Lemmas~5.2 and~5.3]{LambrechtsStanley2008} to check that if $i > n-k+2$ then $\OO^{i} \subseteq \hat{\OO}^{i}$, and if $i \in \{ k-2, k-1 \}$ then $\hat{\OO}^{i} \cap \ker d \subset \OO^{i} \cap \ker d$.
  The only difference is that instead of the hypothesis $d(R^{2}) = 0$, we instead have $d(K_{R}^{2}) = 0$.

  Now, just like~\cite{LambrechtsStanley2008}, we have several cases to check.
  If we have $k \geq n/2 + 2$, or if we have $n$ odd and $k = (n+1)/2 + 1$, then the proof goes through unchanged up to slight changes of notation.
  However, if $n$ is even and $k = n/2 + 1$, more significant adaptations are needed, even if the idea is the same.
  We'd like to check that $\hat{\OO}^{k} \cap \ker d \subset d(\hat{\OO}^{k-1})$.
  We already know that
  \begin{equation}
    \hat{\OO}^{k} \cap \ker d \subset \OO^{k} \cap \ker d = d(\OO^{k-1}) \oplus \langle \alpha_{i} \rangle.
  \end{equation}
  Since $\OO^{k-1} \subset \hat{\OO}^{k-1}$ (the proof is identical to~\cite[Lemma~5.7]{LambrechtsStanley2008}), it is sufficient to check that $\hat{\OO}^{k-1} \cap \langle \alpha_{i} \rangle = 0$.
  Let us write $j : K_{R} \to R$ for the inclusion, $\bar{R} \coloneqq R / \OO$, $\bar{K}_{R} \coloneqq K_{R} / \OO$, $n = 2m$ and $k = n/2 + 1 = m + 1$.
  Then we have long exact sequences in cohomology and morphisms between them:
  \begin{equation}
    \begin{tikzcd}
      0 \rar
      & H^{m}(R) \rar{\pi}
      & H^{m}(\bar{R}) \rar{\delta}
      & H^{m+1}(\OO) \rar
      & 0
      \\
      0 \rar
      & H^{m}(K) \rar{\pi} \uar{j}
      & H^{m}(\bar{K}_{R}) \rar{\delta} \uar{\bar{\jmath}}
      & H^{m+1}(\OO) \rar \uar{=}
      & 0
    \end{tikzcd}
  \end{equation}

  The space $H^{m+1}(\OO)$ is generated by the classes of the $\alpha_{i}$.
  We obtain a Section~$\sigma : H^{m+1}(\OO) \to H^{m}(\bar{K}_{R})$ of $\delta$ by letting $\sigma([\alpha_{i}]) \coloneqq [\gamma_{i}]$.

  Suppose that we have some nonzero element $\alpha \in \langle \alpha_{i} \rangle$; we'd like to show that it is not in $\hat{\OO}^{k}$, i.e., that it is not an orphan in $\hat{K}_{R}$.
  The pairing $H^{m}(\bar{K}_{R}) \otimes H^{m}(\bar{R}) \to \R$ induced by $\bar{\varepsilon}_{R}$ is non-degenerate, and $\sigma(\alpha) \neq 0$, hence there exists some $\beta \in H^{m}(\bar{R})$ such that $\varepsilon(\sigma(\alpha)\beta) \neq 0$.
  But as we saw before, for any $[h] = \sum x_{j} [h_{j}] \in H^{m}(R)$, $\varepsilon_{R}(\sigma(\alpha) h) = 0$, because $\varepsilon_{R}(\gamma_{i} h_{j}) = 0$.
  It follows that $\varepsilon_{R}(\sigma(\alpha) \sigma(\delta(\beta))) = \varepsilon_{R}(\sigma(\alpha) \beta) \neq 0$.
  If we write $\delta(\beta) = \sum_{j} \beta_{j} \gamma_{j} \in H^{m+1}(\OO)$, we can let $w = \sum_{j} \beta_{j} w_{j} \in \hat{R}$, and then by definition $\hat{\varepsilon}_{R}(\alpha w) = \varepsilon_{R}(\sigma(\alpha) \sigma(\delta(\beta))) \neq 0$, thus $\alpha$ is not an orphan.
  This completes the proof that $\hat{\OO}^{k-1} \cap \langle \alpha_{i} \rangle = 0$, and thus that $\hat{\OO}^{k} \cap \ker d \subset d(\hat{\OO}^{k-1})$.
  We have thus covered all the cases to prove that if $\OO$ is $(k-1)$-half-acyclic, then $\hat{\OO}$ is $k$-half-acyclic.

  By induction, we obtain a good pair $(\hat{\rho} : \hat{R} \to R_{\partial}, \hat{\varepsilon}_{R})$ whose ideal of orphans is $(n+1)$-half acyclic (and hence actually acyclic), obtained by a sequence of acyclic cofibrations from $R$.
  It then remains to define $B \coloneqq \hat{R} / \hat{\OO}$ and $\varepsilon_{B}$ to be the map induced by $\hat{\varepsilon}_{R}$ on the quotient to prove Proposition~\ref{prop.more-general-model}.
\end{proof}

Given a PLD model of $(M, \partial M)$, we obtain a diagram (similar to Equation~\eqref{eq.zigzag2}):
\begin{equation}
  \label{eq.nice-diagram}
  \begin{tikzcd}
    {} &
    K \coloneqq \ker \lambda
    \ar[d, hook]
    \ar[dl, bend right, dashed, leftrightarrow, "\text{\tiny non degen.}" {sloped,above}, "\text{\tiny pairing}" {sloped,below}]
    &
    \\
    P \coloneqq B / \ker \theta_{B} &
    B \ar[l, "\sim" swap, "\pi", two heads] \ar[d, "\lambda", two heads] &
    R \ar[l,"f", "\sim" swap] \ar[d, "\rho", two heads] \ar[r, "g" swap, "\sim"] &
    \OmPA^{*}(M) \ar[d, "\operatorname{res}"]
    \\
    {} &
    B_{\partial} &
    R_{\partial} \ar[l, "f_{\partial}", "\sim" swap] \ar[r, "g_{\partial}" swap, "\sim"] &
    \OmPA^{*}(\partial M)
  \end{tikzcd}
\end{equation}

\subsection{Diagonal classes}
\label{sec.diagonal-classes}

We now want to define diagonal classes in PLD pairs.
Let us first quickly recall the well-known case of closed manifolds.
First, let us fix a Poincaré duality CDGA $Q$ of formal dimension $n$.

\begin{definition}
  The diagonal cocycle $\Delta_{Q} \in Q^{\otimes 2}$ of degree $n$ is defined as follows.
  If $\{x_{i}\}$ is a graded basis of $Q$ and $\{x_{i}^{\vee}\}$ is the dual basis (satisfying $\varepsilon(x_{i} x_{j}^{\vee}) = \delta_{ij}$), then
  \begin{equation}
    \Delta_{Q} \coloneqq \sum_{i} (-1)^{|x_{i}|} x_{i} \otimes x_{i}^{\vee}.
  \end{equation}
  This definition is independent of the chosen basis.
\end{definition}

\begin{proposition}[See e.g.~{\cite[Section~8.1.4]{FelixOpreaTanre2008}}]
  The diagonal cocycle satisfies $d \Delta_{Q} = 0$, $(x \otimes 1) \Delta_{Q} = (1 \otimes x)\Delta_{Q}$ for all $x \in Q$, and its image under the switch map:
  \begin{equation}
    (-)^{21} : Q \otimes Q \to Q \otimes Q, \, x \otimes y \mapsto (-1)^{(\deg x)(\deg y)} y \otimes x,
  \end{equation}
  satisfies $\Delta_{Q}^{21} = (-1)^{n} \Delta_{Q}$.
  Moreover, if $\vol_{Q} \in Q^{n}$ is the only element satisfying $\varepsilon_{Q}(\vol_{Q}) = 1$, then the image of $\Delta_{Q}$ under the product $\mu_{Q} : Q^{\otimes 2} \to Q$ is given by $\chi(Q) \cdot \vol_{Q}$, where $\chi(Q)$ is the Euler characteristic.
\end{proposition}

Let us now fix until the end of the section a Poincaré--Lefschetz duality pair $(B \xrightarrow{\lambda} B_{\partial}, \varepsilon)$, with $K = \ker \lambda$.
Recall that we write $\theta \colon P \to K^{\vee}[-n]$ for the isomorphism of $B^{\otimes 2}$-modules induced by the surjective $\theta_{B} : B \to K^{\vee}[-n]$, where $P = B / \ker \theta_{B}$.
We want to generalize the diagonal classes defined for surjective pretty models.

\begin{definition}
  The volume form $\vol_{B} \in K \subset B$ is the unique element satisfying $\varepsilon_{B}(\vol_{B}) = 1$.
\end{definition}

\begin{proposition}
  Let $\vol_{B_{\partial}}$ is any lift of the volume form of $B_{\partial}$ to $B$.
  Then we have $d\vol_{B_{\partial}} = \vol_{B}$.
\end{proposition}
\begin{proof}
  This follows from the Stokes formula for PLD pairs and the uniqueness of volume forms when the orientations are fixed.
\end{proof}

The multiplication $\mu_{K} : K^{\otimes 2} \to K$ can be dualized, using $\theta$, into a morphism of $B^{\otimes 2}$-modules
\begin{equation}
  \delta : P \to P^{\otimes 2}[-n].
\end{equation}

\begin{definition}
  We define a representative of the diagonal class:\index{DeltaP@$\Delta_{P}$}
  \begin{equation}
    \Delta_{P} \coloneqq \delta(1_{P}) \in (P \otimes P)^{n}.
  \end{equation}
\end{definition}

\begin{proposition}
  The diagonal class satisfies $\Delta_{P}^{(21)} = (-1)^{n} \Delta_{P}$ and:
  \begin{equation}
    \forall x \in P, \; \delta(x) = (x \otimes 1) \Delta_{P} = (1 \otimes x). \Delta_{P}
  \end{equation}
\end{proposition}
\begin{proof}
  This follows respectively from the graded commutativity of $\mu_{K}$ and the fact that $\delta$ is a morphism of $B^{\otimes 2}$-modules.
\end{proof}

\begin{definition}
  Let $\{ x_{i} \}$ be some basis of $K$, and let $\{x_{i}^{\vee}\}$ be the dual basis of $P$, satisfying $\varepsilon_{B}(x_{i} x_{i}^{\vee}) = \delta_{ij}$.
  Then we define:\index{DeltaKP@$\Delta_{KP}$}
  \begin{equation}
    \label{def.delta-ak}
    \Delta_{KP} \coloneqq \sum (-1)^{|x_i|} x_{i} \otimes x_{i}^\vee \in (K \otimes P)^{n}.
  \end{equation}
\end{definition}

It is clear that $\Delta_{KP} \in K \otimes P$ is dual to the non-degenerate pairing $K \otimes P \to \R$.
The next proposition also follows directly from the definitions:

\begin{proposition}
  The cocycle $\Delta_{P}$ is the image of $\Delta_{KP}$ under the composite map $\pi \circ \iota \otimes \id : K \otimes P \to B \otimes P \to P \otimes P$.
  Moreover:
  \begin{equation}
    \forall x \in P, \sum_{(\Delta_{KP})} \varepsilon_{B}(x \Delta_{KP}') \Delta_{KP}'' = x.
  \end{equation}
  Denoting by $\mu_{P} \colon P \otimes P \to P$ the multiplication map, $\chi(P)$ the Euler characteristic, and $\vol_P$ the volume form, we have:
  \begin{equation}
    \mu_{P}(\Delta_{P}) = \chi(P) \pi(\vol_{B})
  \end{equation}
\end{proposition}

\begin{corollary}
  If $\partial M \neq \varnothing$, we have:
  \begin{equation}
    \label{eq.mu-delta-zero}
    \mu_{P}(\Delta_{P}) = 0.
  \end{equation}
\end{corollary}
\begin{proof}
  By degree reasons, $\vol_{B}$ is in the kernel of $\theta_{B} : B \to K^{\vee}[-n]$ when $\partial M \neq \varnothing$.
  Hence $\pi(\vol_{B}) = 0$ by definition.
\end{proof}

\section{Configuration spaces, their compactifications, and algebraic structures}
\label{sec:configspaces}
\label{sec.comp-conf-spac}

\subsection{Semi-algebraic realization}

Let $M$ be a compact manifold with boundary.
In order to apply the theory of PA forms of~\cite{HardtLambrechtsTurchinVolic2011} to $M$, we must endow it with a semi-algebraic structure, i.e., a smooth embedding of $M$ into $\R^{N}$ for some (big) $N$ such that $M$ is defined by finite unions of finite intersections of loci of polynomial (in)equalities.

To get such a diffeomorphic semi-algebraic realization of $M$, we first apply the Nash--Tognoli theorem to the double $\tilde M = M \cup_{\partial M} \bar{M}$, and we thus obtain some (semi-)algebraic realization of $\tilde{M}$ in $\R^N$~\cite{Nash1952,Tognoli1973}.
We have that $M$ is singled out by the equation $f>0$, where $f$ is some smooth function on a neighborhood of $\tilde M$ inside $\R^N$, regular at $\p M\subset \tilde M\subset \R^N$.
We then use Stone--Weierstrass to approximate $f$ and its first derivatives ``closely enough'' by an algebraic function $g$.
Our semi-algebraic realization of $M$ is then given by the defining algebraic equations for our realization of $\tilde M$, together with $g>0$.

Alternatively, a $C^1$ realization as a semi-algebraic set may be defined as follows.
First $M$ can be triangulated and is $C^{1}$-diffeomorphic to some simplicial complex in some $\R^{N}$.
Each simplex is a semi-algebraic set, and there is a finite number of them (because $M$ is compact), hence $M$ itself is semi-algebraic.

\subsection{(Axelrod--Singer--)Fulton--MacPherson operad \texorpdfstring{$\FM_n$}{FM\_n}}
\label{sec:axelr-sing-fult}

The Axelrod--Singer--Fulton--MacPherson operad $\FM_n$ introduced by Getzler and Jones~\cite{GJ} is an $E_n$ operad defined by compactifying the configuration spaces of points in $\R^n$.
More concretely, one compactifies the quotient space by the group action of scalings and translations\index{FMn@$\FM_{n}$}
\begin{equation}
  \FM_n(r) = \overline{\Conf_{\R^n}(r)/\R_{>0} \ltimes \R^n }.
\end{equation}
The compactification is such that $\FM_n$ comes with a natural stratification, the strata being indexed by rooted trees with leaf set $\{1,\dots ,r\}$.
Each tree here corresponds to a piece of the boundary in an iterated bordification, see Figure~\ref{fig:strata-fm-n} for an example.

\begin{figure}[htbp]
  \centering
  \begin{tikzpicture}[baseline=-1.5cm, scale=.6]
    \node  {}
    child {
        child { node {1} }
        child { node {2} }
        child {
            child {node {3} }
            child {
                child { node {4} }
                child { node {5} }
              }
          }
      };
  \end{tikzpicture}
  $\longleftrightarrow$
  \begin{tikzpicture}[baseline=1cm]
    \draw (0,0)--(3,0)--(4,2)--(1,2)--cycle;
    \node [int, label={left}:{$\scriptstyle 1$}] (v1) at (1,.5) {};
    \node [int, label={$\scriptstyle 2$}] (v2) at (1,1) {};
    \node [int] (va) at (2,.5) {};
    \begin{scope}[scale = .7, xshift=4cm,yshift=1.5cm]
      \draw (va) -- (0,0) (va) -- (3,0) (va)-- (4,2) (va) -- (1,2);
      \draw[fill=white] (0,0)--(3,0)--(4,2)--(1,2)--cycle;
      \node [int, label={$\scriptstyle 3$}] (v3) at (1,1) {};
      \node [int] (vb) at (2,.5) {};
      \begin{scope}[scale = .7, xshift=2cm,yshift=-4cm]
        \draw (vb) -- (0,0) (vb) -- (3,0) (vb)-- (4,2) (vb) -- (1,2);
        \draw[fill=white] (0,0)--(3,0)--(4,2)--(1,2)--cycle;
        \node [int, label={$\scriptstyle 4$}] (v4) at (1,1) {};
        \node [int, label={$\scriptstyle 5$}] (v5) at (2,.5) {};
      \end{scope}
    \end{scope}
  \end{tikzpicture}
  \caption{Strata for $\FM_n(r)$}
  \label{fig:strata-fm-n}
\end{figure}

For the original definition and more details we refer the reader to~\cite{GJ}.
More details and a gentler presentation can be found in~\cite{Sinha2004} or~\cite[Section~1.2]{Idrissi2018b}.

\subsection{Fiberwise little \texorpdfstring{$n$}{n}-disks operad}\label{sec:fiberwiseLD}
There is a natural $\SO(n)$-action on the operad $\FM_n$.
Given any principal $\SO(n)$-bundle $P\to B$, we may consider the fiberwise little $n$-disks operad
\begin{equation}
  \FM_n^P \coloneqq P \times_{\SO(n)} \FM_n.
\end{equation}
This is an operad in the category of spaces over $B$.
For example, if $B = \{*\}$ is a singleton (thus $P = \SO(n)$), then we obtain $\FM_{n}^{\{*\}} = \FM_{n}$.

A particular instance of this situation that will often be considered in this text is the case where $P$ is obtained as the oriented frame bundle $\Fr Y$ of a rank $n$ vector bundle $\mathbb R^n \hookrightarrow Y\to B$.
In that case, we use the notation $\FM_n^Y \coloneqq\FM_n^{\Fr Y}$.

In particular, given an oriented manifold $M$ denote by $F_M \coloneqq \Fr T_M\to  M$ the oriented frame bundle, and define\index{FMnM@$\FM_{n}^{M}$}
\begin{equation}
  \FM_n^M \coloneqq \FM_n^{F_M}.
\end{equation}
This is an operad in spaces over $M$, and can be understood as a version of the fiberwise little disks operad in the tangent space of $M$.
In~\cite{CDW17}, this operad is an important object to understand in order to model the action of the framed little discs operad on the configuration space of framed points of a closed manifold.

\begin{remark}
  Strictly speaking, we assume the vector bundles above to be endowed with a metric so that we can consider the $\SO(n)$ principal frame bundle instead of the $\mathrm{GL}(n)$ principal frame bundle.

  Using such a metric, one can consider the sphere bundle $S_Y\subset Y$ consisting of vectors of norm $1$.
  One observes that this bundle is canonically isomorphic to the space $\FM_n^Y(2)$ constructed above.
\end{remark}

\begin{remark}
  A this stage it is not essential to require the manifold $M$ to be oriented. In the unoriented case one just considers the full frame bundle $F_M^{unor}$ and the associated operad $\FM_n^M \coloneqq F_M^{unor}\times_{O(n)}\FM_n$, using the $O(n)$ action on $\FM_n$.
\end{remark}

\subsection{Compactified configuration space for closed manifolds}
\label{sec:comp-conf-space-closed}

Let $M$ be a closed manifold.
We denote by $\FM_M$\index{FMM@$\FM_{M}$} the Axelrod--Singer--Fulton--MacPherson type bordification of the configuration space of points on $M$ defined in~\cite{AxelrodSinger1994}.
The space $\FM_M(r)$ is defined by an iterated blowup (or rather, bordification) at the fat diagonal of $M^r$.
If we choose an algebraic realization of $M$ then $\FM_M(r)$ is again a semi-algebraic manifold.

The space $\FM_{M}(r)$ comes with a natural stratification.
Strata correspond to rooted trees with leaves labelled $1,\dots,r$.
The stratum corresponding to a given tree $T$ is given by configuration satisfying the following properties: points $i$ and $j$ are in a subtree if and only if they are infinitesimally close compared to the points outside this subtree.
See Figure~\ref{fig.strata-fmm} for an example of correspondence.

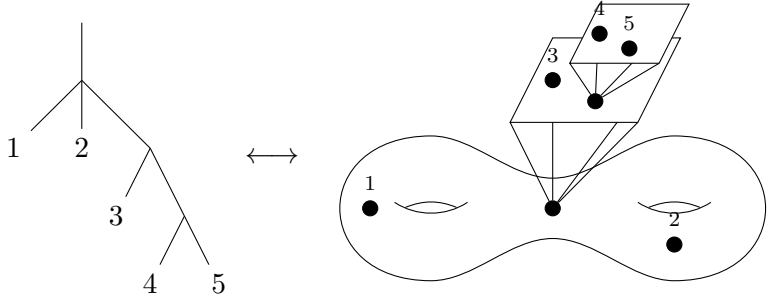
\begin{figure}[htbp]
  \centering
  \begin{tikzpicture}[baseline=-2cm, scale=.6]
    \node  {}
    child {
        child { node {1} }
        child { node {2} }
        child {
            child {node {3} }
            child {
                child { node {4} }
                child { node {5} }
              }
          }
      };
  \end{tikzpicture}
  $\longleftrightarrow$
  \begin{tikzpicture}[baseline=1cm, scale=.8]
    \draw (0,0)
    .. controls (-.7,0) and (-1.3,-.7) .. (-2,-.7)
    .. controls (-4,-.7) and (-4,1.7) .. (-2,1.7)
    .. controls (-1.3,1.7) and (-.7,1) .. (0,1)
    .. controls  (.7,1) and (1.3,1.7) .. (2,1.7)
    .. controls (4,1.7) and (4,-.7) .. (2,-.7)
    .. controls (1.3,-.7) and (.7,0)  .. (0,0);
    \begin{scope}[xshift=-2cm, yshift=.6cm, scale=1.2]
      \draw (-.5,0) .. controls (-.2,-.2) and (.2,-.2) .. (.5,0);
      \begin{scope}[yshift=-.07cm]
        \draw (-.35,0) .. controls (-.1,.1) and (.1,.1) .. (.35,0);
      \end{scope}
    \end{scope}
    \begin{scope}[xscale=-1, xshift=-2cm, yshift=.6cm, scale=1.2]
      \draw (-.5,0) .. controls (-.2,-.2) and (.2,-.2) .. (.5,0);
      \begin{scope}[yshift=-.07cm]
        \draw (-.35,0) .. controls (-.1,.1) and (.1,.1) .. (.35,0);
      \end{scope}
    \end{scope}
    \node [int, label={$\scriptstyle 1$}] (v1) at (-3,.5) {};
    \node [int, label={$\scriptstyle 2$}] (v2) at (2,-.1) {};
    \node [int] (va) at (0,.5) {};
    \begin{scope}[scale = .7, xshift=-1cm,yshift=2.75cm]
      \draw (va) -- (0,0) (va) -- (3,0) (va)-- (4,2) (va) -- (1,2);
      \draw[fill=white] (0,0)--(3,0)--(4,2)--(1,2)--cycle;
      \node [int, label={$\scriptstyle 3$}] (v3) at (1,1) {};
      \node [int] (vb) at (2,.5) {};
      \begin{scope}[scale = .7, xshift=2cm,yshift=2cm]
        \draw (vb) -- (0,0) (vb) -- (3,0) (vb)-- (4,2) (vb) -- (1,2);
        \draw[fill=white] (0,0)--(3,0)--(4,2)--(1,2)--cycle;
        \node [int, label={$\scriptstyle 4$}] (v4) at (1,1) {};
        \node [int, label={$\scriptstyle 5$}] (v5) at (2,.5) {};
      \end{scope}
    \end{scope}
  \end{tikzpicture}
  \caption{Strata for $\FM_M(r)$}
  \label{fig.strata-fmm}
\end{figure}

Recall from~\cite{CDW17} that a right operadic multimodule $\op X$ for an operad $\op P$ in spaces over some base $B$ is a collection of spaces over powers of the base
\begin{equation}
  \op X(r) \to B^r
\end{equation}
with a right $\Sigma_r$-action and operations
\begin{equation}
  \op X(r) \times_{B^r} \op P(s_1)\times \cdots \op P(s_r) \to \op X(s_1+\cdots +s_r)
\end{equation}
satisfying the usual equivariance, commutativity and associativity conditions for an operadic right action.
The notion of operadic right multimodule should be considered as ``the correct generalization'' of the notion of a right module for an operad over a base $B$. (The alternative extension as an operadic right module in the category of spaces over $B$ is unnatural, at least in the present context.)

Note that the spaces $\FM_M(r)$ naturally assemble into a right operadic multimodule over the operad $\FM_n^M$.
If the manifold $M$ is parallelized, then the fiberwise operad $\FM_n^M$ can be identified with $M\times \FM_n$, and the right action gives rise to a right $\FM_n$ action on $\FM_M$.

\subsection{Swiss Cheese operad and fiberwise Swiss Cheese operad}
\label{sec:swiss-cheese-operad}

Given a manifold with boundary $(M, \partial M)$ and finite sets $U, V$, define the colored configuration spaces:\index{ConfUV@$\Conf_{U,V}(M)$}
\begin{equation}
  \Conf_{U,V}(M) \coloneqq \{ c \in \Conf_{U \sqcup V}(M) \mid c(U) \subset \partial M, \; c(V) \subset \mathring{M} \}.
\end{equation}
In particular if $M = \Ha_{n} \subset \R^{n}$ is the upper half-space, then $\Conf_{U,V}(M)$, after modding it out by the group of translations preserving $\Ha_{n}$ and positive dilatations, can be compactified into a space $\SFM_{n}(U,V)$~\cite{Voronov1999}.\index{SFMn@$\SFM_{n}$}
It is a stratified manifold of dimension $n \cdot \# V + (n-1) \cdot \# U - n$ as soon as $\#U + 2 \#V \ge 2$, and it is reduced to a point otherwise.
One has homeomorphisms $\SFM_{n}(U, \varnothing) \cong \FM_{n-1}(U)$ and $\SFM_{n}(\underline{1}, \underline{1}) \cong D^{n-1}$, and homotopy equivalences $\SFM_{n}(\varnothing, V) \simeq \FM_{n}(V)$.

In this way, one obtains a relative operad $\SFM_{n}$ over $\FM_{n}$, weakly equivalent to the Swiss-Cheese operad~\cite{Voronov1999} where operations with no open inputs are allowed (see~\cite[Section~3]{HoefelLivernet2012} or~\cite[Remark~2.2]{Idrissi2017} for an explanation of the difference).
Also note that $\SFM_{n}$ comes equipped with a natural action of the orthogonal group $O(n-1)$, compatible with all structures.

Now for a general oriented manifold with boundary $(M, \partial M)$ of dimension $n$ we may define a relative version of the Swiss Cheese operad, akin to Section~\ref{sec:fiberwiseLD}.
First, as above, we define the operad $\FM_n^M=\Fr_M \ltimes_{\SO(n)} \FM_n$ in the category of spaces over $M$.
Then we consider the collection\index{SFMnM@$\SFM_{n}^{M}$}
\begin{equation}
  \SFM_n^M \coloneqq \Fr_{\p M} \ltimes_{\SO(n-1)} \SFM_n.
\end{equation}
It is naturally equipped with an operadic right action of $\FM_n^M$, i.e., with maps
\begin{equation}
  \SFM_n^M(u,r) \times_{\p M} \FM_n^M(s) \to \SFM_n^M(u,r+s-1)
\end{equation}
satisfying natural compatibility relations.
Furthermore, it carries a compatible structure of an operad in spaces over $M$, so that $\FM_n^M$ and $\SFM_n^M$ together assemble into a two colored operad in spaces over $M$.
(Here the maps $\SFM_n^M(r)\to M^r$ factor through $(\p M)^r$.)
This will be our model for the fiberwise Swiss Cheese operad.

Finally, note that for the present discussion we may go without an orientation on $M$ by just replacing the oriented frame bundle by its unoriented version and replacing $\SO(n-1)$ by $O(n-1)$ above.

\subsection{Compactified configuration space for manifolds with boundary}
\label{sec.constr-sfmm}

Let $M$ now be a compact manifold with boundary $\p M$.
One may define an Axelrod--Singer--Fulton--MacPherson type compactification $\SFM_M(U,V)$ of the configuration space $\Conf_{U,V}(M)$.
This is a stratified manifold of dimension $(n-1) \cdot \# U + n \cdot \# V$.

The quickest way to define it is the following:
We consider the double $\tilde M=M\sqcup_{\p M}M^{op}$. The double comes with a natural orientation reversing involution
\begin{equation}
  I : \tilde M \to \tilde M
\end{equation}
interchanging $M$ and $M^{op}$. The operation $I$ induces a similar operation $I:\FM_{\tilde M}\to \FM_{\tilde M}$ on the compactified configuration spaces of points.

We now define the compactification $\SFM_M(r,s)\subset \FM_{\tilde M}(2r+s)$\index{SFMM@$\SFM_{M}$} to be the subset of configurations $c=(x_1,\dots,x_r,x_1',\dots,x_r',y_1,\dots,y_s)$ that satisfy:
\begin{itemize}
  \item $I c=\operatorname{swap}_r(c) \coloneqq (x_1',\dots,x_r',x_1,\dots,x_r,y_1,\dots,y_s)$
  \item $\pi_j c\in M\subset \tilde M$, where $\pi_j:\FM_{\tilde M}\to \tilde M$ is the projection onto the position of the $j$-th point for $1 \le j \le r$.
\end{itemize}

One can see that space $\SFM_M$ comes again with a natural stratification.
The strata are labelled by rooted trees with two kinds of edges, say ``bulk edges'' (drawn in full) and ``boundary edges'' (dashed), such that every descendant of a bulk edge is again a bulk edge.
Once again, two points are in a subtree iff they are infinitesimally close compared to all the other points.
Figure~\ref{fig.strata-sfmm} provides an example of the correspondence of trees and strata.

\begin{figure}[htbp]
  \centering
  \begin{tikzpicture}[baseline=-3cm]
    \node  {}
    child {
        child {
            child { node {1} }
            child { node {2} }
          }
        child { node {3} }
        child[dashed] {
            child[solid] {node {4} }
            child {
                child[solid] { node {5} }
                child[solid] { node {6} }
                child[dashed] { node {$\bar 1$} }
              }
          }
      };
  \end{tikzpicture}
  $\longleftrightarrow$
  \begin{tikzpicture}[baseline=1cm]
    \draw (0,0)
    .. controls (-.7,0) and (-1.3,-.7) .. (-2,-.7)
    .. controls (-4,-.7) and (-4,1.7) .. (-2,1.7)
    .. controls (-1.3,1.7) and (-.7,1) .. (0,1);
    \begin{scope}[xshift=-2cm, yshift=.6cm, scale=1.2]
      \draw (-.5,0) .. controls (-.2,-.2) and (.2,-.2) .. (.5,0);
      \begin{scope}[yshift=-.07cm]
        \draw (-.35,0) .. controls (-.1,.1) and (.1,.1) .. (.35,0);
      \end{scope}
    \end{scope}
    \draw (0,.5) ellipse (.1cm and .5cm);
    \node [int, label={$\scriptstyle 3$}] (v1) at (-3,.5) {};
    \node [int] (va) at (-0.1,.5) {};
    \node [int] (vc) at (-2,1.4) {};
    \begin{scope}[scale = .7, xshift=-1cm,yshift=2.75cm]
      \draw (va) -- (0,0) (va) -- (3,0) (va)-- (4,2) (va) -- (1,2);
      \draw[fill=white] (0,0)--(3,0)--(4,2)--(1,2)--cycle;
      \node [int, label={$\scriptstyle 4$}] (v3) at (1,1) {};
      \draw (1.5,0)--(2.5,2);
      \node [int] (vb) at (1.8,.6) {};
      \begin{scope}[scale = .7, xshift=2cm,yshift=2cm]
        \draw (vb) -- (0,0) (vb) -- (3,0) (vb)-- (4,2) (vb) -- (1,2);
        \draw[fill=white] (0,0)--(3,0)--(4,2)--(1,2)--cycle;
        \draw (1.5,0)--(2.5,2);
        \node [int, label={$\scriptstyle 5$}] (v4) at (1,.8) {};
        \node [int, label={$\scriptstyle 6$}] (v5) at (1.4,1.4) {};
        \node [int, label={$\scriptstyle \bar 1$}] (v5) at (2,1) {};
      \end{scope}
    \end{scope}
    \begin{scope}[scale = .5, xshift=-6cm,yshift=4cm]
      \draw (vc) -- (0,0) (vc) -- (3,0) (vc)-- (4,2) (vc) -- (1,2);
      \draw[fill=white] (0,0)--(3,0)--(4,2)--(1,2)--cycle;
      \node [int, label={$\scriptstyle 1$}] (v4) at (1,1) {};
      \node [int, label={$\scriptstyle 2$}] (v5) at (2,.5) {};
    \end{scope}
  \end{tikzpicture}
  \caption{A tree indexing a stratum of $\SFM_M(1,6)$ and a point in that stratum}
  \label{fig.strata-sfmm}
\end{figure}

The spaces $\SFM_M$ come equipped with a natural operadic right action of the relative Swiss Cheese operad using insertion of configurations.
\begin{align}
  \circ_{T} : \SFM_{M}(U, V/T) \times_M \FM_{n}^M(T)
   & \to \SFM_{M}(U,V)
   & T \subset V;                \\
  \circ_{W,T} : \SFM_{M}(U/W, V) \times_{\p M} \SFM_{n}^M(W,T)
   & \to \SFM_{M}(U, V \sqcup T)
   & W \subset U.
\end{align}
If $M$ is framed, then this relative Swiss Cheese action gives rise to a to a right $\SFM_{n}$-action:

If we assume that we have chosen an algebraic realization of $\tilde M$ such that $I$ is (semi-)algebraic, then $\SFM_M$ is also a semi-algebraic manifold.
We can thus apply the theory of PA forms (see~\cite{HardtLambrechtsTurchinVolic2011} or~\cite[Section~1.2]{Idrissi2018b}) and consider the collections of CDGAs $\OmPA^{*}(\SFM^{M}_{n})$ and $\OmPA^{*}(\SFM_{M})$.

\begin{remark}\label{rmk:almost}
  Even though the functor $\OmPA^{*}$ is contravariant, $\OmPA^{*}(\SFM_{n})$ does not form a Hopf cooperad, and $\OmPA^{*}(\SFM_{M})$ does not form a Hopf right comodule, because the Künneth quasi-isomorphism goes in the wrong direction.
  Cocomposition structure maps are replaced by a zigzag of maps such as:
  \begin{multline}
    \OmPA^{*}(\SFM^{M}_{n}(U, V)) \xrightarrow{\circ_{T}} \OmPA^{*}(\SFM^{M}_{n}(U,V/T) \times_{M} \FM^{M}_{n}(T)) \\ \xleftarrow{\sim} \OmPA^{*}(\SFM_{n}^{M}(U,V/T)) \otimes_{\OmPA^{*}(M)} \OmPA^{*}(\FM_{n}^{M}(T)),
  \end{multline}
  and similarly for the other structure maps.
  We can call such a collection of CDGAs with such a structure a homotopy Hopf cooperad/comodule.
  Nevertheless, according to~\cite[Section~3]{LambrechtsVolic2014}, we may define a CDGA model of $\SFM_{n}^{M}$ (or $\SFM_{M}$) as a collection of CDGAs $A(U,V)$ equipped with maps $A(U,V) \to \OmPA^{*}(\SFM_{n}^{M}(U,V))$ (or $\SFM_{M}(U,V)$) making the appropriate diagrams commute.
  We will call such a collection of maps an ``almost morphism'' of (homotopy) Hopf cooperads.
\end{remark}

\subsection{Compactified configuration spaces and actions}\label{sec:typeIIcompactification}

\subsubsection{Compactification of $\Conf_r(\p M\times I)$}
We consider $\Conf_r(N \times \R_{>0})$. This space comes with a natural $\R_{>0}$-action by multiplication on the second factor(s).
We define a compactification (recall that the ``a'' stands for ``algebra''):\index{aFMN@$\aFM_{N}$}
\begin{equation}
  \aFM_{N}(r) = \overline{ \Conf_r(N \times \R_{>0})/\R_{>0} }
\end{equation}
of the quotient as follows.
Suppose we are given an embedding of $N$ as a semi-algebraic submanifold of $\R^{D}$ for some big $D$.
Let us define an embedding of $\Conf_r(N \times \R_{>0})/\R_{>0}$ into
\begin{equation}\label{eq:SA embedding of algebra}
  \bigl( S^D \bigr)^{\binom r 2}\times ( \mathbb{RP}^2 )^{\binom{r}{3}} \times N^r,
\end{equation}
where the space $S^D$ is the $D$-dimensional sphere viewed as $\bigl( (\R^D \times \R) \setminus \{\vec{0}\} \bigr) / \R_{>0}$.
Moreover, $\mathbb{RP}^2 = S^D / (\Z/2\Z)$ is the projective plane.
Given some configuration $(x_{1}, h_{1}, \dots, x_{r}, h_{r}) \in \Conf_{r}(N \times \R_{> 0})$, its image in the space above is given by:
\begin{equation}
  \prod_{i<j}
  \frac{\alpha_{ij}}{\|\alpha_{ij}\|}
  \times
  \prod_{i<j<k} \bigl\lbrack \alpha_{ij} : \alpha_{jk} : \alpha_{ik} \bigr\rbrack
  \times \prod x_i,
\end{equation}
where we define vectors $\alpha_{ij} = (h_j x_j - h_i x_i) \oplus (h_j - h_i)$ as elements of $\R^D \oplus \R \cong \R^{D+1}$, which we then project down to the sphere (and similarly for the other indices).
The space $\aFM_{N \times I}(r)$ is then defined to be the closure of the image of this map.

The space $\aFM_{N}(r)$ is a compact SA manifold.
Just like the other compactifications, it can be decomposed into strata.
These strata correspond to rooted trees with leaves labeled by $\{1, \dots, r, *\}$, where $*$ is an extra index representing, roughly speaking, the component $N \times \infty$.
As usual, two leaves belong to a subtree if they are infinitesimally close compared to all the points outside this subtree.
In particular, if leaf $i$ is in a subtree with $*$, it means that $h_{i} / h_{j} = +\infty$ for all $j$ outside this subtree, in other words that the $i$th point is ``infinitesimally close'' to the component $N \times \infty$ compared to the $j$th point.

We have natural gluing maps
\begin{equation}\label{eq.alg-ofm-n-i}
  \aFM_{N}(r)\times \aFM_{N}(s)
  \to
  \aFM_{N}(r+s).
\end{equation}
obtained by gluing the boundary component $N \times \infty$ of $\aFM_{N}(r)$ with the boundary component $N \times 0$ of $\aFM_{N}(s)$ (see the Introduction for a figure).
This gluing is naturally associative and gives the collection $\aFM$ the structure of an algebra object in collections, with the unit given by the unique point of $\aFM_{N}(0)=*$.
On strata, the multiplication map corresponds to grafting the second tree on the leaf labeled by $*$ of the first tree.

Furthermore, we have a right operadic (multimodule-)action of the operad $\FM_n^{TN\times \R}$ in spaces over $N$, compatible with the algebra structure.
Hence, overall the collection $\aFM_{N}$ becomes an algebra object in the category of right $\FM_n^{TN\times \R}$-multimodules.

\subsubsection{Compactifications of $\Conf_{r,0}(M)$ }
\label{sec.ofmm}

Similarly we define compactifications $\mFM_M$\index{mFMM@$\mFM_{M}$} of $\Conf_{r,0}(M)$ (recall that the ``m'' stands for ``module'', i.e., module over the algebra $\aFM_{A_{\partial}}$).
To this end we fix a collar of the boundary $N = \partial M$.
We think of the collar as an ``infinitely extruded'' tube of shape $N\times \R_{>0}$.
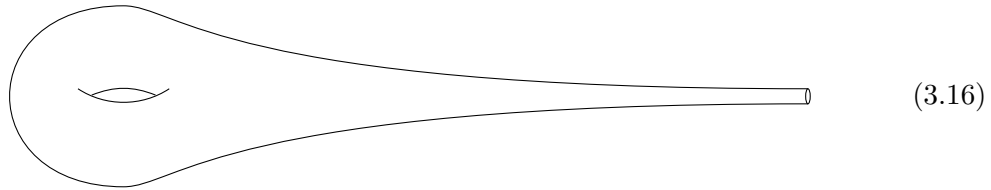
\begin{equation}
  \begin{tikzpicture}[yshift=-.5cm,baseline=-.65ex]
    \draw (7,0.4)
    .. controls (-.7,0.4) and (-1.3,-.7) .. (-2,-.7)
    .. controls (-4,-.7) and (-4,1.7) .. (-2,1.7)
    .. controls (-1.3,1.7) and (-.7,.6) .. (7,.6);
    \begin{scope}[xshift=-2cm, yshift=.6cm, scale=1.2]
      \draw (-.5,0) .. controls (-.2,-.2) and (.2,-.2) .. (.5,0);
      \begin{scope}[yshift=-.07cm]
        \draw (-.35,0) .. controls (-.1,.1) and (.1,.1) .. (.35,0);
      \end{scope}
    \end{scope}
    \draw (7,.5) ellipse (.03cm and .1cm);
  \end{tikzpicture}
\end{equation}
We then define the compactification $\mFM_M(r)$ of $\Conf_{r,0}(M)$ by adding the usual boundary pieces for strata corresponding to points bunching together in the interior of $M$, and we compactify in a manner similar to $\aFM_{\p M}$ on the boundary.
Formally, this is constructed by considering a collar $\p M$ and identifying it with  $\p M\times I$ and identifying its complement with $M^\circ$ the bulk of $M$. This gives us a continuous bijection  $$f\colon \bigsqcup_{k=0}^{r} \FM_{M^\circ}(k)\times \aFM_{\p M}(r-k)  \to \mFM_M(r).$$

The collection of spaces $\mFM_M$ has the following structure:
\begin{itemize}
  \item A right $\FM_n^M$-action of the fiberwise little $n$-cubes operad.
  \item An action of the algebra object $\aFM_{N}$ by gluing at the collar naturally compatible with the right operadic $\FM_n^M$-action.
\end{itemize}

\subsection{Comparison of compactifications: configuration spaces of two points }
\label{sec:FM2discussion}

Of particular importance for us will be the spaces $\SFM_M(2,0)$, $\aFM_{N}(2)$, and $\mFM_M(2)$, because our models will be built from forms pulled back from these spaces via forgetful maps.
For this reason, and as an example, we shall discuss these spaces in detail in the present section.

\subsubsection{$\aFM_{N}(2)$}

In the simplest case, $\aFM_{N}(2)$, we have four strata (see Section~\ref{sec:typeIIcompactification}), corresponding to the trees of Table~\ref{tab.strata-ofm-n-i}.

\begin{table}[htbp]
  \centering
  \footnotesize
  \tikzset{every picture/.style={scale=.5}}
  \begin{tabular}{ccccc}
    \toprule
    Stratum
     & I                           & II & III & IV \\
    \midrule
    Codim.
     & 0                           & 1  & 1   & 1  \\
    {}
     & \begin{tikzpicture}
      \node {}
      child {
          child { node {1} }
          child { node{2} }
          child {node {*} }
        };
    \end{tikzpicture}
     & \begin{tikzpicture}
      \node {}
      child {
          child { node {1}}
          child {
              child { node{2} }
              child {node {*} }
            }
        };
    \end{tikzpicture}
     & \begin{tikzpicture}
      \node {}
      child {
          child { node {2}}
          child {
              child { node{1} }
              child {node {*} }
            }
        };
    \end{tikzpicture}
     & \begin{tikzpicture}
      \node {}
      child {
          child { node {*}}
          child {
              child { node{1} }
              child {node {2} }
            }
        };
    \end{tikzpicture}
    \\
    \bottomrule
  \end{tabular}
  \caption{Boundary strata of $\aFM_{N}$}
  \label{tab.strata-ofm-n-i}
\end{table}

Let us quickly recall the notation.
The leaves $1$ and $2$ represent the two points of the configuration, while the leaf labeled by $*$ represents roughly speaking the component $N \times \infty$.
Two leaves belong to a subtree if they are infinitesimally close compared to all the points outside this subtree.
Thus, stratum~I is the interior of $\aFM_{N}$, with the two points in general position.
Stratum~II represents (roughly speaking) configurations $((x_{1},t_{1}),(x_{2},t_{2})) \in (N \times \R_{>0})^{2}$ where $t_{2}/t_{1}$ is ``infinitesimally close'' to $\infty$.
Similarly, stratum~III represents such configurations with $t_{2}/t_{1}$ infinitesimally close to $0$.
Finally, stratum~IV represents configurations with $(x_{1},t_{1})$ infinitesimally close to $(x_{2},t_{2})$.
From these descriptions, it follows that stratum~I is $\Conf_{2}(N \times \R_{>0})$, strata~II and~III are both $N \times N$, and stratum~IV is the normal bundle of the submanifold $\Delta_N \times \{1\} \subset N^2 \times \R_{>0}$, where $\Delta_N \subset N^2$ is the diagonal.

The poset structure between the different strata by inclusion is reflected in trees by edge contraction, in the sense that a stratum corresponding to tree $T$ contains that corresponding to tree $T'$ iff $T$ may be obtained from $T'$ by iterated edge contraction.
We may draw a schematic picture of $\aFM_{N}(2)$ which reflects the adjacency correctly, with the notation $SN$ denoting the $S^n$-sphere bundle in $\R\times TN$, see Figure~\ref{fig:adj-afm}.
Mind that of course all strata are compact, which is not reflected in the schematic picture.

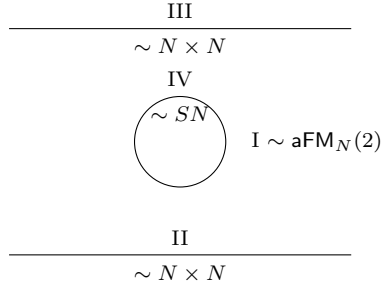
\begin{figure}[htbp]
  \centering
  \begin{tikzpicture}[scale=1.5, every node/.append style = {font=\scriptsize}]
    \draw
    (0,0) -- (3,0) node [midway, above] {II} node [midway, below] {$\sim N \times N$}
    (0,2) -- (3,2) node [midway, above] {III} node [midway, below] {$\sim N \times N$};
    \node[draw, circle, inner sep = 0pt, minimum size = 1.2cm] (circ) at (1.5,1) {};
    \node[above] at (circ.north) {IV};
    \node[below] at (circ.north) {$\sim SN$};
    \node at (2.7,1) {I $\sim \aFM_{N}(2)$};
  \end{tikzpicture}
  \caption{Adjacency of boundary strata in $\aFM_N(2)$}
  \label{fig:adj-afm}
\end{figure}

More generally, strata of $\aFM_{N}(r)$ are classified by rooted trees with $r+1$ leaves (one being labeled by $*$ and the others by $\{1,\dots,r\}$).
The algebra structure from Equation~\eqref{eq.alg-ofm-n-i} is given on strata by inserting a tree in the $*$ vertex, while the $\FM_{n}^{TN \times \R}$ operadic action is given on strata by inserting trees (with no $*$ vertex) in one of the numbered vertices.

\subsubsection{$\mFM_{M}(2)$}

Next let us turn to $\mFM_M(2)$.
The trees describing the stratification are depicted in Table~\ref{tab.strata-ofm-m}.
Again the inclusion relation between strata is captured by edge contraction.
The numbered leaves correspond to points of the configuration, while $*$ represents (roughly speaking) the component $\partial M \times \{ \infty \}$, and the black vertex represents the bulk of $M$.

\begin{table}[htbp]
  \centering
  \tiny
  \tikzset{every picture/.append style={scale=.35}}
  \begin{tabular}{ccccccccc}
    \toprule
    Stratum
     & I                           & II & III & IV & V & VI & VII & VIII \\
    \midrule
    Codim.
     & 0                           & 1  & 1   & 1  & 1 & 2  & 2   & 2
    \\
    {}
     & \begin{tikzpicture}
      \node {}
      child { node[int] {}
          child { node {1} }
          child { node{2} }
          child {node {*} }
        };
    \end{tikzpicture}
     & \begin{tikzpicture}
      \node {}
      child { node[int] {}
          child { node {1}}
          child {
              child { node{2} }
              child {node {*} }
            }
        };
    \end{tikzpicture}
     & \begin{tikzpicture}
      \node {}
      child { node[int] {}
          child { node {2}}
          child {
              child { node{1} }
              child {node {*} }
            }
        };
    \end{tikzpicture}
     & \begin{tikzpicture}
      \node {}
      child { node[int] {}
          child { node {*}}
          child {
              child { node{1} }
              child {node {2} }
            }
        };
    \end{tikzpicture}
     & \begin{tikzpicture}
      \node {}
      child { node[int] {}
          child {
              child { node {*}}
              child { node{1} }
              child {node {2} }
            }
        };
    \end{tikzpicture}
     & \begin{tikzpicture}
      \node {}
      child { node[int] {}
          child {
              child { node {1}}
              child {
                  child { node{2} }
                  child {node {*} }
                }
            }
        };
    \end{tikzpicture}
     & \begin{tikzpicture}
      \node {}
      child { node[int] {}
          child {
              child { node {2}}
              child {
                  child { node{1} }
                  child {node {*} }
                }
            }
        };
    \end{tikzpicture}
     & \begin{tikzpicture}
      \node {}
      child { node[int] {}
          child {
              child { node {*}}
              child {
                  child { node{1} }
                  child {node {2} }
                }
            }
        };
    \end{tikzpicture}
    \\
    \bottomrule
  \end{tabular}
  \caption{Boundary strata of $\mFM_M(2)$}
  \label{tab.strata-ofm-m}
\end{table}

We may draw a schematic picture capturing adjacency, where $SM$ denotes the sphere bundle in $TM$, and $S\p M$ is the restriction of this bundle to the boundary, see Figure~\ref{fig:adj-mfm}.
We are not representing the interior (boundary stratum I) which touches all the other boundary strata.

\begin{figure}[htbp]
  \centering
  \begin{tikzpicture}[scale=1.5, every label/.append style={font=\tiny}, every node/.append style={font=\tiny}]
    \node[int, label=left:{VI}, label=right:{$\sim\p M\times \p M$}] (v1) at (90:1) {};
    \node[int,label=left:{VII}, label=right:{$\sim\p M\times \p M$}] (v2) at (-30:1) {};
    \node[int, label=left:{VIII}, label=right:{$\sim S\p M$}] (v3) at (-150:1) {};
    \draw (0,0) edge (v1) edge (v2) edge (v3)
    (v1) -- (90:3) node[midway, above, sloped] {II} node [midway, below, sloped] {$\sim M \times \p M$}
    (v2) -- (-30:3) node[midway, above, sloped] {III} node [midway, below, sloped] {$\sim \p M \times M$}
    (v3) -- (-150:3) node[midway, above, sloped] {IV} node [midway, below, sloped] {$\sim SM$};
    \node at (.7,0.3) {V $\sim\aFM_{\p M}(2)$};
  \end{tikzpicture}
  \caption{Adjacency of boundary strata in $\mFM_M(2)$}
  \label{fig:adj-mfm}
\end{figure}
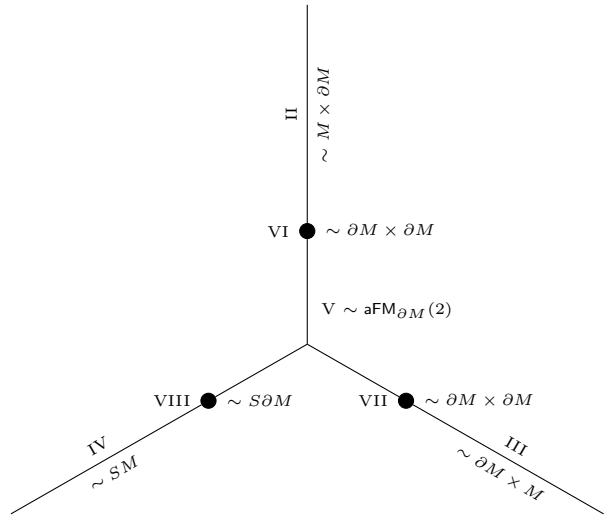

\subsubsection{$\SFM_{M}(0,2)$}

Finally let us turn to the space $\SFM_M(0,2)$, whose boundary strata are in Table~\ref{tab.strata-sfm-m}.
The inclusion of strata is again determined combinatorially by contracting edges.

\begin{table}[htbp]
  \centering
  \tikzset{every picture/.style={scale=.35}}
  \tiny
  \begin{tabular}{ccccccccccc}
    \toprule
    Stratum
     & I & II & III & IV & V & VI & VII & VIII & IX & X \\
    \midrule
    Codim.
     & 0 & 1  & 1   & 1  & 1 & 2  & 2   & 2    & 2  & 3
    \\
    {}
     &
    \begin{tikzpicture}
      \node {}
      child { node[int] {}
          child { node {1} }
          child { node{2} }
        };
    \end{tikzpicture}
     &
    \begin{tikzpicture}
      \node {}
      child { node[int] {}
          child { node {1}}
          child[dashed] {
              child[solid] { node{2} }
            }
        };
    \end{tikzpicture}
     &
    \begin{tikzpicture}
      \node {}
      child { node[int] {}
          child { node {2}}
          child[dashed] {
              child[solid] { node{1} }
            }
        };
    \end{tikzpicture}
     &
    \begin{tikzpicture}
      \node {}
      child { node[int] {}
          child[dashed] {
              child[solid] { node{1} }
              child[solid] {node {2} }
            }
        };
    \end{tikzpicture}
     &
    \begin{tikzpicture}
      \node {}
      child { node[int] {}
          child {
              child { node{1} }
              child {node {2} }
            }
        };
    \end{tikzpicture}
     &
    \begin{tikzpicture}
      \node {}
      child { node[int] {}
          child[dashed] {
              child[solid] { node {1}}
              child[dashed] {
                  child[solid] { node{2} }
                }
            }
        };
    \end{tikzpicture}
     &
    \begin{tikzpicture}
      \node {}
      child { node[int] {}
          child[dashed] {
              child[solid] { node {2}}
              child[dashed] {
                  child[solid] { node{1} }
                }
            }
        };
    \end{tikzpicture}
     &
    \begin{tikzpicture}
      \node {}
      child { node[int] {}
          child[dashed] {
              child[solid] {
                  child { node{1} }
                  child {node {2} }
                }
            }
        };
    \end{tikzpicture}
     &
    \begin{tikzpicture}
      \node {}
      child { node[int] {}
          child[dashed] {
              child[solid] { node{1} }
            }
          child[dashed] {
              child[solid] { node{2} }
            }
        };
    \end{tikzpicture}
     &
    \begin{tikzpicture}
      \node {}
      child { node[int] {}
          child[dashed] {
              child[dashed] {
                  child[solid] { node{1} }
                }
              child[dashed] {
                  child[solid] { node{2} }
                }
            }
        };
    \end{tikzpicture}
    \\
    \bottomrule
  \end{tabular}
  \caption{Boundary strata of $\SFM_M(0,2)$}
  \label{tab.strata-sfm-m}
\end{table}

We may draw a following schematic picture illustrating the inclusion of strata, see Figure~\ref{fig:adj-sfm}.
The notation $\HS$ denotes the ``half sphere bundle''.
Notice that we have not drawn the ``big'' stratum $I\sim \SFM_M(0,2)$ including all others.
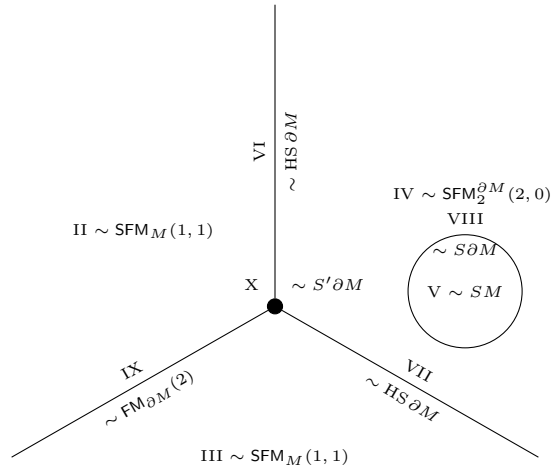
\begin{figure}[htbp]
  \centering
  \begin{tikzpicture}[every node/.append style={font=\tiny}]
    \node[int, label={above left}:{X}, label={above right}:{$\sim S \p M$}] (z) at (0,0) {};
    \draw
    (z) -- (90:4) node [midway, above, sloped] {VI} node [midway, below, sloped] {$\sim \HS \p M$}
    (z) -- (-30:4) node [midway, above, sloped] {VII} node [midway, below, sloped] {$\sim \HS \p M$}
    (z) -- (-150:4) node [midway, above, sloped] {IX} node [midway, below, sloped] {$\sim \FM_{\p M}(2)$};
    \node at (150:2) {II $\sim \SFM_M(1,1)$};
    \node at (-90:2) {III $\sim \SFM_M(1,1)$};
    \node at (30:3) {IV $\sim \SFM_2^{\p M}(0,2)$};
    \node [draw, circle, minimum size = 1.5cm] (c) at (2.5,.2) {V $\sim SM$};
    \node[above] at (c.north) {VIII};
    \node[below] at (c.north) {$\sim S M|_{\partial M}$};
  \end{tikzpicture}
  \caption{Adjacency of boundary strata in $\SFM_M(0,2)$}
  \label{fig:adj-sfm}
\end{figure}

\section{Construction of propagators}
\label{sec.propagator}

\subsection{The propagator for closed manifolds}\label{sec:closed propagator}
The proof of~\cite[Theorem~A]{Idrissi2018b} and the main results of~\cite{CamposWillwacher2016} -- in the case of closed compact orientable manifolds $M$ -- relied on the existence of a ``propagator'' $\varphi$.
Let us briefly recall its construction in this case, before going back to the case $\partial M \neq \varnothing$.

The propagator $\varphi$ is an $(n-1)$-form on $\FM_{M}(\underline{2})$ whose differential is the pullback of the diagonal class of $M^2$.
Moreover, the restriction of the propagator to $\partial \FM_{M}(\underline{2})$ is a global angular form for the $S^{n-1}$-bundle given by the projection $\partial \FM_{M}(\underline{2}) \to M$.
If $M$ is framed, then this bundle is trivial since the operadic right module structure map $\circ_{1} : M \times S^{n-1} \to \p \FM_{M}(\underline{2})$ is an isomorphism of bundles.
In that case, we further assume that $\circ_1^*\varphi$ is equal to $1 \times \vol_{n-1}$.

This propagator was constructed in several steps (see~\cite[Proposition~7]{CamposWillwacher2016} for the details).
We first consider a global angular form of the bundle $\p \FM_M(\underline{2}) \to M$.
We then pull it back to a tubular neighborhood of the boundary $\p \FM_M(\underline{2}) \subset \FM_{M}(\underline{2})$.
We then multiply it by a bump function, and finally we extend it by zero outside of the tubular neighborhood
Moreover, it can be chosen so that it belongs to the subalgebra $\Omega^{*}_{\triv}(\FM_{M}(\underline{2})) \subset \OmPA^{*}(\FM_{M}(\underline{2}))$ of ``trivial'' PA forms~\cite[Appendix~C]{CamposWillwacher2016}.
This implies that it can be integrated along the fibers of the canonical projections $p_{V}$.

In more detail, let us fix representatives of the cohomology, i.e., a quasi-isomorphism $\iota : H(M)\to \OmPA(M)$.
Let $(\gamma_i)$ be representatives of some basis of $H(M)$.
Let also $(\gamma_i^*)$ be the corresponding dual basis.
Note that this fits into the framework of Lemma~\ref{lem:abc} (where there are no $\alpha_{i}$, $\beta_i$ because $\partial M = \varnothing$).
This choice fixes a representative of the diagonal:
\begin{equation}
  \Delta \coloneqq \sum_{i} \iota(\gamma_i) \times \iota(\gamma_i^*) \in \OmPA^{n}(M \times M).
\end{equation}
As a convention, we will often write $x$ for the coordinate of the first $M$ and $y$ for the coordinate of the second $M$, so that e.g., the form $\Delta$ can be written as $\Delta(x,y)$.

The natural pairing of differential forms then also defines a projection onto cohomology
\begin{align*}
  p : \OmPA(M) & \to H(M)                                         \\
  \beta        & \mapsto \sum_i \gamma_i \int_M \gamma_i^* \beta.
\end{align*}

The \emph{propagator}\index{propagator} constructed in~\cite[Proposition 8]{CamposWillwacher2016} adapting earlier ideas in the literature~\cite{BottCattaneo1998} is a form
\begin{equation}
  \varphi\in \Omega^{n-1}_{\triv}(\FM_M(2))
\end{equation}
that satisfies the following properties:
\begin{itemize}
  \item It is (anti-)symmetric:
        \beq{equ:propsymmclosed}
        \varphi(y,x) = (-1)^{n} \varphi(x,y)
        \eeq
  \item Its differential is the pullback of the diagonal class:
        \begin{equation}
          d\varphi(x,y) =\Delta = \sum_{i} \gamma_i(x) \gamma_i^*(y).
        \end{equation}
  \item Its restriction to the boundary $\p\FM_M(\underline{2})$ (which consists of configurations of points that are infinitesimally close) is a fiberwise volume form, i.e., of volume $1$ on each fiber.
  \item For a suitable choice of representatives $\gamma_i$ as above and for all $i$, we have that:
        \begin{equation}
          \int_{y}\varphi(x,y) \gamma_i(x)=0.
        \end{equation}
\end{itemize}
The propagator shall morally be thought of as defining a homotopy $h$ through
\begin{equation}
  h\beta = \int_y \varphi(x,y) \beta(y),
\end{equation}
that satisfies $(dh+hd)(\beta)=\beta - \iota p \beta$.
Unfortunately, due to technical problems with PA forms the fiber integral occurring here is not defined for all PA forms, but only for $\beta\in \Omega_{\triv}(M)$.

\subsection{Diagonal data}
\label{sec.a-data}

In the previous section, we saw that to fix the differential of the propagator on a closed manifold, we needed a fixed diagonal class.
We will also want to fix the differential of our propagators on $\SFM_{M}$, $\aFM_{\partial M}$, and $\mFM_{M}$.
To this end, we define a general notion of ``diagonal data'' in the case of manifolds with boundary.
We will give examples of this notion below.

\begin{definition}
  Diagonal data\index{diagonal data} consists in the following data:
  \begin{itemize}
    \item Two CDGAs $A$ and $A_\p$;
    \item A CDGA morphism (the ``restriction'') $\rho : A\to A_\p$;
    \item A closed element (the  ``diagonal'') $\Delta_A \in A\otimes A$, such that:
          \[
            0 = (\rho \otimes \id) (\Delta_A) \in A_\p \otimes A;
          \]
    \item An element (the ``section'') $\sigma_{A} \in A \otimes A_{\partial}$, such that:
          \[ d\sigma_{A} = (\id \otimes \rho)(\Delta_{A}) \in A \otimes A_{\partial}. \]
  \end{itemize}
  Morphisms of diagonal data are defined to be morphisms of pairs of CDGAs which commute with the morphism and map the distinguished elements to the distinguished elements.
  In particular, quasi-isomorphisms of diagonal data are morphisms of diagonal data which are quasi-isomorphisms on the underlying two pairs of CDGAs.
\end{definition}

Given some diagonal data as in the definition above, we define as a shorthand the following element:
\begin{equation}
  \Delta_{A,A_{\partial}} \coloneqq (\id \otimes \rho)(\Delta_{A}) = d\sigma_{A} \in A \otimes A_{\partial}.
\end{equation}

\begin{example}
  \label{exa.diag-data}
  The reader should keep two examples of diagonal data in mind.
  Consider some compact manifold with boundary $(M,\p M)$.
  \begin{description}
    \item[Combinatorial case]
      Let $A=S (\tilde H(M)\oplus H(M,\p M))$ and $A_\p=S \tilde H(\p M)$ with the zero differentials.
      The map $\rho$ is the map induced from the restriction map $H(M)\to H(\p M)$, with $\rho(H(M,\p M))=0$.
      Recall the bases defined in Section~\ref{sec:note-cohom-comp}.
      The diagonal element is
      \begin{equation*}
        \Delta_A \coloneqq \sum_i d\beta_i \otimes \alpha_i + \sum_j \gamma_j \otimes \gamma_j^* \in H(M,\p M)\otimes H(M),
      \end{equation*}
      considered as an element of $A \otimes A$.
      The ``section'' $\sigma_{A}$ is given by $\sum_{i} \beta_{i} \otimes \alpha_i$.
    \item[Model case]
      Let $A=R$, $A_\p =R_\p$ and $\rho:A\to A_\p$ be CDGA models for $M$, $\p M$ and the inclusion $\p M\to M$ considered in Section~\ref{sec.pretty-nice}, obtained from a PLD model of $M$.
      The diagonal $\Delta_A$ is then the cocycle $\Delta_{R}$ defined in Lemma~\ref{lem.delta-r} below, and $\sigma_{A}$ is the element $\sigma_{R}$ also defined below~\eqref{eq:sigma-r}.
  \end{description}
  In Section~\ref{sec.comp-two-approaches} we explain the advantages of both approaches.
\end{example}

In order to define the elements of the second case, let us fix some PLD model of $M$ as in Equation~\eqref{eq.nice-diagram}.
Recall that the diagonal class $\Delta_{KP} \in K \otimes P$~\eqref{def.delta-ak} is sent to $\Delta_{P} \in P \otimes P$ under the map $K \subset B \xrightarrow{\pi} P$.

\begin{lemma}
  \label{lem.delta-r}
  There exists an element $\Delta_{R} \in R \otimes R$ such that $(\id \otimes \pi) g^{\otimes 2}(\Delta_{R}) = \Delta_{KP}$ and $\mu_{R}(\Delta_{R}) = 0$.
\end{lemma}
\begin{proof}
  The morphisms $R \to B$ and $B \to P$ are both surjective, and $\mu_{P}(\Delta_{P}) = 0$ when $\partial M \neq \emptyset$ (Equation~\eqref{eq.mu-delta-zero}).
  Thus the proof is identical to the proof of~\cite[Proposition~3.3]{Idrissi2018b}.
\end{proof}

\begin{definition}
  \label{def.delta-b}
  For convenience, we define the cocycle $\Delta_{B} \coloneqq g^{\otimes 2}(\Delta_{R}) \in B^{\otimes 2}$.
\end{definition}

Let us now choose a linear section $s : B_{\partial} \to B$ of $\rho : B \to B_{\partial}$.
It does not commute with the differential, in general.
However, $s(dx)-ds(x)$ belongs to $K = \ker \rho$ for all $x \in B$.
Using the Poincaré duality of $B_{\partial}$, we can make $s$ into an element $\sigma_{B} \in B \otimes B_{\partial}$.
This element satisfies $d \sigma_{B} \in K \otimes B_{\partial}$.
Using the fact that we have a surjective quasi-isomorphism (thanks to the five lemma):
\begin{equation}
  \cone(R_{\partial} \otimes \ker \rho) \to \cone(B \otimes K),
\end{equation}
we can find some element
\begin{equation}
  \label{eq:sigma-r}
  \sigma_{R} \in R_{\partial} \otimes R
\end{equation}
with $(g_{\partial} \otimes g)(\sigma_{R}) = \sigma_{B}$ and $d \sigma_{R} \in R_{\partial} \otimes \ker \rho$.

\begin{definition}
  Let $M$ be a manifold with boundary, and let $(A \xrightarrow{\rho} A_{\partial}, \Delta_{A})$ be some diagonal data.
  We say that this diagonal data \textbf{maps into $M$} if we are given maps $g:A\to \Omega_{\triv}(M)$ and $g_\p:A\to \Omega_{\triv}(\p M)$ such that the diagram
  \beq{equ:AdataCD}
  \begin{tikzcd}
    A \ar{d}{\rho} \ar{r}{g} & \Omega_{\triv}(M) \ar{d} \\
    A_\p \ar{r}{g_\p} & \Omega_{\triv}(\p M)
  \end{tikzcd}
  \eeq
  commutes, and such that $g^{\otimes2}(\Delta_{A})$ is a representative of the ``diagonal'' class in cohomology (defined using Poincaré--Lefschetz duality).
\end{definition}

Both of our main examples (Example~\ref{exa.diag-data}) map into $M$, either by definition in the model case, or by choosing closed representatives as in Lemma~\ref{lem:abc}.

\subsection{The propagator on \texorpdfstring{$\SFM_M$}{SFM\_M}}
\label{sec.propagator-sfm}

The propagator $\varphi\in \OmPA(\SFM_M(0,2))$ is constructed ``as usual in physics'' by the method of mirror charges.
The inspiration comes from the construction of $\SFM_{M}$ itself in Section~\ref{sec.constr-sfmm}.

Concretely, let us consider the double manifold:
\begin{equation}
  \tilde M = M \sqcup_{\p M} M^{\mathrm{op}},
\end{equation}
which is a closed compact oriented manifold.
It comes with an orientation reversing involution $I : \tilde M \to \tilde M$ that interchanges the two halves.
We consider a propagator $\tilde \varphi$ on $\FM_{\tilde M}(2)$ as constructed in~\cite[Proposition 8]{CamposWillwacher2016} (see Section~\ref{sec:closed propagator}).
In particular, we have that $\tilde\varphi(x,y) = (-1)^{n}\tilde\varphi(y,x)$.
Moreover, the element $d\tilde \varphi = \tilde \Delta \in \OmPA(M)\otimes \OmPA(M)$ is a representative of the diagonal class in $\tilde M\times \tilde M$.
We then define:
\begin{equation}
  \varphi'(x,y) \coloneqq \tilde \varphi(x,y) - \tilde \varphi(I x,y).
\end{equation}
If $x$ approaches the boundary $\p M$, then $\varphi'(x,y)\to 0$.

\begin{lemma}\label{lem:preprop}
  The form $d\varphi'$ descends to $\OmPA(M,\p M)\otimes \OmPA(M)$ and it is a representative of  the ``diagonal'' element.
\end{lemma}
\begin{proof}
  The statement is independent of the choices involved in the construction of $\varphi'$, so we are free to make these choices conveniently.
  To this end we suppose that we have chosen bases of $H(M)$, $H(M,\p M)$ and $H(\p M)$ as in Lemma \ref{lem:abc} above.
  The cohomology of $H(\tilde M)$ has then a basis represented by $\gamma_i, I^*\gamma_i$, $\tilde \alpha_i= \alpha_i + I^*\alpha_i,$ and $\tilde \beta_i = d\beta_i-I^*d\beta_i$.
  The diagonal in this basis is represented by (schematically)
  \begin{equation}
    \tilde \Delta = \sum_{i} \gamma_{i} \otimes \gamma_{i} + I^*\gamma_{i}\otimes I^*\gamma_{i} + \sum_{j} \tilde \alpha_j \otimes \tilde \beta_j+ \tilde \beta_j \otimes \tilde \alpha_j.
  \end{equation}
  Subtracting $(I^*\otimes \id)\tilde \Delta$ and restricting to $M\times M\subset \tilde M \times \tilde M$, we obtain as desired:
  \begin{equation}
    2 \sum_{i} \gamma_{i}\otimes \gamma_{i} +  2 \sum_{j} d\beta_j \otimes \alpha_j,
    \qedhere
  \end{equation}
\end{proof}

We use the previous lemma to show the following.
Let us fix some diagonal data $(A \xrightarrow{\rho} A_{\partial}, \Delta_{A}, \sigma_{A})$ that maps into $M$, with $g : A \to \OmPA^{*}(M)$ and $g_{\partial} : A_{\partial} \to \OmPA^{*}(\partial M)$.
Recall the description of the boundary strata of $\SFM_{M}(0,2)$ from Section~\ref{sec:FM2discussion}.
In particular, we will denote by
\begin{equation}
  j : \SFM_{M}(1,1) \to \SFM_{M}(0,2)
\end{equation}
the ``inclusion'' whose image is the adherence of the boundary stratum II (see Section~\ref{sec:FM2discussion}) consisting of configurations where one point is infinitesimally close to the boundary and the other remains on the bulk.
This map $j$ is defined using a collar $\partial M \times [0,1) \subset M$ in order to infinitesimally ``push'' a point of the boundary inside $M$.

\begin{proposition}
  \label{prop.propagator-sfmm}
  Let $M$ be a compact orientable $n$-manifold with boundary, for which we fix an algebraic realization as above.
  Let us also fix some diagonal data which maps into $M$ as above.
  Then there is a form $\varphi\in \Omega^{n-1}_{\triv}(\SFM_M(0,2))$ (the propagator\index{propagator}) with the following properties.
  \begin{itemize}
    \item $\varphi$ vanishes on the boundary stratum III of $\FM_M(2)$, using the notation of Section~\ref{sec:FM2discussion};
    \item The restriction of $\varphi$ to the boundary stratum V is fiberwise a volume form, and moreover if $M$ is framed then that restriction is given by $1 \times \vol_{S^{n-1}}$;
    \item We have $d\varphi = g^{\otimes2}(\Delta_{A})$;
    \item The restriction $j^{*}\varphi$ to the boundary stratum II is equal to $(g \otimes g_{\partial})(\sigma_{A})$, while the restriction to boundary stratum III vanishes;
    \item We have the following vanishing properties:
          \begin{align*}
            \forall \alpha \in A, \int_{y} \varphi(y,x) \alpha(y)        & = 0; &
            \forall \beta \in \ker(\rho), \int_{y} \varphi(x,y) \beta(y) & = 0.
          \end{align*}
  \end{itemize}
\end{proposition}

\begin{proof}
  We start with the $\varphi'$ as in Lemma \ref{lem:preprop}, constructed using the method of mirror charges.
  This form already satisfies the boundary conditions required from our propagator.
  By Lemma \ref{lem:preprop} we can find some $\psi\in \OmPA(M,\p M)\otimes \OmPA(M)$ such that $d\varphi'+d\psi=\Delta$.
  We then define $\varphi'' \coloneqq \varphi'+\psi$.
  This addition does not alter the desired boundary behavior and satisfies $d\varphi''=\Delta$.
  It remains to modify $\varphi''$ so as to also satisfy the last property.
  This can be done with a standard trick for ``normalizing'' homotopies as in~\cite[Lemma 3]{CattaneoMneev2010}.
  Concretely, we set
  \begin{multline}\label{eq:phi-prime}
    \varphi_{12}
    \coloneqq \varphi''_{12} - \int_{3} \varphi''_{23} \Delta_{13} - \int_{3} \varphi''_{23} \sigma_{13} - \int_{3} \varphi''_{13} \Delta_{23} - \int_{3} \varphi''_{13} \sigma_{23} \\
    + \int_{3,4} \Delta_{24} \varphi''_{34} \Delta_{13} + \int_{3,4} \sigma_{24} \varphi''_{34} \Delta_{13} + \int_{3,4} \Delta_{24} \varphi''_{34} \sigma_{13} + \int_{3,4} \sigma_{24} \varphi''_{34} \sigma_{13},
  \end{multline}
  where by $\int_{i,j\dots}$ we mean the pushforward (integration along the fibers) along the SA bundle which forgets the listed points, and e.g., $\Delta_{13}$ is the pullback $p_{13}^{*}((g \otimes g)(\Delta_{A}))$ (and so on).

  A careful application of the Stokes formula shows that the differential of the corrective term vanishes, which follows from general properties of the diagonal class and the properties of $\varphi''$ that were already proved.
  We can check on bases of the cohomologies that the difference $j^{*} \varphi - (g_{\partial} \otimes g)(\sigma_{R})$ is an exact form, say $d \xi$ with $\xi \in \Omega^{*}_{\triv}(\SFM_{M}(\{2\}, \{1\}))$.
  The application $j^{*}$ is surjective (because $\operatorname{im} j$ is a submanifold with corners of $\SFM_{M}(\varnothing, \{1,2\})$), hence there exists some $\hat{\xi}$ such that $j^{*} \hat{\xi} = \xi$.
  It then suffices to replace $\varphi$ by $\varphi - d \hat{\xi}$, which still satisfies the first three properties.
\end{proof}

\subsection{The propagator for \texorpdfstring{$\aFM_{N}$}{aFM\_N}}
\label{sec.propagator-afm}

We suppose that $N= \p M$ is a closed compacted oriented (PA) manifold bounding a compact orientable manifold $M$, with $\dim M = n$.
Again we fix some diagonal data which maps into $M$, namely $(A \xrightarrow{\rho} A_{\partial}, \Delta_{A}, \sigma_{A}, g, g_{\partial})$.
For convenience we let $\sigma_{\partial} \coloneqq (g_{\partial} \otimes 1)(\sigma_{A})$, which is, as we recall, ``half'' of the diagonal class of $A_{\partial}$ (in the sense that symmetrizing $\sigma_{\partial}$ is equal to $\Delta_{A_{\partial}}$).

\begin{proposition}\label{prop:oproploc}
  There is a form, the propagator\index{propagator}
  \[
    \varphi\in \Omega^{n-1}_{\triv}(\aFM_{N}(2))
  \]
  with the following properties:
  \begin{itemize}
    \item We have $d\varphi = 0$;
    \item The restriction to the infinitesimal boundary (stratum IV in the language of Section~\ref{sec:FM2discussion}) is a fiberwise volume form of volume one;
    \item On the two infinite boundary strata (respectively stratum II, where the second point goes to infinity, and III, where the first point does) we have the following behavior:
          \begin{align}
            \varphi|_{2 \to \partial M} = \varphi|_{\mathrm{II}} & = \sigma_{\partial},
                                                                 & \varphi|_{1 \to \partial M} = \varphi|_{\mathrm{III}} & = 0;
          \end{align}
    \item We have that
          \begin{equation}
            \label{equ:NxI}
            \forall \alpha \in A_{\partial},\; \int_y \varphi(x,y) \alpha(y) = 0.
          \end{equation}
  \end{itemize}
\end{proposition}

\begin{proof}
  Recall that the space $\aFM_{N}(2)$ has three boundary strata, which we described in Section~\ref{sec:FM2discussion}.
  We begin by picking some extension $\psi\in \Omega_{\mathrm{min}}(\aFM_{N}(2))$ of a fiberwise volume form on the infinitesimal boundary stratum, supported in a neighborhood of that stratum.
  We have that $d\psi$ is a closed form supported away from that diagonal and the infinite boundary strata, and hence represents a cohomology class in
  \begin{equation}
    H(M\times M \times I, M\times M \times \{0,1\})
    =
    (H(M)\otimes H(M) )[-1].
  \end{equation}
  Note that representatives of classes on the right are $\nu_i (x) \nu_j(y) dt$, where $t$ is the coordinate on $I$.
  We have a Poincaré--Lefschetz duality pairing with
  \begin{equation}
    H(M\times M \times I)= H(M)\otimes H(M),
  \end{equation}
  which is represented by classes $\nu_i(x) \nu_j(x)$.
  We can hence compute the cohomology class of $d\psi$ by using this pairing and Stokes' Theorem
  \begin{equation}
    \int_{M\times M\times I} \nu_i(x) \nu_j(y) d\psi(x,y)
    =
    \int_{\aFM_{N}(2)} \nu_i(x) \nu_j(y) d\psi(x,y)
    =\pm
    \int_{N}  \nu_i\nu_j,
  \end{equation}
  where in the last step we used Stokes' Theorem, that $\psi$ vanishes on the infinite boundary, and that $\psi$ is a fiberwise volume form on the infinitesimal boundary stratum.
  We hence find that
  \begin{equation}
    d\psi \pm \Delta_{A} \wedge d(1-t)
  \end{equation}
  represents the zero cohomology class in $H(M\times M \times I)$.
  Hence we may find a $\psi'\in \OmPA(M\times M\times I, M\times M\times \{0,1\})$ such that
  \begin{equation}
    \varphi' \coloneqq \psi - \psi' -\pm \sum_i \Delta_{A} \wedge (1-t)
  \end{equation}
  is a cocycle.
  Moreover, $\varphi'$ has the desired boundary behavior, as $\psi'$ vanishes on the boundary, the term $\psi$ has the desired behavior at the infinitesimal boundary while vanishing at the infinite boundary, and the remaining terms contribute the desired behavior at the infinite boundary, while not affecting the fiberwise properties at the infinitesimal boundary.

  Finally the last property can be ensured by a slight adaptation of a trick from~\cite{CattaneoMneev2010}.
  More concretely, we define:
  \begin{equation}
    \varphi_{12} \coloneqq \varphi'_{12}
    \pm \int_{3} \varphi'_{13} (\sigma_{\partial})_{32} dt
    \pm \int_{3} \varphi'_{23} (\sigma_{\partial})_{31} dt
    \mp \int_{3,4} \varphi'_{34} (\sigma_{\partial})_{23} (\sigma_{\partial})_{14}
  \end{equation}
  Note that the additions do not alter the boundary behavior due to the presence of the $dt$-factors.
\end{proof}

\subsection{The propagator for \texorpdfstring{$\mFM_{M}$}{mFM\_M}}
\label{sec.propagator-mfm}

We consider next our compact connected orientable manifold with non-empty boundary $M$.
We keep the same diagonal data $(A \xrightarrow{\rho} A_{\partial}, \Delta_{A}, \sigma_{A}, g, g_{\partial})$.

\begin{proposition}
  There is a form, the propagator\index{propagator}
  \begin{equation}
    \varphi\in \Omega^{n-1}_{\triv}(\mFM_{M}(2)),
  \end{equation}
  with the following properties:
  \begin{itemize}
    \item We have
          \beq{equ:opropd}
          d\varphi = \Delta_A
          \eeq
    \item On the infinite boundary strata II, III and V (see Section~\ref{sec:FM2discussion} for this notation) we have the following behavior:
          \begin{align}
            \varphi|_{2 \to \partial M} = \varphi|_{\mathrm{II}}        & = \sigma_{A}, \\
            \varphi|_{1 \to \partial M} = \varphi|_{\mathrm{III}}       & = 0,          \\
            \varphi|_{1 \approx 2 \to \partial} = \varphi|_{\mathrm{V}} & = \varphi_\p,
          \end{align}
          where $\varphi_\p \in \Omega^{*}_{\triv}(\aFM_{\p M}(2))$ is the propagator constructed in Proposition \ref{prop:oproploc};
    \item The restriction $\varphi|_{1 \approx 2}$ to the infinitesimal boundary $IV$ is a fiberwise volume form of volume $1$;
    \item We have that
          \beq{equ:propnorm}
          \forall \alpha \in A, \; \int_y \varphi(x,y) \alpha(y) = 0.
          \eeq
  \end{itemize}
\end{proposition}
\begin{proof}
  We first pick a form $\psi\in \Omega^{n-1}_{\triv}(\mFM_M(2))$ that has the desired boundary behavior and continue it to the whole space $\mFM_M(2)$ as usual by considering a neighbourhood isomorphic to the normal bundle and extend it constantly along the fibers.
  Then  $d\psi$ is supported away from diagonal and represents a class in
  \begin{equation}
    H(M\times M, \p M\times M \cup M\times \p M)
    =H(M,\p M) \otimes H(M,\p M).
  \end{equation}
  We may compute this class by using the Poincaré duality pairing with $H(M)\otimes H(M)$.
  Recall the notations $\{\alpha_{i}, \beta_{j}, \gamma_{k}\}$ for the base elements of $H^{*}(M)$, $H^{*}(M, \partial M)$, and $H^{*}(\partial M)$ from Lemma~\ref{lem:abc}; we will use these to compute the cohomology class of $d\psi$.
  Using Stokes' Theorem, we find that
  \begin{align*}
    \int_{M \times M} \alpha_i(x) \alpha_j(y) d\psi(x,y)
     & =
    \int_{\mFM_M(2)} \alpha_i(x) \alpha_j(y) d\psi(x,y)
    \\
     & = \pm
    \int_{\p \mFM_M(2)} \alpha_i(x) \alpha_j(y) \psi(x,y)
    \\
     & = \int_M \alpha_i\alpha_j
    + \int_{\partial M} \alpha_i \beta_k \int_M\alpha_k\alpha_j
    + \int_{\partial M} \alpha_i \beta_k \int_M\alpha_k\alpha_j
    \\&\quad+\int_{\aFM_{N}(2)} \alpha_i(x)\alpha_j(y)\psi(x,y)
    \\
     & =0+0+0+0=0.
  \end{align*}
  Here we use the known boundary behavior of $\psi$, and for the vanishing in the last line the properties \eqref{equ:abc2} and \eqref{equ:NxI}.
  Similarly one computes
  \[
    \int_{M \times M} \alpha_i(x) \gamma_j(y) d\psi(x,y) = \int_{M \times M} \gamma_i(x) \alpha_j(y) d\psi(x,y) = 0
  \]
  and
  \[
    \int_{M \times M} \gamma_i(x) \gamma_j(y) d\psi(x,y) =\int_M \gamma_i \gamma_j= g_{ij}.
  \]
  where $g_{ij}$ is matrix of the (non-degenerate) natural pairing on the image of $H(M,\p M)\to H(M)$.
  As above we hence find that there is a $\psi'$ vanishing on the boundary such that
  \begin{equation}
    \varphi' = \psi-\psi'
  \end{equation}
  satisfies our boundary requirements and \eqref{equ:opropd}.
  We finally define $\varphi$ by the same formula as Equation~\eqref{eq:phi-prime}.
  It is then clear that the propagator $\varphi$ satisfies \eqref{equ:propnorm}.
  Also, at the infinitesimal boundary the additional terms are basic (independent of the fiber coordinate) and hence $\varphi$ is still a fiberwise volume form at that boundary, as $\varphi'$ is.
  Finally, a short computation shows that all the terms added to $\varphi'$ are closed and vanish at the infinite boundary, so that \eqref{equ:opropd} and the infinite boundary conditions continue to hold.
\end{proof}

\section{Graphical models -- recollections}
\label{sec.model-closed}

\subsection{Models for the little disks and Swiss-Cheese operads}
\label{sec.extens-swiss-cheese}

We now quickly describe the Hopf cooperad $\Graphs_{n}$ constructed by Kontsevich~\cite{Kontsevich1999} to prove the formality of the little $n$-disks operad (see also~\cite[Section~1.3]{Idrissi2018b} for notations that match our conventions).
The Hopf cooperad $\Gra_{n}$\index{Gran@$\Gra_{n}$} is defined as
\begin{equation}
  \Gra_{n}(V) \coloneqq \bigl( S(e_{vv'})_{v, v' \in V} / (e_{vv'} - (-1)^{n} e_{v'v}), d = 0 \bigr),
\end{equation}
with the generators $e_{vv'}$ carrying cohomological degree $n-1$.

This definition admits a graphical interpretation: a monomial in $\Gra_{n}(V)$ can be seen as a graph with vertices in bijection with $V$, and an edge between $v$ and $v'$ iff $e_{vv'}$ appears in the monomial.
The cocomposition $\circ_{T}^{\vee}(e_{vv'})$ (for $T \subset V$) is given by $1 \otimes e_{vv'}$ if $v,v' \in W$ and by $e_{[v][v']} \otimes 1$ otherwise, where $[v] \in V/T$ is the class of $v \in V$ in the quotient.
The element\index{mu@$\mu \in \MC(\GC_{n}^{\vee})$}
\begin{equation}
  \label{eq.mu}
  \mu = e_{12}^{\vee}
\end{equation}
dual to the generator $e_{12} \in \Gra_{n}(\underline{2})$ is a Maurer--Cartan element in the deformation complex $\Def(\hoLie_{n}, \Gra_{n})$.
This allows us to produce a new Hopf cooperad $\Tw \Gra_{n}$\index{TwGran@$\Tw\Gra_{n}$} through a procedure called ``operadic twisting''~\cite{Willwacher2014,DolgushevWillwacher2015}.
Elements of $\Tw\Gra_{n}(V)$ can be seen as graphs with ``external'' vertices, which are in bijection with $V$, and indistinguishable ``internal'' vertices of degree $-n$ (typically drawn in black).
The cocomposition is induced by the cocomposition of $\Gra_{n}$, and the product glues graphs along external vertices.
Finally, we mod out by the bi-ideal of graphs with connected component consisting exclusively of internal vertices (a.k.a.\ internal components) to obtain a Hopf cooperad $\Graphs_{n}$.

\begin{theorem}[{Kontsevich~\cite{Kontsevich1999}, Lambrechts and Volić~\cite{LambrechtsVolic2014}}]
  There is a zigzag of quasi-isomorphisms of Hopf cooperads:
  \[ H^{*}(\FM_{n}) \qiso* \Graphs_{n} \xrightarrow[\omega]{\sim} \OmPA^{*}(\FM_{n}). \]
  The first map is given by the quotient by graphs containing internal vertices, and the second map is given by an integral along the fibers of the canonical projections $\FM_{n}(V \sqcup J) \to \FM_{n}(V)$ of explicit minimal forms.
\end{theorem}

\begin{remark}\label{rmk:tadpole}
  It is usual to kill tadpoles (i.e., edges between a vertex and itself) and double edges (i.e., two edges with the same endpoints) in $\Graphs_{n}$, in other words to mod out by $(e_{vv}, e_{vv'}^{2})$ in the definition of $\Gra_{n}$.
  If we denote by $\Graphs'_{n}$ the quotient of our $\Graphs_{n}$ by tadpoles and double edges, then the quotient map $\Graphs_{n} \to \Graphs'_{n}$ is a quasi-isomorphism, see~\cite[Proposition~3.8]{Willwacher2014} and ~\cite[Theorem~2]{WillwacherZivkovic2015}.
  We keep the tadpoles and double edges to remain consistent with the other graph complexes appearing in this paper.
  Note however that by symmetry, tadpoles vanish for even $n$, and double edges vanish for odd $n$.
\end{remark}

The Swiss-Cheese operad is not formal~\cite{Livernet2015,Willwacher2017a}, thus there can be no quasi-iso\-mor\-phism between $H^{*}(\SFM_{n})$ and $\OmPA^{*}(\SFM_{n})$.
Nevertheless, there is a model $\SGraphs_{n}$ established in \cite{Willwacher2015}, which is similar in spirit to the cooperad $\Graphs_{n}$.
The cohomology $H^{*}(\SFM_{n}) \cong H^{*}(\FM_{n}) \otimes H^{*}(\FM_{n-1})$ splits as a Voronov product (see~\cite{Voronov1999} and~\cite[Section~4.3]{Idrissi2017}).
The construction of $\SGraphs_{n}$ contains $\Graphs_{n}$ (as ``aerial'' graphs) and $\Graphs_{n-1}$ (as ``terrestrial'' graphs), and ``intertwines'' the two in a way that corrects the lack of formality.
Let us now describe it.

\begin{remark}
  Our notations differ slightly from the notations of~\cite{Willwacher2015}.
  We call $\SGraphs_{n}$ what is called $\Graphs_{n}^{1}$ there, i.e., the space of operations with output of ``type $1$'' (which corresponds to ``open'').
  This is a relative Hopf cooperad over $\Graphs_{n}$ (which would be $\Graphs_{n}^{2}$ in~\cite{Willwacher2015}, the space of operations with output of ``type $2$'', i.e., closed).
\end{remark}

Let us assume from the start that $n \geq 3$ to avoid some difficulties that arise when $n = 2$.
If we had $n = 2$, then the description of ``connected component'' in $\SGraphs_{n}$ would be slightly different.

The idea is to construct a relative Hopf cooperad $\SGra_{n}$\index{SGran@$\SGra_{n}$} over $\Gra_{n}$, with two types of vertices: aerial ones, corresponding to closed inputs, and terrestrial ones, corresponding to open inputs.
Edges are oriented, and the source of an edge may only be an aerial vertex.
More concretely, one defines:
\begin{equation}
  \SGra_{n}(U,V) \coloneqq S(e_{vu})_{u \in U, v \in V} \otimes S(e_{vv'})_{v, v' \in V}
\end{equation}
where the generators all have degree $n-1$, and the cooperad structure maps are given by:
\begin{equation}
  \begin{aligned}
    \circ_{T}^{\vee} : \SGra_{n}(U,V)
            & \to \SGra_{n}(U, V/T) \otimes \Gra_{n}(T) \\
    e_{vv'} & \mapsto
    \begin{cases}
      1 \otimes e_{vv'}     & \text{if } v, v' \in T; \\
      e_{[v][v']} \otimes 1 & \text{otherwise}.
    \end{cases}                         \\
    e_{vu}  & \mapsto e_{[v]u} \otimes 1.
  \end{aligned}
\end{equation}

\begin{equation}
  \begin{aligned}
    \circ^{\vee}_{W,T} : \SGra_{n}(U, V \sqcup T)
     & \to \SGra_{n}(U/W, V) \otimes \SGra_{n}(W, T) \\
    e_{vv'}
     & \mapsto
    \begin{cases}
      e_{vv'} \otimes 1 & \text{if } v,v' \in V; \\
      1 \otimes e_{vv'} & \text{if } v,v' \in T; \\
      0                 & \text{otherwise}.
    \end{cases}                      \\
    e_{uv}
     & \mapsto
    \begin{cases}
      e_{[u]v} \otimes 1 & \text{if } v \in V;          \\
      1 \otimes e_{uv}   & \text{if } u \in W, v \in T; \\
      0                  & \text{otherwise.}
    \end{cases}
  \end{aligned}
\end{equation}

A monomial in $\SGra_{n}(U,V)$ can be seen as a directed graph with two kinds of vertices: aerial and terrestrial.
The set $U$ is the set of terrestrial vertices, and the set $V$ is the set of aerial vertices.
A generator $e_{ij}$ corresponds to an oriented edge from vertex $i$ to vertex $j$, and an edge always starts at an aerial vertex.

This allows us to produce a first morphism $\omega' : \SGra_{n} \to \OmPA^{*}(\SFM_{n})$.
One can define, for $v, v' \in V$, $\omega'(e_{vv'}) \coloneqq p_{vv'}^{*}(\vol_{n-1})$, where
\begin{equation}
  \vol_{n-1} \in \OmPA^{n-1}(\SFM_{n}(\varnothing, \{v,v'\})) \simeq \OmPA^{n-1}(\FM_{n}(\{v,v'\})) \simeq \OmPA^{n-1}(S^{n-1}).
\end{equation}
Recall that $\SFM_{n}(\{u\}, \{v\})$ is homeomorphic to $D^{n-1}$, and we write $\overline{\vol}^{h}_{n-1}$ for the $(n-1)$-form on $\SFM_{n}(\{u\}, \{v\})$ obtained by pulling back the volume form of $S^{n-1}$ along the map $D^{n-1} \to S^{n-1}$ given by the hyperbolic geodesic (see~\cite[Equation~(8)]{Willwacher2015}).
Then for $u \in U$ and $v \in V$, define $\omega'(e_{vu}) \coloneqq p_{vu}^{*}(\overline{\vol}^{h}_{n-1})$.

If $\Gamma \in \SGra_{n}(U,V)$ is a graph, let\index{c@$c \in \MC(\SGC_{n}^{\vee})$}
\begin{equation}
  \label{eq.konts-coeff}
  c(\Gamma) \coloneqq \int_{\SFM_{n}(U,V)} \omega'(\Gamma).
\end{equation}
Note that these are analogous to the coefficients that appear in Kontsevich's~\cite{Kontsevich2003} universal $L_{\infty}$ formality morphism $T_{\mathrm{poly}} \to D_{\mathrm{poly}}$, defined for $n = 2$.

The cooperad $\SGra_{n}$ is then twisted with respect to the sum of the Maurer--Cartan element $\mu \in \Def(\hoLie_{n} \to \Gra_{n})$ with the Maurer--Cartan element defined by $c$, to obtain a relative Hopf cooperad $\Tw \SGra_{n}$\index{TwSGran@$\Tw \SGra_{n}$} over $\Tw \Gra_{n}$.

Concretely, $\Tw \SGra_{n}(U,V)$ is spanned by graphs with $2 \times 2 = 4$ types of vertices: they can be either aerial or terrestrial, and either internal or external.
Internal terrestrial vertices are of degree $1-n$ and indistinguishable among themselves, and internal aerial vertices are of degree $-n$ and indistinguishable among themselves.
Edges remain oriented and of degree $n-1$.
External terrestrial vertices are in bijection with $U$, and external aerial vertices are in bijection with $V$.
The cooperad structure maps collapse subgraphs, and the product glues graphs along external vertices.

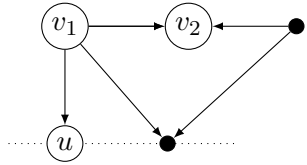
\begin{figure}[htbp]
  \centering
  \begin{tikzpicture}

    \pgfdeclarelayer{bg}
    \pgfsetlayers{bg,main}

    \node (0) {};
    \node (o) [right = 3cm of 0] {};
    \begin{pgfonlayer}{bg}
      \draw[dotted] (0) to (o);
    \end{pgfonlayer}
    \node[ext, fill=white] (u) [right = 0.5cm of 0] {$u$};
    \node[int] (i1) [right = 1cm of u] {};
    \node[ext] (v1) [above = 1cm of u] {$v_{1}$};
    \node[ext] (v2) [right = 1cm of v1] {$v_{2}$};
    \node[int] (i2) [right = 1cm of v2] {};
    \draw[-latex] (v1) edge (u) edge (i1) edge (v2);
    \draw[-latex] (v1) to (v2);
    \draw[-latex] (i2) edge (v2) edge (i1);
  \end{tikzpicture}
  \caption{A colored graph in $\Tw \SGra_{n}(\{u\}, \{v_{1}, v_{2}\})$.
    Terrestrial vertices are drawn on a dotted line to distinguish them, even though they are not ordered.
    This graph is of degree $5(n-1) - n - (n-1)$.}
  \label{fig.exa-sgra_n}
\end{figure}

The differential has several summands:
\begin{itemize}
  \item A first summand (corresponding to $\mu$) contracts edges between two aerial vertices, with at least one being internal.
  \item A second summand is given by contracting subgraphs $\Gamma' \subset \Gamma$ with at most one external vertex, which must be terrestrial, to obtain $\Gamma / \Gamma'$, with coefficient $c(\Gamma')$.
        One should note that in $\Gamma/\Gamma'$, the new vertex representing the collapsed subgraph is terrestrial, even if $\Gamma'$ is fully aerial.
  \item Finally, a third summand is given by forgetting some internal vertices and keeping a subgraph $\Gamma' \subset \Gamma$, with coefficient $c(\Gamma / \Gamma')$.
\end{itemize}

In the second and third cases, if the graph $\Gamma / \Gamma'$ contains an edge whose source is terrestrial, then the summand is defined to be zero (see~\cite[§6.4.2.2]{Kontsevich2003}).

One then checks that there is an extension $\omega : \Tw \SGra_{n} \to \OmPA^{*}(\SFM_{n})$.
Given a graph $\Gamma \in \SGra_{n}(U \sqcup I, V \sqcup J) \subset \Tw \SGra_{n}(U,V)$, then $\omega(\Gamma)$ is given by:
\[ \omega(\Gamma) = \int_{\SFM_{n}(U \sqcup I, V \sqcup J) \to \SFM_{n}(U,V)} \omega'(\Gamma) = (p_{U,V})_{*}(\omega'(\Gamma)). \]
It remains to mod out by the bi-ideal of graphs with internal components to obtain a Hopf cooperad $\SGraphs_{n}$\index{SGraphsn@$\SGraphs_{n}$}, and to check that $\omega$ factors through the quotient.

\begin{theorem}[{\cite{Willwacher2015}}]
  The morphism $\omega : \SGraphs_{n} \to \OmPA^{*}(\SFM_{n})$ is a quasi-isomorphism of relative Hopf cooperads.
\end{theorem}

\subsection{Graph complexes \texorpdfstring{$\GC_{n}$}{GC\_n} and \texorpdfstring{$\SGC_{n}$}{SGC\_n}}
\label{sec.graph-complexes}

Let us now introduce several Lie algebras that stem from Kontsevich's proof of the formality of $\FM_{n}$ and from the model of the Swiss-Cheese operad.
Later on, we will introduce decorated variants that will be endowed with actions of the non-decorated complexes.

Let us define\index{fGCn@$\fGC_{n}$}
\begin{equation}
  \fGC_{n} \coloneqq \Tw \Gra_{n}(\varnothing)[-n]
\end{equation}
to be the full graph complex.
It is spanned by graphs containing only internal vertices, with a degree shift: if $\gamma \in \fGC_{n}$ is a graph with $k$ edges and $l$ vertices, then its degree is $k(n-1) - l n + n$.
The differential is given by contracting edges.
Its suspension $\fGC_{n}[n]$ is a CDGA, and the product is the disjoint union of graphs.
Since the differential cannot create new connected components, this a free CDGA, and we have:
\begin{equation}
  \fGC_{n} = S(\GC_{n}[n])[-n]
\end{equation}
where $\GC_{n}$\index{GCn@$\GC_{n}$} is the submodule of connected graph.
The dual $\GC_{n}^{\vee}$ is a (pre)Lie algebra using insertion of graphs.
The variants $\GC_{n}^{2}$ (where bivalent vertices are allowed but univalent vertices forbidden) and $\GC_{n}^{\circlearrowright}$ (where loops are allowed) are defined similarly.

\begin{remark}
  Our notations are slightly nonstandard.
  It is often the dual complex $\GC_{n}^{\vee}$ which is called the graph complex, with a differential that is dually given by creating internal vertices.
  This dual is a Lie algebra using insertion of graphs.
\end{remark}

The homology of this dg-module is particularly hard to compute.
It is known that $H^{0}(\GC_{2}^{\vee})$ is isomorphic to the Grothendieck--Teichmüller Lie algebra $\mathfrak{grt}_{1}$ and that $H^{<0}(\GC_{2}^{\vee}) = 0$~\cite{Willwacher2014}; it is a conjecture that $H^{1}(\GC_{2}) = 0$.
The vanishing of $H^{*}(\GC^{2,\vee}_{n})$ in some degrees shows that the operad of little $n$-disks is intrinsically formal as a Hopf cooperad~\cite{FresseWillwacher2015}.
And this homology computes the homotopy groups of the space of rational automorphisms of the little $n$-disks operad~\cite{FresseTurchinWillwacher2017}.

There is also a Swiss-Cheese version of that graph complex appearing in the description of the model for the Swiss-Cheese operad from~\cite{Willwacher2015}:\index{SGCn@$\SGC_{n}$}\index{fSGCn@$\fSGC_{n}$}
\begin{equation}
  \label{eq.fsgcn}
  \fSGC_{n} \coloneqq \Tw \SGra_{n}(\varnothing, \varnothing)[1-n] = S(\SGC_{n}[n-1])[1-n]
\end{equation}
is the full Swiss-Cheese graph complex, a symmetric algebra on its subcomplex of connected graphs.
The dual\index{KGCn@$\KGC_{n}$}
\begin{equation}\label{eq.kgc}
  \KGC_{n} \coloneqq \SGC_{n}^{\vee}
\end{equation}
is a (pre)Lie algebra, using insertion of graphs, as well as a module over the Lie algebra $\GC_{n}^{\vee}$.
The Kontsevich integrals $c \in \KGC_{n}$ (see Equation~\eqref{eq.konts-coeff}) define a Maurer--Cartan element in this Lie algebra.

\subsection{Model for configuration spaces of closed manifolds}
\label{sec.model-conf-space-closed}

Let us recall the construction from~\cite{Idrissi2018b,CamposWillwacher2016} for configuration spaces of closed oriented manifold $M$.
Our aim, in this work, will be to generalize these two constructions for manifolds with boundary, in a framework that encompasses the two approaches.

In~\cite{CamposWillwacher2016}, one defines a sequence of CDGAs $\Graphs_M$\index{GraphsM@$\Graphs_{M}$} that model the real homotopy type of the configuration spaces of points $M$:
\begin{equation}
  \Graphs_M(n) \xrightarrow{\sim} \OmPA(\FM_M(n)).
\end{equation}

\begin{remark}
  In~\cite{CamposWillwacher2016}, the notation $\Graphs_{M}$ was used for the dual of what we call $\Graphs_{M}$, and $^{*}\Graphs_{M}$ was used for what we call $\Graphs_{M}$.
  We exchanged the notations to be coherent with the rest of the paper, where cooperads/comodules appear far more often than their duals.
\end{remark}

As a vector space, $\Graphs_M(n)$ is spanned by graphs with $n$ labeled external vertices and an unspecified number of indistinguishable internal vertices such that every vertex is decorated by an element of the unital symmetric algebra of the reduced cohomology of $M$, $S(\tilde H(M))$, under the condition that there are no connected components without external vertices.
A typical element of $\Graphs_{M}(n)$ is represented in the following picture, where black vertices are internal, the others are external, and the decorations are denoted by dotted edges (with the $\omega_{i} \in S(\tilde{H}(M))$):
\begin{equation}
  \begin{tikzpicture}[scale=1.4]
    \node[ext] (v1) at (0,0) {$\scriptstyle 1$};
    \node[ext] (v2) at (.5,0) {$\scriptstyle 2$};
    \node[ext] (v3) at (1,0) {$\scriptstyle 3$};
    \node[ext] (v4) at (1.5,0) {$\scriptstyle 4$};
    \node[int] (w1) at (.25,.5) {};
    \node[int] (w2) at (1.5,.5) {};
    \node[int] (w3) at (1,.5) {};
    \node (i1) at (1.7,1) {$\scriptstyle \omega_1$};
    \node (i2) at (1.3,1) {$\scriptstyle \omega_1$};
    \node (i3) at (-.4,.5) {$\scriptstyle \omega_2$};
    \node (i4) at (1.9,.4) {$\scriptstyle \omega_3$};
    \node (i5) at (0.25,.9) {$\scriptstyle \omega_4$};
    \draw (v1) edge (v2) edge (w1) (w1)  edge (v2) (v3) edge (w3) (v4) edge (w3) edge (w2) (w2) edge (w3);
    \draw[dotted] (v1) edge (i3) (w2) edge (i2) edge (i1) (v4) edge (i4) (w1) edge (i5);

    \node at (3,.2) {$\in \Graphs_M(4)$};
  \end{tikzpicture}
\end{equation}

The degree of a graph is $(m-1)\#\text{edges}-m\#\text{int. vertices} + \text{deg. of decorations}$ and the commutative product is given by the union of graphs with superposition of external vertices.

The differential $\delta$ splits as $\delta=\delta_{contr} + \delta_{cut}$, where $\delta_{contr}$ contracts edges adjacent to at least one internal vertex and $\delta_{cut}$ splits any edge into two decorations given by the diagonal class of $M$. Notice that due to the constraint of not allowing connected components without external vertices, $\delta_{cut}$ might create \textit{forbidden} graphs and in those cases such connected components are replaced by their image under the partition function ${Z_M} \colon \mathsf{GC}_{H{(M)}} \to \mathbb R$.

Let us now describe the map $\omega : \Graphs_M(n) \to \OmPA(\FM_M(n))$.
Let $p_{ij}\colon \FM_M(n)\to \FM_M(2)$ (resp.\ $p_{i} : \FM_{M}(n) \to M$) be the projections that forget all but two (resp.\ one) vertices, for $1 \le i, j \le n$, and recall the propagator $\varphi$ from Section~\ref{sec:closed propagator}.
The sub-CDGA of graphs with no internal vertices is generated by edges and decorations.
The map $\omega$ sends an edge connecting vertices $i$ and $j$ to $\varphi_{ij} = p_{ij}^*\varphi$, and it sends a decoration $\alpha \in \tilde{H}(M) \subset S(\tilde{H}(M))$ on vertex $i$ to $p_{i}^{*}(\alpha)$.

The extension of this map to the whole space $\Graphs_M(n)$ is given by the usual Feynman rules.
Let $\Gamma \in \Graphs_{M}(n)$ be a graph with $k$ internal vertices.
To define $\omega_{\Gamma}$, we first define a graph $\Gamma' \in \Graphs_{M}(n+k)$ obtained from $\Gamma$ by replacing the internal vertices by external ones.
We get a form $\omega_{\Gamma'} \in \OmPA(\FM_{M}(n+k))$ using the definition above.
We then consider the projection $\FM_{M}(n+k) \to \FM_{M}(n)$ which forgets the additional $k$ points, and we define $\omega_{\Gamma}$ to be the integral of $\omega_{\Gamma'}$ along the fibers of this projection.
We can summarize this procedure by the formula:
\beq{equ:feynmanclosed}
\Gamma \mapsto \omega_\Gamma = \int_{\fiber} \bigwedge_{(i,i)} \pi_{ij}^*\varphi
\eeq

\begin{theorem}[\cite{CamposWillwacher2016}]\label{thm:CW}
  The prescription above defines a quasi-isomorphism of differential graded commutative algebras
  \[ \Graphs_M(n) \to \OmPA(\FM_M(n)). \]
\end{theorem}

Given a trivialization of the tangent bundle of $M$, one can consider the right operadic action of $\FM_m$ on $\FM_M$. At the level of $\Graphs_M$ this can be expressed as a right $\Graphs_m$ comodule structure given by collapsing subgraphs containing the desired external vertices.

\begin{theorem}[\cite{CamposWillwacher2016}]
  For $M$ be a parallelized manifold, the previous map induces a quasi-isomorphism of Hopf comodules
  \[ \left(\Graphs_M, \Graphs_m\right) \longrightarrow \left( \OmPA(\FM_M), \OmPA(\FM_m)\right). \]
\end{theorem}

\begin{remark}
  Let us denote $A= \Graphs_M(1)$ (or, if desired, its genus 0 part), which is a model for $M$.
  The previous theorem gives us a model for the right action of the fiberwise little disks operad since for $M$ parallelized $\FM_n^M = M \times \FM_n$.
  The result above can be restated as a quasi-isomorphism from the $A$-multicomodule $\left(\Graphs_M, A\otimes \Graphs_m \right)$ to the $\OmPA(M)$-multicomodule $\left( \OmPA(\FM_M),  \OmPA(M\times \FM_m)\right).$
\end{remark}

In~\cite{Idrissi2018b}, one must assume that $M$ is simply connected.
Then one starts from a Poincaré duality model of $M$, say $P$, connected to $\OmPA^{*}(M)$ via some CDGA $A$.
This Poincaré duality model $A$ is used to build a ``small'' model for $\Conf_{k}(M)$, the Lambrechts--Stanley model:\index{GP@$\GG{P}$}
\begin{equation}
  \GG{P}(k) \coloneqq \bigl( P^{\otimes k} \otimes H^{*}(\Conf_{k}(\R^{n})) / (p_{i}^{*}(x) \cdot \omega_{ij} = p_{j}^{*}(x) \cdot \omega_{ij}), d \omega_{ij} = \Delta_{P} \bigr),
\end{equation}
where $H^{*}(\Conf_{k}(\R^{n}))$ is generated by the classes $\omega_{ij} \in H^{n-1}(\Conf_{k}(\R^{n}))$ satisfying the classical Arnold relations~\cite{Arnold1969,Cohen1976}, $\Delta_{P} \in (P \otimes P)^{n}$ is a representative of the diagonal class of $M$ (defined using the Poincaré duality of $P$) and, for $x \in P$ and $1 \leq i \leq n$, we have $p_{i}^{*}(x) = 1 \otimes \dots \otimes 1 \otimes x \otimes 1 \otimes \dots \otimes 1 \in P^{\otimes k}$ with $x$ in position $i$.

One also build a graphical comodule $\Graphs_{A}$, similar to $\Graphs_{M}$ above, but vertices are decorated by elements of $A$ rather than elements of $S(\tilde{H}^{*}(M))$.
There is a quasi-isomorphism (or really, a zigzag, due to technicalities of PA forms) $\Graphs_{A} \to \OmPA^{*}(\FM_{M})$, and under the assumption that $\dim M \geq 4$ and $M$ simply connected, a quasi-isomorphism $\Graphs_{A} \to \GG{P}$.
These maps are compatible with the symmetric group actions; moreover, if $M$ is framed, they are compatible with the action of the operad $\FM_{n}$.
Thus:

\begin{theorem}[{\cite{Idrissi2018b}}]
  Let $M$ be a smooth, simply connected, closed manifold of dimension at least $4$.
  Then we have a zigzag of quasi-isomorphisms, compatible with the symmetric group actions:
  \[ \GG{P}(k) \simeq \OmPA^{*}(\FM_{M}(k)). \]
  Moreover, if $M$ is framed, then $\GG{P}$ is a right Hopf $H^{*}(\FM_{n})$-comodule, and the zigzag is compatible with the actions of $H^{*}(\FM_{n})$ and $\OmPA^{*}(\FM_{n})$ using the zigzag between these two Hopf cooperads defined by Kontsevich formality~\cite{Kontsevich1999}.
\end{theorem}

\subsection{Functoriality of the models}
\label{sec.functoriality-models}

The graphical models above depend on several choices: the choice of CDGAs which are models for the base manifold, and the choice of some Maurer--Cartan elements, namely the partition functions and the MC elements used to twist the base cooperads.
We now make precise the notion that the models are ``functorial'' in terms of these data, and that ``weak equivalences'' (quasi-isomorphisms or gauge equivalences) produce quasi-isomorphic objects.

We will make extensive use of this functoriality in the rest of this paper and will not repeat the arguments.
In our notation for our different models ($\SGraphs_{A,A_{\partial}}$, $\mGraphs_{A}$, etc.) we will not write down the Maurer--Cartan elements, with the understanding that the object still depends on them -- it will be clear from the context what the (gauge equivalence classes) of the Maurer--Cartan elements are.

Let us first deal with the Maurer--Cartan elements.
Two kinds of Maurer--Cartan elements appear in the discussions above:
\begin{itemize}
  \item a first kind is used to twist (co)operads and comodules over them in the sense of~\cite{Willwacher2014,DolgushevWillwacher2015};
  \item a second kind of used to remove ``vacuum type'' components (connected components with only internal vertices).
\end{itemize}
We will write our arguments in a general setting which will then be applicable to our graphical models.

Suppose first we are given a cooperad $\CC$ equipped with a morphism to the cooperad of shifted homotopy Lie algebras $\hoLie_{k}^{\vee}$.
Using the general framework of~\cite{Willwacher2014,DolgushevWillwacher2015}, this morphism correspond to a Maurer--Cartan element $\mu$ in the deformation complex $\mathfrak{g} \coloneqq \Def(\CC, \hoLie_{k}^{\vee})$.
It can be used to produce a ``twisted'' cooperad $\Tw_{\mu} \CC$, whose coalgebras are of the form $(A, d + [x,-])$, where $A$ is a $\CC$-coalgebra and $x$ is a Maurer--Cartan element in $A$ (one must take special care of duals in order to make sense of this).
Moreover, if $\MM$ is a cooperadic right comodule over $\CC$, then it can be twisted into a right comodule $\Tw_{\mu} \MM$ over $\Tw_{\mu} \CC$.

We now reuse the arguments of~\cite[Section~7.2]{CamposWillwacher2016} in order to show that two gauge equivalent Maurer--Cartan elements $\mu \sim \mu'$ produce quasi-isomorphic twists.
By definition, $\mu$ and $\mu'$ are gauge equivalent if there exists a Maurer--Cartan element
\begin{equation}
  \hat{\mu} \in \mathfrak{g}[t,dt] \coloneqq \mathfrak{g} \otimes \Omega^{*}(\Delta^{1}),
\end{equation}
where $\Omega^{*}(\Delta^{1}) = \Bbbk[t,dt]$ is a path object for $\Bbbk$ in the model category of CDGAs, such that $\hat{\mu}|_{t = 0} = \mu$ and $\hat{\mu}|_{t = 1} = \mu'$.
This is equivalent to the existence of an element $\gamma \in \exp \mathfrak{g}$ in the exponential group (equal to $\mathfrak{g}^{0}$ with product given by the Baker--Campbell--Hausdorff formula) and such that the flow of $\gamma$ sends $\mu$ to $\mu'$, i.e.
\begin{equation}
  \mu' = \gamma \cdot \mu \coloneqq \exp(\operatorname{ad}_{\gamma})(\mu) = \mu + \operatorname{ad}_{\gamma}(\mu) + \frac{1}{2} \operatorname{ad}_{\gamma}^{2}(\mu) + ({\dots}).
\end{equation}
This flows induces an isomorphism of twisted Lie algebras $\mathfrak{g}^{\mu} \cong \mathfrak{g}^{\mu'}$.
One then checks, as is done in~\cite[Section~7.2]{CamposWillwacher2016}, that this induces an isomorphism of twisted cooperads and comodules:
\begin{align}
  \Tw_{\mu} \CC & \cong \Tw_{\mu'} \CC, & \Tw_{\mu} \MM \cong \Tw_{\mu'} \MM.
\end{align}
This can be applied to $\CC = \Graphs_{n}$, $\MM = \Graphs_{M}$, for example, and takes care of the first kind of Maurer--Cartan elements.

Let us now deal with the second kind of Maurer--Cartan elements.
Let us assume that we are given
\begin{itemize}
  \item a Hopf cooperad $\CC$ and a Hopf right comodule $\MM$;
  \item a Lie algebra $V^{\vee}$, which can equivalently be seen as a quasi-free CDGA $fV = (S(V[1]), \delta)$ such that $\delta(V) \subset V \oplus S^{2}(V)$;
  \item an $fV$-module structure on each $\MM(k)$ such that the comodule structure maps $\circ_{W}^{\vee} : \MM(U) \to \MM(U/W) \otimes \CC(W)$ preserve this module structure;
  \item a Maurer--Cartan element $Z \in \MC(V^{\vee})$, or equivalently a CDGA morphism $Z : fV \to \Bbbk$; this makes $\Bbbk$ into a left $fV$-module.
\end{itemize}

Then we can define the ``reduced'' $\CC$-comodule $\MM_{Z}$, given in each arity by:
\begin{equation}
  \MM_{Z}(k) \coloneqq \MM(k) \otimes_{fV} \Bbbk.
\end{equation}

Now let us assume that $Z$ is gauge equivalent to some other Maurer--Cartan $Z'$.
Suppose moreover that $\MM(k)$ is cofibrant as an $fV$ module.
This is for example the case if $\MM(k)$ is quasi-free as an $fV$-module and is equipped with a good filtration.

We then use a proof technique similar to what is done in~\cite[Section~4.1]{Idrissi2018b} to show that $\MM_{Z}$ is quasi-isomorphic to $\MM_{Z'}$.
The gauge equivalence $Z \sim Z'$ can be seen as a homotopy between the two CDGAs morphism $Z$ and $Z'$:
\begin{equation}
  \begin{tikzcd}
    {} & fV \ar[dl, bend right, "Z" swap] \ar[d, dashed, "\exists h"] \ar[dr, bend left, "Z'"] \\
    \Bbbk & \Omega^{*}(\Delta^{1}) = \Bbbk[t,dt] \ar[l, "\operatorname{ev}_{0}" swap, "\sim"] \ar[r, "\operatorname{ev}_{1}", "\sim" swap] & \Bbbk
  \end{tikzcd}
  .
\end{equation}
We can thus define a third Hopf right $\CC$-comodule $\hat{\MM}$ by $\hat{\MM}(k) \coloneqq \MM(k) \otimes_{fV} \Omega^{*}(\Delta^{1})$, where the action of $fV$ on $\Omega^{*}(\Delta^{1})$ is induced by $h$.
We also have morphisms of comodules $\MM_{Z} \gets \hat{\MM} \to \MM_{Z'}$, induced by evaluation at $t = 0$ and $t = 1$.
Because we have assumed that each $\MM(k)$ is cofibrant as a right $fV$-module, the functor $\MM(k) \otimes_{fV} -$ preserves quasi-isomorphisms, hence both morphisms of comodules are quasi-isomorphisms in each arity.

It then remains to show that the models are functorial in terms of the CDGA models for the base manifold.
The following general method will work for all of the graphical models we build.
Let $f : A \to A'$ be a quasi-isomorphism of CDGAs which sends the MC element in $A$ to the MC element in $A'$.
It induces a morphism between the corresponding graphical models, by applying $f$ to the labels of all the vertices.
We can filter the graphical model by the total number of vertices, as the differential always decreases it.
The morphism $f$ being a quasi-isomorphism, we obtain an isomorphism on the $E^{1}$ pages of the associated spectral sequences.
Standard arguments (the filtration is bounded in each degree for a fixed number of external vertices) then show that $f$ induces a quasi-isomorphism on the graphical models.

We must also deal with the choice of MC elements corresponding to the partition function for this case.
If we are given an MC element $z$ with values in $A$, then we can simply apply $f$ to it to obtain an MC element $z' = f(z)$ with values in $A'$.
Conversely, if we are given an MC element $z'$ with values in $A'$, then we apply the Goldman--Millson theorem to show that $z'$ is gauge equivalent to $f(z)$ for some MC element with values in $A$.
As we have already proved that gauge equivalent elements yield quasi-isomorphic objects, we then obtain a zigzag (and not necessarily a direct map) of quasi-isomorphism between the model defined by $(A,z)$ and the one defined by $(A',z')$.

\subsection{Model for the fiberwise little \texorpdfstring{$n$}{n}-disks operad}\label{sec:model-fiberwiseLD}
Recall from~\cite{CDW17} the construction of a model of the action of the fiberwise little disks operad on $\FM_M$ for non-parallelized manifolds (see Section~\ref{sec:fiberwiseLD} for the definition of $\FM_{m}^{M}$).

For $n=\dim M$, and a CDGA $A$ one defines a symmetric sequence of graded commutative algebras\index{GranA@$\Gra_{n}^{A}$}
\begin{equation}
  \Gra_n^A(r) \coloneqq A\otimes \Gra_n(r).
\end{equation}

Suppose we are given a CDGA map $f:A\to \OmPA(M)$ and a fiberwise volume form $\varphi\in \Omega_{\triv}(\FM_n^M(2))$ on the sphere bundle of $M$, which we assume to be (anti-)symmetric and such that $d\varphi=f(E)$, with $E\in A$. Note that for $n$ even we can assume $E=0$.
Then, using the usual Feynman rules, we can define a map of Hopf symmetric sequences
\beq{equ:GraMFeynman}
\begin{aligned}
  (\Gra_n^A, d_A + d + E T\cdot) & \to \OmPA(\FM_n^M)                                              \\
  a\otimes \Gamma                & \mapsto a\wedge \bigwedge_{(ij)\in E_{\Gamma}} p_{ij}^*\varphi,
\end{aligned}
\eeq
where the operation $T\cdot$ is the action of the tadpole
\begin{equation}
  T \coloneqq
  \begin{tikzpicture}
    \node (v0) at (0.5,1) [int] {};
    \draw (v0) to [out=45,in=135,loop] ();
  \end{tikzpicture}
  \in \GC_n^{\vee},
\end{equation}
removing one edge from the graph.

\begin{remark}
  Here we do not allow tadpoles in graphs in $\Gra_n$.
  We could however extend the formula to graphs with tadpoles as follows.
  Note that the action $T\cdot$ does not remove tadpole edges, so that the ``differential of a tadpole edge'' is zero.
  Hence we may consistently understand $p_{ii}^*\varphi \coloneqq 0$ in the above formula, to extend it to the tadpole case.
\end{remark}

Also note that the map \eqref{equ:GraMFeynman} is compatible with the cooperad structures on both sides.
(I.e., the structure of a cooperad in CDGAs under $A$ is intertwined with the structure of almost cooperad in CDGAs under $\OmPA(M)$, using the map $f$ from $A$ to $\OmPA(M)$, see Remark~\ref{rmk:almost}.)

We may now twist the map \eqref{equ:GraMFeynman}.
This is done by formally introducing an additional arity-zero cogenerator to $\Gra_n^A$ producing the collection of CDGAs $A\otimes \Tw\Gra_n(r)$
built from graphs with numbered external and unnumbered ``internal'' vertices.
The map to $\OmPA(\FM_M)$ is extended by taking a fiber integral over the positions of the points corresponding to internal vertices
\begin{equation}
  \begin{aligned}
    A\otimes \Tw\Gra_n & \to \OmPA(\FM_n^{M})                                                        \\
    a\otimes \Gamma    & \mapsto a\wedge \int_{\fiber} \bigwedge_{(ij)\in E_\Gamma} p_{ij}^*\varphi.
  \end{aligned}
\end{equation}

We mind that this map is a priori not compatible with the differential as is, due to boundary terms resulting from applying the fiberwise Stokes' formula.
More concretely, for $A=\OmPA(M)$, the arity zero piece produces a Maurer--Cartan element\index{zomega@$z_{\Omega}$}
\begin{equation}
  z_\Omega \coloneqq ET + \sum_\gamma \left( \scalebox{.8}{$ \int_{\fiber} \bigwedge_{(ij)\in E\Gamma} p_{ij}^*\varphi$} \right) \otimes \gamma \in \OmPA(M) \hotimes \GC^{\vee}_n.
\end{equation}
such that (only) twisting with this MC element makes the above map compatible with the differential.
Fortunately, there is the vanishing Lemma:
\begin{lemma}\label{lem:vanI}
  In the situation above, $z_\Omega=ET$.
\end{lemma}
\begin{proof}[Sketch of proof, recollected from \cite{CamposWillwacher2016}]
  This vanishing result is shown in the following way (see \cite{CamposWillwacher2016} for details):
  If the graph $\gamma$ has a vertex of valence 1 (or 0), then a simple degree counting argument asserts that the integral is zero.
  If the graph has a vertex of valence $2$, then using the symmetry property \eqref{equ:propsymmclosed} one can show that the integral vanishes by a reflection argument due to Kontsevich. Hence we may assume that $\gamma$ has no vertices of valence $\leq 2$.
  But then if $n\geq 3$ the integral vanishes again by a simple degree counting argument.
  If $n=2$ one needs to use a more sophisticated vanishing Theorem of Kontsevich.
\end{proof}

We finally pass to the quotients:\index{GraphsmA@$\Graphs_{m}^{A}$}
\begin{equation}
  \Graphs_n^A(r) \coloneqq A \otimes \Graphs_n(r)
\end{equation}
and equip it with the differential
\begin{equation}
  d_A + d + E T \cdot (-).
\end{equation}
One can show in particular:
\begin{proposition}
  If $f:A\to \OmPA(M)$ is a quasi-isomorphism of CDGAs, then the map $\Graphs_n^A\to \OmPA(\FM_n^M)$ is a quasi-isomorphism of collections of CDGAs, compatible with the (almost) cooperad structures (see Remark~\ref{rmk:almost}).
\end{proposition}

One can furthermore show that the action of the fiberwise little disks operad $\FM_n^M$ on $\FM_M$ naturally intertwines with the right coaction of $\Graphs_n^A$ on $\Graphs_M$
\begin{equation}
  \begin{aligned}
    \Delta_i \colon \ntstG_A & \to \Graphs_A \otimes_A (\Graphs_n^A)                                                                    \\
    \Gamma                   & \mapsto \sum_{\Gamma' \text{ subgraph at } i} \Gamma/\Gamma' \otimes_A (\omega_{\Gamma'}\otimes \Gamma')
  \end{aligned}
\end{equation}
where $\omega_{\Gamma'}$ represents the cohomological decorations of $\Gamma'$.
Here we assume that one makes choices so that $\Graphs_A$ supports an $A$-action.
\begin{proposition}[\cite{CDW17}]\label{prop:CDW}
  Let $M$ be a closed oriented manifold. There exists a quasi-isomorphisms of right comultimodules
  \[
    \left(\Graphs_A, \Graphs_n^A\right) \longrightarrow \left(\OmPA(\FM_M), \OmPA(\FM_n^M)\right)
  \]
  extending the map from Theorem \ref{thm:CW}.
\end{proposition}

\subsection{The local Maurer--Cartan element I}\label{sec:localMC1}

We want to highlight one particular aspect of the constructions recalled above, and slightly generalize the setting.
To this end, suppose more generally that $\R^n\hookrightarrow Y\to B$ is a bundle of $n$-dimensional oriented vector spaces over a connected base $B$.
The example to keep in mind should be $B = M$ and $Y = TM$.

We may form the associated sphere bundle $S^{n-1}\to S_Y\xrightarrow{\pi} B$.
We choose a fiberwise volume form on $S_Y$, that is, an element $\omega \in \Omega^{n-1}(S_Y)$ such that $\int_{\fiber}\omega=1$ and $d\omega=\pi^* E$ for some $E\in \Omega^n(B)$.
By definition $E$ is the Euler class of $Y$.
Also, $\omega$ is unique up to addition of exact elements and basic forms, i.e., pullbacks under $\pi$ of closed forms representing elements in $H^{n-1}(B)$.

For simplicity (though this is not strictly necessary) we furthermore assume that $\omega$ is (anti-)symmetric under the involution $I$ of the spheres:
\beq{equ:omegasymm}
I^*\omega = (-1)^n \omega.
\eeq
This can always be achieved by symmetrization $\omega\mapsto (\omega+(-1)^n I^*\omega)/2$.
(If we did not ask for \eqref{equ:omegasymm}, then we would just replace $\GC_n$ by the quasi-isomorphic complex of directed graphs.)

Using the usual configuration space formulas, we create an element\index{zomega@$z_{\Omega}$}
\begin{equation}
  z_\Omega \coloneqq E \otimes \tadpole + \sum_{\gamma} \bigl( \scalebox{.8}{$\int_{\fiber} \bigwedge_{ij}\pi_{ij}^*\omega$} \bigr) \otimes \gamma\in \OmPA(B)\hotimes \GC_n^{\vee},
\end{equation}
where we integrate over the fiber in $\FM_n^Y$, the fiberwise configuration space of points on $Y$ modulo scaling as defined in Section~\ref{sec:fiberwiseLD}.
Notice that there is an identification of $B$-bundles $\FM_n^Y(2)=S_Y$.
By the Stokes formula, $z_\Omega$ is a Maurer--Cartan element.

We may now essentially repeat the constructions of Section~\ref{sec:model-fiberwiseLD}.
Suppose a $f: A\to \OmPA(B)$ is a CDGA quasi-isomorphism, and suppose further that $z\in A\otimes\GC_n^{\vee}$\index{z@$z$} is a Maurer--Cartan element satisfying $f(z)=z_\Omega$.
Then, as above, we may construct a quasi-isomorphism of collections of CDGAs\index{GraphsmA@$\Graphs_{m}^{A}$}
\begin{equation}
  \Graphs_n^A \coloneqq (A\otimes \Graphs_n)^z \to \OmPA(\FM_n^Y)
\end{equation}
compatible with the cooperadic cocompositions.
It is also clear that in this form the construction is naturally functorial in $(A,z)$.
Furthermore, changing the MC element $z$ by gauge transformations produces quasi-isomorphic collections of CDGAs.
Hence we are interested in determining the gauge equivalence class of $z$ or $z_\Omega$.
Unfortunately, the simple argument of Lemma \ref{lem:vanI} fails in general, if the dimension of the base is greater than $n$.

Let us nevertheless note:
\begin{lemma}\label{lem:gaugeindep}
  The gauge equivalence class of $z_\Omega$ is independent of the choice of $\omega$.
\end{lemma}
\begin{proof}
  Changing the propagator by some exact form $\beta$, we may just apply the construction to the propagator $\omega+d(t\beta)$ for the bundle $Y\times I\to B\times I$ (with $I=[0,1]$ the interval) to find an explicit gauge equivalence.
  Changing the propagator by $\pi^*\alpha$, with $\alpha$ a closed basic form, changes $m$ by gauge transformation with $\tadpole \alpha$.
  (Attention: Here it is important that we include the tadpole in our complex.)
\end{proof}

The above construction is furthermore functorial in the bundle $Y$.
Hence we immediately arrive at:
\begin{proposition}\label{prop:fromzu}
  Let $\tau \colon M\to B\SO(n)$ be the classifying map of the frame bundle.
  We have that $z_{\Omega}$ as above is gauge equivalent to $\tau^*z_{u}$, where\index{zu@$z_{u}$}
  \[ z_u\in H^{*}(B\SO(n)) \hotimes \GC_n \]
  is the universal Maurer--Cartan element constructed in~\cite{KhoroshkinWillwacher2017}.
\end{proposition}

\begin{remark}
  The universal Maurer--Cartan element $z_{u}$ can equivalently be seen as applying the above construction for the universal bundle $(Y \to B) = (E\SO(n) \to B\SO(n))$, which governs the $\SO(n)$-action on the little $n$-disks operad.
\end{remark}

\begin{remark}\label{rem:univprop}
  The careful reader might object that in the universal case we apply the construction to the infinite dimensional base $B\SO(n)$, then we pull back forms.
  However, all constructions here can be conducted in finite dimensions effectively, using the standard ``Chern--Weil''-type map from the equivariant forms on a space to the forms on the associated bundle.

  Concretely, for the equivariant forms (i.e., forms on the universal sphere bundle) we may take the Cartan model $A_{\text{Cartan}}=(S(\alg g^*[-2])\otimes \Omega(S^{n-1}))^G$. An explicit formula for the propagator there has been given in~\cite[Appendix A]{KhoroshkinWillwacher2017}.

  Choose some $\alg g$-valued connection $A\in \Omega^1(FY)$ on the frame bundle $Y$, i.e., $\iota_x A=x$ for all $x\in \alg g$ and $A$ is $G$-invariant. Then we have a map
  \begin{equation}
    \begin{aligned}
      A_{\text{Cartan}} & \to \Omega(FY\times S^{n-1})         \\
      \alpha            & \mapsto \exp(\pm \iota A) \alpha(F),
    \end{aligned}
  \end{equation}
  where $F=dA+[A,A]/2$ is the curvature, to be ``inserted'' in the factors $\alg g^*$. One checks that the resulting form is $G$-basic and hence descends to a form{\tiny } on the quotient $S_Y$.
  This is the desired propagator. As in~\cite[Section~6]{KhoroshkinWillwacher2017} one checks that the resulting MC element is precisely the pullback of the universal one as stated.
\end{remark}

By the above, it suffices to understand the universal MC element $m_u$.
This has been done in~\cite[Theorem 7.1]{KhoroshkinWillwacher2017} with the following result.
\begin{theorem}[\cite{KhoroshkinWillwacher2017}]\label{thm:KW}
  The element $z_u$  is gauge equivalent to
  \[
    \begin{cases}
      E\tadpole                   & \text{for $n$ even,} \\
      \sum_{j\geq 1}
      \frac{p_{2n-2}^{j}}{4^j}
      \frac{1}{2(2j+1)!}
      \begin{tikzpicture}[baseline=-.65ex]
        \node[int] (v) at (0,.5) {};
        \node[int] (w) at (0,-0.5) {};
        \draw (v) edge[bend left=50] (w) edge[bend right=50] (w) edge[bend left=30] (w) edge[bend right=30] (w);
        \node at (2,0) {($2j+1$ edges)};
        \node at (0,0) {$\scriptstyle\cdots$};
      \end{tikzpicture} & \text{for $n$ odd,}
    \end{cases}
  \]
  where $p_{2n-2}\in H^{8n-8}(B\SO(n))$ is the top Pontryagin class.
\end{theorem}

From this we can immediately (re)derive a weaker version of Lemma \ref{lem:vanI}:
\begin{corollary}\label{cor:locvanishing1}
  If $\dim(B)=n$, then $z_\Omega$ is gauge equivalent to $E\tadpole$.
\end{corollary}
For later purposes let us give yet another proof of the result, that generalizes well.
\begin{proof}[Elementary independent proof of Corollary \ref{cor:locvanishing1}]
  For our purposes it in fact suffices to record that, using the $O(n)\supset \SO(n)$-action on $S^{n-1}$, the element $m_u$ is $\Z_2=\pi_0(O(n))=O(n)/\SO(n)$-invariant, with $\Z_2$ acting on $\GC_n$ by multiplication by $(\pm)^{\text{loop number}}$.
  This invariance statement is shown by a simple and elementary reflection argument, cf.~\cite[Lemma D.1]{KhoroshkinWillwacher2017}.
  One merely notes that $\omega$ can be taken reflection anti-invariant, and hence the integrals inherit that property.
  Recall also that the $\Z_2$-action on $H(B\SO(n))$ is trivial on all Pontryagin classes and is the sign action on the Euler class.
  Finally it is known that the $k$-loop part (with $k\geq 2$) of $H(\GC_n)$ satisfies the following degree bounds (cf. Lemma \ref{lem:GCdegbounds})
  \beq{equ:GCdegbds}
  H^{>-(k+2)(n-3)-3}_{k\text{-loop}}(\GC_n) = 0
  \eeq
  In particular, if $\dim(B)=n\geq 3$, then only 1-loop classes can occur in the MC element $m$ by degree reasons.
  By the above parity argument the coefficient of such 1-loop classes must involve one copy of the Euler class at least (since all Pontryagin classes are even) and hence by degree reasons there are no further possible corrections on top of the leading term, and we hence arrive at the conclusion that
  \[
    z\sim E \tadpole. \qedhere
  \]
\end{proof}

\begin{remark}\label{rem:tadpoles}[A remark on tadpoles]
  We note that using the definitions of~\cite{CamposWillwacher2016} tadpoles, i.e., edges connecting a vertex to itself are not allowed in graphs in $\Graphs_M$.
  We use different conventions, tadpoles are generally allowed in all graphs and graph complexes.
  We extend the Feynman rules \eqref{equ:feynmanclosed} by assigning to a tadpole the zero form, so that eventually the construction of~\cite{CamposWillwacher2016} is recovered.
  However, this rule might seem inconsistent at first sight, because the piece $\delta_{cut}$ of the differential of a tadpole edge should formally map a tadpole to a representative of the Euler class, which is generally not exact.
  But note that this piece of the differential is precisely cancelled by the action of the Maurer--Cartan element  $z$ above, which removes a tadpole and replaces it by (a representative of) the Euler class, making all formulas formally consistent.
\end{remark}

\section{A model for \texorpdfstring{$\SFM_{M}$}{SFM\_M}}
\label{sec.model-right-hopf}

\subsection{The local Maurer--Cartan element II and models for the fiberwise Swiss-Cheese operad}\label{sec:locMC2}

We can repeat the discussion of Section~\ref{sec:localMC1} in the local case ``with boundary''.
More precisely, we now consider an oriented vector bundle
\begin{equation}
  \R^{n}\hookrightarrow Y|_{\partial B} \to \partial B,
\end{equation}
where $\partial B$ is the boundary of some manifold $B$ and $Y$ is a vector bundle on $B$.
The most notable example for us will be $\partial B = \partial M$ and $Y = TM$.

Let $\Ha_n$\index{Han@$\Ha_{n}$} be the upper half-space of $\R^n$, which we consider with its natural $G=\SO(n-1)$-action.
The group $G$ acts on $\Fr Y|_{\partial B}$ naturally.
To $Y$ we may then associate the half-space bundle
\begin{equation}
  \Fr Y|_{\partial B} \times_G \Ha_n.
\end{equation}
By analogy with Section~\ref{sec:localMC1}, we may consider compactified fiberwise configuration spaces of points\index{SFMnY@$\SFM_{n}^{Y}$}
\begin{equation}
  \SFM_{n}^{Y}(r,s) \coloneqq \Fr Y|_{\partial B} \times_{\SO(n-1)} \SFM_{n}(r,s)
\end{equation}
in these half-spaces, which form a colored operad in spaces over $\partial B$.

If $(r,s)=(0,2)$, each fiber is a ``Kontsevich eye''~\cite[Section~5.2]{Kontsevich2003}, i.e., the space $\SFM_{n}(0,2)$ (see Figure~\ref{fig:eye} adapted from~\cite{Kontsevich2003}).
To explain the figure, note that up to translation and rescaling, we can fix the first point (e.g., at $(0,\dots,0,1) \in \Ha_{n}$).
Then the second point moves in $\Ha_{n}$ with that first point removed.
It can either come infinitesimally close to the first point (the ``iris''), to $\R^{n-1} \times \{0\} \subset \R^{n}$ (the ``lower lid''), or go to infinity, which can equivalently be seen as the first point being infinitesimally close to $\R^{n-1} \times \{0\}$ (the ``upper lid'')

\begin{figure}[htbp]
  \centering
  \begin{tikzpicture}[every node/.style={font=\tiny, inner sep=1pt}, decoration={snake, amplitude=1pt}]
    \path[draw, fill=black!10] (-2,0) -- (2,0) arc[radius=2, start angle=0, end angle=180];
    \node[circle,draw,fill=white,inner sep=0.2cm] (O) at (0,1) {};
    \node (iris) at (1.5,2) {iris};
    \draw[->, very thin, decorate] (iris) -- (O);
    \node (lower lid) at (-2.5,1.5) {lower lid};
    \draw[->, very thin, decorate] (lower lid) -- (-.7,0);
    \node (upper lid) at (-1.5,2.5) {upper lid};
    \draw[->, very thin, decorate] (upper lid) -- (120:2);
  \end{tikzpicture}
  \caption{The Kontsevich eye for $n=2$}
  \label{fig:eye}
\end{figure}

\begin{proposition}
  There exists a ``propagator''\index{propagator}
  \[ \tilde \omega \in \Omega^{n-1}(\SFM_n^{Y}(0,2)) \]
  satisfying the following properties:
  \begin{itemize}
    \item restricted to each fiber (Figure~\ref{fig:eye}) the propagator becomes a volume form at the iris (point $2$ is infinitesimally close to point $1$), and zero on the lower lid (point $2$ is infinitesimally close to $\R^{n-1}$).
    \item at the iris, the volume form is (anti-)symmetric.
    \item $d\tilde \omega = 0$.
  \end{itemize}
\end{proposition}
\begin{proof}
  The first property is achieved by the local construction of the propagator~\cite{Willwacher2015} which is given as a pullback of the volume form on $S^{n-1}$ and is therefore closed.
  One can then follow the usual globalization argument using the \v{C}ech--de Rham complex of a good cover on $B$ and we get that the obstruction to globalize the propagator to a closed form is the Euler class  of $Y|_{\partial B} \times \R\to B$.
  This class vanishes because the latter bundle has a nonzero global section.
\end{proof}

The boundary value at the iris of a propagator $\tilde \omega$ is a ``non-boundary-case''-propagator $\omega \in \Omega^{n-1}(S_{Y|_{\partial B}})$ in the sense of Section~\ref{sec:localMC1}.
Therefore it gives rise to a MC element\index{zomega@$z_{\Omega}$}
\begin{equation}
  z_{\Omega} \in \Omega^*(\partial B) \hotimes \GC_n^{\vee}.
\end{equation}

We may extend this to a Maurer--Cartan element in the semi-direct product:
\begin{equation}
  z_{\Omega} + z_{\Omega}^{\partial} \in \Omega^{*}(B) \hotimes \GC^{\vee}_n \ltimes \Omega^{*}(\partial B) \hotimes \SGC^{\vee}_n,
\end{equation}
where\index{zomegad@$z_{\Omega}^{\partial}$}
\begin{equation}
  z_{\Omega}^{\partial} = \sum_{\Gamma} \bigl( \scalebox{.8}{$\int_{\fiber} \bigwedge_{ij} \pi_{ij} \tilde \omega$} \bigr) \otimes \Gamma \in \Omega^{*}(\partial B) \hotimes \SGC_{n}^{\vee}.
\end{equation}

\begin{remark}
  Applied to the trivial bundle over $B = \Ha_{n}$, this yields the Kontsevich coefficients $c \in \SGC_{n}^{\vee}$ of Equation~\eqref{eq.konts-coeff}.
\end{remark}

We also note that, since $\SFM_n(2,0)/(\text{lower lid})\simeq S^{n-1}$, picking a propagator in the boundary case is topologically the same as picking a propagator in the previous subsection.
We then have the result analogous to Lemma \ref{lem:gaugeindep}:
\begin{lemma}
  The gauge equivalence class of the Maurer--Cartan element $z_{\Omega} +z_{\Omega}^{\partial}$ is independent of the choice of $\tilde \omega$.
  \qed
\end{lemma}

The construction is functorial in $Y$.
We have a result analogous to Proposition \ref{prop:fromzu}:
\begin{proposition}
  Let $\tau \colon M\to B\SO(n)$ be the classifying map of the frame bundle.
  We have that the Maurer--Cartan element $z_{\Omega} + z_{\Omega}^{\partial}$ above is gauge equivalent to $\tau^*(z_{u}+z^{\partial}_u)$, where\index{zu@$z_{u}$}\index{zud@$z_{u}^{\partial}$}
  \[ z_u + z^{\partial}_u \in H^{*}(B\SO(n-1)) \hotimes \GC_n^{\vee} \ltimes H^{*}(B\SO(n)) \hotimes \SGC_n^{\vee} \]
  is the universal Maurer--Cartan element, constructed as above for $(Y \to B) = (E\SO(n) \to B\SO(n))$, governing the $\SO(n)$-action on the little $n$-disks operad.
  \qed
\end{proposition}

\begin{remark}
  The universal propagator $\tilde \omega_u$ in the boundary case can be constructed by the method of ``mirror charges'' from the ``no boundary'' universal propagator $\omega_u$ whose explicit formula is given in~\cite{KhoroshkinWillwacher2017}, (restricted to $\SO(n-1)\subset\SO(n)$).
  Concretely, (with slightly imprecise notation)
  \[
    \tilde \omega_u(x,y) = (\omega_u(x,y)-\omega_u(Ix,y))|_{\alg{so}_{n-1}},
  \]
  where $X$ is the reflection of $\R^n$ at the boundary.
\end{remark}

We would now like to proceed as in Section~\ref{sec:localMC1} to evaluate $z_u+z^{\partial}_u$. Unfortunately, a result analogous to Theorem \ref{thm:KW} is currently not available (although it could be done if necessary).
Fortunately, however, the short independent proof of Corollary \ref{cor:locvanishing1} still can be transcribed to the ``boundary''-situation.

\begin{proposition}
  \label{prop:locvanishing2}
  In the case that $\dim \partial B = n-1$ and $n\geq 4$, we have that
  \[
    z_{\Omega} + z_{\Omega}^{\partial} \sim  0 + z^{\partial}_{0} + z^{\partial}_{1}
  \]
  where:\index{z0d@$z_{0}^{\partial} + z_{1}^{\partial}$}
  \begin{align*}
    z_{0}^{\partial}
     & \coloneqq
    \sum_{n \ge 0} \frac{1}{n!}
    \begin{tikzpicture}[baseline=.5cm]
      \draw (-1,0)--(1,0);
      \node[int] (v) at (0,1) {};
      \node[int] (v1) at (-.8,0) {};
      \node[int] (v2) at (-.4,0) {};
      \node at (0,.5) {$\scriptstyle \cdots$};
      \node[int] (v3) at (.4,0) {};
      \node[int] (v4) at (.8,0) {};
      \draw[-latex] (v) edge (v1) edge (v2) edge (v3) edge (v4);
      \draw [
        thick,
        decoration={
            brace,
            mirror,
            raise=0.1cm
          },
        decorate
      ] (v1.south west) -- (v4.south east)
      node [pos=0.5,anchor=north,yshift=-0.2cm] {$\scriptstyle n\times$};
    \end{tikzpicture}
    ,
     & z_{1}^{\partial}
     & \coloneqq
    E \cdot
    \biggl(
    \sum_{n\geq 0}
    \begin{tikzpicture}[baseline=.5cm]
      \draw (-1,0) -- (1,0);
      \node[int] (v) at (0,1) {};
      \node[int] (v1) at (-.8,0) {};
      \node[int] (v2) at (0,0) {};
      \node[int] (v3) at (.4,0) {};
      \node[int] (v4) at (.8,0) {};
      \draw[-latex] (v) edge[bend left] (v1) edge[bend right] (v1) edge (v2) edge (v3) edge (v4);
      \draw [
        thick,
        decoration={
            brace,
            mirror,
            raise=0.1cm
          },
        decorate
      ] (v2.south west) -- (v4.south east)
      node [pos=0.5,anchor=north,yshift=-0.2cm] {$\scriptstyle n\times$};
    \end{tikzpicture}
    \pm
    \sum_{n\geq 0}
    \begin{tikzpicture}[baseline=.5cm]
      \draw (-1,0) -- (1,0);
      \node[int] (v) at (0,1) {};
      \node[int] (v1) at (-.4,0) {};
      \node[int] (v2) at (0,0) {};
      \node[int] (v3) at (.4,0) {};
      \draw[-latex] (v) edge[loop] (v) edge (v1) edge (v2) edge (v3);
      \draw [
        thick,
        decoration={
            brace,
            mirror,
            raise=0.1cm
          },
        decorate
      ] (v1.south west) -- (v3.south east)
      node [pos=0.5,anchor=north,yshift=-0.2cm] {$\scriptstyle n\times$};
    \end{tikzpicture}
    \biggr),
  \end{align*}
  and $E\in \Omega^{n-1}(B)$ is (a representative of) the Euler class of $Y$.
\end{proposition}

Before we give the proof, let us check:
\begin{lemma}
  The element above is indeed a Maurer--Cartan element.
\end{lemma}
\begin{proof}
  We have $E^2=0$ by degree reasons.
  Therefore we just have to check that $\delta z^{\partial}_0 + \frac 1 2 [z^{\partial}_0,z^{\partial}_0]=0$ and that $z^{\partial}_1$ is $z^{\partial}_0$-closed, i.e., $\delta z^{\partial}_1 + [z^{\partial}_0,z^{\partial}_1]=0$.
  Both facts are easy explicit computations.
\end{proof}

\begin{proof}[Proof of Proposition~\ref{prop:locvanishing2}]
  We will first consider the universal MC element $z_u+Z_u$ from~\cite{KhoroshkinWillwacher2017} (see Proposition~\ref{prop:fromzu}; the piece $z_u$ may in fact be evaluated by using Theorem \ref{thm:KW}, but we shall not need to use it here.)
  As has been shown in Proposition~\ref{prop.kgc-z-0}, we have that
  \begin{equation}
    H^{*}(\SGC_n^{\vee,z^{\partial}_0}) = H(\GC_n)[-1]\oplus H(\GC_{n-1}).
  \end{equation}
  By the same reflection argument as in the elementary proof of Corollary \ref{cor:locvanishing1}, we see that any odd-loop-order graph must have a coefficient divisible by the Euler class, which is of degree $n-1$.
  In particular, pulling back via $H(BSO(n-1))\to \Omega(B)$, by degree counting we see that the 1-loop graphs do not contribute to the MC element except possibly for the 1-loop graph  in $\GC_{n-1}$ with one vertex, which corresponds to the term $z^{\partial}_1$ above.

  By the degree bounds \eqref{equ:GCdegbds}, however, we find that higher loop classes can also not contribute non-trivially, using that $n-1\geq 3$.
  This shows that the loop order zero piece $0+z^{\partial}_0$ is deformable in at most one direction, namely $z^{\partial}_1$.
  One then shows by an explicit integral computation (which we leave to the reader) that the pieces $z^{\partial}_0$ and $z^{\partial}_1$ indeed occur with the weights shown.
\end{proof}

We can now mimic the constructions of Sections~\ref{sec:model-fiberwiseLD} and~\ref{sec:localMC1} and build models of the fiberwise Swiss Cheese operad $(\FM_m^M,\SFM_m^M)$.
To this end we consider models $A$ and $A_\p$ for $M$ and $\partial M$, fitting into a commutative diagram:
\beq{equ:AAddiag}
\begin{tikzcd}
  A \ar{d} \ar{r} & \OmPA(M) \ar{d} \\
  A_\p \ar{r} & \OmPA(\p M).
\end{tikzcd}
\eeq
We can then consider the collection of CDGAs $A_\p \otimes \SGra_n$.
It comes equipped with maps of CDGAs
\[
  (A_\p \otimes \SGra_n, d_{A_\p} + d) \to \OmPA(\SFM_n^M).
\]
These maps are compatible with the cooperad structures and the right coaction of $A\otimes \Gra_n$ (respectively $\OmPA(\FM_n^M)$).

Again we may twist the map, extending the twist discussed in Section~\ref{sec:model-fiberwiseLD}.
We assume that were are given MC element $(z,z^{\partial})\in A\otimes \GC_n \ltimes A_\p \otimes \KGC_n$ mapping to $(z_\Omega,z^{\partial}_\Omega)$ above under the given maps $A\to \OmPA(M)$, $A_\p\to \OmPA(\p M)$.
Then we obtain a map of collection of CDGAs\index{SGraphsnA@$\SGraphs_{n}^{A_{\partial}}$}
\begin{equation}
  \SGraphs_n^{A_\p} \coloneqq (A_\p \otimes \SGraphs_n)^{z,z^{\partial}} \to \OmPA(\SFM_n^M)
\end{equation}
compatible with the cooperadic structure and with the right action of $\SGraphs_n^{A}$.
If the horizontal arrows in \eqref{equ:AAddiag} are quasi-isomorphisms, then we obtain accordingly a quasi-isomorphism of colored Hopf (almost, see Remark~\ref{rmk:almost}) cooperads
\begin{equation}
  (\Graphs_n^A,\SGraphs_n^{A_\p}) \to (\OmPA(\FM_M), \OmPA(\SFM_n^M)),
\end{equation}
which provides our desired combinatorial model for the fiberwise Swiss Cheese operad.
Furthermore the construction is functorial in the pair $A\to A_\p$, and changing the MC element $(z,z^{\partial})$ to a gauge equivalent one produces a quasi-isomorphic colored Hopf cooperad, see Section~\ref{sec.functoriality-models}.
We may hence assume that the MC element has the simple form discussed above.

\subsection{Colored labeled graphs}
\label{sec.colored-labeled-graphs}

We now introduce a colored version of the graph comodules from Section~\ref{sec.model-conf-space-closed}.

\begin{definition}
  Given diagonal data $(A, A_\p, \rho, \Delta_A, \sigma_{A})$ (see Section~\ref{sec.a-data}), we define the CDGA of $(A,A_\p)$-labeled graphs on the finite sets $U$, $V$ as:\index{SGraAAd@$\SGra_{A,A_{\partial}}$}
  \begin{equation}
    \SGra_{A, A_\p}(U,V) \coloneqq \bigl( (A_\p)^{\otimes U} \otimes A^{\otimes V} \otimes \SGra_{n}(U,V), d \bigr)
  \end{equation}
  with differential defined by
  \begin{align}
    de_{vv'}
     & =
    \begin{cases}
      \iota_{vv'}(\Delta_{A}) & v \neq v' \in V, \\
      0                       & v = v' \in V.
    \end{cases}
    \\
    de_{vu}
     & = \iota_{vu}(\Delta_{A, A_{\p}}), \; u \in U, v \in V.
  \end{align}
\end{definition}

\begin{remark}
  The fact that the differential of a tadpole edge is zero should be understood as being due to the action of the MC element $ET$.
  More precisely, the above definition is reasonable if we assume below that the distinguished element $E\in A$ of Section~\ref{sec:model-fiberwiseLD} is equal to the image (say $E'$) of $\Delta_A\in A\otimes A$ under the multiplication map $A\otimes A\to A$.
  More generally, we would define the differential of the tadpole edge to be $E-E'$.
  In either case, we can show the following result.
\end{remark}

\begin{proposition}
  The bisymmetric collection $\SGra_{A, A_\p}$ is a Hopf right $(\Gra_n^A,\SGra_{n}^{A_\p})$-comodule.
\end{proposition}

\begin{proof}
  The proof is similar to the proof of \cite[Proposition~3.8]{Idrissi2018b}.
  General considerations imply that if we forget the differential $d$, then the collection $\{ (A_{\partial})^{\otimes U} \otimes A^{\otimes V} \otimes \SGra_{n}(U,V)\}_{U,V}$ is a Hopf right $(\Gra_{n}^{A}, \SGra_{n}^{A_{\partial}})$-comodule.
  Thus it is sufficient to check that the comodule structure maps are compatible with the differential on all the generators.
  This is clear on the generators coming from $A$ and $A_{\partial}$, and we have defined the differential on the edges precisely so that it is compatible with the structure maps.
\end{proof}

We note that the construction of $\Gra_{A,\p A}$ is evidently functorial in the data $(A,A_\p,\rho,\Delta_A)$.

This comodule has a graphical description similar to $\SGra_{n}$.
The difference is that aerial points (corresponding to the interior of $M$) are labeled by $A$, while terrestrial points (corresponding to the boundary) are labeled by $A_{\partial}$.

Given a graph $\Gamma$, the differential $d \Gamma$ is a sum over the set of edges of $\Gamma$.
For each summand, one removes the edge from the graph and multiplies the endpoints of the edge by either $\Delta_{A}$ (if both endpoints are aerial) or $\Delta_{A,A_{\partial}}$ (if one endpoint is aerial and the other terrestrial).
We call this ``splitting'' the edge, see Figure~\ref{fig.d-split}.
Recall that we write $\Delta_{A} = \sum_{(\Delta_{A})} \Delta_{A}' \otimes \Delta_{A}'' \in A^{\otimes 2}$.

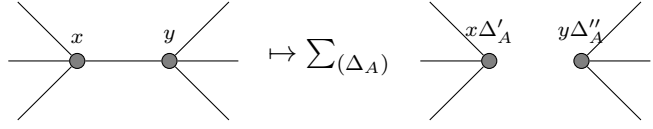
\begin{figure}[htbp]
  \centering
  \begin{tikzpicture}[baseline]
    \node[unkv, label = {\scriptsize $x$}] (1) {};
    \node[unkv, label = {\scriptsize $y$}] (2) [right = 1cm of 1] {};
    \node (x1) [left = 0.8cm of 1] {};
    \node (y1) [above left = 1cm of 1] {};
    \node (z1) [below left = 1cm of 1] {};
    \node (x2) [right = 0.8cm of 2] {};
    \node (y2) [above right = 1cm of 2] {};
    \node (z2) [below right = 1cm of 2] {};
    \draw (1) -- (2);
    \draw (x1) -- (1); \draw (y1) -- (1); \draw(z1) -- (1);
    \draw (x2) -- (2); \draw (y2) -- (2); \draw(z2) -- (2);
  \end{tikzpicture}
  $\mapsto  \sum_{(\Delta_A)}$
  \begin{tikzpicture}[baseline]
    \node[unkv, label = {\scriptsize $x \Delta_{A}'$}] (1) {};
    \node[unkv, label = {\scriptsize $y \Delta_{A}''$}] (2) [right = 1cm of 1] {};
    \node (x1) [left = 0.8cm of 1] {};
    \node (y1) [above left = 1cm of 1] {};
    \node (z1) [below left = 1cm of 1] {};
    \node (x2) [right = 0.8cm of 2] {};
    \node (y2) [above right = 1cm of 2] {};
    \node (z2) [below right = 1cm of 2] {};
    \draw (x1) -- (1); \draw (y1) -- (1); \draw(z1) -- (1);
    \draw (x2) -- (2); \draw (y2) -- (2); \draw(z2) -- (2);
  \end{tikzpicture}
  \caption{The splitting differential of $\SGra_{A,A_\p}$ (a gray vertex can be of any kind).}
  \label{fig.d-split}
\end{figure}

The product glues graphs along their vertices (multiplying the labels), and the comodule structure collapses subgraphs, multiplying the
labels and applying $\rho$ to them if necessary.

\begin{definition}\label{def.m-realized}
  Let $M$ be a compact oriented semi-algebraic manifold with boundary $\p M$.
  We define $M$-realized diagonal data\index{realized diagonal data} to be diagonal data $(A \xrightarrow{\rho} A_{\partial}, \Delta_{A}, \sigma_{A})$ mapping into $M$ through $(g,g_{\partial})$ (see the end of Section~\ref{sec.a-data}), together with a propagator $\varphi \in \Omega_{\triv}^{*}(\SFM_{M}(0,2))$ as built in Proposition~\ref{prop.propagator-sfmm}.
  In particular, $d\varphi = (g \otimes g)(\Delta_{A})$.
\end{definition}

Note that in the definition above, we use the inclusion $j : \SFM_{M}(1,1) \to \SFM_{M}(0,2)$ defined using a collar $\partial M \times [0,1) \subset M$ around the boundary.

\begin{proposition}
  \label{prop.omega-prime-colored}
  Given $M$-realized diagonal data $(A \xrightarrow{\rho} A_\p,\Delta_A,\sigma_{A},g,g_\p,\varphi)$, there is a morphism of bisymmetric collections $\omega' : \SGra_{A, A_\p} \to \OmPA^{*}(\SFM_{M})$ characterized by
  \begin{align*}
    \omega'(\iota_{v}(x))
     & = p_{v}^{*}(g(x))
     & v \in V, \; x \in A          \\
    \omega'(\iota_{u}(x))
     & = p_{u}^{*}(g_\p(x))
     & u \in U, \; x \in A_\p       \\
    \omega'(e_{vv'})
     & = p_{vv'}^{*}(\varphi)
     & v, v' \in V                  \\
    \omega'(e_{vu})
     & = p_{vu}^{*}(j^{*}(\varphi))
     & u \in U, \; v \in V
  \end{align*}
  The morphism $\omega'$ is naturally compatible with the action of the fiberwise Swiss-Cheese operad, and we can extend $\omega'$ by the maps of Section~\ref{sec:locMC2} to a Hopf (relative) comodule map
  \[
    \omega' : (\SGra_{A, A_\p}, \Gra_n^A,  \SGra_{n}^{A_\p}) \to  (\OmPA^{*}(\SFM_{M}), \OmPA^{*}(\FM_{n}^M), \OmPA^{*}(\SFM_{n}^M))
  \]
\end{proposition}
\begin{proof}
  Due to the description of $\SGra_{A,A_{\partial}}$, it is clear that we can get an algebra map with the prescribed behavior on the generators.
  The properties of $\varphi$ (see Proposition~\ref{prop.propagator-sfmm}, most notably the fact that $d\varphi = \Delta_{A}$) and the fact that $g$, $g_{\partial}$ are CDGA maps also show that this map is compatible with the differentials on both sides.

  It remains to check that this is a morphism of comodules.
  This is clear on the generators coming from $A$ and $A_{\partial}$.
  For edges, we use again the properties of the propagator.
  Note that for odd $n$, when we apply the comodule structure map to the propagator $\varphi \in \OmPA^{n-1}(\SFM_{M}(0, 2))$, we get a sum of the local propagator $\varphi \in \OmPA^{*}(\FM_{n}(2))$ plus the Euler class of $\partial M$, which is consistent with the fact that applying the comodule structure map $\circ_{\varnothing,\underline{2}}^{\vee}$ to the edge $e_{12} \in \SGra_{A,A_{\partial}}(\varnothing, \underline{2})$ produces a sum $1 \otimes e_{12} + e_{**} \otimes 1 \in \SGra_{A,A_{\partial}}(\{*\}, \varnothing) \otimes \Gra_{n}(\underline{2})$ and that tadpoles are sent to the Euler class of $\partial M$.
  For even $n$ there is no such difficulty (and tadpoles are mapped to zero).
\end{proof}

Next we want to extend the operadic twist from $(\Gra_n^A,  \SGra_{n}^{A_\p})$ to $(\Graphs_n^A, \SGraphs_n^{A_\p})$ discussed above to also include $\SGra_{A, A_\p}$.
Again we may proceed by passing to a collection of CDGAs $\Tw\SGra_{A, A_\p}$ obtained by formally adding a zero-ary cogenerator.
From Sections~\ref{sec:localMC1} and~\ref{sec:locMC2}, we have that that $M$-realized diagonal data produces Maurer--Cartan elements:
\begin{equation}
  z_{\Omega} + z_{\Omega}^{\partial} \in \Omega^{*}_{\triv}(M) \hotimes \GC_{n}^{\vee} \ltimes \Omega^{*}_{\triv}(\partial M) \hotimes \SGC_{n}^{\vee},
\end{equation}
We assume that we are given MC elements\index{zzd@$z+z^{\partial}$}
\begin{equation}
  z+z^{\partial} \in A \hotimes \GC_{n}^{\vee} \ltimes A_{\partial} \hotimes \SGC_{n}^{\vee}
\end{equation}
mapping to $z_{\Omega}+z^{\partial}_{\Omega}$ under the given map $A \to \Omega^{*}_{\triv}(M)$, $A_{\partial} \to \Omega^{*}_{\triv}(\partial M)$.
The following definition will depend on this choice, although we will suppress it from the notation.

\begin{definition}
  The twisted colored labeled graph comodule $\Tw \SGra_{A, A_\p}$\index{TwSGraAAd@$\Tw \SGra_{A,A_{\partial}}$} is the Hopf right $\Tw(A\otimes \Gra_n, A_\p\otimes \SGra_{n})$-comodule obtained by twisting $\SGra_{A, A_\p}$ with respect to the Maurer--Cartan element $(z, z_\p)$.
\end{definition}

Let us give a graphical description of the module $\Tw \SGra_{A, A_\partial}(U,V)$.
It is spanned by graphs with four types of vertices: external aerial, internal aerial, external terrestrial, internal terrestrial.
Recall that aerial vertices are labeled by $A$ while terrestrial ones are labeled by $A_{\partial}$.
External aerial vertices are in bijection with $V$, while external terrestrial vertices are in bijection with $U$.
Internal aerial vertices are of degree $-n$, while internal terrestrial vertices are of degree $-(n-1)$.
Both kinds of internal vertices are indistinguishable among themselves.
Finally, edges are of degree $n-1$, and the source of an edge may only be aerial.
See Figure~\ref{fig.exa-col-lab-graph} for an example.

\begin{figure}[htbp]
  \centering
  \begin{tikzpicture}[baseline = 0.5cm]
    \node (0) {};
    \node (o) [right = 3cm of 0] {};
    \draw[dotted] (0) to (o);
    \node[ext, label = {above left}:{\scriptsize $x_{1}$}, fill=white] (u) [right = 0.5cm of 0] {$u$};
    \node[int, label = {\scriptsize $x_{2}$}] (i1) [right = 1cm of u] {};
    \node[ext, label = {\scriptsize $y_{1}$}] (v1) [above = 1cm of u] {$v_{1}$};
    \node[ext, label = {\scriptsize $y_{2}$}] (v2) [right = 1cm of v1] {$v_{2}$};
    \node[int, label = {\scriptsize $y_{3}$}] (i2) [right = 1cm of v2] {};
    \draw[-latex] (v1) edge (u) edge (i1) edge (v2);
    \draw[-latex] (v1) to (v2);
    \draw[-latex] (i2) edge (v2) edge (i1);
  \end{tikzpicture}
  \quad ($x_{1}, x_{2} \in A_{\partial}$ and $y_{1}, y_{2}, y_{3} \in A$)
  \caption{A colored labeled graph in $\Tw \SGra_{A, A_\p}(\{u\}, \{v_1, v_2\})$.}
  \label{fig.exa-col-lab-graph}
\end{figure}
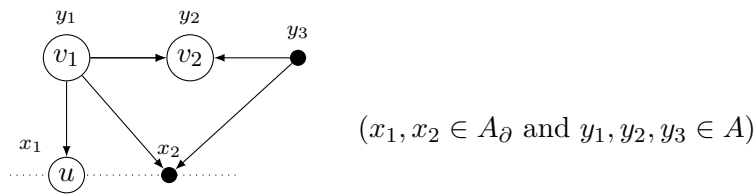

The product glues graphs along external vertices (multiplying the labels).
The comodule structure maps collapse subgraphs, and the label of the collapsed subgraph is the product of all the labels inside that subgraph (applying $\rho : A \to A_{\partial}$ as needed).
Finally, the differential has several parts (in addition to the internal differentials of $A$ and $A_{\partial}$):
\begin{itemize}
  \item A first part comes from $\SGra_{A, A_\p}$: it splits edges between vertices of any type, and then multiplies the endpoints of the removed edge by either $\Delta_{A}$ or $\Delta_{A,A_{\partial}}$.
  \item A second part contracts a subgraph $\Gamma'$ with only aerial vertices, at most one of them being external.
        The result of that contraction is an aerial vertex (external if $\Gamma'$ contained one, internal otherwise).
        The label of that vertex by the products of all the labels in $\Gamma'$, multiplied by $z(\bar{\Gamma}')$, where $\bar{\Gamma}' \in \GC_{n}$ is $\Gamma'$ with all the labels removed.
        See Figure~\ref{fig.diff-tw-sgra} for an example.
  \item A third part contracts a subgraph $\Gamma''$ with at most one external vertex.
        The label of the contracted subgraph is the product of all the labels inside it, multiplied by $z^{\partial}(\bar{\Gamma}'')$, where $\bar{\Gamma}'' \in \SGC_{n}$ is the subgraph with the labels removed.
        If the result contains a ``bad edge'' (whose source is terrestrial), then the summand vanishes.
        Again, see Figure~\ref{fig.diff-tw-sgra}.
\end{itemize}

\begin{figure}[htbp]
  \centering
  \tikzset{every label/.append style={font=\scriptsize}, every picture/.append style={baseline=1cm, scale=0.7}}
  \begin{tikzpicture}
    \draw[dotted] (0,0) -- (4,0);
    \node[ext, label={$a_{1}$}] (a1) at (1,2) {1};
    \node[ext, label={$a_{2}$}] (a2) at (3,2) {2};
    \node[int, label={above right}:{$a_{3}$}] (a3) at (2,1) {};
    \node[int, label={$a_{4}$}] (a4) at (2,2) {};
    \node[ext, label=below:{$b_{1}$}] (b1) at (1,0) {1};
    \node[int, label=below:{$b_{2}$}] (b2) at (3,0) {};
    \draw[-latex]
    (a1) edge (a3) edge (a4)
    (a3)  edge (b1) edge (b2)
    (a4) edge (a2) edge (a3);
    \node[draw, dashed, fit=(a1) (a3) (a4), label=left:$\Gamma'$] (circ) {};
  \end{tikzpicture}
  $\mapsto$
  \begin{tikzpicture}
    \draw[dotted] (0,0) -- (4,0);
    \node[ext, label={above left}:{$a_{1} a_{3} a_{4} z(\bar{\Gamma}')$}] (a1) at (1.5,1.5) {1};
    \node[ext, label={$a_{2}$}] (a2) at (3,2) {2};
    \node[ext, label=below:{$b_{1}$}] (b1) at (1,0) {1};
    \node[int, label=below:{$b_{2}$}] (b2) at (3,0) {};
    \draw[-latex]
    (a1) edge (b1) edge (b2) edge (a2);
  \end{tikzpicture}
  \\
  \;
  \begin{tikzpicture}
    \draw[dotted] (0,0) -- (4,0);
    \node[ext, label={$a_{1}$}] (a1) at (1,2) {1};
    \node[ext, label={$a_{2}$}] (a2) at (3,2) {2};
    \node[int, label={above right}:{$a_{3}$}] (a3) at (2,1) {};
    \node[int, label={$a_{4}$}] (a4) at (2,2) {};
    \node[ext, label=below:{$b_{1}$}] (b1) at (1,0) {1};
    \node[int, label=below:{$b_{2}$}] (b2) at (3,0) {};
    \draw[-latex]
    (a1) edge (a3) edge (a4)
    (a3) edge (b1) edge (b2)
    (a4) edge (a2) edge (a3);
    \node[draw, dashed, fit=(a3) (b1) (b2), label=right:$\Gamma''$] (circ) {};
  \end{tikzpicture}
  $\mapsto$
  \begin{tikzpicture}
    \draw[dotted] (0,0) -- (4,0);
    \node[ext, label={$a_{1}$}] (a1) at (1,2) {1};
    \node[ext, label={$a_{2}$}] (a2) at (3,2) {2};
    \node[int, label={$a_{4}$}] (a4) at (2,2) {};
    \node[ext, label=below:{$\rho(a_{3}) b_{1} b_{2} z^{\partial}(\bar{\Gamma}'')$}] (b1) at (1.5,0) {1};
    \draw[-latex]
    (a1) edge (b1) edge (a4)
    (a4) edge (b1) edge (a2);
  \end{tikzpicture}
  \caption{Second and third part of the differential in $\Tw \SGra_{A,A_\p}$}
  \label{fig.diff-tw-sgra}
\end{figure}
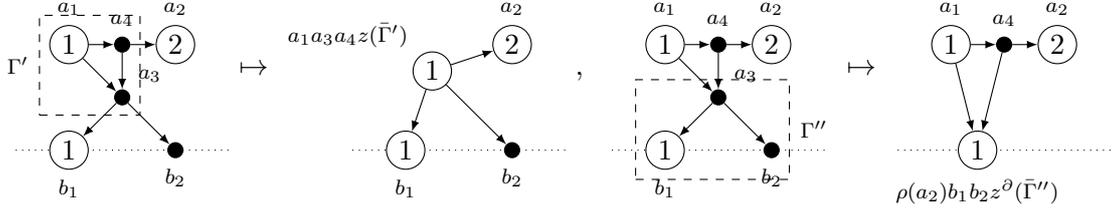

\begin{remark}
  The second part of the differential has often a simpler description.
  Assume that $z \in \Omega^{*}(M) \otimes \GC_{n}^{\vee}$ is simply the Maurer--Cartan element given by $1 \otimes \mu$, where $\mu \in \GC_{n}^{\vee}$ is the graph with two vertices and one edge as usual (see Equation~\eqref{eq.mu}).
  Then that second part of the differential simply contracts an edge between an (aerial) internal vertex and an aerial vertex of any kind, multiplying the labels.
  Note that this includes dead ends (i.e., edges connected to a univalent internal vertex), unlike the differential in $\SGraphs_{n}$.
\end{remark}

\begin{remark}
  Note that there are several differences compared to the description of the differential from Section~\ref{sec.extens-swiss-cheese}: dead ends are contractible, and the last part of the differential which forgets some internal vertices is not present.
  This comes from the fact that in the definition of the twisting of a right comodule over a cooperad, the Maurer--Cartan element $z + z^{\partial}$ can only act from the right on the comodule.
\end{remark}

\begin{remark}
  \label{rmk.stokes}
  One should not forget that when a subgraph $\Gamma' \subset \Gamma$ with only internal vertices is contracted (the third part), the result may be a terrestrial vertex even though the subgraph contains only aerial vertices, see Figure~\ref{fig.collapse-aerial-terrestrial}.
\end{remark}

\begin{figure}[htbp]
  \centering
  \begin{tikzpicture}[baseline = 0.5cm]
    \node (0) {};
    \node (o) [right = 1.5cm of 0] {};
    \draw[dotted] (0) to (o);
    \node (u) [right = 0.5cm of 0] {};
    \node[int, label = {\scriptsize $x$}] (i1) [above = 0.5cm of u] {};
  \end{tikzpicture}
  $\mapsto$
  \begin{tikzpicture}[baseline = 0.5cm]
    \node (0) {};
    \node (o) [right = 1.5cm of 0] {};
    \draw[dotted] (0) to (o);
    \node[int, label = {\scriptsize $\rho(x)$}] (u) [right = 0.5cm of 0] {};
  \end{tikzpicture}
  \caption{Collapsing an aerial vertex into a terrestrial vertex}
  \label{fig.collapse-aerial-terrestrial}
\end{figure}
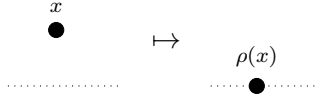

\begin{proposition}
  The morphism $\omega' : \SGra_{A, A_\partial} \to \OmPA^{*}(\SFM_{M})$ of Proposition~\ref{prop.omega-prime-colored} extends to a morphism of bisymmetric collections $\omega : \Tw \SGra_{A, A_\partial} \to \OmPA^{*}(\SFM_{M})$, given on a graph $\Gamma \in \SGra_{A, A_\partial}(U \sqcup I, V \sqcup J) \subset \Tw \SGra_{A, A_\p}(U,V)$ by:
  \[ \omega(\Gamma) \coloneqq (p_{U,V})_{*}(\omega'(\Gamma)) = \int_{\mathrlap{\SFM_{M}(U \sqcup I, V \sqcup J) \to \SFM_{M}(U,V)}} \quad \omega'(\Gamma). \]
  This morphism is compatible with the fiberwise Swiss Cheese action, i.e., we have a map of relative right Hopf multimodules
  \[ \omega : ( \Tw\SGra_{A, A_\partial}, \Tw\Gra_n^A, \Tw\SGra_{n}^{A_\p}) \to (\OmPA^{*}(\SFM_{M}),\OmPA(\FM_n^M), \OmPA^{*}(\SFM_{n}^M)). \]
\end{proposition}
\begin{proof}
  We have chosen the propagator $\varphi$ so that it is a trivial form, and we have assumed that the morphisms $A \to \OmPA^{*}(M)$ and $A_{\partial} \to \OmPA^{*}(\partial M)$ factor through the sub-CDGAs of trivial forms.
  Hence for any graph $\Gamma$, $\omega'(\Gamma)$ is a trivial form and can be integrated along the fiber of $p_{U,V}$.

  We can now reuse the proof of~\cite[Proposition~3.14]{Idrissi2018b}.
  The difference is the description of the decomposition of the fiberwise boundary of $p_{U,V}$ used to show that $\omega$ is a chain map through the application of Stokes' formula.
  This description is similar to the one implicitly used by~\cite{Willwacher2015} (see also~\cite[Section~5.2.1]{Kontsevich2003}) with some variations accounting from the fact that no normalization is done to compactify $M$ (see the discussion before the proof of~\cite[Lemma~3.18]{Idrissi2018b}).
  More concretely, the boundary of $\SFM_{M}(U,V)$ is given by:
  \[ \partial \SFM_{M}(U,V) = \bigcup_{\hspace{-2em} \mathrlap{T \in \BF'(V)}} \im (\circ_{T}) \cup \bigcup_{\hspace{-2em} \mathrlap{(W,T) \in \BF''(U;V)}} \im (\circ_{W,T}), \]
  where:
  \begin{equation}
    \begin{aligned}
      \BF'(V)    & \coloneqq \{ T \subset V \mid \# T \geq 2 \},                                        \\
      \BF''(U;V) & \coloneqq \{ (W,T) \mid W \subset U, \; T \subset V, \; 2 \cdot \#T + \#W \geq 2 \}.
    \end{aligned}
  \end{equation}
  Note that in the description of $\partial \SFM_{n}(U,V)$, there is an additional condition $W \cup T \subsetneq U \cup V$.
  Indeed in $\SFM_{n}$ the normalization by the affine group prevents the points from becoming infinitesimally close all at once; in $\SFM_{M}$, no such normalization occurs.

  Then the fiberwise boundary of the canonical projection $p_{U,V}$ is given by:
  \begin{equation}
    \SFM^{\partial}_{M}(U,V) =  \bigcup_{\hspace{-2em} \mathrlap{T \in \BF'(V,J)}} \im (\circ_{T}) \cup \bigcup_{\hspace{-2em} \mathrlap{(W,T) \in \BF''(U, I; V, J)}} \im (\circ_{W,T}) \subset \SFM_{M}(U \sqcup I, V \sqcup J),
  \end{equation}
  where the subsets $\BF'(V,J) \subset \BF'(V \sqcup J)$ and $\BF''(U,I;V,J) \subset \BF''(U \sqcup I; V \sqcup J)$ are defined by:
  \begin{equation}
    \begin{aligned}
      T \in \BF'(V,J)           & \iff \#(T \cap J) \leq 1,                         \\
      (W,T) \in \BF''(U,I; V,J) & \iff V \cap T = \varnothing, \#(U \cap W) \leq 1.
    \end{aligned}
  \end{equation}

  One can then check that the boundary faces of that decomposition correspond to the summands of the differential.
\end{proof}

\begin{definition}
  Define the full colored graph complex to be:\index{fSGCAAd@$\fSGC_{A,A_{\partial}}$}
  \[ \fSGC_{A,A_\p} \coloneqq \Tw \SGra_{A, A_\p}(\varnothing, \varnothing). \]
\end{definition}

This is the CDGA of colored, labeled graphs with only internal vertices.
The product is the disjoint union of graphs, thus $\fSGC_{A,A_\p}$ is free as an algebra, generated by the graded module $\SGC_{A,A_\p}$\index{SGCAAd@$\SGC_{A,A_{\partial}}$} of connected graphs.
Each $\Tw \SGra_{A,A_\p}(U,V)$ is a module over the CDGA $\fSGC_{A,A_\p}$ by adding connected components.

\begin{remark}
  \label{rmk.maurer-cartan}
  Since $\fSGC_{A,A_\p}$ is a CDGA, the dual module $\SGC_{A,A_\p}^{\vee}$ is naturally an $\hoLie_{1}$-algebra.
  The differential of $\fSGC_{A,A_{\p}}$ cannot create more than two connected components, thus $\SGC_{A,A_{\partial}}^{\vee}$ is actually a $\Lie_{1}$-algebra.
  The differential blows up vertices (like in $\GC_{n}^{\vee}$) and joins pairs of vertices by an edge, while the Lie bracket joins two graphs by an edge.
\end{remark}

We may consider the differential graded Lie algebra\index{gAAd@$\alg g_{A,\partial A}$}
\begin{equation}
  \label{eq.g-a-a-del}
  \alg g_{A,\p A} \coloneqq \SGC_{A,A_\p}\rtimes \SGC_{n}^{A_\partial} \rtimes \GC_n^A.
\end{equation}
A CDGA morphism
\begin{equation}
  \fSGC_{A, A_\p}(\varnothing, \varnothing) \to \K
\end{equation}
may be interpreted as an extension of the Maurer--Cartan element $z^\p + z$ to a Maurer--Cartan element
\begin{equation}
  Z+z^\p+z \in \alg g_{A,\p A}.
\end{equation}
The universal example of such a MC element, in some sense, is given by the following definition.

\begin{definition}\label{def.szphi}
  The colored partition function $Z_{\Omega} : \fSGC_{\Omega^{*}_{\triv}(M),\Omega^{*}_{\triv(\partial M)}} \to \R$\index{ZOmega@$Z_{\Omega}$} is the CDGA morphism given by the restriction in empty arity
  \[ Z_{\Omega} \coloneqq \omega|_{(\varnothing,\varnothing)} : \Tw \SGra_{\Omega_{\triv}(M), \Omega^{*}_{\triv}(\partial M)}(\varnothing, \varnothing) \to \OmPA^{*}(\SFM_{M}(\varnothing, \varnothing)) = \OmPA^{*}(\{*\}) = \R. \]
\end{definition}

We will next assume that we are given a Maurer--Cartan element (the partition function) $Z+z^\p+z \in \alg g_{A,\p A}$\index{Z@$Z$} that maps to $Z_{\Omega}+z^{\partial}_{\Omega}+z_\Omega$ under the given maps $A\to \Omega_{\triv}(M)$, $A_\p\to \Omega_{\triv}(\p M)$.

\begin{definition}
  The reduced colored labeled graph comodule $\SGraphs_{A, A_\p}$\index{SGraphsAAd@$\SGraphs_{A,A_{\partial}}$} is the bisymmetric collection given in each arity by:
  \[ \SGraphs_{A, A_\p}(U, V) \coloneqq \R \otimes_{\fSGC_{A,A_\p}} \Tw \SGra_{A, A_\p}(U,V), \]
  where we understand $\R$ as an $\fSGC_{A,A_\p}$-module via the given partition function $Z$.
\end{definition}

\begin{proposition}\label{prop:mainmapdef}
  The bisymmetric collection $\SGraphs_{A, A_\p}$ forms a Hopf right multimodule over $(\SGraphs_{n}^{A_\p},\Graphs_n^{A})$, and the map $\omega : \Tw \SGra_{A, A_\p} \to \OmPA^{*}(\SFM_{M})$ factors through a Hopf right comodule morphism:
  \[ \omega : (\SGraphs_{A,A_\p}, \SGraphs_{n}^{A_\p},\Graphs_n^A) \to (\OmPA^{*}(\SFM_{M}), \OmPA^{*}(\SFM_{n}^M),\OmPA^*(\FM_n^M)). \]
\end{proposition}
\begin{proof}
  This is identical to the proof of~\cite[Proposition~3.23]{Idrissi2018b} and follows from the general properties of integration along fibers, most prominently the double-pushforward formula of \cite[Proposition 8.13]{HardtLambrechtsTurchinVolic2011}.
\end{proof}

\subsection{Cohomology and quasi-isomorphism theorem}
\label{sec:sfm-main-thm}

One of the main results of this paper is the following, more precise version of Theorem~\ref{thm.intro.sgraphs-qiso}:
\begin{theorem}
  \label{thm.main-qiso}
  Suppose we are given realized diagonal data $(A \xrightarrow{\rho} A_\p, \Delta_{A}, \sigma_{A}, g, g_{\partial}, \varphi)$ fitting into a commutative diagram~\eqref{equ:AdataCD}, as well as a Maurer--Cartan element $Z+z^\p+z\in \alg g_{A,A_\p}$ mapping to the canonical Maurer--Cartan element from Definition~\ref{def.szphi}.
  Suppose that the horizontal maps in \eqref{equ:AdataCD} are quasi-isomorphisms.
  Then the map
  \[ \omega : (\SGraphs_{A,A_\p}, \SGraphs_{n}^{A_\p},\Graphs_n^A) \to (\OmPA^{*}(\SFM_{M}), \OmPA^{*}(\SFM_{n}^M),\OmPA^*(\FM_n^M)). \]
  of Proposition \ref{prop:mainmapdef} is a quasi-isomorphism.
  The same result holds in the ``combinatorial'' case $A = S(\tilde{H}^{*}(M) \oplus H^{*}(M,\partial M))$, $A_{\partial} = S(\tilde{H}^{*}(\partial M))$ from Example~\ref{exa.diag-data}.
\end{theorem}

The construction of $(\SGraphs_{A,A_\p}, \SGraphs_{n}^{A_\p},\Graphs_n^A)$ is functorial in $A\to A_\p$, see Section~\ref{sec.functoriality-models}.
Hence it sufficient to consider the universal case with $A=\Omega_{\triv}(M)$, $A_\p=\Omega_{\triv}(\p M)$ and the Maurer--Cartan element equal to the universal one.

We will need the following technical result:

\begin{proposition}\label{prop:fiberqiso}
  Suppose we have a commutative diagram of augmented CDGAs
  \[
    \begin{tikzcd}
      A \ar{r}{\sim} \ar{d} & A' \ar{d} \\
      B \ar{r}{f} & B'.
    \end{tikzcd}
  \]
  Suppose that in addition the induced map between the derived tensor products
  \[
    B \lotimes_A \K \to B' \lotimes_A \K
  \]
  is a quasi-isomorphism.
  Suppose further that $A$ is cohomologically connected (so that there is a unique augmentation up to homotopy) and of finite cohomological type.
  Then $f$ is a quasi-isomorphism.
\end{proposition}

\begin{proof}
  The proof is essentially done by using the Serre spectral sequence.
  We may pick a resolution $C\to A$ which is connected.
  It suffices to show that the middle arrow in the following zigzag is a quasi-isomorphism:
  \begin{equation}
    B \xleftarrow{\sim} B \lotimes_{ C}  C \to B' \lotimes_{ C} C \xrightarrow{\sim} B'.
  \end{equation}
  To this end we realize the derived tensor products by using normalized bar resolutions
  \begin{equation}
    M \lotimes_C M' \coloneqq \operatorname{B}_{*}(M, C, M') = \bigl( \bigoplus_{n\geq 0} M \otimes (\bar C [1])^{\otimes n} \otimes M', d \bigr),
  \end{equation}
  where $\bar{C} = \ker(C \to \K)$.

  We filter $B \lotimes_{C}  C$ and $B' \lotimes_{ C} C$ by the cohomological degree of the last factor $C$.
  The associated graded is identified with $(B \lotimes_{C} \K) \otimes_{\K} C$ and $(B' \lotimes_{C} \K) \otimes_{\K} C$ respectively.
  By our assumption, the map is an isomorphism on the $E^1$-page, and moreover our filtration is exhaustive and bounded below, whence the claim follows.
\end{proof}

\begin{proof}[Proof of Theorem~\ref{thm.main-qiso}]
  We only need to show that the map $\SGraphs_{A,A_{\partial}} \to \OmPA(\SFM_M)$ of the theorem is a quasi-isomorphism.
  We use a proof technique by induction similar to~\cite[Chapter~8]{LambrechtsVolic2014}, with some extra work to ``push'' the decorations away from a given external vertex (similar to~\cite[Lemmas~4.27--28]{Idrissi2018b}), adapted to the Swiss-Cheese case.

  More concretely, we proceed by an induction on the arities $r$ (in the boundary) and $s$ (in the interior).
  Say we begin with $s$.
  For the induction step, we apply the Proposition \ref{prop:fiberqiso} to the diagram
  \begin{equation}
    \begin{tikzcd}
      \SGraphs_{A,A_{\partial}}(r,s) \ar{r}{\sim} \ar{d} & \OmPA(\SFM_{M}(r,s)) \ar{d} \\
      \SGraphs_{A,A_{\partial}}(r,s+1) \ar{r} & \OmPA(\SFM_{M}(r,s+1)).
    \end{tikzcd}
  \end{equation}

  Note that
  \begin{equation}
    \OmPA(\SFM_{M}(r,s+1)) \lotimes_{\OmPA(\SFM_{M}(r,s))} \K
  \end{equation}
  is a model for the homotopy fiber $\SFM_{M}(r,s+1) \to \SFM_{M}(r,s)$, and since this map is a fibration, we obtain a model for the fiber.
  This fiber is homotopy equivalent to $M \setminus \underline{s}$, i.e., the manifold $M$ with $s$ points removed from the interior.

  Using the Mayer--Vietoris exact sequence inductively on the number of removed points one computes
  \beq{equ:fibercohom0}
  H^*(M \setminus \underline{s}) = H^*(M) \oplus (\R[1-n])^{s}.
  \eeq
  On the other hand let us compute
  \begin{equation}\label{eq:V}
    V \coloneqq \SGraphs_{A,A_{\partial}}(r,s+1) \lotimes_{\SGraphs_{A,A_{\partial}}(r,s)} \K.
  \end{equation}
  The module $\SGraphs_{A,A_{\partial}}(r,s+1)$ is free over the algebra $\SGraphs_{A,A_{\partial}}(r,s)$, therefore the derived tensor product agrees with the ordinary one.
  As a graded vector space, this tensor product can be represented by graphs in which the (external aerial) vertex labeled by $1$ is $\geq 1$ valent, together with a copy of $\R$ in degree zero accounting for the remaining graphs using the evaluation map $\SGraphs_{A,A_{\partial}}(r,s)\to \R$.
  As a chain complex, whenever the differential produces a ``forbidden'' (0-valent on vertex $1$) graph, one replaces it by its evaluation.

  To compute the cohomology of $V$, we split the complex into pieces
  \begin{equation}\label{equ:Vfiltr}
    V =
    \begin{tikzcd}
      V_{0} \ar[phantom]{r}{\oplus}
      & V_{H} \ar[phantom]{r}{\oplus} \ar[loop above]
      & V_{in} \ar[phantom]{r}{\oplus} \ar[loop above]
      \ar[bend left]{ll} \ar[bend right]{l} \ar[bend right]{r}
      & V_{\geq 1} \ar[loop above]
    \end{tikzcd}
    .
  \end{equation}
  The splitting depends on whether we are in the model-theoretic case or the combinatorial case $A = S(\tilde{H}^{*}(M) \oplus H^{*}(M,\partial M))$.
  We have:
  \begin{itemize}
    \item $V_0$ is the piece where vertex 1 has valence zero and is decorated by the unit $1 \in A^{0}$;
    \item $V_H$ is the piece where vertex 1 has valence zero and is either decorated by an element of $A^{> 0}$ (in the model case), or by a single element of $\tilde{H}^{*}(M)$ and not $H^{*}(M,\partial M)$ (in the combinatorial case);
    \item $V_{in}$ is the piece where vertex 1 has valence one, the edge is incoming, and it is decorated by the unit in $A^{0}$;
    \item $V_{\geq 1}$ consists of all other graphs.
  \end{itemize}
  Clearly $V_0 = A^{0} = \R$ and $V_H$ is either $A^{>0}$ or $\tilde{H}^{*}(M)$.
  There are various pieces of the differential between the summands as shown by arrows above.

  Let us filter $V_{0}$, $V_{H}$, and $V_{in}$ by the number of edges, and $V_{\ge 1}$ by the number of edges shifted by $1$.
  Then in the associated spectral sequence, the first differential is the one pointing furthest right in \eqref{equ:Vfiltr}, i.e., the map from $V_{in}\to V_{\geq 1}$ given by the contraction of the edge at vertex $1$.
  All the other parts of the differential strictly decrease the filtration and thus vanish.

  This map is surjective.
  Its kernel is composed of graphs where either
  \begin{enumerate*}[label={(\roman*)}]
    \item the edge is connected to another external vertex, necessarily aerial;
    \item the edge is connected to an internal vertex, decorated by an element of $H^{*}(M)$ in the combinatorial case;
    \item the edge is connected to an internal vertex of valence 2, with one incoming and 1 outgoing edge.
  \end{enumerate*}

  The next page's differential comes from the piece of the differential in \eqref{equ:Vfiltr} mapping the summands to themselves, and is only non-trivial on the $V_{in}$ factor.
  One can check that this differential kills all remaining terms in the $V_{in}$ except for the graphs where vertex $1$ is connected to one external vertex.
  The spectral sequence converges here to \eqref{equ:fibercohom0}: elements of $H(M)$ are represented by $V_{H}$ (i.e., vertex $1$ is decorated by $1$ or $\tilde{H}(M)$) and the classes of $\R[1-n]^{s}$ are represented by elements of the form $e_{ij}$ of $V_{in}$ where $1$ is connected to another external vertex $j$.
  The classes of $H(M)$ are by definition sent to representatives in $\OmPA(M)$ (pulled back to $\OmPA(\SFM_{M}(r,s+1))$ and then restricted to the fiber).
  The element $e_{ij}$ is sent to the propagator, whose restriction to the fiber we consider here is assumed (Proposition~\ref{prop.propagator-sfmm}) to be a volume form of the sphere obtained by letting point $j$ rotate around point $i$, which is precisely the class found by the Mayer--Vietoris argument above.
  So we have finished the induction step on $r$.

  We are thus reduced to the case $s = 0$, i.e., all the (external) vertices are on the boundary.
  Next we proceed with the induction on $r$.
  We proceed in the same manner as above, starting with the diagram
  \begin{equation}\label{eq:2}
    \begin{tikzcd}
      \SGraphs_{A,A_{\partial}}(r,0) \ar{r}{\sim} \ar{d} & \OmPA(\SFM_{M}(r,0)) \ar{d} \\
      \SGraphs_{A,A_{\partial}}(r+1,0) \ar{r} & \OmPA(\SFM_{M}(r+1,0)).
    \end{tikzcd}
  \end{equation}
  The fiber of $\FM_{M}(r+1,0)\to \FM_{M}(r,0)$ is homotopy equivalent to $\p M$ minus $r$ points, with cohomology
  \beq{equ:fibercohom0bdry}
  H^{<n-1}(\p M) \oplus (\R[1-n])^{r-1}.
  \eeq

  We have to compare this to the cohomology of
  \begin{equation}
    \SGraphs_{A,A_{\partial}}(r+1,0) \lotimes_{\SGraphs_{A,A_{\partial}}(r,0)} \K
    =
    \SGraphs_{A,A_{\partial}}(r+1,0) \otimes_{\SGraphs_{A,A_{\partial}}(r,0)} \K
    .
  \end{equation}
  The latter space can again be identified with graphs ``modulo'' internally connected components not connected to external vertex $1$ (note that it is now terrestrial).

  We want to compute the cohomology of that space by splitting it as above.
  First let us define ``wedges'', which are bivalent aerial internal vertices, decorated by the unit of $A$ and connected to two terrestrial vertices; and stubs, which are univalent aerial internal vertices, decorated by an element of $\ker(A \to A_{\partial})$ and connected to a terrestrial vertex.
  See Figure~\ref{fig:wedge-stub} for examples.

  \begin{figure}[htbp]
    \centering
    \begin{subfigure}[t]{.45\linewidth}
      \centering
      \begin{tikzpicture}
        \draw[dotted] (-1,0)--(1,0);
        \node[ext] (v)  at (0,0) {$i$};
        \node[int] (w) at (0,1) {};
        \node at (0,1.3) {\small $\alpha \in \ker(A \to A_{\partial})$};
        \draw[-latex] (w) edge (v);
      \end{tikzpicture}
      \caption{Stub}
    \end{subfigure}
    ~
    \begin{subfigure}[t]{.45\linewidth}
      \centering
      \begin{tikzpicture}[scale=1.3]
        \draw[dotted] (-1,0) -- (1,0);
        \node[int] (v) at (0,.5) {};
        \node[ext] (v1) at (-.5,0) {$i$};
        \node[ext] (v2) at (.5,0) {$j$};
        \draw[-latex] (v) edge (v1) edge (v2);
      \end{tikzpicture}
      \caption{Wedge}
    \end{subfigure}
    \caption{Wedges and stubs}
    \label{fig:wedge-stub}
  \end{figure}
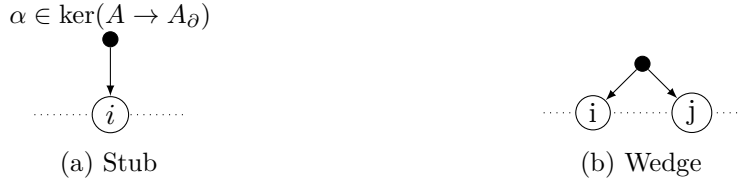

  We may now split our complex as:
  \begin{equation}
    \label{equ:Wfiltr}
    W \coloneqq
    \begin{tikzcd}
      W_0 \ar[phantom]{r}{\oplus} & W_A \ar[phantom]{r}{\oplus} & W_w \ar[bend right]{l} \ar[bend left]{r} \ar[phantom]{r}{\oplus} & \ar[bend left]{l} W_{\geq 1}
    \end{tikzcd},
  \end{equation}
  where:
  \begin{itemize}
    \item $W_0=\K$ are graphs for which vertex 1 has valence zero and is decorated by the unit $1 \in A^{0}_{\partial}$;
    \item $W_A$ are graphs where vertex 1 has valence zero and is either decorated by an element of $A_{\partial}^{> 0}$ (in the model case) or by a single element of $\tilde{H}^{*}(\partial M)$ (in the combinatorial case);
    \item $W_w$ is the piece where $1$ has valence one, is decorated by the unit $1 \in A^{0}_{\partial}$, and is either connected by a wedge to some other terrestrial vertex or to a stub;
    \item $W_{\geq 1}$ are all remaining graphs.
  \end{itemize}

  Again we take a spectral sequence such that the leading differential is the piece ``pointing the most right'' in \eqref{equ:Wfiltr}, i.e., the piece $W_w\to W_{\geq 1}$.
  This is the piece contracting our wedge (unless at the wedge there is another wedge, or an $A$-decoration, or a stub).
  The map is surjective, and its kernel consists graphs of the following types:
  
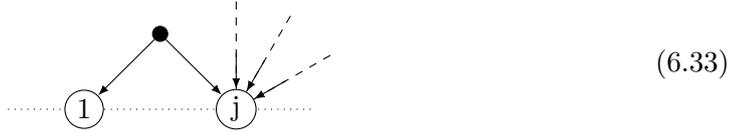
\begin{equation}
    \begin{tikzpicture}[baseline=.5cm]
      \draw[dotted] (0,0) -- (4,0);
      \node[ext] (a) at (1,0) {$1$};
      \node[int] (i) at (2,1) {};
      \node[ext] (b) at (3,0) {$j$};
      \draw[-latex] (i) edge (a) edge (b);
      \foreach \ang in {30,60,90}
      \draw (b) edge[dashed] +(\ang:1.5)  edge[latex-] +(\ang:.75);
    \end{tikzpicture}
  \end{equation}

  The next differential consists of the pieces acting on each summand in \eqref{equ:Wfiltr} separately.
  For $W_A$ the piece of the cohomology is $\tilde H(\p M)$.
  For $W_w$ only the piece remains where vertex $1$ is connected to one other external vertex via a wedge.
  The remaining differentials in the spectral sequence then let one of the wedges to an external vertex kill the piece of $H(W_A)\cong \tilde H(M)$ corresponding to the decoration by the top class.
  Together we reproduce \eqref{equ:fibercohom0bdry} and hence we are done.
\end{proof}

\subsection{Small model and canonical combinatorial model}
\label{sec.comp-two-approaches}

Note that the models discussed above depend on the following data:
\begin{enumerate*}[label={(\roman*)}]
  \item a CDGA model $A\to A_\p$ for the map $\p M\to M$;
  \item a Maurer--Cartan element $z+z^{\partial}+Z\in \mathfrak{g}_{A,A_{\partial}}$ (see Equation~\eqref{eq.g-a-a-del}).
\end{enumerate*}
We thus obtain two different kinds of similar but different models for $\Conf_{r,s}(M)$, both derived from the two examples of diagonal data in Example~\ref{exa.diag-data}.

In the first case, we assume that $(M, \partial M)$ admits a Poincaré--Lefschetz duality model as in Section~\ref{sec.pretty-nice}.
We may then take our diagonal data to the the CDGAs $(B, B_{\partial})$ together with their diagonal classes, and we obtain a gauge equivalence class of Maurer--Cartan elements in $z+z^{\partial}+Z \in \mathfrak{g}_{B,B_{\partial}}$.
The advantage of this model is that it is rather smaller than the next one, and can be used to obtain the generalization of the Lambrechts--Stanley model of configuration spaces~\cite{Idrissi2018b} in Section~\ref{sec.model-conf-spac}.
Its disadvantage is its cost in terms of connectivity and dimension assumption on $M$.

In the second case, we take the diagonal data obtained by setting $A = S(\tilde{H}^{*}(M) \oplus H^{*}(M, \partial M))$ and $A_{\partial} = S(\tilde{H}^{*}(\partial M))$.
The construction of $\SGraphs_{A,A_{\partial}}$ then depends on the choice of the Maurer--Cartan elements $z+z^{\partial}+Z \in \mathfrak{g}_{A,A_{\partial}}$ as before.
While this model is much bigger than the previous one ($A$ and $A_{\partial}$ are infinite dimensional), it has advantages for several purposes:
\begin{itemize}
  \item we only need to assume that $M$ is compact and orientable and that $\dim M\geq 3$ (for $\dim M=2$ the boundary is a union of circles, which can be handled by the constructions of \cite{CamposWillwacher2016});
  \item it agrees with spaces of Feynman diagrams of AKSZ theories studied in physics, and thus gives a topological meaning to these objects;
  \item the collection $\SGraphs_{A,A_{\p}}$ consists of quasi-free algebras, which is important for example when computing homotopy automorphisms of these objects.
\end{itemize}
We shall also see that the tree (i.e., 0-loop) piece of the Maurer--Cartan element $Z$ precisely encodes the real homotopy type of the map $\p M\to M$.

\begin{remark}\label{rmk:direct-map}
  Since the combinatorial model is cofibrant as an algebra, we can get a direct quasi-isomorphism from the combinatorial model to the small model in each arity, by general theory.
  However, the small model is not fibrant as a Hopf right comodule, so these general arguments do not necessarily produce a direct map of Hopf right comodules.
  To get a direct map between the two, we can proceed as in the following sketch of proof.
  Let $A = S(H(M) \oplus H(M,\p M))$ and $A_{\partial} = H(\p M)$.
  Picking a Poincaré--Lefschetz duality model $(B, B_{\partial})$, we can define $\SGraphs_{S(B), S(B_\partial)}$ like we had done for $\SGraphs_{A,A_{\partial}}$.
  By picking representatives of cohomology classes, we can get a quasi-isomorphism $\SGraphs_{A,A_{\partial}} \to \SGraphs_{S(B), S(B_\partial)}$.
  The latter model depends on a certain partition function, which can be reduced to its tree piece as in the next section.
  As explained in the next section, this tree piece encodes Massey products of $(B, B_{\partial})$, but since $(B, B_{\partial})$ is already a model for $(M, \partial M)$, all these Massey products vanish, so the partition function is trivial up to gauge equivalence.
  Using the general arguments of Section~\ref{sec.functoriality-models}, we can thus replace the partition function by the trivial one and obtain an isomorphic model $\SGraphs^{0}_{S(B), S(B_{\partial})}\cong \SGraphs_{S(B), S(B_\partial)}$, which can then be projected down to $\SGraphs_{B,B_{\partial}}$ by a quasi-isomorphism.
\end{remark}

\subsection{The global Maurer--Cartan element and homotopy invariance}
\label{sec.global-maurer-cartan}

In this section, we use the ``canonical combinatorial'' model approach, i.e., we use the diagonal data obtained by setting $A = S(\tilde{H}^{*}(M) \oplus H^{*}(M,\partial M))$ and $A_{\partial} = S(\tilde{H}^{*}(\partial M))$, see Example~\ref{exa.diag-data}.

Due to its definition as an integral, the Maurer--Cartan element $Z \in \SGC_{A,A_{\partial}}^{\vee}$ depends, a priori, on more than just the cohomology type of $M$ and the classes of the maps $H(B\SO(n))\to \Omega(M)$ and $H(B\SO(n-1))\to \Omega(\p M)$, but it is not necessarily clear how much information about $M$ it ``knows''.
We now show that $Z$ actually captures all the information about the real homotopy type of $M$.
We moreover show that under good conditions, the real homotopy type of $(M, \partial M)$ determines the gauge equivalence class of $Z$, and thus the real homotopy type of $\SFM_{M}$ (including the real homotopy type of configuration spaces of $M$).

We generally cannot give an explicit formula for the integrals defining this MC element, except perhaps in simple enough cases.
However, one can understand the loop order zero piece of $Z$, i.e., the ``classical piece'' in the physics slang.
As shown in Corollary~\ref{cor:ZtreedM} below, this piece just encodes the real homotopy types of the spaces $M$ and $\p M$ and the real homotopy type of the inclusion $\p M\hookrightarrow M$.
Furthermore, if $M$ and $\p M$ are simply connected and $n=\dim M\geq 5$, then we will see that the classical piece is already not deformable and fully determines the MC element $Z$.

\subsubsection{Loop order zero (classical) piece of $Z$}

For ease of notation we will define the Lie algebra (with $A = S(\tilde{H}^{*}(M) \oplus H^{*}(M,\partial M))$, $A_{\partial} = S(\tilde{H}^{*}(\partial M))$ the diagonal data of the canonical combinatorial model):\index{KGCM@$\KGC_{M}$}
\begin{equation}
  \KGC_{M} \coloneqq \SGC_{A,A_{\partial}}^{\vee}.
\end{equation}
We first consider its quotient\index{KGCMtree@$\KGC_{M}^{\tree}$}
\begin{equation}
  \KGC_{M}^{\tree} \coloneqq \KGC_{M} / I,
\end{equation}
where $I$ is the ideal spanned by all graphs with at least one loop.
Our MC element $Z \in \KGC_{M}$ projects onto an MC element $Z^{\tree} \in \KGC_{M}^{\tree}$.

Let $\hCom = \Omega(\Com^{\text{¡}})$\index{Cominf@$\hCom$} be the operad governing commutative and homotopy associative algebras (also known as $C_{\infty}$-algebras), obtained by Koszul duality from the operad $\Com$ of commutative and associative algebras.

Consider the Lie algebra $\alg g$ which governs (i.e., whose Maurer--Cartan elements are) the following data:
a $\hCom$-structure on $H(M)$;
a $\hCom$ structure on $H(\p M)$;
an $\infty$-map $H(M)\to H(\p M)$ between these $\hCom$ algebras.
More concretely,
\[
  \alg g \coloneqq \Harr(\tilde H(M), H(M)) \oplus \Harr(\tilde H(M), H(\p M)) \oplus \Harr(\tilde H(\p M),H(\p M)),
\]
where we work with the normalized Harrison complexes for simplicity (the map from the unnormalized one is a quasi-isomorphism~\cite[Proposition~1.6.5]{Loday1992}).
We write $\alg g_{1}$, $\alg g_{2}$, $\alg g_{3}$ for the three summands in the definition of $\alg g$.

Using Poincaré duality, $\alg g$ may as well be interpreted as a graph complex, whose elements are:
\begin{enumerate}
  \item rooted trees with leaves labelled by elements of $\tilde H_\bullet(M)$ and root labelled by $H_\bullet(M, \p M)$, governing the $\hCom$-structure on $H(M)$,
  \item similar rooted trees, but with root labelled by $H_\bullet(\p M)$, governing the map, and
  \item rooted trees with leaves labelled by $\tilde H_\bullet(\p M)$ and root by $H_\bullet (\p M)$, governing the $\hCom$-structure on $H(\p M)$.
\end{enumerate}
Keep in mind that the trees in the Harrison complexes are ``Lie trees'' in the sense that they are $3$-valent and satisfy the IHX relations.

We claim that we have a map
\beq{equ:MCmap}
\MC(\KGC_M^{\tree}) \to \MC(\alg g).
\eeq
Concretely, given a Maurer--Cartan element on the left we may produce a Maurer--Cartan element on the right as follows.
First recall that given a contraction data
\begin{equation}
  \begin{tikzcd}
    H(M) \ar[shift left]{r}{\iota} & \OmPA(M) \ar[shift left]{l}{p} \ar[loop right]{l}{h}
  \end{tikzcd}
\end{equation}
we may put a $\hCom$-structure on $H(M)$ by the homotopy transfer theorem, and simultaneously extend the morphisms $p$ and $\iota$ to $\hCom$-quasi-isomorphisms, see~\cite[Section~10.3]{LodayVallette2012} or~\cite{Berglund2014}.
Furthermore, the $\hCom$ structure and the morphisms are given by explicit combinatorial formulas built from $\iota,p,h$ only.
For example, the $\hCom$ operation associated to a rooted tree is obtained by interpreting the tree as a composition tree, with a map $h$ applied on any edge, and multiplication in $\OmPA(M)$ on vertices, see loc. cit. for details.

In particular, given contraction data for $\OmPA(M)$ and $\OmPA(\p M)$ as above, we may build a MC element in $\alg g$, such that the $\hCom$-structures on $H(M)$ and $H(\p M)$ are those given by transfer, and the $\hCom$ map is the composition of $\hCom$-morphisms
\begin{equation}
  H(M) \to \OmPA(M) \to \OmPA(\p M) \to H(\p M).
\end{equation}
Furthermore, note that we may interpret the homotopy on $M$ as an edge, the homotopy on $\p M$ as a boundary edge.
Hence the $\hCom$ map is combinatorially completely determined by the numeric values of tree graphs such as occurring in $Z_M^{\tree}$.
This defines the map \eqref{equ:MCmap}.
Note that it is given by explicit combinatorial formulas, whose precise form will not be too relevant for us.

Given \eqref{equ:MCmap} we may use graded polarization as in~\cite[Section~1.3]{FresseTurchinWillwacher2017} to obtain an $L_\infty$ map
\beq{equ:KGCgmap}
\KGC_M^{\tree} \to \alg g.
\eeq

We will not need the explicit combinatorial formula for the map \eqref{equ:KGCgmap} is not so important.
Let us describe the leading order pieces, which will be enough to prove that we have an $L_{\infty}$ quasi-isomorphism:
\begin{itemize}
  \item The coefficient of a rooted tree $T\in \alg g_1$ is the same as the coefficient of the identical tree in $\KGC_M$, with all edges oriented towards the root.
  \item The coefficient of a rooted tree $T\in \alg g_3$ is the same as the coefficient of the identical tree in $\KGC_M$, where all vertices are terrestrial and edges are interpreted as wedges in the sense of Figure~\ref{fig:wedge-stub}.
  \item The coefficient of a rooted tree $T\in \alg g_2$ is a sum of coefficients of various trees in $\KGC_M$, but the piece with fewest terrestrial vertices is the same tree as $T$, with edges oriented towards the root, and with the root connected to a single terrestrial vertex with the root decoration. For example:
        \begin{equation}
          \begin{tikzpicture}[yscale=.5, baseline=1cm]
            \node {$\alpha$} [grow'=up]
            child {
                child {
                    child {}
                    child {}
                  }
                child {}
              };
          \end{tikzpicture}
          \leadsto
          \begin{tikzpicture}[yscale=.5, baseline=1cm]
            \node[int, label={above right}:{$\alpha$}] (i) {} [grow'=up]
            child[<-] { node[int] {}
                child {
                    node[int] {}
                    child { node[int] {} }
                    child { node[int] {} }
                  }
                child { node[int] {}}
              };
            \draw[dotted] ($(i)-(1,0)$) -- ($(i)+(1,0)$);
          \end{tikzpicture}
          + (\dots), \quad \alpha \in H^{*}(\partial M)
        \end{equation}
\end{itemize}

Furthermore, let us consider the MC element $m\in \alg g$ which is the image of $Z \in \MC(\KGC_{M})$, and the twisted $L_\infty$ morphism
\beq{equ:KGCgmap2}
\KGC_M^{\tree,Z} \to (\alg g^m)',
\eeq
where $(\alg g^m)'$ is the Lie subalgebra in which the products and lowest component of the map are fixed.

Knowing the leading order terms above suffices to show:
\begin{proposition}\label{prop:kgc-harr-qiso}
  The map \eqref{equ:KGCgmap2} constructed above is a quasi-isomorphism.
\end{proposition}

For this we consider the following Lemma from~\cite[Section~8]{CamposWillwacher2016}.

\begin{lemma}[{\cite[Definition/Proposition~61]{CamposWillwacher2016}}]\label{lem:tree to harr}
  The quotient map $\GC_{\p M}^{\tree}\to \Harr(\tilde H(\p M),H(\p M))$ is a quasi-isomorphism.
\end{lemma}

\begin{proof}
  The map can be written as a composition of two quasi-isomorphisms
  \begin{equation}
    \GC_{\p M}^{\tree} \stackrel{q}{\to} \GC_{\p M}^{\Lie} \stackrel{r}{\to}  \Harr(\tilde H(\p M),H(\p M)),
  \end{equation}
  where the space of rooted Lie trees $\GC_{\p M}^{\Lie}$ is the quotient of $\GC_{\p M}^{\tree}$ by graphs that are $\geq 4$-valent and by the IHX relations.

  Let us first check that the quotient map $q : \GC_{\p M}^{\tree} \to \GC_{\p M}^{\Lie}$ is a quasi-isomorphism.
  We can first filter both complexes by $\deg \Gamma - \# E_{\Gamma}$ for $\Gamma \in \GC_{\p M}^{\tree}, \GC_{\p M}^{\Lie}$, the map $q$ being obviously compatible with this filtration.
  On the associated graded, the differential increases the number of vertices by exactly one.
  We can then filter both complexes by the number of decorations.
  The associated graded of the first complex is then isomorphic to the complex $\bigoplus_{k \ge 0} \Lie_{\infty}(k)_{\Z/(k+1)\Z}$, while the second is isomorphic to $\bigoplus_{k \ge 0} \Lie(k)_{\Z/(k+1)\Z}$, where the action of $\Z/(k+1)\Z$ is given by the cyclic structures of the operads $\Lie_{\infty}$ and $\Lie$ (permuting the root with the other vertices in a circle).
  The map $q$ induces the usual quotient map $\Lie_{\infty} \to \Lie$.
  Since this map is a quasi-isomorphism, we can deduce that $q$ is a quasi-isomorphism.

  The map $r$ sends a graph to the sum over all possible ways of rooting a decoration in $H(\p M)$.
  One sees that this is a quasi-isomorphism by considering a spectral sequence in which the first differential sees only the splitting of a vertex decorated by $\omega \in H(\p M)$ into a 3-valent vertex decorated by $\alpha$ and $\beta$, where $\alpha \beta =\omega$.
  A second filtration by the number of non-unital decorations gives us the result.
\end{proof}

\begin{proof}[Proof of Proposition~\ref{prop:kgc-harr-qiso}]
  We start by considering a filtration by the number of aerial vertices.
  We split $\KGC_M^{\tree} = \GC_{\p M}^{\tree} \oplus A$, where $A$ is spanned by all other graphs.
  On the associated spectral sequence, if we restrict to the piece of the map from $\GC_{\p M}^{\tree} \to \Harr(\tilde H(\p M),H(\p M))$, we obtain the projection from the previous Lemma which is a quasi-isomorphism.

  On the second page it remains
  \begin{equation}
    H(A,d_0) \to \Harr(\tilde H( M),H( M))\oplus \Harr(\tilde H(M), H(\p M)),
  \end{equation}
  where the differential $d_1$ that creates a single aerial vertex now acts by producing Lie trees with decorations in $H(M)$.

  Next, by the same argument as Lemma \ref{lem:tree to harr} we obtain (after taking the corresponding spectral sequence) cyclic Lie trees on the left hand side.

  Finally, the identification is given by the $\gamma$-decorations contributing to the remaining piece of $\alg g_1$ and simultaneously the $\alpha$-decorations giving the remaining map piece of $\alg g_2$.
\end{proof}

As a consequence, we find that the Maurer--Cartan elements up to gauge equivalence are in 1:1 correspondence and hence arrive at the following corollary:

\begin{corollary}\label{cor:ZtreedM}
  The tree level part $Z^{\tree}$ of the MC element $Z$ encodes precisely the (naive) real homotopy type of the inclusion $\p M\to M$.
\end{corollary}

\subsubsection{Higher loop orders in high dimensions and homotopy invariance}
Let us assume now that $n=\dim M\geq 5$, and that $M$ and $\p M$ have no rational homology in degree 1.
Then by the Proposition \ref{prop:HKGCM} we see that there is no cohomology in $\KGC_M$ in higher loop orders in the relevant degrees, i.e., the map
\begin{equation}
  \KGC_M\to \KGC_M^{\tree}
\end{equation}
induces an isomorphism on cohomology in degrees $\geq 0$.
We find that in this case the spaces of MC elements modulo gauge in $\KGC_M^{Z}$ and $\KGC_M^{\tree,Z}$ are isomorphic.

The last thing left to check is the dependence on the Euler class, which appears in the expression of $z_{1}^{\partial}$ in Proposition~\ref{prop:locvanishing2}.
Given a CDGA model for the pair $(M, \partial M)$, the cohomology class of the Euler class is only defined up to a nonzero scalar.
Given two Euler class $E$ and $E'$, we can first apply a rescaling by the loop order, sending a graph $\Gamma$ to $\lambda^{g} \Gamma$ (where $g$ is the genus of the graph) for some $\lambda \in \R \setminus \{ 0 \}$.
The element $z_{0}^{\partial}$ is invariant under this gauge equivalence, and choosing $\lambda$ correctly it sends $E'$ to $E + d\alpha$ for some $\alpha$.
Then we can use $\alpha$ to produce a second gauge equivalence and show that, whichever Euler class we choose, we get quasi-isomorphic graph complexes $\SGraphs_{M}$.

We therefore obtain Corollary~\ref{cor:mainSC}:
\begin{corollary}
  \label{cor.inv-htp1}
  If $n=\dim M\geq 5$ and $H^1(M)=H^1(\p M)=0$ then the (naive) real homotopy type of $\SFM_M$ depends only on the real homotopy type of the inclusion $\p M\to M$.
\end{corollary}

\section{A model for \texorpdfstring{$\mFM_M$}{mFM\_M} and \texorpdfstring{$\aFM_{\partial M}$}{aFM\_dM}}\label{sec:secondmodel}

The goal of this Section~is to construct a graphical model for the spaces $\mFM_M$ of Section~\ref{sec.ofmm}, i.e., for the configuration space of points in the interior of a manifold $M$ with non-empty boundary $\p M$.
Our model will be compatible with all algebraic structures, in particular with the action of $\aFM_{\p M}$, a model for which we construct as well.
We will not write the constructions in full detail, as they are quite similar to the constructions of $\Graphs_{A}$ and $\SGraphs_{A,A_{\partial}}$ that we already carried out.

\subsection{Graphical models \texorpdfstring{$\aGraphs_{A_{\partial}}$}{aGraphs\_Ad} and \texorpdfstring{$\mGraphs_{A}$}{mGraphs\_A}}
\label{sec:agraphs-mgraphs}

Throughout this section, we choose some diagonal data $(A \xrightarrow{\rho} A_{\partial}, \Delta_{A}, \sigma_{A}, g, g_{\partial})$ which maps into $M$ (see Section~\ref{sec.a-data} for the definition).
We also choose some compatible propagators on $\aFM_{\partial M}$ and $\mFM_{M}$ as in Sections~\ref{sec.propagator-afm} and~\ref{sec.propagator-mfm}.
We will also define $\sigma_{\partial}$ to be the image of $\sigma_{A}$ in $A_{\partial}^{\otimes 2}$, recalling that it is ``half'' the diagonal class of $\partial M$ (i.e., $\Delta_{A_{\partial}} = \sigma_{\partial} + \sigma_{\partial}^{21}$).

\subsubsection{Graphical model for $\aFM_{\partial M}$}

We define a graph complex $\aGraphs_{A_{\partial}}(r)$\index{aGraphsAd@$\aGraphs_{A_{\partial}}$} spanned by graphs of the following type:
\begin{itemize}
  \item there are $r$ numbered (``external'') vertices and a finite but arbitrary number of ``internal vertices'';
  \item each vertex is decorated by an element of $A_{\partial}$;
  \item edges are directed and have degree $n-1$, internal vertices have degree $-n$, and we add the degrees of the decorations to the degree of the graph;
  \item there are no connected components with only internal vertices.
\end{itemize}

We moreover have the following algebraic structures on $\aGraphs_{A_{\partial}}$:
\begin{itemize}
  \item the right coaction of $\Graphs_{n}^{A_{\partial}}$ is given by multiplication and subgraph contraction at an external vertex;
  \item the coproduct, which corresponds to the algebra structure on configuration spaces (by gluing), is obtained by decomposing the graph in two parts, and then replacing any edge connecting the two parts by $\sigma_\p$;
  \item we obtain a corresponding graph complex $\aGC_{A_{\partial}}$\index{aGCAd@$\aGC_{A_{\partial}}$} of connected graphs with no external vertices; the coalgebra structure yields a Lie bracket on the dual complex $\aGC^{\vee}_{A_{\partial}}$, which acts on on $\aGraphs_{A_{\partial}}$.
\end{itemize}

Finally, the differential is defined as follows.
The usual configuration space integral formulas give rise to a Maurer--Cartan element $w \in \aGC_{A_{\partial}}^{\vee}$, to be defined precisely in Equation~\eqref{eq:def-w} below.
The differential on $\aGraphs_{A_{\partial}}$ is given by edge contraction, plus a twist by $w$ (see Figure~\ref{fig.diff-tw-sgra} for an idea of how this works).
If anything produces a graph containing an internal component (i.e., a connected component with only internal vertices), then this component is replaced by a real coefficient given by $w$.

Note that this construction is functorial in terms of the diagonal data (see Section~\ref{sec.functoriality-models}).
Using the propagator of Section~\ref{sec.propagator-afm}, we may build maps
\begin{equation}
  \aGraphs_{A_{\partial}} \to \OmPA^{*}(\aFM_{\partial M})
\end{equation}
which are compatible with the algebraic structures (CDGA, module over the Fulton--MacPherson operad, $E_{1}$-algebra).

\subsubsection{Graphical model for $\mFM_{M}$}

We now consider a graph complex $\mGraphs_{A}(r)$\index{mGraphsA@$\mGraphs_{A}$}, spanned by graphs of the following type:
\begin{itemize}
  \item there are $r$ numbered (``external'') vertices and an finite but arbitrary number of ``internal vertices'';
  \item each vertex is decorated by an element of $A$;
  \item edges are directed and have degree $n-1$, internal vertices have degree $-n$, and we add the degrees of the decorations to the degree of the graph.
  \item there are no connected components with only internal vertices.
\end{itemize}

The following algebraic structures are defined on $\mGraphs_{A}$:
\begin{itemize}
  \item a right coaction of $\Graphs_{n}^{A}$, given by subgraph contraction at an external vertex using the multiplication from $A$;
  \item a coaction of $\aGraphs_{A_{\partial}}$, corresponding to the module structure on configuration spaces (again, by gluing), obtained by decomposing the graph in two parts, replacing any edge connecting the two parts by either $\sigma_{A}$ if the tip of the edge goes to the boundary and zero otherwise, and applying $\rho : A \to A_{\partial}$ to all the labels in the graph going to the boundary;
  \item we obtain a corresponding graph complex $\mGC_{A}$\index{mGCA@$\mGC_{A}$} of connected graphs as above with no external vertices; the dual space is a dg Lie algebra $\mGC^{\vee}_A$ which comes with an action of $\mGC_{A_{\partial}}$, i.e., we have a Lie algebra $\mGC^{\vee}_{A} \rtimes \aGC^{\vee}_{A_{\partial}}$, which acts on $\mGraphs_{A}$.
\end{itemize}

Finally we describe the differential.
The usual configuration space integral formulas give rise to a Maurer--Cartan element $W \in \mGC_{A}^{\vee}$, see Equation~\eqref{eq:def-W2} -- more precisely, $w+W$ is a MC element in the semidirect product  $\mGC_A^{\vee} \ltimes \aGC^{\vee}_{A_{\partial}}$.
The differential on $\mGraphs_{A}$ is given by edge contraction, plus replacing one edge by $\Delta_{A}$, plus a twist by $w+W$.

\begin{remark}
  If we consider the combinatorial diagonal data $A = S(\tilde{H}^{*}(M) \oplus H^{*}(M,\partial M))$ (see Section~\ref{sec.a-data}), then we may define an undirected version of the graph complex.
  In this version, edges are undirected (formally an edge is identified with $(-1)^{n}$ times its opposite), and decorations may only be in $\tilde{H}^{*}(M)$, not $H^{*}(M, \partial M)$.
  The differential of the edge is $\Delta_{A}$, and the boundary value of all edges is $\sigma_{A}$, regardless of whether the shaft or the tip of the edge goes to the boundary.
\end{remark}

This construction is again functorial in terms of the diagonal data, and using the propagator of Section~\ref{sec.propagator-mfm}, we build maps:
\begin{equation}
  \mGraphs_{A} \to \OmPA^{*}(\mFM_{M})
\end{equation}
which are compatible with all the algebraic structures (CDGA, module over the Fulton--MacPherson operad, module over the $E_{1}$-algebra $\aFM_{\partial M}$).

\subsection{The Maurer--Cartan element}\label{sec:oGCMC}
\subsubsection{Boundary}

For convenience, we let $N \coloneqq \partial M$.
Using the usual Feynman rules and the propagator, we can define a Maurer--Cartan element:\index{w@$w$}
\begin{equation}\label{eq:def-w}
  w \coloneqq \sum_{\gamma} \bigl( \scalebox{.8}{$\int_{\aFM_{N}} \bigwedge_{(ij)} \varphi(x_i,x_j)$} \bigr) \cdot \gamma \in \aGC^{\vee}_{A_{\partial}}.
\end{equation}

We have $w = w_{0} + ({\cdots})$ is a sum of\index{w0@$w_{0}$}
\beq{equ:oGCMC1}
w_0 \coloneqq
\sum_i
\begin{tikzpicture}
  \node[int,label={$\scriptstyle \alpha_i\beta_i$}] at (0,0) {};
\end{tikzpicture}
,
\eeq
plus terms with at least one edge or more than two decorations.

Note that the dg Lie algebra $\aGC^{\vee}_{A_{\partial}}$ is equipped with a descending complete filtration $\mF^p$ by loop order.
In particular, we note that the tree part $w_{\tree}$ of $w$ determines a Maurer--Cartan element in
\begin{equation}
  \aGC^{\vee}_{A_{\partial}} / \mF^1\aGC^{\vee}_{A_{\partial}}.
\end{equation}

The piece $w_{\tree}$ encodes precisely the cyclic $\hCom$-structure on $H(N)$, i.e., the real homotopy type of $N$ (cf.~\cite[Section~8]{CamposWillwacher2016}, where one finds an analogous computation).
A priori, there might be higher loop pieces in $w$.
We will next show, however, that they vanish up to gauge equivalence.

We consider first the case $\dim N=1$.
Since $N$ is compact, we have $N=S^1 \sqcup \cdots \sqcup S^1$.
In fact, it is sufficient to consider $N=S^1$, as the MC element for a union of circles is essentially just a sum of MC elements for each circle separately.
For $N=S^1$, we have that $\aFM_{N}$ is homotopy equivalent to the configuration space of points on the cylinder.
The spaces $\aFM_{N}$ and graphical models for them have been constructed in~\cite{Willwacher2016}.
Among other things, it was shown there that for a suitable choice of propagator
\begin{equation}
  w =
  \begin{tikzpicture}
    \node[int,label={$\scriptstyle \omega$}] at (0,0) {};
  \end{tikzpicture},
\end{equation}
where $\omega$ is a volume form on $S^1$, representing the top cohomology class.

Next consider the case $\dim N\geq 2$.
Recall from Proposition~\ref{prop:oGCNI} that the dg Lie subalgebra
\begin{equation}
  \aGC_{A_{\partial}}^{\vee,\geq 3} \subset (\aGC^{\vee}_{A_{\partial}}, \delta + [w_0,-]),
\end{equation}
spanned by diagrams all of whose vertices have valence at least 3, is quasi-isomorphic to the full dg Lie algebra in degrees $\geq 1$.
It follows that, up to gauge equivalence, we may assume that our MC element $w$ takes values in this dg Lie subalgebra.
However, by the degree counting result of Lemma~\ref{lem:oGCdegcounting}, we see that $\aGC_{A_{\partial}}^{\vee}$ has no elements in degree $1$ in loop order $\geq 1$.
Hence, we find that the tree piece of our MC element already agrees with the MC element.

Let us summarize these findings:

\begin{proposition}\label{prop:ozinvariance}
  The gauge equivalence class of the MC element $w$ is encoded by and encodes the real homotopy type of $N$.
  Furthermore, the MC element $w$ above is gauge equivalent to a Maurer--Cartan element
  \[
    w_0+ w'_{\tree}
  \]
  where $w'_{\tree} \in \aGC_{A_{\partial}}^{\vee,\geq 3}\subset (\aGC^{\vee}_{A_{\partial}}, \delta + [w_0,-])$ contains no graphs of loop order $\geq 1$.
\end{proposition}

\subsubsection{Bulk}

We may also define an element:\index{W@$W$}
\begin{equation}\label{eq:def-W2}
  W \coloneqq \sum_{\Gamma}\bigl( \scalebox{.8}{$\int_{\mFM_M} \bigwedge_{(ij)} \varphi(x_i,x_j)$} \bigr) \cdot \Gamma \in \mGC^{\vee}_A,
\end{equation}
so that the pair $(w,W)\in \aGC^{\vee}_{A_{\partial}} \ltimes \mGC^{\vee}_A$ is a Maurer--Cartan element.
Concretely, $W$ has the form $W = W_{0} + ({\cdots})$, where\index{W0@$W_{0}$}
\beq{equ:oGCMC2}
W_0 \coloneqq
\sum_i
\begin{tikzpicture}[baseline=-.65ex]
  \node[int,label={$\scriptstyle \gamma_i\gamma_i^*$}] at (0,0) {};
\end{tikzpicture},
\eeq
and $({\cdots})$ are terms with at least one edge or three or more decorations.

Once again the complex $\mGC_A^{\vee}$ is equipped with a descending complete filtration $\mF^p$ by loop order.
The tree piece $W_{\tree}$ extends $w_{\tree}$ above to encode a Maurer--Cartan element in the complex
\begin{equation}
  \aGC^{\vee}_{A_{\partial}} \ltimes \mGC_A /  \mF^1\left(\aGC_{A_{\partial}}^{\vee} \ltimes \mGC_A^{\vee}\right).
\end{equation}
This tree piece can be seen (as in Section~\ref{sec.global-maurer-cartan} above) to encode the real homotopy type of the map
\begin{equation}
  \p M\to M.
\end{equation}

Furthermore, we have the subcomplex of graphs all of whose vertices are at least trivalent
\begin{equation}
  \aGC_{A_{\partial}}^{\vee,\geq 3} \ltimes \mGC_A^{\vee,\geq 3} \subset (\aGC_{A_{\partial}}^{\vee} \ltimes \mGC_A^{\vee}, z_0+Z_0).
\end{equation}
Proposition~\ref{prop:incl-agc-mgc} shows that the inclusion is a quasi-isomorphism in degrees $\geq 0$, so that we may in particular gauge change our MC element $w-w_0+W-W_0$ deforming $w_0+W_0$ to one which is contained in the subcomplex.
But then, if $H^1(M)=0$ and $\dim M\geq 4$, the degree counting argument of Lemma~\ref{lem:oGCdegcounting2} shows that there are no elements in that subcomplex of degree 1 and loop orders $\geq 1$. It follows that our gauge changed MC element is in fact equal to the tree piece.
We summarize this as:
\begin{proposition}
  \label{prop:wW-invar}
  If $\dim M\geq 4$ and $H^1(M)=0$, then the gauge equivalence class of the Maurer--Cartan element $w+W \in \aGC_{A_{\partial}}^{\vee} \ltimes \mGC_A$ encodes and is encoded by the real homotopy type of the inclusion $\p M\to M$.
  We can find representative without terms of loop order $\geq 1$, and all vertices in graphs at least trivalent, with the exception of those occurring in $w_0+W_0$.
\end{proposition}

\subsection{Cohomology, and proof of quasi-isomorphism property}
\label{sec:afm-mfm-main-thm}

\begin{proposition}
  The map $\aGraphs_{A_{\partial}} \to \OmPA(\aFM_{N})$ of the preceding section is a quasi-isomorphism.
\end{proposition}

\begin{proof}
  We proceed by an induction on the arity $r$, using a proof technique similar to the one of Section~\ref{sec:sfm-main-thm}.
  For the induction step, we apply Proposition~\ref{prop:fiberqiso} to the diagram
  \begin{equation}
    \begin{tikzcd}
      \aGraphs_{A_{\partial}}(r) \ar{r}{\sim} \ar{d} & \OmPA(\aFM_{N}(r)) \ar{d} \\
      \aGraphs_{A_{\partial}}(r+1) \ar{r} & \OmPA(\aFM_{N}(r+1)).
    \end{tikzcd}
  \end{equation}
  Note that
  \begin{equation}
    \OmPA(\aFM_{N}(r+1)) \lotimes_{\OmPA(\aFM_{N}(r))} \K
  \end{equation}
  is a model for the homotopy fiber, but since $\aFM_{N}(r+1)\to \aFM_{N}(r))$ is a fibration it is a model for the fiber.
  This fiber is just $N\times I \setminus *^{\sqcup r}$, the manifold $N\times I$ with $r$ points removed.
  As in \eqref{equ:fibercohom0}, the cohomology is easy to compute by an Mayer--Vietoris inductive argument and agrees with
  \beq{equ:fibercohom}
  H(N) \oplus (\K[1-n])^{\otimes r}.
  \eeq
  On the other hand we compute
  \begin{equation}
    \aGraphs_{A_{\partial}}(r+1) \lotimes_{\aGraphs_{A_{\partial}}(r)} \K.
  \end{equation}
  Since the left-hand side is a free module over the algebra, the homotopy tensor product agrees with the ordinary one.
  This in turn can be identified with graphs, where graphs which have internally connected components not connected to vertex $1$ are evaluated using the map $\aGraphs_{A}(r) \to \K$.

  To compute the cohomology of that complex (say $V$), we take a filtration by the number of edges.
  All pieces of the differential reduce that number.
  The differential on the first page of the spectral sequence is given by the piece of the differential killing exactly one edge.
  It has two pieces: The contraction of an edge, plus letting a vertex attached to a dingle edge go to infinity.

  We next mimic the trick of~\cite{LambrechtsVolic2014}.
  We impose another filtration by the arity of the first vertex.
  \begin{equation}
    V=V_0 \oplus V_1 \oplus V_{\geq 1}.
  \end{equation}
  As in loc.\ cit.\ the piece of the differential  $V_1\to V_{\geq 1}$ is surjective, and the kernel consists of graphs such that the first vertex has valence 1 and is connected to either another external vertex or is decorated by $A_{\partial}^{>0}$.
  Note also that $V_0=\K$.
  The spectral sequence abuts here, so we that $H(V)$ agrees with \eqref{equ:fibercohom}.
  Hence invoking Proposition~\ref{prop:fiberqiso} we are done.
\end{proof}

\begin{remark}
  Together with the characterisation of the MC element $w$ in Proposition \ref{prop:ozinvariance} the above result shows in particular that the real homotopy type of the configuration space of points on $N\times I$ only depends on the real homotopy type of $N$, without further conditions on $N$.
  It is interesting to compare this result to the main result of~\cite{RaptisSalvatore2016}, who show that the homotopy type of (a space related to) the configuration space of two points on $N \times X$, for $X$ contractible and not equal to a point, is a homotopy invariant of $N$.
  Both results are not true without taking a product with the contractible space $I$ or $X$~\cite{LoSa2005}.
  Hence both results confirm the picture that thickening our manifold takes away the ``interesting information'' in the homotopy type of the configuration space.
\end{remark}

We now turn our attention to configurations in $M$.

\begin{proposition}
  The map $\mGraphs_{A} \to \OmPA(\mFM_{M})$ defined above is a quasi-isomorphism.
\end{proposition}

\begin{proof}
  We proceed as in the ``boundary-case''-proof of the preceding section, by an induction on the arity $r$.
  For the induction step, we now apply the Proposition \ref{prop:fiberqiso} to the diagram
  \begin{equation}
    \begin{tikzcd}
      \mGraphs_{A}(r) \ar{r}{\sim} \ar{d} & \OmPA(\mFM_{M}(r)) \ar{d} \\
      \mGraphs_{A}(r+1) \ar{r} & \OmPA(\mFM_{M}(r+1)).
    \end{tikzcd}
  \end{equation}
  The fiber of the map $\mFM_{M}(r+1)\to \mFM_{M}(r)$ agrees with $M$ with $r$ points removed. The cohomology is hence
  \beq{equ:fibercohom2}
  H(M) \oplus (\K[1-n])^{\otimes r}.
  \eeq
  On the other hand let us compute
  \begin{equation}
    \mGraphs_{A}(r+1) \lotimes_{\mGraphs_{A}(r)} \K
    =
    \mGraphs_{A}(r+1) \otimes_{\mGraphs_{A}(r)} \K.
  \end{equation}
  Again the right-hand side can be identified with graphs, where graphs which have internally connected components not connected to vertex $1$ are evaluated using the map $\mGraphs_{A}(r) \to \K$.

  To compute the cohomology of that complex, we take a filtration by the number of edges.
  All pieces of the differential reduce that number.
  The differential on the first page of the spectral sequence is given by the piece of the differential killing exactly one edge.
  It has now three pieces: the contraction of an edge; letting a vertex attached to a single edge go to the boundary; a term replacing an edge by the section $\sigma_{A}$.

  We next mimic again the trick of~\cite{LambrechtsVolic2014}.
  We impose another filtration by the arity of vertex $1$.
  \begin{equation}
    V=V_0 \oplus V_1 \oplus V_{\geq 1}.
  \end{equation}
  As above, one can check that $H(V)$ agrees with \eqref{equ:fibercohom2}.
  Indeed, if we filter by loop order, then on the first page of the associated spectral sequence we get the same complex but without edge cutting in the differential.
  Hence invoking Proposition \ref{prop:fiberqiso} we are done.
\end{proof}

This ends the proof of Theorem~\ref{thm:B}.

\begin{proof}[Proof of Corollary~\ref{cor:mainoGC}]\label{proof:cor-d}
  The idea is the same as the one in Section~\ref{sec.global-maurer-cartan}.
  The key point of the proof is Proposition~\ref{prop:wW-invar}, i.e., the vanishing results of Corollaries~\ref{cor:vanish-agc} and~\ref{cor:vanish-mgc}.
\end{proof}

\section{Relation and small models for ``good'' spaces}
\label{sec.model-conf-spac}

From now on, let us assume that $M$ is a smooth, simply connected manifold with simply connected boundary, such that $n=\dim M \ge 7$.
We fix a Poincaré--Lefschetz duality model of $M$ (see Section~\ref{sec.definition}).

We will use the notations of Equation~\eqref{eq.nice-diagram}.
Let  $B \xrightarrow{\lambda} B_{\partial}$ be the model of $\partial M \hookrightarrow M$, $\varepsilon : \cone(\lambda) \to \R[-n+1]$ be the orientation inducing the Poincaré--Lefschetz pairing, $K = \ker(\lambda)$ be the model for $\Omega^{*}(M,\partial M)$, $P = B / \ker \theta_{B}$ be the model for $\Omega^{*}(M)$, and $\theta : P \cong K^{\vee}[-n]$ be the isomorphism of Poincaré--Lefschetz duality.
We also recover diagonal data as in Example~\ref{exa.diag-data}.

\subsection{The dg-module model of \texorpdfstring{$\Conf_{k}(M)$}{Conf\_k(M)}}
\label{sec.dgmodel}

Given a finite set $V$ and an element $v \in V$, define the canonical injection $\iota_{v} : P \to P^{\otimes V}$ by
\begin{equation}
  \iota_{v}(x) \coloneqq 1 \otimes \dots \otimes 1 \otimes \underbrace{x}_{v} \otimes 1 \otimes \dots \otimes 1.
\end{equation}

\begin{definition}
  \label{def.ga}
  Define a symmetric collection of CDGAs $\GG{P}$ by:\index{GP@$\GG{P}$}
  \begin{equation*}
    \GG{P}(V) \coloneqq \bigl( P^{\otimes V} \otimes \enV(V) / (\iota_{v}(x) \cdot \omega_{vv'} = \iota_{v'}(x) \cdot \omega_{vv'}), \; d(\omega_{vv'}) = (\iota_{v} \cdot \iota_{v'})(\Delta_{P}) \bigr)
  \end{equation*}
  with the obvious actions of the symmetric groups.
\end{definition}

This definition also makes sense when $\partial M = \varnothing$, and it yields the symmetric collection of CDGAs considered in~\cite{Idrissi2018b}.

When $M$ is a closed manifold, as soon as its Euler characteristic $\chi(M)$ vanishes, then there is a structure of Hopf right $\enV$-comodule on $\GG{P}$~\cite[Proposition~2.1]{Idrissi2018b}, with cocomposition structure maps characterized by (compare with the definition of $\Gra_{n}$ at the beginning of Section~\ref{sec.extens-swiss-cheese}):
\begin{equation}
  \begin{aligned}
    \circ^{\vee}_{T}(\omega_{vv'})
     & = 1 \otimes \omega_{vv'},
     & \text{if } \{v, v'\} \subset T;     \\
    \circ^{\vee}_{T}(\omega_{vv'})
     & = \omega_{[v][v']} \otimes 1,
     & \text{if } \{v,v'\} \not \subset T; \\
    \circ_{T}^{\vee}(\iota_{v}(x))
     & = \iota_{[v]}(x)
     & \text{for } x \in A, \; v \in V;
  \end{aligned}
\end{equation}
where $[v] \in V/T$ is the class of $v$ in the quotient.

\begin{proposition}
  \label{prop.ga-comod}
  If $\partial M \neq \varnothing$, then the symmetric collection of CDGAs $\GG{P}$ forms a right Hopf $\enV$-comodule, with the same formulas.
\end{proposition}
\begin{proof}
  Comparing with the proof of~\cite[Proposition~2.1]{Idrissi2018b}, we see that almost all the arguments are the same.
  The only difficulty is to check that the cocomposition is compatible with the differential, which required that the Euler characteristic vanished in the boundaryless case.

  It is immediate to check that $d(\circ_{T}^{\vee}(\iota_{v}(x))) = \circ_{T}^{\vee}(d(\iota_{v}(x))) = \iota_{[v]}(dx)$, thus it suffices to check that the same equality holds on the generators $\omega_{vv'}$.
  If either $v \not\in T$ or $v' \not\in T$, this is again immediate; hence it suffices to check that
  \[ d(\circ_{T}^{\vee}(\omega_{vv'})) = \circ_{T}^{\vee}(d(\omega_{vv'})) \textnormal{ for } v,v' \in T \]

  The LHS of that equation always vanishes.
  On the other hand, the RHS is equal to $\iota_{*}(\mu_{P}(\Delta_{P}))$, where $\mu_{P} : P \otimes P \to P$ is the product.
  But by Equation~\eqref{eq.mu-delta-zero}, $\mu_{P}(\Delta_{P}) = 0$.
\end{proof}

\begin{example}
  \label{exa.model-dn-deux}
  Recall from Example~\ref{exa.model-dn} the model for $(D^{n}, S^{n-1})$, with $P = \R$, and $\Delta_{P} = 0$.
  It follows that in this case, $\GG{P}$ is isomorphic to $\enV$ seen as a Hopf right comodule over itself.
  This is not surprising, given that $\FM_{n}$ is formal as an operad, and hence as a module over itself, and that $\SFM_{D^{n}}(\varnothing, -)$ is weakly equivalent to $\FM_{n}$ as a right $\FM_{n}$-module.
\end{example}

\subsection{Computing the homology}
\label{sec.computing-homology}

We now prove that $\GG{P}$ has the right cohomology, in the spirit of~\cite{LambrechtsStanley2008a} and using the methods of~\cite{CordovaBulensLambrechtsStanley2015a} to deal with manifolds with boundary.
From then on and until the end of this section, we fix some integer $k \geq 0$.
We can work over $\Q$ in this section.

The general idea goes as follows.
If $W$ is a manifold with boundary, and $X \subset W$ is a sub-polyhedron, then by~\cite{CordovaBulensLambrechtsStanley2015a}, it suffices to know a CDGA model of the square of inclusions
\begin{equation}
  \label{eq.square}
  \begin{tikzcd}
    \partial W \ar[hook]{r} & W \\
    \partial_{W}X \coloneqq X \cap \partial W \ar[hook]{r} \ar[hook]{u} & X \ar[hook]{u}
  \end{tikzcd}
\end{equation}
to obtain a complex computing the cohomology of $W - X$.

Therefore, to compute the cohomology of $\Conf_{k}(M)$, we need to find such models for $W = M^{k}$ and\index{Deltak@$\Delta_{(k)}$}
\begin{equation}
  X = \Delta_{(k)} \coloneqq \bigcup_{1 \le i, j \le k} \Delta_{ij},
\end{equation}
where $\Delta_{ij} \coloneqq \{ x \in M^{k} \mid x_{i} = x_{j} \}$.
Since the sub-polyhedron $\Delta_{(k)}$ can be decomposed into the sub-polyhedra $\Delta_{ij}$, we can use the techniques of~\cite{LambrechtsStanley2008a} to further simplify the description of the dg-module model as a ``total cofiber'' indexed by graphs, which will be isomorphic to $\GG{A}$.

Let us now give the details.
Let $E = \{ (i,j) \mid 1 \leq i < j \leq k \}$ be a set of pairs, and let $\Gamma$ be the poset of subsets of $E$ ordered by reverse inclusion.
We can see an element $\gamma \in \Gamma$ as a graph on $k$ vertices, with an edge between $i$ and $j$ iff $(i,j) \in \gamma$.
In particular $\varnothing \in \Gamma$ is the ``empty'' graph with no edges (but $k$ vertices).
Using this point of view, we can define the ``zeroth homotopy group'' $\pi_{0}(\gamma)$ of a graph $\gamma \in \Gamma$, which is a partition of $\{1, \dots, k\}$.

We obtain a functor $\nabla$\index{$\nabla$} from $\Gamma$ to the category of topological spaces defined by
\begin{equation}
  \gamma \mapsto \nabla(\gamma) \coloneqq \bigcap_{e \in E_{\gamma}} \Delta_{e} \subset M^{k},
\end{equation}
where $\Delta_{(i,j)}$ is simply the small diagonal $\Delta_{ij}$.
Note that $\nabla(\varnothing) = M^{k}$, and that if $\gamma' \supset \gamma$ then there is an inclusion $\nabla(\gamma') \subset \nabla(\gamma)$.
The space $\nabla(\gamma)$ is homeomorphic to the product $M^{\pi_{0}(\gamma)}$, and under these homeomorphisms, the inclusion $\nabla(\gamma') \subset \nabla(\gamma)$ is the cofibration induced by iterations of the diagonal map $M \to M \times M$.
We thus obtain that:
\begin{equation}
  \label{eq.delta-colim}
  \Delta_{(k)} = \bigcup_{1 \leq i < j \leq k} \Delta_{ij} = \colim_{\gamma \in \Gamma} \nabla(\gamma) = \colim_{\gamma \in \Gamma} M^{\pi_{0}(\gamma)}.
\end{equation}

\begin{lemma}
  The space $\Delta_{(k)}$ is the homotopy colimit of $\gamma \mapsto M^{\pi_0(\gamma)}$.
\end{lemma}
\begin{proof}
  We are going to check that the functor $\nabla = M^{\pi_0(-)}$ is cofibrant in the (Reedy) category (see e.g., \cite[Section~5]{Hovey1999} for details).
  Let $\gamma \in \Gamma$ be some graph.
  The latching space $L_\gamma \nabla$ is defined as:
  \begin{equation}
    L_\gamma \nabla \coloneqq \colim_{\gamma' \supset \gamma, \; \gamma' \neq \gamma} M^{\pi_0(\gamma)},
  \end{equation}
  and we have to check that the canonical map $L_\gamma \nabla \to \nabla(\gamma)$ (induced by the inclusions $\nabla(\gamma') \subset \nabla(\gamma)$ for $\gamma' \supset \gamma$) is a cofibration.
  We can either have $L_\gamma \nabla = \nabla(\gamma)$ if $\pi_0(\gamma)$ is just a point (i.e., all the vertices are connected in $\gamma$), and the map is just the identity, which is a cofibration.
  Otherwise, if $\pi_0(\gamma)$ has at least two components, then we have
  \begin{equation}
    L_\gamma \nabla \cong \{ x \in \nabla(\gamma) = M^{\pi_0(\gamma)} \mid \exists [i] \neq [j] \in \pi_0(\gamma) \text{ s.t. } x_i = x_j \}.
  \end{equation}
  In other words, $L_\gamma \nabla$ is the subspace of $M^{\pi_0(\gamma)}$ of configurations where at least two coordinates are equal.
  The inclusion is clearly a cofibration (for example, it is a union of submanifolds that intersect transversally).
\end{proof}

We will first aim to build a CDGA model for the square~\eqref{eq.square} (with $W = M^{k}$ and $X = \Delta_{(k)}$) out of the diagram~\eqref{eq.nice-diagram}.
The previous description of $\Delta_{(k)}$ as a homotopy colimit tells us that a model for $\Delta_{(k)}$ is given by the homotopy limit $\holim_{\gamma \in \Gamma^{\mathrm{op}}} B^{\otimes \pi_{0}(\gamma)}$, where the maps in the diagram are induced by iterations the multiplication $\mu_{B}$ of $B$.
Since $\mu_{B}$ is surjective and the limit is directed, this homotopy limit is actually equivalent to the classical limit.
The inclusion $\Delta_{(k)} \subset M^{k}$ is modeled by the canonical map from $B^{\otimes k} = B^{\otimes \pi_{0}(\varnothing)}$ to the limit, obtained by patching together the various maps $B^{\otimes k} \to B^{\otimes \pi_{0}(\gamma)}$.

It remains to find a model for $\partial_{W} X = \Delta_{(k)} \cap \partial(M^{k})$ and models for the inclusion maps.
The morphism $B \to B_{\partial}$ is surjective, hence $B_{\partial}$ is isomorphic to $B / K$ where $K \coloneqq \ker (B \to B_{\partial})$.
We get:

\begin{lemma}
  \label{lem.square}
  For all $i \geq 0$, the left-hand side square is a CDGA model for the right-hand side square, where the horizontal maps are the diagonal maps:
  \[ \begin{tikzcd}[cramped]
      B \ar[two heads, d] & B^{\otimes i} \ar[two heads, d] \ar[l, "\mu_{B}^{(i)}" swap] \\
      B / K & B^{\otimes i} / K^{\otimes i} \ar[l, "\mu_{B}^{(i)}" swap]
    \end{tikzcd}
    \text{ is a model for }
    \begin{tikzcd}[cramped]
      M \ar[r, hook, "\delta"] & M^{i} \\
      \partial M \ar[hook, u] \ar[r, hook, "\delta"] & \partial \bigl( M^{i} \bigr) \ar[hook, u]
    \end{tikzcd} \]
\end{lemma}
\begin{proof}
  The idea is the same as in~\cite[Proposition~5.1]{CordovaBulensLambrechtsStanley2015a}.
  We work by induction.

  The case $i = 1$ is obvious (as $B_{\partial} \cong B/K$), and they prove the case $i = 2$.
  Now let us assume that the proposition is true for a given $i \geq 2$.
  There is a diagram, where all the inclusions are either induced by diagonal maps or induced by $\partial M \subset M$:
  \begin{equation}
    \begin{tikzcd}[cramped]
      M \times M^{i} && M \ar[ll, hook] \\
      \partial(M \times M^{i}) \ar[u, hook] & M \times \partial(M^{i}) \ar[l, hook] \\
      (\partial M) \times M^{i} \ar[u, hook] & (\partial M) \times (\partial(M^{i})) \ar[l, hook] \ar[u, hook] \ar[ul, phantom, "\text{\tiny (ho.\ pushout)}"] & \partial M \ar[l, hook] \ar[uu, hook]
    \end{tikzcd}.
  \end{equation}
  The diagram of the proposition is the ``outer'' diagram, and the bottom left square is a (homotopy) pushout.

  Let $P$ be the (homotopy) pullback in CDGAs of
  \[ B_{\partial} \otimes B^{\otimes i} \to B_{\partial} \otimes (B^{\otimes i} / K^{\otimes i}) \gets B \otimes (B^{\otimes i} / K^{\otimes i}). \]
  Then the induction hypothesis and the fact that homotopy pushouts of spaces become homotopy pullbacks of models imply that the following diagram is a CDGA model of the previous one (where the maps are either induced by $\mu_{B} : B^{\otimes 2} \to B$ or $\lambda : B \to B_{\partial}$):
  \begin{equation}
    \begin{tikzcd}[cramped]
      B \otimes B^{\otimes i} && B \ar[ll, leftarrow] \\
      P \ar[u, leftarrow] & B \otimes (B^{\otimes i} / K^{\otimes i}) \ar[l, leftarrow] \\
      B_{\partial} \otimes B^{\otimes i} \ar[u, leftarrow] & B_{\partial} \otimes (B^{\otimes i} / K^{\otimes i}) \ar[l, leftarrow] \ar[u, leftarrow] \ar[ul, phantom, "\text{\tiny (ho.\ pullback)}"] & B_{\partial} \ar[l, leftarrow] \ar[uu, leftarrow]
    \end{tikzcd}
  \end{equation}

  Now, as in the proof~\cite[Lemma~5.3]{CordovaBulensLambrechtsStanley2015a}, it is clear that the natural map $B^{\otimes (i+1)} \to P$ is surjective and that its kernel is $K^{\otimes(i+1)}$, in other words that $P \cong B^{\otimes(i+1)} / K^{\otimes(i+1)}$.
  The proposition then follows immediately.
\end{proof}

The space $\Delta_{(k)} \cap \partial(M^{k})$ admits a description as a colimit similar to the one Equation~\ref{eq.delta-colim}.
Indeed, a point of $M^{k}$ is in the boundary iff one of the coordinates is in the boundary of $M$.
Now if a point is in both $\nabla(\gamma)$ and $\partial(M^{k})$, then at least one of the coordinates (say $x_{i}$) is in the boundary, and thus all the points indexed by some $j$ in the same connected component as $i$ in $\gamma$ is also in the boundary.
We then obtain:
\begin{equation}
  \Delta_{(k)} \cap \partial(M^{k}) = \colim_{\gamma \in \Gamma} \colim_{\varnothing \subsetneq S \subset \pi_{0}(\gamma)} (\partial M)^{S} \times M^{\pi_0(\gamma) - S}.
\end{equation}
For a fixed $\gamma$, the inner colimit is precisely the image of $\partial(M^{\pi_{0}(\gamma)})$ under the diagonal embedding $M^{\pi_{0}(\gamma)} \hookrightarrow M^{k}$.
Combining this with the previous lemma, we then obtain:

\begin{proposition}
  A model for the square~\eqref{eq.square}, with $W = M^{k}$ and $X = \Delta_{(k)}$, is given by:
  \[ \begin{tikzcd}
      B^{\otimes k} \ar[r, two heads, "\alpha_{k}"] \ar[d, "\xi_{k}"] & B^{\otimes k} / K^{\otimes k} \ar[d] \\
      \lim_{\gamma \in \Gamma^{\mathrm{op}}} R^{\otimes \pi_{0}(\gamma)} \ar[r, "\beta_{k}"] & \lim_{\gamma \in \Gamma^{\mathrm{op}}} R^{\otimes \pi_{0}(\gamma)} / K^{\otimes \pi_{0}(\gamma)}
    \end{tikzcd}
    \qed \]
\end{proposition}

The map $\alpha_{k}$ is surjective, thus we have a canonical quasi-isomorphism $\ker \alpha_{k} = K^{\otimes k} \qiso \hoker \alpha_{k}$.
Therefore, by~\cite[Proposition~3.1]{CordovaBulensLambrechtsStanley2015a}:
\begin{corollary}
  \label{cor.cone-cohom}
  The cohomology of the cone
  \[ \cone \bigl( (\hoker \beta_{k})^{\vee}[-nk] \xrightarrow{\bar{\xi}_{k}} (K^{\otimes k})^{\vee}[-nk] \bigr) \]
  of the map induced by $\xi_{k}$ on the kernels is isomorphic, as a graded vector space, to the cohomology of $\Conf_{k}(M) = M^{k} - \Delta_{(k)}$.
  \qed
\end{corollary}

Armed with this corollary, we can now define a cubical diagram $C_{\bullet}$ (see~\cite[Section~7]{LambrechtsStanley2008a}) whose total cofiber computes the cohomology of $\Conf_{k}(M)$ as a graded vector space.
Given $\gamma \in \Gamma$, we define the chain complex
\begin{equation}
  C_{\gamma} \coloneqq (K^{\otimes \pi_{0}(\gamma)})^{\vee}.
\end{equation}
Note in particular that $C_{\varnothing} = (K^{\otimes k})^{\vee}$.
If $\gamma' \supset \gamma$, then the map $C_{\gamma' \supset \gamma} : C_{\gamma'} \to C_{\gamma}$ is induced by the dual of the multiplication of $K$.

Recall that the total cofiber of $C_{\bullet}$ is given by the chain complex~\cite[Definition~7.2]{LambrechtsStanley2008a}:
\begin{equation}
  \TotCof C_{\bullet} \coloneqq \biggl( \bigoplus_{\gamma \in \Gamma} C_{\gamma} \cdot y_{\gamma}, \; D \biggl),
\end{equation}
where $y_{\gamma}$ is some variable of degree $- \# E_{\gamma}$, $\deg(x \cdot y_{\gamma}) = \deg(x) + \deg(y_{\gamma})$, and
\begin{equation}
  D(x \cdot y_{\gamma}) = \pm (dx) \cdot y_{\gamma} + \sum_{e \in E_{\gamma}} \pm (C_{\gamma' \supset \gamma} x) \cdot y_{\gamma - e}.
\end{equation}

\begin{proposition}
  The total cofiber of $C_{\bullet}$ computes the cohomology of $\Conf_{k}(M)$ as a graded vector space, up to suspension by $nk$.
\end{proposition}
\begin{proof}
  The proof is almost identical to the proof of~\cite[Theorem~9.2]{LambrechtsStanley2008a}.
  First define an auxiliary cubical diagram $C'_{\bullet}$ just like $C_{\bullet}$, except for $C_{\varnothing}$ which we set equal to $(\hoker \beta_{k})^{\vee}$.
  The maps $C'_{\gamma \supset \varnothing}$ are induced by the inclusions $K^{\otimes \pi_{0}(\gamma)} \to B^{\otimes \pi_{0}(\gamma)}$, which go through the definition of $\hoker \beta_{k}$.
  By Poincaré--Lefschetz duality, $(\hoker \beta_{k})^{\vee}$ is a model for $\Delta_{(k)} = \bigcup_{\gamma \in \Gamma} \nabla(\gamma)$, each $C_{\gamma}$ is a model for $\nabla(\gamma)$, and the maps between the $C'_{\gamma}$ are models for the inclusions by Lemma~\ref{lem.square}.
  We thus obtain that $\TotCof C'_{\bullet}$ is acyclic by~\cite[Proposition~9.1]{LambrechtsStanley2008a} and the homotopy invariance of total cofibers.

  The morphism $C'_{\gamma} \to C_{\gamma}$ is given by the identity if $\gamma \neq \varnothing$, and it is given by $\bar{\xi}_{k}^{\vee}$ if $\gamma = \varnothing$.
  This yields a morphism of cubical diagrams $C'_{\bullet} \to C_{\bullet}$.
  Define $C''_{\bullet}$ to be the object-wise mapping cone of $C'_{\bullet} \to C_{\bullet}$, so that there is a short exact sequence
  \begin{equation}
    0 \to C'_{\bullet} \to C_{\bullet} \to C''_{\bullet} \to 0.
  \end{equation}
  For $\gamma \neq \varnothing$, the map $C'_{\gamma} \to C_{\gamma}$ is the identity, hence $C''_{\gamma}$ is acyclic.
  It follows that $\TotCof C''_{\bullet}$ is quasi-isomorphic to the cone of $\bar{\xi}_{k}^{\vee} : C'_{\varnothing} \to C_{\varnothing}$, which computes the cohomology of $\Conf_{k}(M)$ by Corollary~\ref{cor.cone-cohom}.

  There is a long exact sequence between the homologies of the total cofibers $C'_{\bullet}$, $C_{\bullet}$, and $C''_{\bullet}$: total cofibers commute with mapping cones up to homotopy, because both are types of homotopy colimits.
  The proposition then follows from the fact that $H_{*}(\TotCof C'_{\bullet}) = 0$.
\end{proof}

\begin{theorem}
  \label{thm.ga-same-cohom}
  Let $M$ be a simply connected manifold with simply connected boundary and which admits a Poincaré--Lefschetz duality model.
  Let $P$ be the model of $M$ obtained from this PLD model (see Section~\ref{sec.pretty-nice}).
  Then there is an isomorphism of graded vector spaces between $H^{*}(\GG{P}(k); \Q)$ and $H^{*}(\Conf_{k}(M); \Q)$.
\end{theorem}
\begin{proof}
  There is in fact an isomorphism of dg-modules
  \begin{equation}
    \GG{P}(k) \cong (\TotCof C_{\bullet})[-nk].
  \end{equation}
  For $k = 1$, this is the Poincaré--Lefschetz isomorphism $\theta : P \to K^{\vee}[-n]$, and more generally for $k \geq 2$ it is induced by this isomorphism.
  We use the crucial fact that the multiplication of $K$ is dual, under this isomorphism, to the map $P \to P^{\otimes 2}$ defined by $x \mapsto (x \otimes 1) \Delta_{P} = (1 \otimes x) \Delta_{P}$ (which is true by definition of $\Delta_{P}$ in our setting).
  The proof is then similar to the proof of~\cite[Lemma~5.2]{Idrissi2018b}, and is carried out by comparing the standard bases of the components of the Lie operad and its dual cooperad.
\end{proof}

In other words, $\GG{P}(k)$ is a dg-module model of $\Conf_{k}(M)$.
Unfortunately, in general if $\partial M \neq \varnothing$ then $\GG{P}(k)$ is \emph{not} an actual model of $\Conf_{k}(M)$: the algebra structure is not the correct one.
This is not surprising: $\GG{P}(k)$ only depends on the homotopy type of $M$, whereas it is known that the homotopy type of $\Conf_{k}(M)$ should depend on the homotopy type of the map $\partial M \to M$.
\subsection{The perturbed model}
\label{sec.perturbed-model}

We define now a ``perturbed'' version of $\GG{P}(V)$, which will be the actual model for $\Conf_{k}(M)$, and we will prove that it is isomorphic to $\GG{P}(V)$ as a dg-module.

Recall the construction of the element $\sigma_{B} \in B \otimes B_{\partial}$ from Section~\ref{sec.a-data}, defined using a Section~of $B \to B_{\partial}$ and such that $d\sigma_{B} = \Delta_{B,B_{\partial}}$.
Let\index{sigmaP@$\sigma_{P}$}
\begin{equation}
  \sigma_{P} = (\pi \otimes 1)(\sigma_{B}) \coloneqq \sum_{i=1}^{N} \sigma'_{i} \otimes \sigma''_{i} \in P \otimes B_{\partial}.
\end{equation}

We define the perturbed model to be:\index{GtP@$\GGt{P}$}
\begin{equation}
  \label{eq.perturbed-model}
  \GGt{P}(V) \coloneqq\bigl( P^{\otimes V} \otimes S(\tilde\omega_{vv'})_{v,v' \in V} / J, d \tilde\omega_{vv'} = (\iota_{v} \cdot \iota_{v'})(\Delta_{P}) \bigr)
\end{equation}
where the ideal of relations is generated by $\tilde\omega_{vv'}^{2} = 0$ and, for all $b \in B$ and all subsets $T \subset \underline{k}$ of cardinality $\#T \geq 2$:
\begin{equation}\label{eq:bigrel}
  \sum_{v \in T} \pm \bigl( \iota_{v}(\pi(b)) \cdot \prod_{v \neq v' \in T} \tilde\omega_{vv'} \bigr) + \sum_{i_{1}, \dots, i_{k}} \pm \varepsilon_{\partial} \bigl( \rho(b) \prod_{v \in T} \sigma''_{i_{v}} \bigr) \prod_{v \in T} \iota_{v}(\sigma'_{i_{v}}).
\end{equation}
Note that these relations are precisely the ones obtained by considering the differential of a graph in $\mGraphs_{A}(k)$ with exactly one internal vertex.

The relations for $\# T = 2$ are perturbations of the symmetry relation of $\GG{P}(V)$, in the sense that it is the sum of the standard relation and terms which have a strictly lower number of generators $\tilde\omega_{vv'}$):
\begin{equation}
  \iota_{1}(b)\tilde\omega_{12} - (-1)^{n} \iota_{2}(b) \tilde\omega_{21} + \sum_{i,j} \pm \varepsilon_{\partial}(\rho(b) \sigma''_{i} \sigma''_{j}) \sigma'_{i} \otimes \sigma'_{j} = 0 \in \GGt{P}(\underline{2}).
\end{equation}
Note in particular that if we take $b=1$, we see that $\tilde\omega_{vv'}$ is not necessarily equal to $\pm \tilde\omega_{v'v}$.
Similarly, the relations for $\# T = 3$ are perturbations of the Arnold relations in $\enV(V)$.

\begin{remark}\label{rmk:T-2-3}
  The relations for $\# T \ge 4$ can actually be obtained from the relations for $\# T \in \{ 2, 3 \}$, which simplifies the presentation of Equation~\eqref{eq.perturbed-model} (but makes the relationship with $\mGraphs_{A}(k)$ less immediate).
  This can be proved ``by hand'' (but tediously) using the properties of the section $\sigma_{P}$ and its relationship with the evaluation $\varepsilon_{\partial}$.
  An alternative easy way to see it is to consider a graph in $\mGraphs_{A}(k)$ with exactly two connected internal vertices, apply the differential twice (thus getting zero) and projecting the result onto $\GGt{P}(k)$.
  Precisely one term of $d^2$ will be the desired relation (the one starting by contraction of the connecting edges) while the other ones can be expressed using smaller relations.
  For example, let $b \in B$ and let $R_{T}(b)$ be the relation of Equation~\eqref{eq:bigrel} for some $T$.
  If we apply this procedure to the following graph (for $b \in B$):
  \begin{center}
    \begin{tikzpicture}
      \node[ext] at (0,0) (e1) {1};
      \node[ext] at (1,0) (e2) {2};
      \node[ext] at (2,0) (e3) {3};
      \node[ext] at (3,0) (e4) {4};
      \node[int, label={$b$}] at (1,1) (i1) {};
      \node[int] at (2,1) (i2) {};
      \draw
      (i1) edge (e1) edge (e2) edge (i2)
      (i2) edge (e3) edge (e4);
    \end{tikzpicture}
    ,
  \end{center}
  we find that the relation $R_{\{1,2,3,4\}}(b)$ is equivalent up to signs to:
  \begin{multline*}
    \iota_{1}(\sigma'_{1}) \iota_{2}(\sigma'_{2}) R_{\{3,4\}}(\sigma'_{3}) \varepsilon_{\partial}(b \sigma''_{1} \sigma''_{2} \sigma''_{3}) + R_{\{1,2\}}(b \sigma'_{3}) \varepsilon_{\partial}(\sigma''_{1} \sigma''_{2} \sigma''_{3}) \iota_{3}(\sigma'_{1}) \iota_{4}(\sigma'_{2}) \\
    + \iota_{1}(b) \tilde{\omega}_{12} R_{\{1,3,4\}}(1) + \iota_{2}(b) \tilde{\omega}_{21} R_{\{2,3,4\}}(1) + R_{\{1,2,3\}}(b) \tilde{\omega}_{34} + R_{\{1,2,4\}}(b) \tilde{\omega}_{43}
  \end{multline*}
\end{remark}

\begin{example}
  Let us consider $M = S^{1} \times [0,1]$ (even though it doesn't satisfy our assumptions about connectivity).
  We can find a PLD model for $M$ where:
  \begin{itemize}
    \item $B_{\partial} = H^{*}(\partial M)$ is four-dimensional, generated by $1$, $t$, $d\varphi$, and $t d\varphi$ with $t$ being idempotent;
    \item $P = H^{*}(M) = H^{*}(S^{1})$ is two-dimensional, generated by $1$ and $d\varphi$;
    \item $K = H^{*}(M, \partial M)$ is two-dimensional, generated by $dt$ and $dt \wedge d\varphi$.
  \end{itemize}

  Then one of the nontrivial relations in $\GGt{P}(2)$ is given by $(d\varphi \otimes 1) \tilde\omega_{21} + (1 \otimes d\varphi) \tilde\omega_{12} + (d\varphi \otimes d\varphi) = 0$.
  To obtain some intuition about this relation, consider that $\Conf_{2}(M) \simeq \Conf_{2}(\R^{2} - \{0\})$ is homotopy equivalent to $\Conf_{3}(\R^{2})$ (by the Fadell--Neuwirth fibration, which has a contractible base in this case).
  The element $d\varphi$ correspond to the two points are rotating around the origin.
  Thus we can identify $\tilde\omega_{12}$ with $\omega_{12} \in H^{*}(\Conf_{3}(\R^{2}))$, $d\varphi \otimes 1$ with $\omega_{13}$, and $1 \otimes d\varphi$ with $\omega_{23}$.
  The perturbed relation in $\GGt{A}(2)$ is then nothing but the usual Arnold relation in $\enV[2](3) = H^{*}(\Conf_{3}(\R^{2}))$.
\end{example}

\begin{proposition}
  There is an isomorphism of dg-modules between $\GG{P}(V)$ and $\GGt{P}(V)$.
\end{proposition}

\begin{proof}
  Let us fix $V = \{ 1, \dots, k \}$ for some $k \geq 0$.
  Consider the standard basis of $\enV(k)$ given by monomials of the type:
  \begin{equation}
    \omega_{i_{1}j_{1}} \dots \omega_{i_{r}j_{r}},
  \end{equation}
  with $1 \leq i_{1} < \dots < i_{r} \leq r$ and $i_{l} < j_{l}$ for all $l$.
  By choosing some basis $\{ x_{1}, \dots, x_{m} \}$ of $P$, we obtain a basis of $\GG{P}(k)$ by labeling the last element of each connected component of such a monomial by some $x_{i}$.

  We claim that if we replace all the $\omega_{ij}$ by $\tilde\omega_{ij}$ in this basis, then we obtain a basis of $\GGt{P}(V)$.
  Using the perturbed Arnold relations, it's clear that any element of $\GGt{P}(V)$ can be written as a linear combination of these elements.
  Moreover, using the same argument that proves that there is no nontrivial relation between the elements of the standard basis of $\enV(k)$, we can prove that there is non nontrivial relation between the elements of our claimed basis.

  There is thus a linear isomorphism $\GG{P}(V) \to \GGt{P}(V)$ which is defined on the basis by replacing all the $\omega_{ij}$ by $\tilde\omega_{ij}$.
  It's then clear that this map preserves the internal differential of $P$ and the part of the differential which splits an $\omega_{i_{l}j_{l}}$ (which can be written down explicitly in the basis: it merely splits a connected component into two).
\end{proof}

\begin{corollary}
  There is an isomorphism of graded vector spaces $H^{*}(\Conf_{k}(M)) \cong H^{*}(\GGt{P}(k))$.
\end{corollary}
\begin{proof}
  This follows immediately from the previous proposition and Theorem~\ref{thm.ga-same-cohom}
\end{proof}

\begin{remark}
  It is often the case that $\GG{P}$ and $\GGt{P}$ are actually equal.
  For example, if $M$ is obtained by removing a point from a closed manifold, then $B_{\partial} = H^{*}(S^{n-1}) = S(v)/(v^{2})$, and $\sigma_{P} = 1 \otimes v$.
  Then in all the ``corrective terms'' in the definition of $\GGt{P}(V)$, $v$ appears at least twice, hence the summand vanishes.
\end{remark}

In particular:
\begin{example}
  \label{exa.model-dn-trois}
  If $M = D^{n}$ and we use the PLD model of Example~\ref{exa.model-dn}, then $\GGt{P} = \GG{P}$ (see Example~\ref{exa.model-dn-deux}).
\end{example}

\subsection{The connection between $\SGraphs_{A,A_{\partial}}$ and $\mGraphs_{A}$}
\label{sec.conn-betw-sgraphs}

We now compare our two graphical models, before going back to the comparison with our small model.
We first deal with the ``local'' case, i.e., $\SGraphs_{n}$ and $\Graphs_{n}$.

Let us apply the constructions of Section~\ref{sec:locMC2} to a trivial bundle over a point: we then obtain the cooperad $\SGraphs_{n}$ of Section~\ref{sec.extens-swiss-cheese}.
The Maurer--Cartan element $z+z^{\partial} \in \GC_{n}^{\vee} \ltimes \SGC_{n}^{\vee}$ is then equal to $\mu + c$, where $\mu$ is the graph with two vertices from Equation~\eqref{eq.mu}, and $c$ is given by the Kontsevich coefficients of Equation~\eqref{eq.konts-coeff}.
By the arguments of Section~\ref{sec:locMC2}, $c$ is gauge equivalent to $z^{\partial}_{0} + 0 \cdot z^{\partial}_{1}$ (the Euler class of $\Ha_{n}$ vanishes).
For convenience, we will use the notation\index{c0@$c_{0}$}
\begin{equation}
  c_{0} \coloneqq z^{\partial}_{0},
\end{equation}
and we assume in this section that $c_{0}$ is used to construct $\SGraphs_{n}$ instead of $c$.
As usual, since the two elements are gauge equivalent, this produces a quasi-isomorphic Hopf cooperad (see Section~\ref{sec.functoriality-models}).

By considering graphs with no terrestrial external vertices, we obtain a $\Graphs_{n}$-comodule $\SGraphs_{n}(\varnothing, -)$.
This comodule is not isomorphic to $\Graphs_{n}$ seen as a comodule over itself, because its graphs can contain internal terrestrial vertices.
However, $\SFM_{n}(\varnothing, \{*\})$ is a point, and there is a weak equivalence of $\FM_{n}$-modules given by:
\begin{equation}
  \FM_{n}(U) \cong \SFM_{n}(\varnothing, \{*\}) \times \FM_{n}(U) \xrightarrow{\circ_{U}} \SFM_{n}(\varnothing, U).
\end{equation}

This is modeled by the following proposition.
Let us define the map\index{nu@$\nu$}
\begin{equation}
  \nu : \SGraphs_{n}(\varnothing, V) \to \Graphs_{n}(V)
\end{equation}
as follows.
Given some $\Gamma \in \SGraphs_{n}(\varnothing, V)$, if $\Gamma$ only has univalent terrestrial vertices, then $\nu(\Gamma)$ is the graph with these univalent vertices and their incident edges removed.
Otherwise, if $\Gamma$ has terrestrial vertices of valence greater than one, then $\nu(\Gamma) = 0$.
Note that $\Gamma$ cannot have isolated terrestrial vertices, as they are all internal.
See Figure~\ref{fig:sketch-nu} for a sketch of the map $\nu$.

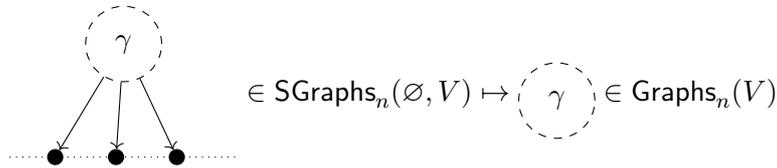
\begin{figure}[htbp]
  \centering
  \[
    \begin{tikzpicture}[baseline=.75cm]
      \draw[dotted] (0,0) -- (3,0);
      \node[draw, dashed, circle, minimum size=1cm] (g) at (1.5,1.5) {$\gamma$};
      \node[int] (i1) at (.6,0) {};
      \node[int] (i2) at (1.4,0) {};
      \node[int] (i3) at (2.2,0) {};
      \draw[->] (g) edge (i1) edge (i2) edge (i3);
    \end{tikzpicture}
    \in \SGraphs_{n}(\varnothing, V)
    \mapsto
    \tikz[baseline]{ \node[draw, dashed, circle, minimum size=1cm] {$\gamma$}; }
    \in \Graphs_{n}(V)
  \]
  \caption{Sketch of the map $\nu$}
  \label{fig:sketch-nu}
\end{figure}

\begin{proposition}
  \label{prop.qiso-sc}
  The maps $\nu : \SGraphs_{n}(\varnothing, V) \to \Graphs_{n}(V)$ define a quasi-isomorphism of Hopf right $\Graphs_{n}$-modules.
\end{proposition}
\begin{proof}
  For a given $V$, the map is the composite of two quasi-isomorphisms of CDGAs (hence it is itself a CDGA map and a quasi-isomorphism):
  \begin{itemize}
    \item The comodule structure map
          \begin{equation}
            \circ_{V}^{\vee} : \SGraphs_{n}(\varnothing, V) \to \SGraphs_{n}(\varnothing, \{*\}) \otimes \Graphs_{n}(V),
          \end{equation}
          which is a quasi-isomorphism because it models the weak equivalence $\FM_{n}(\varnothing, \{*\}) \times \FM_{n}(V) \to \SFM_{n}(\varnothing, V)$;
    \item The map $c_{0} : \SGraphs_{n}(\varnothing, \{*\}) \to \R$.
          This is a quasi-isomorphism because $\SGraphs_{n}(\varnothing, \{*\})$ is a model for $\SFM_{n}(\varnothing, \{*\})$, which is a point, and $c_{0}$ is a nontrivial cocycle.
  \end{itemize}

  The fact that $\SGraphs_{n}$ is a relative operad over $\Graphs_{n}$ also shows that this is a morphism of $\Graphs_{n}$-comodules.
\end{proof}

We now get back to the global case.
We still assume that the Maurer--Cartan element $z+z^{\partial}+Z \in \mathfrak{g}_{A,A_{\partial}}$ (see Equation~\eqref{eq.g-a-a-del}) used to define $\SGraphs_{A,A_{\partial}}$ is equal to the simpler one $\mu + z^{\partial}_{0} + E z^{\partial}_{1} + Z$, in order to have a simpler description for the map described below.
We also assume that the Maurer--Cartan element $w + W \in \mGC^{\vee}_{A_{\partial}} \ltimes \aGC^{\vee}_{A}$ used to define $\mGraphs_{A}$ (see Section~\ref{sec:oGCMC}) is also equal to the simple one $w_{0} + W_{0}$.

For each finite set of external vertices $V$, there is a map\index{nuA@$\nu_{A}$}
\begin{equation}
  \nu_{A} : \SGraphs_{A, A_{\partial}}(\varnothing, V) \to \mGraphs_{A}(V)
\end{equation}
defined as follows.
Given some graph $\Gamma \in \SGraphs_{A, A_{\partial}}(\varnothing, V)$, we remove all the edges from an aerial vertex to a terrestrial vertex and we multiply their endpoints by $\sigma_{A} \in A_{\partial} \otimes A$.
We then obtain a graph where all the terrestrial vertices are zero-valent, and we identify those with the real number given by the integral over $\partial M$ of their labels.
In this way, we obtain an element of $\mGraphs_{A}(V)$, and it is not hard to check that this produces a morphism of CDGAs.

\begin{proposition}
  The maps $\nu_{A} : \SGraphs_{A, A_{\partial}}(\varnothing, -) \qiso \mGraphs_{A}$ define a quasi-iso\-mor\-phism of Hopf right $\Graphs^{A}_{n}$-comodules.
\end{proposition}
\begin{proof}
  We use an argument similar to the one of~\cite[Section~4.2]{Idrissi2018b}.
  Filter both complexes by the number of edges minus the number of vertices.
  The only part of the differential which remains on the $\EE^{0}$ page is the contraction of aerial edges and the contraction of a subgraph into a terrestrial vertex induced by $z^{\partial}_{0}$, which is nonzero on trees only (see Proposition~\ref{prop:locvanishing2}), while the summand induced by $E z^{\partial}_{1}$ vanishes because it strictly decreases the filtration.

  It follows that both complexes split as a direct sum in terms of connected components, just like in~\cite[Lemma~4.21]{Idrissi2018b}.
  We can now reuse the same ``trick'' we used for the proof of~\cite[Lemma~4.28]{Idrissi2018b}, to show that in cohomology, there is only one label for each connected component.
  We thus reduce to the morphism $\nu : \SGraphs_{n}(\varnothing, V) \to \Graphs_{n}(V)$ from Proposition~\ref{prop.qiso-sc} tensored with the identity of $A$ for each connected component.
  Since that morphism is a quasi-isomorphism, we obtain that the induced morphism on the $\EE^{1}$ page is an isomorphism.
  It follows by standard spectral sequence arguments (the filtration is bounded below for a fixed $V$) that the morphism of the proposition is a quasi-isomorphism.
\end{proof}

\subsection{Comparison with the perturbed model and proof of Theorem~\ref{thm.A}}
\label{sec.comp-with-pert}

We consider the diagonal data $(B \xrightarrow{\lambda} B_{\partial}, \Delta_{B}, \sigma_{B})$ obtained from the Poincaré--Lefschetz duality model of $M$ (see Section~\ref{sec.pretty-nice}).
In particular we let $P$ be the quotient of $B$ by the orphans, which is also a model for $M$.
We would now like to connect $\mGraphs_{B}(V)$ with $\GGt{P}(V)$.

Recall that we can assume that the Maurer--Cartan element $w+W$ used to define $\mGraphs_{B}$ is the simple one $w_{0} + W_{0}$, see Section~\ref{sec:oGCMC}.
Note that since we are in the ``model-theoretic'' case, the pieces $w_{\tree}$, $W_{\tree}$ vanish as they encode higher Massey products -- they only appear in the ``canonical combinatorial'' case.

It is not necessarily the case that the two morphisms $\varepsilon, \varepsilon' : \operatorname{cone}(\rho) \to \R[-n+1]$, defined respectively by the composites $(\int_{M} g(-), \int_{\partial M} g_{\partial}(-))$ and $(\varepsilon_{B}f, \varepsilon_{B_{\partial}}f_{\partial})$, are equal.
Nevertheless, up to rescaling $\varepsilon$ and $\varepsilon_{\partial}$ (which induces an automorphism of $\mGraphs_{B}$), they induce the same map up to quasi-isomorphism, because $H^{n-1}(\cone(\rho))$ is one-dimensional.
Therefore we can further assume (using a gauge equivalence, see Section~\ref{sec.functoriality-models}) that the Maurer--Cartan element is compatible with the Poincaré--Lefschetz duality of $(B,B_{\partial})$.

Although the pair $(P, B_{\partial})$ is not an example of diagonal data (there is no direct map $P \to B_{\partial}$ in general, much less a surjective one), we can still mod out the graphs in $\mGraphs_{B}$ which contains vertices labeled by an element of $\ker(B \to P)$, which we check forms a dg-Hopf comodule bi-ideal.
We then obtain a graph complex $\mGraphs_{P}$, similar to $\mGraphs_{B}$ except that the vertices are labeled by elements of $P$.
The descriptions of the differential and the comodule structure are similar (they are induced by the structure maps on $\mGraphs_{B}$ and they are compatible with the quotient).
Moreover, the quotient map $\mGraphs_{B} \to \mGraphs_{P}$ is a quasi-isomorphism: we can filter by the total number of vertices both complexes, and thus obtain an isomorphism on the $E^{1}$ pages of the associated spectral sequences, following from the fact that $B \to P$ is a quasi-isomorphism.
This filtration is bounded in each degree, hence the initial map is a quasi-isomorphism by standard spectral sequences arguments.

We can now define a morphism $\Graphs_{P} \to \GGt{P}$ which maps all graphs with internal vertices to zero, and which sends an edge $e_{vv'}$ between external vertices to $\tilde{\omega}_{vv'}$.
\begin{proposition}
  \label{qiso-graphs-ga}
  This defines a quasi-isomorphism of right Hopf comodules:
  \[ (\mGraphs_{P}, \Graphs_{n}) \to (\GGt{P}, \enV). \]
\end{proposition}
\begin{proof}
  We can mimic the proof of~\cite[Proposition~4.13]{Idrissi2018b} in a straightforward way.
  Indeed, as soon as we filter by $\# \text{edges} - \# \text{vertices}$, the perturbed relations of $\GGt{P}$ become the usual relations of $\GG{P}$, and the rest of the argument is identical.
\end{proof}

Then the first case of Theorem~\ref{thm.A} follows:
\begin{theorem}
  \label{thm.model-ga}
  Let $M$ be a simply connected, smooth, compact manifold with boundary of dimension at least $7$.
  Let $P$ be the CDGA model built from any Poincaré--Lefschetz model.
  Then the symmetric collection of CDGAs $\GGt{P}$ is quasi-isomorphic to $\OmPA^{*}(\SFM_{M}(\varnothing,-))$.
  If moreover $M$ is framed, then the Hopf right comodule $(\GGt{P}, \enV)$ is quasi-isomorphic to $(\OmPA^{*}(\SFM_{M}(\varnothing, -)), \OmPA^{*}(\FM_{n}))$.
  \qed
\end{theorem}

Let us finally deal with simply connected manifolds with simply connected boundary of dimension $\dim M \in \{4,5,6\}$.
It is known that such a manifold is formal (see e.g.~\cite[Proposition~2.99]{FelixOpreaTanre2008}), thus as a model for $M$ we may simply take $H^{*}(M)$.
Note however that while $\partial M$ is formal too, the inclusion $\partial M \to M$ is not formal in general (take e.g., $M = S^{2} \times S^{2} \setminus D^{4}$), and the restriction map $H^{*}(M) \to H^{*}(\partial M)$ is not surjective anyway.

Consider the model $\mGraphs_{A}$ of $\Conf_{\bullet}(M)$ built in Section~\ref{sec:secondmodel}, where $A$ is some cofibrant model of $M$ part of some diagonal data.
As $M$ is formal, there exists a quasi-isomorphism $A \qiso H^{*}(M)$, necessarily surjective.
We define $\GGt{H^{*}(M)}(r)$ to be the quotient:
\begin{equation}
  \GGt{H^{*}(M)}(r) \coloneqq \mGraphs_{A}(r) / I,
\end{equation}
where $I$ is the CDGA ideal generated by graphs with internals vertices and graphs containing a label in the kernel of $A \to H^{*}(M)$.
This CDGA admits a description similar to Equation~\ref{eq.perturbed-model}: $\GGt{H^{*}(M)}(r)$ can be seen as the subcomplex of $\mGraphs_{A}(r)$ of graphs with no internal vertices, modded out by differentials of graphs with exactly one internal vertices (as the differential decreases the number of internal vertices by at most one).
Note that $\GGt{H^{*}(M)}$ forms an $\enV$-comodule, and that the quotient map
\begin{equation}
  (\mGraphs_{A}, \Graphs_{n}) \to (\GGt{H^{*}(M)}, \enV)
\end{equation}
defines a morphism of comodules.

\begin{theorem}
  \label{thm.ext-formal}
  Let $M$ be a simply connected manifold with simply connected boundary, with $\dim M \in \{4,5,6\}$.
  Let $\GGt{H^{*}(M)}$ be the collection of CDGAs defined above.
  Then each $\GGt{H^{*}(M)}(r)$ is a CDGA model of $\Conf_{r}(M)$.
  If moreover $M$ is framed, then the Hopf right comodule $(\GGt{H^{*}(M)}, \enV)$ is a quasi-isomorphic to $(\OmPA^{*}(\SFM_{M}(\varnothing, -)), \OmPA^{*}(\FM_{n}))$.
\end{theorem}

\begin{proof}
  Given that $(\mGraphs_{A}, \Graphs_{n})$ is known to be a model by Section~\ref{sec:secondmodel}, this is a simple matter of checking that the ideal $I$ is acyclic, using the same methods developed above for $\dim M \geq 7$.
\end{proof}

\begin{example}
  We can apply this to $M = D^{n}$, using the PLD model from Example~\ref{exa.model-dn}.
  Recall that in this case, $P = \R$, $\Delta_{A} = 0$, $\sigma_{A} = 0$, and $\GGt{P}$ is isomorphic to $\enV$ as a Hopf right comodule over itself (see Example~\ref{exa.model-dn-trois}).
  We then ``recover'' the already known fact that $\FM_{D^{n}}$ is (Hopf) formal as a right $\FM_{n}$-module -- though that ``proof'' is of course circular.

  The Hopf right $\Graphs_{n}$-comodule $\mGraphs_{P}$ is isomorphic to $\Graphs_{n}$ seen as a comodule over itself.
  The augmentation $\varepsilon_{B} : B \to \R[-n+1]$ yields a Maurer--Cartan element $\ze$ in the abelian $\hoLie_{n}$-algebra $\mGC_{P}^{\vee}$ (see Remark~\ref{rmk.maurer-cartan}), given in the dual basis by the graph with a single vertex labeled by $\vol_{n}$.
  The twisted $\hoLie_{n}$-algebra $\mGC_{P}^{\vee,\ze}$ is then isomorphic to a shift of $\GC^{\vee}_{n}$ from Section~\ref{sec.graph-complexes} -- the twisting ensures that dead ends are not contractible.
\end{example}

\appendix

\section{Cohomology of some graph complexes}
\label{sec:appxA}

In this section, we obtain bounds on the cohomology of the various graph complexes which appear in this paper.

\subsection{Directed and undirected graph complexes}\label{sec:dGC}

We call a \emph{stub} a univalent vertex with incoming edge:
\begin{equation}
  \begin{tikzpicture}
    \coordinate (v) at (0,0);
    \node[int] (w) at (1,0) {};
    \draw[-latex] (v) -- (w);
  \end{tikzpicture}
\end{equation}
Let us consider three versions of graph complexes $\GC_n$, $\dGC_n$, $\dGC_n'$ of undirected graphs, directed graphs, and directed graphs modulo graphs with stubs.
We work here with connected graphs, but allow all valences of vertices.

We recall:
\begin{lemma}[{\cite[Appendix K]{Willwacher2014}}]
  The map $\GC_n\to \dGC_n$ sending a graph to the sum of all graphs obtained by adding directions is a quasi-isomorphism of complexes.
\end{lemma}

One may also show the following.
\begin{lemma}\label{lem:dGC}
  The quotient map $\dGC_n\to \dGC_n'$ is a quasi-isomorphism of complexes up to one class of degree zero represented by a zero-valent vertex, i.e.,
  \begin{equation}
    H(\dGC_n)\oplus \K \cong H(\dGC_n') .
  \end{equation}
\end{lemma}
\begin{proof}
  The proof uses essentially the same techniques as the proof of~\cite[Proposition 3.4]{Willwacher2014}.
  Concretely, let us abbreviate $U=\dGC_n$ and $V=\dGC_n'$. These complexes decompose into a direct sum of subcomplexes
  \begin{align}
    U & = U_{lin} \oplus U_{loop} \oplus U_3
      &
    V & = V_{lin} \oplus V_{loop} \oplus V_3,
  \end{align}
  which consist of loop order zero graphs with at most bivalent vertices, loop order 1 graphs with only bivalent vertices and graphs with at least one trivalent vertex.
  The quotient map $U\to V$ clearly maps each subcomplex of $U$ to its counterpart in $V$.
  Also note that $U_{loop}=V_{loop}$.
  The spaces $U_{lin}$ and $V_{lin}$ are spanned by linear, i.e., strings of consecutive edges.
  We next claim that $H(U_{lin})=0$.
  To shows this, first write
  \begin{equation}
    U_{lin} = (U_{lin}')^{\Z_2},
  \end{equation}
  where $U_{lin}'$ is the space of linear graphs with labelled endpoints, and the $\Z_2$-action swaps the endpoints.
  Since we are working over a field of characteristic 0, It suffices to show that $H(U_{lin}')=0$.
  We say that the antenna at the first endpoint are all consecutive edges of like direction starting from that endpoint.
  For example, the antenna in the following graph has length 2.
  
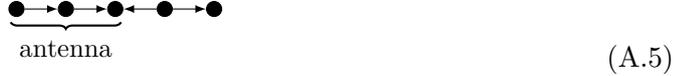
\begin{equation}
    \begin{tikzpicture}[scale=1.3]
      \node[int] (v1) at (0,0) {};
      \node[int] (v2) at (0.5,0) {};
      \node[int] (v3) at (1,0) {};
      \node[int] (v4) at (1.5,0) {};
      \node[int] (v5) at (2,0) {};
      \draw[-latex] (v1) edge (v2) (v2) edge (v3) (v4) edge (v5) edge (v3);
      \draw [
        thick,
        decoration={
            brace,
            mirror,
            raise=0.1cm
          },
        decorate
      ] (v1.south west) -- (v3.south east)
      node [pos=0.5,anchor=north,yshift=-0.2cm] {\small antenna};
    \end{tikzpicture}
  \end{equation}
  We filter by the number of non-antenna vertices and consider the associated spectral sequence.
  The associated graded complex splits into two subcomplexes of graphs such that (i) the antenna is the whole graph or (ii) the antenna is not the whole graph. In the subcomplex of such full-antenna graphs the differential makes the graph one longer if it has an odd number of vertices, and maps it to zero otherwise. The resulting complex is acyclic. In the other subcomplex the differential maps an antenna of odd length to a one longer antenna, and the graph to zero otherwise. The complex is also acyclic.
  Using that our filtration is complete and bonded above we are hence done and can conclude that $H(U_{lin})=0$.

  For $V_{lin}$ we proceed in exactly the same way. Just note that in the graphs the edges at the endpoints must now all be inwards oriented, so that now the subcomplex of full-antenna graphs is one-dimensional, spanned by the graph
  \begin{equation}
    \begin{tikzpicture}
      \node[int] at (0,0) {};
    \end{tikzpicture}\, .
  \end{equation}
  Hence the subcomplex of full-antenna graphs has now 1-dimensional cohomology, so that at the end $H(V)=\K$.

  We next claim that the quotient map
  \begin{equation}
    U_3\to V_3
  \end{equation}
  is a quasi-isomorphism.
  We call an outwards antenna a string of consecutive outwards pointing edges starting at a valence one vertex.
  
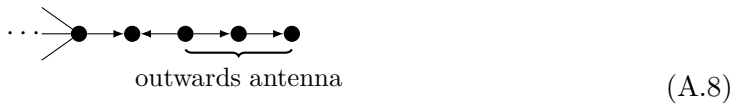
\begin{equation}
    \begin{tikzpicture}[scale=.7]
      \node[int] (v1) at (0,0) {};
      \node[int] (v2) at (1,0) {};
      \node[int] (v3) at (2,0) {};
      \node[int] (v4) at (3,0) {};
      \node[int] (v5) at (4,0) {};
      \node at (-1,0) {$\cdots$};
      \draw (v1) edge (-.7,0) edge (-.7,.5) edge (-.7,-.5);
      \draw[-latex] (v4) edge (v5) (v3) edge (v4) edge (v2) (v1) edge (v2);
      \draw [
        thick,
        decoration={
            brace,
            mirror,
            raise=0.1cm
          },
        decorate
      ] (v3.south) -- (v5.south)
      node [pos=0.5,anchor=north,yshift=-0.2cm] {\small outwards antenna};
    \end{tikzpicture}
  \end{equation}
  Of course, in $V_3$ all graphs that have such antennas are declared zero.
  We filter $U_3$ (and trivially $V_3$) by the number of vertices not in outwards antennas and consider the associated spectral sequence. The first differential makes some outwards antenna by one longer.
  By a slight variation of the argument from before (see also the proof of~\cite[Proposition 3.4]{Willwacher2014}) one shows that the subcomplex of graphs with inwards antennas is acyclic. Hence one concludes that the map $U_3\to V_3$ is indeed a quasi-isomorphism as desired.
\end{proof}

\subsection{Decorated graph complex}\label{sec:decoGC}
Consider a finite dimensional graded vector space $V$.
Consider a graph complex $\GC_{n,V}$ of undirected graphs, defined just as $\GC_n$, but with graphs whose vertices can be optionally decorated by zero, one or more elements of $V$ (formally, an element of $S(V)$).
\begin{equation}
  \begin{tikzpicture}[baseline=-.65ex]
    \node[int] (v1) at (0,0) {};
    \node[int, label=0:{$\scriptstyle v_1v_2$}] (v2) at (0:1) {};
    \node[int, label=90:{$\scriptstyle v_3$}] (v3) at (120:1) {};
    \node[int] (v4) at (-120:1) {};
    \draw (v1) edge (v2) edge (v3) edge (v4) (v4) edge (v2) edge (v3) (v2) edge (v3);
  \end{tikzpicture}
  \quad
  \text{with $v_j\in V$}
\end{equation}
The differential splits vertices, and distributes the decorations in all possible ways on the two vertices created (i.e., it uses the cocommutative coproduct on $S(V)$).

\begin{remark}
  This complex is not the same as the complex of $V$-hairy graphs which occur in the rational homotopy theory of string links~\cite{SonghafouoTsopmeneTurchin2015}.
  The difference is that here the differential may create vertices with one decoration and one incident edge -- in other words, ``hairs'' -- whereas in the complex from the theory of string links, the hairs and their colors is fixed.
\end{remark}

We have:
\begin{proposition}\label{prop:decoGC}
  The cohomology of $\GC_V$ satisfies:
  \begin{equation}
    H(\GC_{n,V}) = H(\GC_n) \oplus V,
  \end{equation}
  where the first summand is represented by graphs without any decoration, and the second by a graph consisting of a single $V$-decorated vertex.
\end{proposition}
\begin{proof}
  Clearly $\GC_{n,V}$ splits into subcomplexes according to the number of decorations by $V$.
  The piece with zero such decorations is isomorphic to $\GC_n$, and gives rise to the piece $H(\GC_n)$ in the statement of the proposition.
  We claim that the remaining piece of graphs with at least one decoration has cohomology $V$, meaning in particular that all subcomplexes with more than one decoration are acyclic.
  To this end it in fact suffices to consider the case that graphs are decorated by $k\geq 1$ identifiable decorations, say $e_1,\dots, e_n$. We say vertex $1$ is the one decorated by $e_1$.
  Then one may copy a trick of Lambrechts--Volić~\cite{LambrechtsVolic2014} and decompose our complex as
  \begin{equation}
    \begin{tikzcd}
      U_1 \ar[r, phantom, "\oplus"] & U_2 \ar[r, bend left, "f"] \ar[loop above] \ar[r, phantom, "\oplus"] & U_{r} \ar[loop right]
    \end{tikzcd}\, ,
  \end{equation}
  where $U_1$ is spanned by graphs where vertex one has valence 1 (i.e., nothing else than $e_1$ is connected to it, which means that by connectedness the graph has one vertex, no edge and $k=1$), $U_2$ is spanned by graphs where vertex $1$ has only one edge incident to it, and $U_{r}$ is spanned by the remainder of graphs. The arrows indicate the pieces of the differential between the subspaces.
  By picking a bounded filtration we arrive at a spectral sequence whose first differential is $f$.
  The map $f$ is surjective, the kernel is spanned by graphs such that the vertex at which the unique edge at vertex one ends is incident to exactly one edge (and no decoration) or two edges and no decoration.
  One easily check that that complex becomes acyclic on the next page, e.g., by filtering on the number of vertices not in a ``string'' starting at 1.
  We hence see that $U_1$ spans the cohomology as claimed.
\end{proof}

We will also consider the version of the above graph complex allowing non-connected graphs $\fGC_{n,V}$.
Clearly, this is just a free shifted symmetric algebra on the connected part:
\begin{corollary}\label{cor:decofGC}
  We have the following isomorphism:
  \begin{equation}
    H(\fGC_{n,V}) \cong S^+( H(\GC_n)[-n]\oplus V[-n])[n].
  \end{equation}
\end{corollary}

\subsection{Decorated \texorpdfstring{$M$}{M}-relative graph complex}
Let again $V$ be some vector space of additional decorations.
We want to consider a graph complex similar to $\GC_M$ ($M$ closed), but with additional decorations by $V$.

To this end, we assume we picked some finite dimensional dg vector space $A$ with a degree $n$ non-degenerate pairing.
We suppose furthermore that $A = \K \oplus \overline{A}$ has a ``unit'' (see \cite{CamposWillwacher2016} for details).
We may build the graph complex $\GC_{A}$,  whose vertices may be $A$-decorated, and whose differential will have a piece that denoted by $\nabla$ connecting two vertices, using the pairing on $A$.

Moreover, let us suppose that $Z$ is a Maurer--Cartan element of $\GC_A$.
If we twist the Lie algebra $\GC_A$ by $Z$, the differential of a graph will have a new piece ``gluing'' (the summands of) $Z$ to the graph along decorations via the pairing in $A$.
In the discussion that follows one can assume that $\GC_A$ is twisted by a non-trivial Maurer--Cartan element.

Define the graph complex $\GC_{A,V}$ to be composed of $A$-decorated graphs as before, but additionally vertices may again be decorated by elements of $V$.
Bear in mind that decorations in $A$ and $V$ are not on the same footing. For instance, there is a piece in the differential creating a $V$-decorated stub, but non creating an $A$-decorated stub.

Note that $\GC_{A}$ is a dg Lie subalgebra of $\GC_{A,V}$ trivially.
Furthermore, $A\otimes V$ is a subcomplex, understood as graphs with one vertex.
In the discussion that follows one can also assume $\GC_{A,V}$ to be twisted by $Z\in \GC_A\subset \GC_{A,V}$.
We have the following:
\begin{proposition}\label{prop:GCAV}
  \begin{equation}
    H(\GC_{A,V}) = H(\GC_{A}) \oplus H(A) \otimes V.
  \end{equation}
\end{proposition}
\begin{proof}
  The complex $\GC_{A,V}$ again splits into subcomplexes according to the number of decorations in $V$.
  The subcomplex with zero decorations is isomorphic to $\GC_{A}$ and hence produces the first term in the statement of the Proposition.
  To compute the cohomology of the remaining complex, we use the complete filtration by the number of loops.
  The first differential in the associated spectral sequence will not see the part joining two decorations into an edge.
  It will also not see the non-tree part of $Z$.

  Filtering again by the number of vertices and taking another spectral sequence we can next replace $A$ by $H(A)$ and restrict to the piece of the differential creating exactly one vertex, by splitting an edge, or joining the one-vertex piece of $Z$.
  We can now run an argument similar to that of the proof of Proposition \ref{prop:decoGC}.
  First we may assume that the decorations in $V$ are, say $e_1, \dots, e_k$, each occurring exactly once.
  We call the vertex decorated by $e_1$ the first vertex and decompose our complex as
  \begin{equation}
    \begin{tikzcd}
      U_1 \ar[r, phantom, "\oplus"] & U_2 \ar[bend left, r, "f"] \ar[loop above] \ar[r, phantom, "\oplus"] & U_{r} \ar[loop right]
    \end{tikzcd}\, ,
  \end{equation}
  where $U_1$ consists of graphs with a single vertex decorated by $e_1$ and an element of $H(A)$ (including 1),
  $U_2$ consists of graphs for which the first vertex has valence 2 and precisely one incident edge and $u_r$ consists of all remaining graphs.
  Again the map $f$ is surjective, a one sided inverse being given by contracting the edge. The kernel consists of graphs for which the first vertex has precisely one incident edge, and the vertex it connects to has valence at most 2, with no $V$-decorations.
  Again this complex can easily seen to be acyclic.
\end{proof}

\newcommand{\hfGC}{\widehat{\fGC}}

We also consider a non-connected version $\fGC_{A,V}$ of the graph complex $\GC_{A,V}$, i.e., we now allow also disconnected diagrams.
Note that (-in contrast to the previous section-) this complex is not merely the symmetric product of the connected graph complex, because the term $\nabla$ of the differential can connect two different connected components.

For technical reasons we will also temporarily consider a version $\fGC_{A,V}'=\K\oplus \fGC_{A,V}[-n]$, where we shift the degree and add an element representing the empty graph, so that, as graded vector spaces
\begin{equation}
  \fGC_{A,V}' = S(\GC_{A,V}).
\end{equation}
Note also that this notation allows us to write the differential on $\fGC_{A,V}'$ as
\beq{equ:dformulagcp}
\delta = e^{-Z} d_0 e^Z,
\eeq
where
\beq{equ:d0pieces}
d_0 = d_{split}+\nabla \quad \text{ with $\nabla=\sum_i \frac {\p}{\p a_i} \frac{\p }{\p a_i^*} $ } + d_A
\eeq
is the untwisted differential.
We claim that we have a map of complexes
\begin{equation}
  F : S(V\otimes A)\otimes \fGC_{A}' \to \fGC_{A,V}'.
\end{equation}
To define this map $F$ we consider elements $v\otimes a\in V\otimes A$ as derivations of $\fGC_{A,V}'$ which remove one decoration dual to $a$ and replace it with a decoration $v$, formally
\beq{equ:appder}
v \frac{\p }{\p a^*}.
\eeq
These derivations all commute, hence we can naturally associate to an element $f\in S(V\otimes A)$ an operator $D_f$ by composing the derivations. We then define the map $F$ above as
\beq{equ:Fdef}
F(f\otimes x) \coloneqq e^{-Z} D_f e^Z x.
\eeq
We check that this map commutes with the differentials, using \eqref{equ:dformulagcp}:
\begin{align*}
  \delta F(f\otimes x) & = \delta e^{-Z} D_f e^Z x
  = e^{-Z} d_0  D_f e^Z x
  \\
                       & = e^{-Z} (  D_f d_0 + D_{d_A f} ) e^Z x
  =
  e^{-Z}  D_f e^Z  e^{-Z} d_0e^Z x
  +
  e^{-Z} D_{d_A f}  e^Z x
  \\
                       & = F(f\otimes \delta x + d_Af \otimes x).
\end{align*}
Here we also used that all operators \eqref{equ:appder}, and hence also $D_f$, commute with the operations $d_{split}$ and $\nabla$ in \eqref{equ:d0pieces}.

We note that by the quantity
\begin{equation}
  \# \text{edges} - \# \text{vertices}
\end{equation}
we obtain a descending complete filtration on $\fGC_{A,V}'$.
We call this filtration the EC (Euler characteristic) filtration.
We note that the differential on the associated graded does not ``see'' the pieces $\nabla$ and the non-tree parts of $Z$.
Below we shall need the following result:
\begin{lemma}\label{lem:Fgr}
  The map \eqref{equ:Fdef} induces a quasi-isomorphism on the associated graded complexes with respect to the EC filtration.
\end{lemma}
\begin{proof}
  It follows from the same argument used in the proof of Proposition \ref{prop:GCAV} above.
\end{proof}

\begin{remark}
  It does not readily follow from the lemma that $F$ is a quasi-isomorphism, as the EC filtration is not exhaustive.
  However, the map \eqref{equ:Fdef} cannot increase the number of connected components.
  It is therefore defined on the non-complete (wrt.\ the number of connected components) analogs of the complexes, even though on these complexes the expression $e^Z$ is not defined.
  Lemma \ref{lem:Fgr} does imply that the map $F$ is a quasi-isomorphism on the non-complete complexes.
\end{remark}

\subsection{The cohomology of $\KGC_{n}$}
\label{sec.cohom-kgc-n}

Let $n \geq 3$.
Recall the graph complex $\KGC_{n} \coloneqq \SGC_{n}^{\vee}$ from Equation~\eqref{eq.kgc} and the Maurer--Cartan element $z_0^\p \in \MC(\KGC_{n})$ from Proposition~\ref{prop:locvanishing2}.
We want to compute here the cohomology of
\begin{equation}
  \KGC_{n}^{z_0^\p} \coloneqq (\KGC_n, \delta +[z_0^\p,-])
\end{equation}

First, note that one has a right action
\[
  \bullet :  (\KGC_n, \delta) \otimes (\GC_n,\delta) \to (\KGC_n,\delta)
\]
by Lie algebra derivations, through the insertion of a graph at aerial vertices (cf. also~\cite{Willwacher2016}).
It follows that for the (or in fact any) MC element $z_0^\p\in (\KGC_n, \delta)$ we obtain a map of complexes
\beq{equ:KGCtotcx}
z_0^\p \bullet - : \GC_n[-1] \to (\KGC_n,\delta+[z_0^\p,-])
\eeq
Concretely, to see that this map respects the differentials, one just acts with $\bullet \gamma$ (for some $\gamma\in \GC_n$) on the MC equation and uses the derivation property to obtain:
\begin{equation}
  0 = (\delta z_0^\p +[z_0^\p,z_0^\p]/2) \bullet \gamma
  =
  \delta (z_0^\p\bullet \gamma) + z_0^\p\bullet \delta \gamma
  +
  [z_0^\p, z_0^\p\bullet \gamma].
\end{equation}

Next, we define a map of complexes
\begin{equation}
  \pi : (\KGC_n,\delta+[z_0^\p,-])  \to (\GC_{n-1},\delta)
\end{equation}
as follows:
\begin{itemize}
  \item We say that a wedge is an aerial vertex connected to two terrestrial vertices and no other vertices:
        
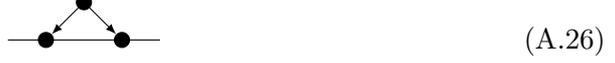
\begin{equation}
          \begin{tikzpicture}
            \draw (-1,0) -- (1,0);
            \node[int] (v) at (0,.5) {};
            \node[int] (v1) at (-.5,0) {};
            \node[int] (v2) at (.5,0) {};
            \draw[-latex] (v) edge (v1) edge (v2);
          \end{tikzpicture}
        \end{equation}
  \item If $\Gamma\in \KGC_n$ contains any aerial vertices which are not wedges, then we set $\pi(\Gamma)=0$.
  \item If $\Gamma\in \KGC_n$ contains only wedges, then we set $\pi(\Gamma)$ to the graph in $\GC_{n-1}$ with the vertices the same as the terrestrial vertices of $\Gamma$, and with every wedge in $\Gamma$ replaced by an edge in $\pi(\Gamma)$.
\end{itemize}

The map $\pi$ respects the differentials: the only terms in the differential that are not mapped to zero by $\pi$ stem from the bracket with the wedge piece in $z_0^\p$, and those terms precisely reproduce the differential on $\GC_{n-1}$.

Finally, for technical reasons, let us also consider the quasi-isomorphic subcomplex
\begin{equation}
  \GC_n^{2e} \xhookrightarrow{\sim} \GC_n
\end{equation}
spanned by graphs with at least two edges.
Concretely, the subcomplex $\GC_n^{2e}\subset \GC_n$ has codimension 2, with the two ``missing'' dimensions spanned by the graphs
$
  \begin{tikzpicture}[baseline=-.65ex]
    \node[int] at (0,0) {};
  \end{tikzpicture}$
and
$
  \begin{tikzpicture}[baseline=-.65ex]
    \node[int] (v) at (0,0) {};
    \node[int] (w) at (0.5,0) {};
    \draw (v) edge (w);
  \end{tikzpicture}
$.

The subcomplex is chosen so that the composition of maps
\beq{equ:KGCcpx}
\GC_n^{2e}[-1] \xrightarrow{z_0^\p\bullet -}  \KGC_n \xrightarrow{\pi} \GC_{n-1}^{\geq 1}
\eeq
is zero, as is easily verified.

The main result is now the following:
\begin{proposition}\label{prop:totacyclic}
  The total complex of \eqref{equ:KGCcpx} is acyclic.
\end{proposition}
\begin{proof}
  Let $T$ be the total complex (i.e., the mapping cone) of $\KGC_n\to \GC_{n-1}$.
  We will show equivalently that the induced map
  \begin{equation}
    H(\GC_n)\to H(T)
  \end{equation}
  is an isomorphism.
  We will filter $T$ by the total number of non-wedge aerial vertices and take the associated spectral sequence.
  Here we understand graphs in the piece $\GC_{n-1}$ as containing no such vertices.

  The first differential is composed of the piece of $[W,-]$ creating a wedge, and the map between $\KGC_n$ and $\GC_{n-1}$.
  We note that $\KGC_n$ with the first piece has the shape of a decorated graph complex in the sense of Section~\ref{sec:decoGC}, where we have (possibly disconnected) graphs of terrestrial vertices, the edges being the wedges, and the decorations the incident edges from the rest of the graph.
  Using this identification and replicating the proof of Proposition \ref{prop:decoGC} one sees that the cohomology of $\KGC_n$ with that differential is composed of one piece of shape $\GC_{n-1}$ and another one which is spanned by ``hairy wedgeless graphs'', i.e., graphs with only univalent terrestrial vertices and no wedges.
  Hence we find that $H(\gr T)$ is precisely the space of such hairy wedgeless graphs.
  We filter the next page in the spectral sequence again by the total number of vertices (aerial and terrestrial alike).
  The differential on the associated graded is $[I,-]$, which makes one hair into a stub.
  It is more or less obvious that the cohomology of that piece is the same as hairless graphs, modulo stubs.
  The only caveat is that by removing the graph which is a single wedge from our complex we retain an additional class
  \begin{equation}
    X =
    \begin{tikzpicture}[baseline=1cm]
      \draw(-1,0) -- (1,0);
      \node[int] (v) at (0,0) {};
      \node[int] (w) at (0,1.3) {};
      \node[int] (ww) at (0.8,.5) {};
      \draw[-latex] (w) edge (v) edge (ww);
    \end{tikzpicture}
  \end{equation}
  We hence find (on the next page of our inner spectral sequence) a complex which is isomorphic to $\dGC_n'$, cf. Section~\ref{sec:dGC}.
  Using Lemma \ref{lem:dGC} we hence see that our cohomology on the subsequent page is
  \begin{equation}
    H(\GC_n) \oplus \K U \oplus \K X,
  \end{equation}
  where
  \begin{equation}
    U =
    \begin{tikzpicture}[baseline, scale=.5]
      \draw(-1,0) -- (1,0);
      \node[int] (w) at (0,1) {};
    \end{tikzpicture}
  \end{equation}
  The classes in the piece $H(\GC_n)$ stem from the map $\GC_n\to \KGC_n$, hence have representatives closed in the full complex, and hence all further differentials in our spectral sequences will be zero on them.
  Finally, one easily checks that on the next page of the outer spectral sequence $U$ kills $X$, and hence the proposition follows.
\end{proof}

One can then consider the long exact sequence associated to \eqref{equ:KGCcpx} (see \cite[Lemma 5.6]{Willwacher2014})
\beq{equ:KGClongexact}
\cdots \to H(\GC_n)[-1]\to H(\KGC_n) \to H(\GC_{n-1}) \to H(\GC_n)[-2]\to \cdots
\eeq
Furthermore we see that all maps respect the (complete) grading by loop order, so that we may always consider subcomplexes of fixed loop order $g$.
First consider loop order 1.
Suppose $n$ is even. Then the 1-loop classes in $H(\GC_n)$ live in degrees $1-n+4j$, $j=0,1,2,\dots$.
If $n$ is odd they live in degrees $3-n+4j$, $j=0,1,2,\dots$.
This means that the 1-loop classes in $H(\GC_{n-1})$ live in degrees $4-n+4j$ is $n$ is even and $2-n+4j$ if $n$ is odd.
Hence in either case the live in degrees of opposite parity, so that the map $H(\GC_{n-1}) \to H(\GC_n)[-2]$ must necessarily be trivial in loop order $1$.

\begin{lemma}
  \label{lem:GCdegbounds}
  In loop order $g> 2$ the cohomology of $H(\GC_n)$ is concentrated in degrees
  \begin{equation}
    \{-g (n - 2)  ,\dots,  -g(n-3)-3 \}
  \end{equation}
\end{lemma}
\begin{proof}
  The upper bound is due to the fact the we may consider at least trivalent vertices, so that, in the extreme case one has
  \begin{align*}
    e & = 3v /2 & g & =e-v+1 & d = n(v-1) -(n-1)e,
  \end{align*}
  which one easily solves for the degree $d$ to yield the desired bound (see also~\cite[Section~9.4]{FresseTurchinWillwacher2017}).

  The lower bound comes from the quasi-isomorphic identification of $\GC_n$ with the complex of (Hopf-)biderivations of the homology operad of the little $n$-disks operad, see \cite{Willwacher2014}.
  The biderivation complex there has the given lower degree bound, which hence is also valid for $H(\GC_n)$.

  More concretely, recall from \cite[Section~5]{Willwacher2014} the identification between the homology of $\mathsf{Der}(e_n)=\prod_N \left( e_n(N) \otimes e_n(N) \otimes sgn^{\otimes n}    [(1-N)n] \right)^{\mathbb S_n} $ and $H(GC_n)$.
  Under this identification the genus is given by genus of the graph we obtain by drawing a line with $N$ vertices, a graph representing a Gerstenhaber word of arity $N$ above the line (i.e., multiple Lie trees whose leafs form a subset of the $N$ vertices) and a similar below the line.

  Then, on a graphical representation $\Gamma$ of an element of $\mathsf{Der}(e_n)$ and denoting by $c$  its the number of connected components, $e$ its number of edges, $v$ its number of vertices, and $g$ its genus we obtain

  \begin{align}
    \deg \Gamma & = (e-v)(1-n) -n(1-N)                                                     \\
                & =(e-v-N)(1-n) +N(1-n) -n(1-N)                                            \\
                & =(g-c)(1-n) +N-n                                                         \\
                & = -g(n-2) + \underbrace{(c-1)}_{\geq 0}\underbrace{(n-1)}_{>0} +(N-g-1).
  \end{align}

  The conclusion follows from noting that since the graph would have genus zero after deletion of the $N$ middle vertices, then $N\geq g+1$.
\end{proof}

We also note that in loop order $g=2$ there is cohomology in $\GC_n$ only for $n$ even, and only one class spanned by the theta-graph, in degree $n-3(n-1)=-2n+3$.

Overall we find that the cohomologies of $H(\GC_{n-1})$ and $H(\GC_n)[-2]$ occupy disjoint degree regions, and hence the map between them in the long exact sequence \eqref{equ:KGClongexact} has to be trivial.
Hence we conclude from this long exact sequence the following.
\begin{proposition}\label{prop.kgc-z-0}
  We have that
  \begin{equation}
    H(\KGC_n,\delta + [z_0^\p,-]) \cong H(\GC_n)[-1]\oplus H(\GC_{n-1}).
  \end{equation}
\end{proposition}

\subsection{The cohomology of \texorpdfstring{$\aGC_{A_{\partial}}$}{aGC\_\{A_d\}}}

We next consider the graph complex (dg Lie algebra) $(\aGC^{\vee}_{A_{\partial}},\delta+[w,-])$, where $w=w_0+\cdots$ is the Maurer--Cartan element constructed in Section~\ref{sec:oGCMC}.
We will be mostly interested in $(\aGC_{A_{\partial}}^{\vee},\delta+[w_0,-])$, of which the above complex is a deformation.

First note that we have an obvious map of complexes
\beq{equ:GCtooGC}
\GC^{\vee}_n[1-n]\to \aGC^\vee_{A_{\partial}}
\eeq
by simply mapping a diagram to itself.
We also have a map in the other direction
\begin{equation}
  (\aGC^{\vee}_{A_{\partial}}, \delta+[w_0,-])\to \GC^{\vee}_n[1-n]
\end{equation}
projecting onto diagrams without decoration.
To see that this is indeed a map of complexes, note that the twist piece of the differential does not have a piece adding two edges joined by an undecorated vertex, owed to the fact that $w_0$ does not have pieces with two $\alpha_i$ or $\beta_i$ decorations.
Mind however that the $\omega$-term in $w_0$ can create more than one edge, eventually owed to our conventions regarding decorations by $1$.

This shows that we have a direct sum decomposition
\begin{equation}
  (\aGC^{\vee}_{A_{\partial}}, \delta+[w_0,-]) \cong \GC^{\vee}_{n}[1-n] \oplus X
\end{equation}
where $X$ is the subcomplex spanned by graphs with at least one decoration in $\tilde H(N)$.
This subcomplex $X$ has also been considered previously. It is closely related to the hairy graph complexes with hair decoration in $\tilde H(N)$, considered for example in~\cite[Section~2]{SonghafouoTsopmeneTurchin2016}.

We will not try to make this link more precise, but just state some results.
Let
\begin{equation}
  \aGC_{A_{\partial}}^{\vee,\geq k} \subset \aGC_{A_{\partial}}^{\vee}
\end{equation}
be the subcomplex spanned by graphs all of whose vertices have valence $\geq k$, with $k=2,3$. (We leave it to the reader to verify that this subspace is indeed closed under the differential.)

\begin{proposition}\label{prop:oGCNI}
  The inclusion
  \begin{equation}
    \aGC_{A_{\partial}}^{\vee,\geq 2}\subset (\aGC_{A_{\partial}}, \delta+[w_0,-])
  \end{equation}
  is a quasi-isomorphism. The inclusion
  \begin{equation}
    \aGC_{A_{\partial}}^{\vee,\geq 3}\subset (\aGC_{A_{\partial}}, \delta+[w_0,-])
  \end{equation}
  is a quasi-isomorphism, up to the cohomology of the $1$-loop part of $\GC_n[1-n]$, living in the cohomology of the right-hand side.
\end{proposition}

\begin{proof}
  For the first statement one takes a spectral sequence from the filtration on the number of ``non-antenna vertices'', similar to the proof of~\cite{Willwacher2016}, or Lemma \ref{lem:dGC} above.
  For the second statement one similarly filters on the number of non-bivalent vertices.
\end{proof}

Finally we want to note that the complex $(\aGC_{A_{\partial}}^{\vee},\delta+[w,-])$ is naturally filtered by loop number.
A simple counting argument shows:
\begin{lemma}\label{lem:oGCdegcounting}
  Graphs in $\aGC_{A_{\partial}}^{\vee,\geq 3}$ of loop number $l$ have degree
  \begin{equation}
    \leq -(l-1)(n-3).
  \end{equation}
\end{lemma}
\begin{proof}
  For the following argument we ignore symmetries of graphs.
  Suppose we are given a graph of loop order $l$ of maximal degree.
  We may assume that all vertices are exactly trivalent, for otherwise splitting higher valent vertices would produce a graph of higher degree.
  In case there is a vertex with 3 decorations in $\tilde H(N)$, this is the whole graph, and $l=0$.
  The degree of such a graph is at most
  \begin{equation}
    n (\text{\# vertices}) - (n-1)(\text{\# edges})-(\text{degree of decorations})
    \leq n-0-3=n-3,
  \end{equation}
  in accordance with the formula.
  Next suppose there is a vertex with exactly 2 decorations.  Then performing the substitution
  \begin{equation}
    \begin{tikzpicture}[baseline=-.65ex]
      \node[int] (v) at (0,0) {};
      \node[int,label={$\scriptstyle ab$}] (w) at (0,0.5) {};
      \draw (v) edge (w) edge +(-.5,-.5)  edge +(.5,-.5);
    \end{tikzpicture}
    \mapsto
    \begin{tikzpicture}[baseline=-.65ex]
      \node[int,label={$\scriptstyle a$}] (v) at (0,0) {};
      \draw (v) edge +(-.5,-.5)  edge +(.5,-.5);
    \end{tikzpicture}
  \end{equation}
  at most increases the degree. We may hence assume that there are no vertices with exactly two decorations.
  Next suppose we have a vertex with exactly one decoration.
  Then the substitution
  \begin{equation}
    \begin{tikzpicture}[baseline=-.65ex]
      \node[int,label=$a$] (v) at (0,0) {};
      \draw (v) edge +(-.5,0) edge +(.5,0);
    \end{tikzpicture}
    \mapsto
    \begin{tikzpicture}[baseline=-.65ex]
      \draw (-.5,0) -- (.5,0);
    \end{tikzpicture}
  \end{equation}
  can at most increase the degree. Note that this substitution can be performed iff the two edges incident at the vertex on the left are distinct, i.e., if the graph is not equal to the following
  \begin{equation}
    \begin{tikzpicture}[baseline=-.65ex]
      \node[int,label=$a$] (v) at (0,0) {};
      \draw (v) edge[loop below] (v);
    \end{tikzpicture}\, .
  \end{equation}
  In this latter case $l=1$ and the graph has degree at most $n-(n-1)-1=0$, in accordance with the claim in the lemma.
  For higher $l\geq 2$ it hence suffices to consider trivalent graphs without any decorations. Such graphs have  degree
  \begin{equation}
    n(2l-2)-(n-1)(3l-3) = -(l-1)(n-3).
    \qedhere
  \end{equation}
\end{proof}

Note also that the loop order 1 part of $\GC_n[1-n]$ is concentrated in non-positive degrees.

Hence we obtain:
\begin{corollary}\label{cor:vanish-agc}
  If $\dim N\geq 2$, then the graded vector space $\mF^1\aGC_{A_{\partial}}^{\vee}$ does not contain elements of positive degrees.
  \qed
\end{corollary}

\subsection{The cohomology of \texorpdfstring{$\mGC_A$}{mGC\_A}}
We may discuss similarly the cohomology of the complex $\mGC_A$, or rather extend the discussion of the complex $\aGC_{A_\partial}$ from the preceding Section~to the semi-direct product dg Lie algebra
\begin{equation}
  \aGC_{A_{\partial}} \ltimes \mGC_A.
\end{equation}
Again we have the descending complete filtration $\mF^p$ by loop order. The associated graded space can again be related to (a twisted version of) the hairy graph complexes occurring for example in~\cite{SonghafouoTsopmeneTurchin2015}.
We have a Maurer--Cartan element $w+W=w_0+W_0 +(\cdots)$, cf. \eqref{equ:oGCMC1}, \eqref{equ:oGCMC2}.

We note that in $(\mGC_A, \delta + [w_0+W_0,-])$ we have classes represented by the loop graphs
\begin{equation}
  L_k
  =
  \begin{tikzpicture}[baseline=-.65ex]
    \node[int] (v1) at (0:1) {};
    \node[int] (v2) at (60:1) {};
    \node[int] (v3) at (120:1) {};
    \node[int] (v4) at (180:1) {};
    \node[int] (v5) at (240:1) {};
    \node (v6) at (300:1) {$\cdots$};
    \draw (v1) edge (v2) edge (v6) (v3) edge (v2) edge (v4) (v5) edge (v6) edge (v4);
  \end{tikzpicture}
  \quad \text{($k$ vertices)}
\end{equation}
with
\beq{equ:krange}
k =
\begin{cases}
  1,5,9,13,\dots & \text{for $n$ even} \\
  3,7,11,\dots   & \text{for $n$ odd}
\end{cases}\, .
\eeq
The $L_k$ have cohomological degree $k+1$.
It is easy to check that the $L_k$ are closed, reminding the reader that here we consider only the differential arising by twisting with the leading term $w_0+W_0$ of the MC element $w+W$.
\begin{lemma}\label{lem:oGCMloops}
  The $L_k$, for $k$ as in \eqref{equ:krange}, represent non-trivial cohomology classes in $(\mGC_A, \delta + [w_0+W_0,-])$.
\end{lemma}
\begin{proof}
  We filter the complex by loop order and consider the associated spectral sequence.
  We note that the associated graded differential leaves invariant not only the loop number, but also the number of decorations in $\tilde H(M)$.

  The cohomology in the associated graded in loop order 0 is a free cyclic Lie algebra generated by $\tilde H(M)$, represented by trivalent diagrams with leaves decorated by $\tilde H(M)$, modulo IHX relations.
  In particular all such diagrams with $v$ vertices must have $v+2$ decorations in $\tilde H(M)$.
  The differential on the next page (increasing the loop order by one) adds one edge, removing one or two decorations.
  So it is clear that only diagrams with at least one decoration (even per vertex) can be produced in this way. In particular all the $L_k$ must survive in the spectral sequence, and hence give rise to non-trivial cohomology classes.
\end{proof}

We will again consider bivalent and trivalent versions, in which graphs are required to contain only vertices of valence at least $k$, $k=2$ or $3$:
\begin{equation}
  \aGC_{A_{\partial}}^{\vee,\geq k} \ltimes \mGC_A^{\geq k}
  \subset
  (\aGC^{\vee}_{A_{\partial}} \ltimes \mGC_A^{\vee}, \delta + [w_0+W_0,-]).
\end{equation}
Note in particular that we consider the right-hand side as twisted with the leading piece of the MC element, which does not live in our subcomplexes.
\begin{proposition}\label{prop:incl-agc-mgc}
  The inclusion
  \begin{equation}
    \aGC_{A_{\partial}}^{\vee,\geq 2} \ltimes \mGC_A^{\vee,\geq 2}
    \subset
    (\aGC_{A_{\partial}}^{\vee} \ltimes \mGC_A^{\vee}, \delta + [w_0+W_0,-]).
  \end{equation}
  is a quasi-isomorphism.
  The inclusion
  \begin{equation}
    \aGC_{A_{\partial}}^{\vee,\geq 3} \ltimes \mGC_A^{\vee,\geq 3}
    \subset
    (\aGC_{A_{\partial}}^{\vee} \ltimes \mGC^{\vee}_A, \delta + [w_0+W_0,-]).
  \end{equation}
  is a quasi-isomorphism up to the 1-loop classes living in $H(\aGC_{A_{\partial}})$ as in Proposition \ref{prop:oGCNI}, and the 1-loop classes in $H(\mGC_A)$ as in Lemma \ref{lem:oGCMloops}.
  Both statements continue to hold for the associated graded complexes under the filtration by loop order.
\end{proposition}
\begin{proof}
  We first filter both sides by loop order and consider the associated spectral sequence.
  This removes the term from the differential that creates an edge by joining two $\gamma_j$-decorations.
  Afterwards the proof proceeds as that of Proposition \ref{prop:oGCNI}.
\end{proof}

We also have an analog of Lemma \ref{lem:oGCdegcounting}.

\begin{lemma}\label{lem:oGCdegcounting2}
  If $H^1(M)=0$ then graphs in $\mGC_{A}^{\vee,\geq 3}$ of loop order $l$ have degree
  \begin{equation}
    \leq -(l-1)(n-3)+1.
  \end{equation}
\end{lemma}
\begin{proof}
  The proof is identical to the one of Lemma \ref{lem:oGCdegcounting}, with the caveat that the formula for the degree of a graph has an extra ``$+1$'' in the present case, thus accounting for the different degree bound in this case.
\end{proof}

\begin{corollary}\label{cor:vanish-mgc}
  If $n=\dim M\geq 4$ and if $H^1(M)=0$, then the graded vector space $\mF^1 \mGC_{A}^{\vee,\geq 3}$ is concentrated in non-positive degrees.
\end{corollary}
\begin{proof}
  From Lemma \ref{lem:oGCdegcounting2} it follows immediately that $\mF^1 \mGC_{A}^{\vee,\geq 3}$ is concentrated in non-positive degrees, even without the condition on $H^1(M)$.
  In loop order one, the proof of Lemma \ref{lem:oGCdegcounting} shows that the graphs of the highest possible degree $1$ have all decorations in $H^1(M)$. Hence if $H^1(M)=0$, then these graphs are no longer present in the complex and the highest possible degree is (at most) 0.
\end{proof}

\begin{remark}\label{rmk:big-picture}
  The discussions in the present and the previous subsections provide a quite satisfying picture of the cohomology and the MC spaces of the graph complex $\aGC_{A_{\partial}}^{\vee} \ltimes \mGC_A^{\vee}$.
  Essentially, what we have shown is the following.
  The graph complex is quasi-isomorphic to a similar graph complex built from a diagram whose vertices are all at least trivalent.
  Under suitable conditions, this new graph complex does not have elements of positive degrees in positive loop orders.
  It follows that its Maurer--Cartan space is completely controlled by its tree part, which is easier to compute: as evidenced by Corollary~\ref{cor:ZtreedM}, it simply encodes the real homotopy type of the underlying manifold and its boundary.
  Furthermore, there is a spectral spectral sequence exhibiting those complexes as deformations of hairy graph complexes.
  These hairy graph complexes have been used and computed in various settings (see e.g.~\cite{SonghafouoTsopmeneTurchin2015}) and we can therefore use these results to obtain information on $\aGC_{A_{\partial}}^{\vee} \ltimes \mGC_{A}^{\vee}$.
\end{remark}

\subsection{The cohomology of \texorpdfstring{$\KGC_M$}{KGC\_M}}
Let us next consider the complex $\KGC_M \coloneqq \SGC_{A,A_{\partial}}^{\vee}$, or rather the semi-direct product $A \otimes \KGC_n \ltimes \KGC_M$.
We have a Maurer--Cartan element $Z+z^{\p}=Z^{\tree}+z_0^\p+z_{1}^{\partial}+\cdots$ as constructed in Section~\ref{sec.model-right-hopf}.

We will derive degree bounds on the cohomology of the complex $\KGC_M$ by considering various spectral sequences.
First, note that -- as with all complexes in this paper -- we have a descending complete filtration by loop order.
We consider the associated spectral sequence, whose first differential does not create loops.
Note that this ``removes'' the pieces of the differential coming from the higher loop part of the MC elements $Z$ and $z_0$, and also the differential $\nabla$ replacing two decorations by a new edge, either in the bulk or on the boundary.

Secondly, let us filter $\KGC_M$ again by the number of aerial vertices.
The differential on the associated graded has two terms.
The first term connects the tree piece of $Z$ without aerial vertices to our graph:
\begin{equation}
  \begin{tikzpicture}[baseline=.5cm]
    \draw[dotted] (-1,0) -- (2,0);
    \node[int, label={above left}:{$x_{1}$}] (t1) {};
    \node[int, label={above right}:{$x_{2}$}, right=of t1] (t2) {};
    \node[int, label={$y_{1}$}, above right=of t1] (a1) {};
    \node[int, label={$y_{2}$}, right=1cm of a1] (a2) {};
    \draw[->]
    (a1) edge (t1) edge (t2) edge (a2)
    (a2) edge (t1);
  \end{tikzpicture}
  \mapsto
  \sum_{(Z^{\tree})}
  \;
  \begin{tikzpicture}[baseline=.5cm]
    \draw[dotted] (-1,0) -- (5,0);
    \node[int, label={above left}:{$x_{1}$}] (t1) {};
    \node[int, label={above right}:{$x_{2}$}, right=of t1] (t2) {};
    \node[int, label={$y_{1}$}, above right=of t1] (a1) {};
    \node[int, label={$y_{2}$}, right=1cm of a1] (a2) {};
    \draw[->]
    (a1) edge (t1) edge (t2) edge (a2)
    (a2) edge (t1);
    \draw[dashed] (3,0) -- ++(0,1.5) -- ++(1.5,0);
    \node[int, right=1.5cm of a2] (i) {};
    \draw[<-]
    (a2) -- (i)
    (i) edge +(0:.6) edge +(30:.6) edge +(-30:.6);
    \node[below=.5cm of i] {$\gamma \in Z^{\tree}$};
  \end{tikzpicture}
  .
\end{equation}

The second term splits boundary vertices, creating one boundary edge, or joins two boundary decorations by $\tilde H(\p M)$.
We filter again by twice the number of boundary edges, plus three times the number of bulk edges.
On the associated graded complex we see only the edge splitting piece of the differential.
By Proposition~\ref{prop:GCAV} the cohomology of this complex can be identified with the space of graphs all of whose external vertices have exactly one incident edge and at most one decoration by $\tilde H(M)$, plus a space of ``boundary'' graphs without aerial vertices, quasi-isomorphic to $\GC_{\p M}$.
On the next page we see the following piece of the differential, which a decoration by  $d\beta_j$ upstairs and links it to a boundary vertex.
\begin{equation}
  \begin{tikzpicture}[baseline=-.65ex]
    \draw (0,0)--(1,0);
    \node[int,label=-90:{$\scriptstyle d\beta_j$}] (v) at (.5,.7) {};
    \draw (v) edge +(-.5,.5) edge +(0,.5)  edge +(.5,.5);
  \end{tikzpicture}
  \mapsto
  \begin{tikzpicture}[baseline=-.65ex]
    \draw (0,0)--(1,0);
    \node[int] (v) at (.5,.7) {};
    \node[int,label=-90:{$\scriptstyle \beta_j$}] (w) at (.5,0) {};
    \draw (v) edge +(-.5,.5) edge +(0,.5)  edge +(.5,.5);
    \draw[-latex] (v) to (w);
  \end{tikzpicture}
\end{equation}
The cohomology with respect to this differential can be identified with the space of hairy graphs, whose hairs are not decorated by $\beta_j$'s, and whose aerial vertices are not decorated by $d\beta_j$'s.
At this stage, our innermost spectral sequence abuts.

The differential on the following page of the next outer spectral sequence increases the number of aerial vertices by one.
To compute the cohomology of this second page we filter our complex yet again by the total number of vertices (aerial + terrestrial).
The differential on the first page of the resulting spectral sequence removes a terrestrial vertex, making it an aerial vertex. There is only one such piece in the differential contributed by the bracket with the leading term of $z_0^\p$, namely
\begin{equation}
  \begin{tikzpicture}[baseline=-.65ex]
    \draw (0,0)--(1,0);
    \node[int] at (.5,.7) {};
  \end{tikzpicture}
  .
\end{equation}
The bracket with this term lifts a terrestrial vertex decorated by $\alpha_j$ (or nothing) to an aerial vertex:
\begin{equation}
  \begin{tikzpicture}[baseline=-.65ex]
    \draw (0,0)--(1,0);
    \node[int,label=-90:{$\scriptstyle \alpha_j$}] (v) at (.5,0) {};
    \coordinate (w) at (.5,1.3);
    \draw[-latex]  (w) to (v);
  \end{tikzpicture}
  \mapsto
  \begin{tikzpicture}[baseline=-.65ex]
    \draw (0,0)--(1,0);
    \node[int,label=-90:{$\scriptstyle \alpha_j$}] (v) at (.5,.7) {};
    \coordinate (w) at (.5,1.3);
    \draw[-latex]  (w) to (v);
  \end{tikzpicture}\, .
\end{equation}
(It should be noted that there are further terms that do not create vertices, however they raise the loop order, and hence are not seen here.)
The cohomology of that differential can be identified with non-hairy graphs, modulo $\alpha_j$-decorated stubs, i.e., modulo vertices of the form
\begin{equation}
  \begin{tikzpicture}
    \draw (-.5,0)--(.5,0);
    \node[int,label=0:{$\scriptstyle \alpha_j$}] (v) at (.5,.7) {};
    \draw[-latex] (0,1) to (v);
  \end{tikzpicture}\, .
\end{equation}
including undecorated such vertices (i.e., decorations by $\alpha_1=1$).

The differential on the next page creates one aerial vertex and increases the number of vertices by one.
There are two such terms: One arises from the splitting of vertices, and one arises from the part of $z_0^\p$ containing only one vertex, which is of aerial.
Also recall that at this stage our complex consists of directed graphs, with only aerial vertices, which may be decorated by $\alpha_j$, $\gamma_j$ or $\tilde \gamma_j$, modulo $\alpha$-stubs.

To compute further we filter by the number of vertices of valence $\geq 3$.
The first differential on the spectral sequence will create one at most bivalent vertex.
As in the proof of Proposition \ref{prop:ozinvariance}, one can see that the cohomology of this differential can be identified with a sum of two spaces: (i) 1-loop graphs of undirected bivalent vertices and (ii) undirected graphs all of whose vertices are of valence $\geq 3$, with decorations in $\tilde H(M)$.
We hence can use the degree counting argument of Lemma \ref{lem:GCdegbounds} and below to conclude the following.

\begin{proposition}\label{prop:HKGCM}
  Suppose that $\dim M\geq 5$ and $H^1(M)=H^1(\p M)=0$. Then the cohomology of the higher loop part
  \begin{equation}
    \mF^1 \KGC_M
  \end{equation}
  is concentrated in non-positive degrees.
  \qed
\end{proposition}

\printindex

\printbibliography

\end{document}